\documentclass[10pt, a4paper]{amsart}
\pdfoutput=1
\usepackage{hyperref}
\hypersetup{colorlinks=false}
\usepackage{graphicx}
\usepackage{amscd}
\usepackage{amssymb}
\usepackage[all]{xy}
\usepackage[T1]{fontenc}
\usepackage[sc]{mathpazo}
\makeindex 

\author{Amnon Yekutieli}
\title[Nonabelian Multiplicative Integration]{Nonabelian 
Multiplicative Integration on Surfaces}
\address{Department of Mathematics, Ben Gurion University,
Be'er Sheva 84105, Israel}
\email{amyekut@math.bgu.ac.il}
\date{6 March 2015}
\thanks{
Primary: 22E30; Secondary: 22E60, 57Q15, 57R19, 58A10, 53C08.} 
\thanks{{\em Funding}: This research was supported by the Israel Science 
Foundation.\\
\indent {\em Publication}: To appear as a book with World Scientific.}

\newtheorem{thm}[equation]{Theorem}
\newtheorem{cor}[equation]{Corollary}
\newtheorem{prop}[equation]{Proposition}
\newtheorem{lem}[equation]{Lemma}
\theoremstyle{definition}
\newtheorem{dfn}[equation]{Definition}
\newtheorem{rem}[equation]{Remark}
\newtheorem{exa}[equation]{Example}

\newtheorem{conj}[equation]{Conjecture}

\newtheorem{conv}[equation]{Convention}
\numberwithin{equation}{section}
\newcommand{\iso}{\xrightarrow{\simeq}}

\newcommand{\xar}{\xrightarrow}
\newcommand{\opn}{\operatorname}

\newcommand{\bdot}{\bsym{\cdot}}

\newcommand{\ol}{\overline}

\newcommand{\rmitem}[1]{\item[\text{\textup{(#1)}}]}
\newcommand{\mfrak}[1]{\mathfrak{#1}}
\newcommand{\mcal}[1]{\mathcal{#1}}

\newcommand{\mbf}[1]{\mathbf{#1}}
\newcommand{\mrm}[1]{\mathrm{#1}}
\newcommand{\mbb}[1]{\mathbb{#1}}

\newcommand{\smfrac}[2]{{\textstyle \frac{#1}{#2}}}
\newcommand{\tup}[1]{\textup{#1}}
\newcommand{\bsym}[1]{\boldsymbol{#1}}
\newcommand{\boplus}{\bigoplus\nolimits}

\newcommand{\bosum}{\sum\nolimits}
\newcommand{\boprod}{\prod\nolimits}

\newcommand{\til}[1]{\tilde{#1}}

\newcommand{\norm}[1]{\lVert #1 \rVert}
\newcommand{\Norm}[1]{\bigl\lVert\, #1 \, \bigr\rVert}
\newcommand{\abs}[1]{\lvert #1 \rvert}
\renewcommand{\vert}{\, | \,}

\newcommand{\K}{\mbb{K}}
\newcommand{\R}{\mbb{R}}
\newcommand{\Z}{\mbb{Z}}
\newcommand{\N}{\mbb{N}}
\newcommand{\g}{\mfrak{g}}
\newcommand{\h}{\mfrak{h}}

\newcommand{\f}{\mfrak{f}}
\newcommand{\bwedge}{{\textstyle \bigwedge}}

\renewcommand{\d}{\mrm{d}}

\newcommand{\lb}{\linebreak}

\begin{document}

%
%
%
%

\begin{abstract}
We construct a {\em $2$-dimensional twisted nonabelian multiplicative 
integral}. This is done in the context of a {\em Lie crossed module}
(an object composed of two Lie groups interacting), and a {\em pointed
manifold}. The integrand is a {\em connection-curvature pair}, that consists of
a Lie algebra valued $1$-form and a Lie algebra valued $2$-form, satisfying a
certain differential equation. The geometric cycle of the integration is a {\em
kite} in the pointed manifold. A kite is made up of a $2$-dimensional simplex
in the manifold, together with a path connecting this simplex to the base
point of the manifold. The multiplicative integral is an element of the second
Lie group in the crossed module. 

We prove several properties of the multiplicative integral. Among them is the
{\em $2$-dimensional nonabelian Stokes Theorem}, which is a generalization of
Schlesinger's Theorem. Our main result is the {\em $3$-dimensional nonabelian
Stokes Theorem}. This is a totally new result. 

The methods we use are: the CBH Theorem for the nonabelian exponential map;
piecewise smooth geometry of polyhedra; and some basic algebraic topology. 

The motivation for this work comes from {\em twisted deformation quantization}
and {\em descent for nonabelian gerbes}. 
Similar questions arise in {\em nonabelian gauge theory}.
\end{abstract}

\maketitle

\tableofcontents

\cleardoublepage


\setcounter{page}{1}
\renewcommand\thepage{\arabic{page}} 
\cleardoublepage
\setcounter{section}{-1}
\section{Introduction}
\numberwithin{equation}{subsection}

\subsection{}
In this paper we establish a theory of {\em twisted nonabelian multiplicative
integration of $2$-forms on surfaces}. 

Let us begin the exposition with a discussion of 
{\em $1$-dimensional nonabelian multiplicative integration}. 
This goes back to work of Volterra in the 19-th
century, and is also known by the names ``path ordered exponential integration''
and ``holonomy of a connection along a path''. See the book \cite{DF} and
the paper \cite{KMR} for history and various properties. 

In our setup the $1$-dimensional multiplicative integral looks like this. 
Let $X$ be an $n$-dimensional manifold (by which we mean a differentiable
manifold with corners), and let $G$ be a Lie group with Lie algebra $\g$, 
all over the field  $\R$. We denote by
\[ \Omega(X) = \boplus_{p = 0}^n \ \Omega^p(X) \]
the de Rham algebra of $X$ (i.e.\ the algebra of smooth differential forms).
By smooth path in $X$ we mean a smooth map 
$\sigma: \bsym{\Delta}^1 \to X$, where $\bsym{\Delta}^1$ is the $1$-dimensional
simplex. Let
\[ \alpha \in  \Omega^1(X) \otimes \g . \]
The {\em multiplicative integral of $\alpha$ on $\sigma$} is an element
\begin{equation} \label{eqn:163}
\opn{MI}(\alpha \vert \sigma) \in G , 
\end{equation}
obtained as the limit of {\em Riemann products}.
This operation is re-worked and extended in Section \ref{sec:dim1} of our
paper.

\subsection{}
For reasons explained in Subsection \ref{subsec:motiv} of the Introduction, we
found it necessary to devise a theory of {\em $2$-dimensional nonabelian
multiplicative integration}. Our work was guided by the problem at
hand, plus ideas borrowed from the papers \cite{BM, BS, Ko}. 

Instead of a single Lie group, the $2$-dimensional operation involves a pair
$(G, H)$ of Lie groups, with a certain interaction between them. 
This structure is called a {\em Lie crossed module}.
A Lie crossed module
\begin{equation} \label{eqn:230}
\mbf{C} = (G, H, \Psi, \Phi)
\end{equation}
consists, in addition to the Lie groups $H$ and $G$, of an analytic action
$\Psi$ of $G$ on $H$ by automorphisms of Lie groups,
which we call the {\em twisting}; and of a map of Lie groups
$\Phi : H \to G$, called the {\em feedback}. The conditions are that $\Phi$ is
$G$-equivariant (with respect to $\Psi$ and to the conjugation action
$\opn{Ad}_G$ of $G$ on itself); and 
\begin{equation} \label{eqn:165}
\Psi \circ \Phi = \opn{Ad}_H .
\end{equation} 
See  \cite{BM, BS}.
The integrand is now a pair $(\alpha, \beta)$, with 
\begin{equation} \label{eqn:162}
\alpha \in \Omega^1(X) \otimes \g \ \text{ and } \
\beta \in  \Omega^2(X) \otimes \h . 
\end{equation}
Here $\h$ is the Lie algebra of $H$.

Let $x_0$ be a point in $X$, so the pair $(X, x_0)$ is a pointed manifold.
The geometric data (the cycle) for the multiplicative integration is a {\em
kite} $(\sigma, \tau)$ in $(X, x_0)$. 
Let us denote by $v_0, \ldots, v_p$ the vertices of the 
$p$-dimensional simplex $\bsym{\Delta}^p$. By definition a {\em smooth
triangular kite} $(\sigma, \tau)$ in $(X, x_0)$ consists of smooth maps
$\sigma : \bsym{\Delta}^1 \to X$ and
$\tau : \bsym{\Delta}^2 \to X$, such that
$\sigma(v_0) = x_0$ and
$\sigma(v_1) = \tau(v_0)$. 
See Figure \ref{fig:80} for an illustration.

\begin{figure} 
\includegraphics[scale=0.27]{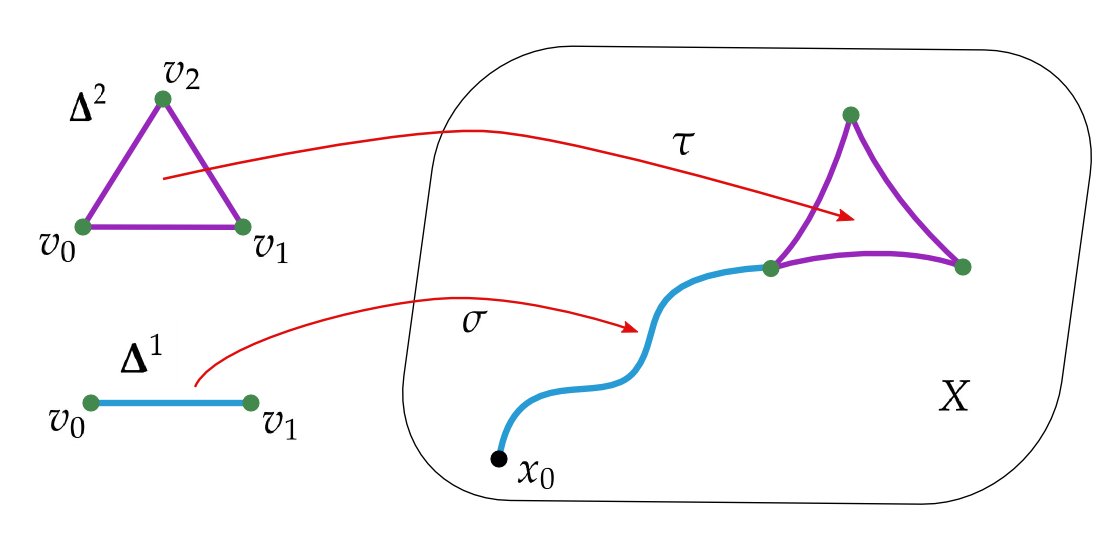}
\caption{A smooth triangular kite $(\sigma, \tau)$ in the pointed manifold 
$(X, x_0)$.} 
\label{fig:80}
\end{figure}

The next theorem summarizes our construction. It encapsulates many results
scattered throughout the paper.

\begin{thm}[Existence of MI on Triangles] \label{thm:18}
Let 
$\mbf{C} = (G, H, \Psi, \Phi)$
be a Lie crossed module, $(X, x_0)$ a pointed manifold, and
$(\alpha, \beta)$ a pair of differential forms as in \tup{(\ref{eqn:162})}.
Given any smooth triangular kite
$(\sigma, \tau)$ in $(X, x_0)$, there is an element
\[ \opn{MI} (\alpha, \beta \vert \sigma, \tau) \in H  \]
called the {\em twisted multiplicative integral of $(\alpha, \beta)$ on 
$(\sigma, \tau)$}.

The operation $\opn{MI}(-,-)$ enjoys these properties:
\begin{enumerate}
\rmitem{a} The element 
$\opn{MI} (\alpha, \beta \vert \sigma, \tau)$
has an explicit formula as the limit of Riemann products.
\rmitem{b} The operation $\opn{MI}(-,-)$ is functorial in 
$\mbf{C}$ and $(X, x_0)$.
\rmitem{c} If $H$ is abelian and $G$ is trivial, then 
\[ \opn{MI} (\alpha, \beta \vert \sigma, \tau) = 
\opn{exp}_H  \bigl( \int_{\tau} \beta \bigr) . \]
\end{enumerate}
\end{thm}

More on the construction in Subsections 
\ref{subsec:methods}-\ref{subsec:binary} and \ref{subsec:triangles} of the
Introduction.

As far as we know, this is the first construction of a nonabelian
multiplicative integration on surfaces of such generality. 
The very special case $G = H = \opn{GL}_m(\R)$ was done by Schlesinger in the
1920's; cf.\ \cite{DF, KMR}. 

\subsection{} \label{subsec:actually}
Actually, in the body of the paper we work in a much more complicated situation.
Instead of a Lie crossed module (\ref{eqn:230}), we work with a {\em Lie
quasi crossed module with additive feedback} (See Section \ref{sec:LQC}). 
The reason for the more complicated setup is that this is what was 
required, at the time, for our work on {\em twisted deformation quantization}
in the paper \cite{Ye3}. More on this in Subsection \ref{subsec:motiv} below.
In the Introduction we stick to the simpler
setup of a Lie crossed module, which is interesting enough. Note
however that all the results mentioned in the Introduction are valid also in the
more complicated setup.

\subsection{}
In this subsection we explain the nonabelian $2$-dimensional Stokes Theorem.
For this to hold it is necessary to
impose a condition on the integrand $(\alpha, \beta)$. 
Recall the feedback $\Phi : H \to G$. Consider the induced Lie algebra
homomorphism
\[ \opn{Lie}(\Phi) : \h \to \g . \]
By tensoring this induces a homomorphism of differential graded Lie algebras
\[ \phi : \Omega(X) \otimes \h \to \Omega(X) \otimes \g . \]
We say that $(\alpha, \beta)$ is a {\em connection-curvature pair}
for $\mbf{C} / X$ if
\begin{equation}  \label{eqn:164}
\phi(\beta) = \d(\alpha) + \smfrac{1}{2} [\alpha, \alpha]   
\end{equation}
in $\Omega^2(X) \otimes \g$.
(In \cite{BM} this condition is called the {\em vanishing of the fake
curvature}.) 

The boundary of a triangular kite $(\sigma, \tau)$ is the closed path 
$\partial (\sigma, \tau)$ depicted in Figure \ref{fig:81}.

\begin{figure} 
\includegraphics[scale=0.30]{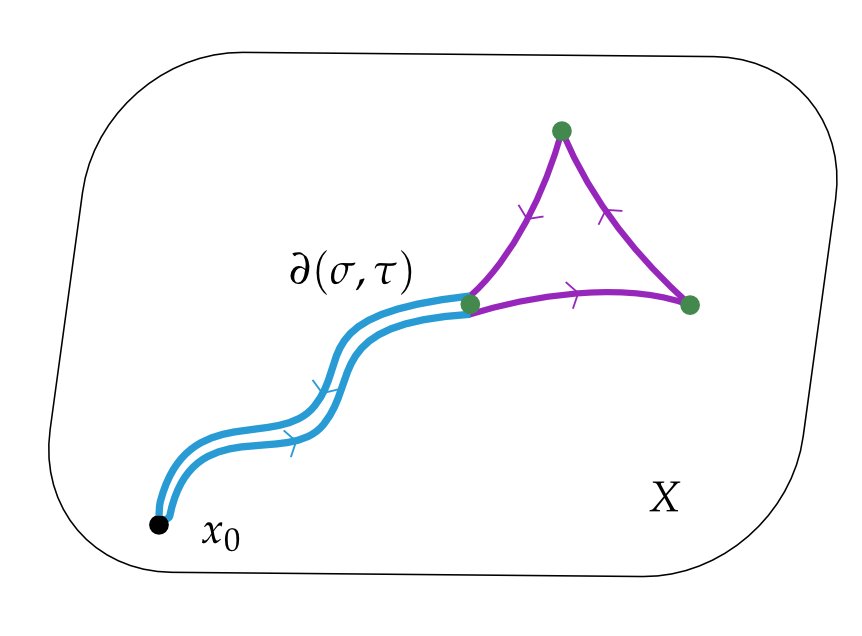}
\caption{The boundary $\partial(\sigma, \tau)$ of the
kite $(\sigma, \tau)$  from Figure \ref{fig:80}.} 
\label{fig:81}
\end{figure}

\begin{thm}[Stokes Theorem for Triangles] \label{thm:19}
Let $\mbf{C} = (G, H, \Psi, \Phi)$
be a Lie crossed module, $(X, x_0)$ a pointed manifold, and
$(\alpha, \beta)$ a connection-curvature pair for $\mbf{C} / X$.
Given any smooth triangular kite 
$(\sigma, \tau)$ in $(X, x_0)$, one has
\[ \Phi \bigl( \opn{MI} (\alpha, \beta \vert \sigma, \tau) \bigr)
= \opn{MI} (\alpha \vert \partial(\sigma, \tau)) \]
in $G$.
\end{thm}

The special case $G = H = \opn{GL}_m(\R)$ is Schlesinger's Theorem 
(see \cite{DF, KMR}).

\subsection{}
We now approach the main result of the paper, namely the nonabelian
$3$-dimensional Stokes Theorem. 

A {\em smooth triangular balloon} in $(X, x_0)$ is a pair $(\sigma, \tau)$,
consisting of smooth maps
$\sigma : \bsym{\Delta}^1 \to X$ and
$\tau : \bsym{\Delta}^3 \to X$, such that
$\sigma(v_0) = x_0$ and
$\sigma(v_1) = \tau(v_0)$. 
See Figure \ref{fig:82} for an illustration.

\begin{figure} 
\includegraphics[scale=0.27]{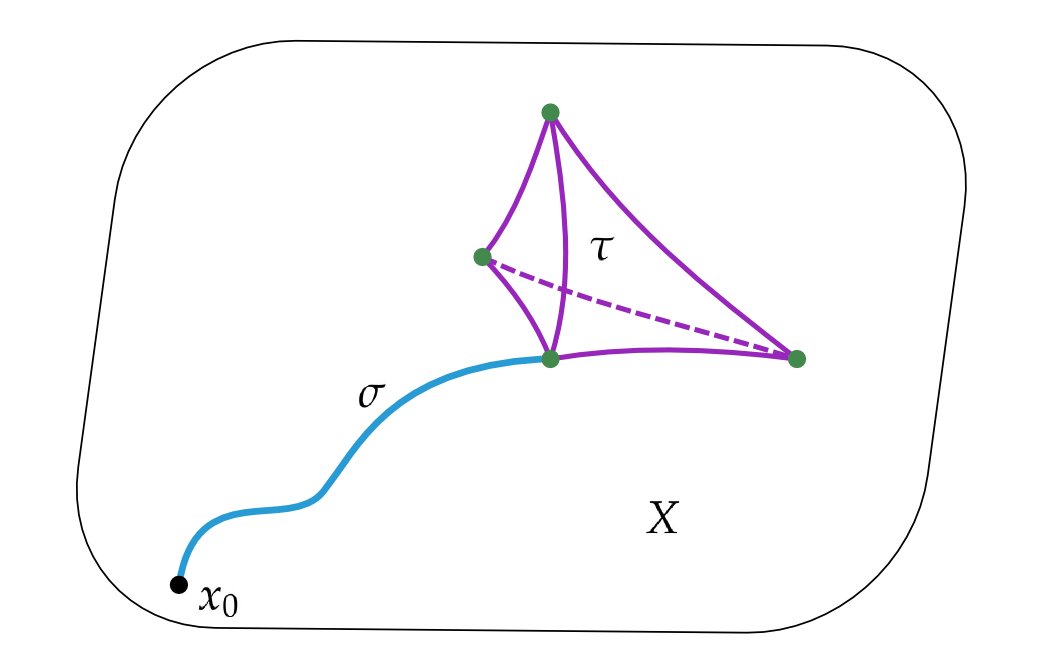}
\caption{A smooth triangular balloon $(\sigma, \tau)$ in the pointed manifold 
$(X, x_0)$.} 
\label{fig:82}
\end{figure}

The {\em boundary} of a balloon $(\sigma, \tau)$ is a sequence
\[ \partial (\sigma, \tau) =
\bigl( \partial_1 (\sigma, \tau),\, \partial_2 (\sigma, \tau),\, 
\partial_3 (\sigma, \tau),\, \partial_4 (\sigma, \tau) \bigr) \]
of triangular kites. See Figure \ref{fig:86}. We write
\begin{equation} \label{eqn:166}
\opn{MI}(\alpha, \beta \vert \partial (\sigma, \tau)) :=
\prod_{1 = 1}^4 \ \opn{MI}(\alpha, \beta \vert \partial_i (\sigma, \tau)) , 
\end{equation}
where the order of the product is left to right.

\begin{figure} 
\includegraphics[scale=0.36]{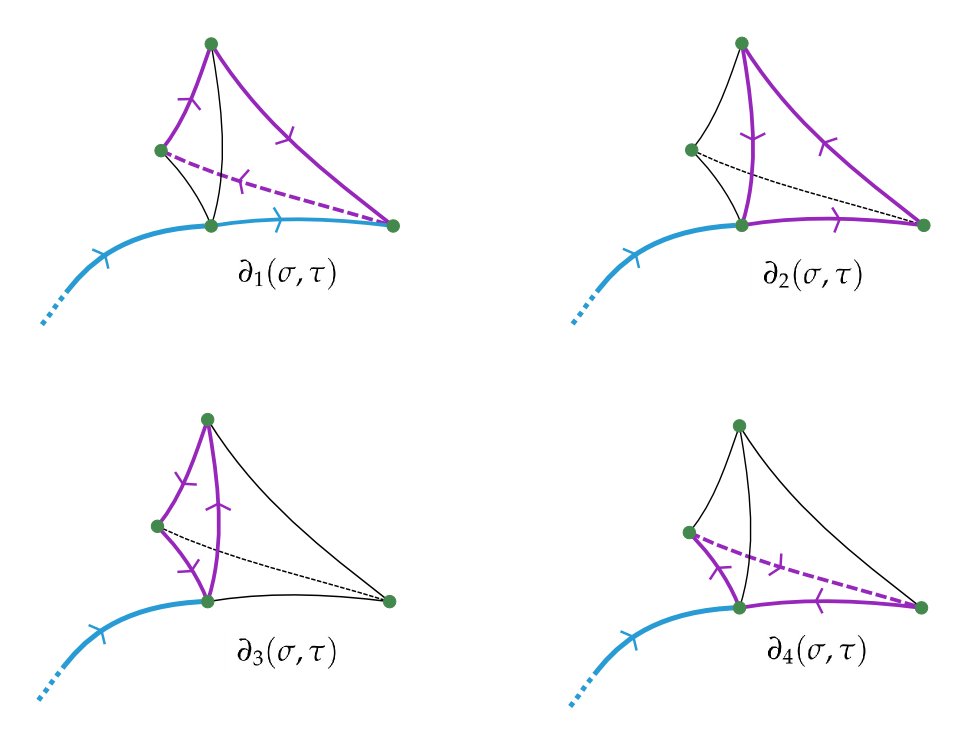}
\caption{The four triangular kites that make up the sequence 
$\partial (\sigma, \tau)$. Here $(\sigma, \tau)$ is the triangular balloon
from Figure \ref{fig:82}.} 
\label{fig:86}
\end{figure}

Let $H_0 := \opn{Ker}(\Phi)$, which is a closed Lie subgroup of $H$. We call it
the {\em inertia subgroup}, and 
$\h_0 := \opn{Lie}(H_0)$ is called the {\em inertia subalgebra}.
Note that $H_0$ is contained in the center of $H$, so $H_0$ is an abelian Lie
group.  A form 
\[ \gamma \in \Omega^p(X) \otimes \h_0  \]
is called an {\em inert $p$-form}.

We say that a form 
$\alpha \in \Omega^1(X) \otimes \g$
is {\em tame connection} if it belongs to some 
connection-curvature pair $(\alpha, \beta)$.
We show (in Subsection \ref{subsec:inert}) that given a tame connection 
$\alpha$ and an inert $3$-form $\gamma$, there is a well defined
{\em twisted abelian multiplicative integral}
\[ \opn{MI}(\alpha, \gamma \vert \sigma, \tau) \in H_0 . \]

For $g \in G$ there is a Lie group automorphism
$\Psi(g) : H \to H$. Its derivative is
\[ \Psi_{\h}(g) := \opn{Lie}(\Psi(g)) : \h \to \h . \]
As $g$ varies we get a map of Lie groups
\begin{equation} \label{eqn:231}
\Psi_{\h} : G \to \opn{GL}(\h) . 
\end{equation}
The derivative of $\Psi_{\h}$ is a map of Lie algebras
\[ \opn{Lie}(\Psi_{\h}) : \g \to  
\mfrak{gl}(\h) = \opn{End}(\h) . \]
This extends by tensor product to a homomorphism of differential graded Lie
algebras
\[ \psi_{\h} :  \Omega(X) \otimes \g \to \Omega(X) \otimes \opn{End}(\h) . \]
Given a connection-curvature $(\alpha, \beta)$, its {\em $3$-curvature}
is the form
\begin{equation} \label{eqn:167}
\gamma := \d(\beta) + \psi_{\h}(\alpha)(\beta) \in 
\Omega^3 \otimes \h .
\end{equation}
(This is the name given in \cite{BM}.)

We can now state the main result of the paper. (See Theorem \ref{thm:22} for the
full version.)

\begin{thm}[Stokes Theorem for Tetrahedra] \label{thm:20}
Let $\mbf{C} = (G, H, \Psi, \Phi)$
be a Lie crossed module, $(X, x_0)$ a pointed manifold, and
$(\alpha, \beta)$ a connection-curvature pair for $\mbf{C} / X$.
Let $\gamma$ be the $3$-curvature of $(\alpha, \beta)$. Then:
\begin{enumerate}
\item The form $\gamma$ is inert.
\item Given any smooth triangular balloon 
$(\sigma, \tau)$ in $(X, x_0)$, one has
\[ \opn{MI} (\alpha, \beta \vert \partial(\sigma, \tau)) =
\opn{MI} (\alpha, \gamma \vert \sigma, \tau) \]
in $H$.
\end{enumerate}
\end{thm}

In this generality, Theorem \ref{thm:20} appears to have no precursor
in the literature. The special (but most important) case, namely
$\gamma = 0$, was taken to be true in some papers (e.g.\ \cite{Ko, BS}), but
no proofs were provided there. 

Assertion (1) in the theorem can be viewed as a {\em generalized Bianchi
identity}.

\subsection{} \label{subsec:methods}
In this and the next few subsections we will explain some of the work
leading to the theorems mentioned above.

The methods used in the constructions and proofs are of two kinds: geometric
and infinitesimal. The geometric methods consist of dividing kites into smaller
ones, and studying the effect on the corresponding approximations (Riemann
products). The infinitesimal methods involve Taylor expansions and estimates
for the nonabelian exponential map. 

Throughout most of the paper we work with {\em linear quadrangular kites}
in {\em polyhedra}, and with {\em piecewise smooth differential forms},
rather than with smooth triangular kites in manifolds and smooth differential
forms. The reason for choosing quadrangular kites is mainly that it is much
easier to carry out calculations of Taylor expansions on squares, as compared to
triangles. Also the combinatorics of quadrangular kites and their binary
subdivisions is simpler than that of triangles. See Subsections 
\ref{subsec:kites}-\ref{subsec:bintes}. A review of the piecewise linear
geometry of poyhedra, and of piecewise smooth differential forms, can be found
in Section \ref{sec:pws}.

The key technical result that allows us to calculate nonabelian products is 
Theorem \ref{thm:6} on estimates for the nonabelian exponential map. These
estimates are gotten from the CBH formula. Section \ref{sec:expon} is devoted
to the proof of these estimates.

For heuristic purposes we introduce the concept of ``tiny scale''. By tiny scale
(depending on context of course) we mean
geometric or algebraic elements that are so small that the relevant estimates
(arising from the CBH formula or Taylor expansion) apply to them. See Remark
\ref{rem:3}.

\subsection{}
For $n \geq 0$ let $\mbf{I}^n$ be the $n$-dimensional cube. We give names 
to some vertices:
$v_0 := (0, \ldots, 0)$, and
\[ v_i := (0, \ldots, 1, \ldots, 0) \]
with $1$ in the $i$-th position, for $i = 1, \ldots, n$.
The base point for $\mbf{I}^n$ is $v_0$. 

Let $(X, x_0)$ be a pointed polyhedron. By {\em linear string} in 
$X$ we mean a sequence
$\sigma = (\sigma_1, \ldots, \sigma_m)$ of 
piecewise linear maps $\sigma_i : \mbf{I}^1 \to X$,
such that $\sigma_i(v_1) = \sigma_{i+1}(v_0)$. 
The maps $\sigma_i$ are called the pieces of $\sigma$. 
In Section \ref{sec:dim1} we construct the nonabelian multiplicative 
integral $\opn{MI}(\alpha \vert \sigma)$
of a piecewise smooth differential form
$\alpha \in \Omega^1_{\mrm{pws}}(X) \otimes \g$
on a string $\sigma$. This construction is quite simple and essentially the
same as the classical one. We also prove a few basic properties of this
MI.

\subsection{} \label{subsec:binary}
In Section \ref{sec.MI} we construct the $2$-dimensional
nonabelian multiplicative integral
$\opn{MI}(\alpha, \beta \vert \sigma, \tau)$. 
Here $(\sigma, \tau)$ is a linear quadrangular kite in the pointed polyhedron
$(X, x_0)$. By definition this means that $\sigma$ is a string in $X$, and
$\tau : \mbf{I}^2 \to X$ is a linear map. The conditions are that the initial
point of $\sigma$ is $x_0$, and the terminal point of $\sigma$ is 
$\tau(v_0)$. The integrand $(\alpha, \beta)$ consists of piecewise smooth
differential forms:
\[ \alpha \in \Omega^1_{\mrm{pws}}(X) \otimes \g \ \text{ and } \
\beta \in  \Omega^2_{\mrm{pws}}(X) \otimes \h . \]

The coarse approximation of the $2$-dimensional
nonabelian multiplicative integral is as follows. 
Given a kite  $(\sigma, \tau)$ in $(X, x_0)$, let
\[ g := \opn{MI}(\alpha \vert \sigma) \in G ; \]
so applying the operator $\Psi_{\h}(g)$ from (\ref{eqn:231})
we have a new (twisted) $2$-form
\[ \Psi_{\h}(g)(\beta) \in \Omega_{\mrm{pws}}^2(X) \otimes \h . \]
We define the {\em basic Riemann product} 
of $(\alpha, \beta)$ on $(\sigma, \tau)$ to be
\begin{equation} \label{eqn:238}
\opn{RP}_0(\alpha, \beta \vert \sigma, \tau) :=
 \exp_H \Bigl( \int_{\tau} \Psi_{\h}(g)(\beta) \Bigr) \in H .
\end{equation}
(Actually in the body of the paper we use another formula for \lb
$\opn{RP}_0(\alpha, \beta \vert \sigma, \tau)$, that converges
faster; see Definition \ref{dfn:11} and Remark \ref{rem:4}.)

For the limiting process we introduce the {\em binary
subdivisions} of $\mbf{I}^2$. The $k$-th binary subdivision is the cellular
decomposition of $\mbf{I}^2$ into $4^k$ equal squares, and we denote it by
$\opn{sd}^k \mbf{I}^2$. The $1$-skeleton of $\opn{sd}^k \mbf{I}^2$ is denoted
by $\opn{sk}_1 \opn{sd}^k \mbf{I}^2$. Its fundamental group (based at $v_0$) is
denoted by
$\bsym{\pi}_1(\opn{sk}_1 \opn{sd}^k \mbf{I}^2)$. 
It is very important that the group 
$\bsym{\pi}_1(\opn{sk}_1 \opn{sd}^k \mbf{I}^2)$ is a  free group on $4^k$
generators. 
We say that a kite $(\sigma, \tau)$ is {\em patterned on $\opn{sd}^k \mbf{I}^2$}
if for every piece $\sigma_i$ of the string $\sigma$, the image
$\sigma_i(\mbf{I}^1)$ is a $1$-cell of $\opn{sd}^k \mbf{I}^2$, and
$\tau(\mbf{I}^2)$ is a $2$-cell of $\opn{sd}^k \mbf{I}^2$.

A {\em tessellation} of $\mbf{I}^2$ is by definition a sequence
$\bigl( (\sigma_i, \tau_i) \bigr)_{i = 1, \ldots, 4^k}$
of square kites in $(\mbf{I}^2, v_0)$, each patterned on 
$\opn{sd}^k \mbf{I}^2$, satisfying the following topological condition.
Let us denote by $[\partial (\sigma_i, \tau_i)]$ the element of
$\bsym{\pi}_1(\opn{sk}_1 \opn{sd}^k \mbf{I}^2)$
represented by the closed string $\partial (\sigma_i, \tau_i)$. Then
\begin{equation} \label{eqn:168}
\prod_{1 = 1}^{4^k}\ [\partial (\sigma_i, \tau_i)] = 
[\partial \mbf{I}^2] 
\end{equation}
in $\bsym{\pi}_1(\opn{sk}_1 \opn{sd}^k \mbf{I}^2)$.

For the construction we choose a particular tessellation for every $k$. It is
called the {\em $k$-th binary tessellation}, and the notation is
$\opn{tes}^k  \mbf{I}^2$. The actual definition of $\opn{tes}^k  \mbf{I}^2$
is not important (since it is
quite arbitrary, and chosen for convenience). All that is important are its two
properties:
\begin{itemize}
\item It is a tessellation; i.e.\ equation (\ref{eqn:168}) holds.
\item It has a recursive (self similar) nature.
\end{itemize}

For fixed $k$ and a kite $(\sigma, \tau)$ we obtain, by a simple geometric
operation, the $k$-binary tessellation 
\[ \opn{tes}^k (\sigma, \tau) = 
\bigl( \opn{tes}^k_i (\sigma, \tau) \bigr)_{i = 1, \ldots, 4^k} \]
of $(\sigma, \tau)$, which is a sequence of kites in $(X, x_0)$. 
(See Definition \ref{dfn:44}.)

For every $i \in \{ 1, \ldots, 4^k \}$ we have the basic Riemann product
\[ \opn{RP}_0(\alpha, \beta \vert \opn{tes}^k_i (\sigma, \tau)) \]
on the kite $\opn{tes}^k_i (\sigma, \tau)$. 
We define the {\em $k$-th Riemann product} to be 
\begin{equation} \label{eqn:170}
\opn{RP}_k(\alpha, \beta \vert \sigma, \tau) :=
\prod_{1 = 1}^{4^k}\ 
\opn{RP}_0(\alpha, \beta \vert \opn{tes}^k_i (\sigma, \tau)) .
\end{equation}

In Theorem \ref{thm:12} we prove that the limit
\[ \lim_{k \to \infty} \opn{RP}_k(\alpha, \beta \vert \sigma, \tau) \]
exists in $H$. The proof goes like this: for sufficiently large $k$ the 
kites $\opn{tes}^k_i (\sigma, \tau)$ are all tiny. We use estimates 
to show that the limits
\[ \lim_{k' \to \infty} \opn{RP}_{k'}(\alpha, \beta \vert 
\opn{tes}^k_i (\sigma, \tau) ) \]
exist for every $i$. Due to the recursive nature of the binary tessellations,
this is enough.  We can finally define
\begin{equation} 
\opn{MI}(\alpha, \beta \vert \sigma, \tau) := 
\lim_{k \to \infty} \opn{RP}_k(\alpha, \beta \vert \sigma, \tau) . 
\end{equation}

\subsection{}
In Section \ref{sec:stokes2} we prove the $2$-dimensional Stokes Theorem.
It is the same as Theorem \ref{thm:19}, except that it talks about 
piecewise smooth forms on a pointed polyhedron and linear quadrangular kites. 
Again the proof is by reduction (using the recursive nature of the binary
tessellations) to the case of a tiny kite. And then we use approximations 
(both of Taylor expansions and the exponential map) to do the
calculation. 

A very important technical consequence of the $2$-dimensional Stokes Theorem
is:

\begin{thm}[Fundamental Relation] \label{thm:21}
Let $\mbf{C} = (G, H, \Psi, \Phi)$
be a Lie crossed module, $(X, x_0)$ a pointed polyhedron, 
$(\alpha, \beta)$ a piecewise smooth \lb connection-curvature pair for 
$\mbf{C} / X$, and $(\sigma, \tau)$ a linear quadrangular kite in $(X, x_0)$.
Consider the elements
\[ g := \opn{MI} \bigl( \alpha \vert \partial (\sigma, \tau) \bigr) \in G \]
and
\[ h :=  \opn{MI} (\alpha, \beta \vert \sigma, \tau) \in H . \]
Then
\[ \Psi(g) = \opn{Ad}_H(h) \]
as automorphisms of the Lie group $H$. 
\end{thm}

In Section \ref{sec:stokes-3} we prove the first version of the main 
result of the paper,
namely the $3$-dimensional Stokes Theorem (Theorem \ref{thm:10}). 
This is like Theorem \ref{thm:20}, only for piecewise smooth
connection-curvature pairs, and for linear quadrangular balloons. 

The strategy of the proof is this. Using Theorem \ref{thm:21}, and a lot of
combinatorics of free groups (done in Section \ref{sec:puzzles}), we reduce the
problem to the case of a tiny balloon. And for a tiny balloon we use Taylor
expansions and the estimates from Theorem \ref{thm:6}.

\subsection{} \label{subsec:triangles}
Finally, in Section \ref{sec:simpl}, we show how to pass from quadrangular
kites  to triangular ones. 
This is very easy, using the following trick. The triangle 
$\bsym{\Delta}^2$ is naturally embedded in the square $\mbf{I}^2$; and 
there is a piecewise linear retraction 
$h : \mbf{I}^2 \to \bsym{\Delta}^2$. 
Consider the ``universal triangular kite'' 
$(\sigma', \tau')$ , and the ``universal quadrangular kite''
$(\sigma', \tau'')$, both 
in $(\mbf{I}^2, v_0)$; see Figures \ref{fig:55}, \ref{fig:84} and \ref{fig:85}.
There is a piecewise linear retraction 
\[ g : \mbf{I}^2 \to \sigma'(\mbf{I}^1) \cup \tau'(\mbf{I}^2) . \]
These maps are
related by 
\[ \tau' \circ h = g \circ \tau'' , \]
as maps $\mbf{I}^2 \to \mbf{I}^2$.

Given a smooth kite $(\sigma, \tau)$ in a pointed manifold
$(X, x_0)$, there a piecewise smooth map
$f :  \mbf{I}^2 \to X$
such that 
$f \circ \sigma' = \sigma$ and
$f \circ \tau' = \tau$. Given a smooth connection-curvature pair
$(\alpha, \beta)$ for $\mbf{C} / X$, the pair
\[ (\alpha', \beta') := (f^*(\alpha), f^*(\beta)) \]
is a piecewise smooth connection-curvature pair
$(\alpha, \beta)$ for $\mbf{C} / \mbf{I}^2$.
We define
\begin{equation} \label{eqn:173}
\opn{MI}(\alpha, \beta \vert \sigma, \tau) :=
\opn{MI}(\alpha', \beta' \vert \sigma', \tau'') . 
\end{equation}

The results proved for quadrangular kites can be easily transferred to
triangular kites by similar geometric tricks.

\subsection{} \label{subsec:motiv}
The reason we became interested in nonabelian multiplicative integration is its
application to twisted deformation quantization of algebraic
varieties, as developed in our paper \cite{Ye3} (see also the survey article
\cite{Ye4}). Twisted deformations of an algebraic variety $X$ are very close to
gerbes and to stacks of algebroids (in the sense of \cite{Ko}). 
Our original strategy was to classify these twisted deformations in
terms of Maurer-Cartan solutions in a DG Lie algebra gotten as the commutative
\v{C}ech resolution of a sheaf of DG Lie algebras on $X$. The multiplicative
integration was supposed to be a key component of this construction. Some of the
details are explained in Subsections \ref{subsec:QTDGLie}, \ref{subsec:CosDGLie}
and  \ref{subsec:conj.desc}.

In the meanwhile we discovered a more direct approach to the classification of
twisted deformations, using descent data in cosimplicial crossed groupoids, and
a related invariance result (see Sections 5-6 in the latest version of
\cite{Ye3}, and Theorem 0.1 in \cite{Ye6}). 

In the version of \cite{Ye3} dated 26 August 2009 there was a ``result''
labeled Theorem 11.2, about cosimplicial quantum type DG Lie algebras,
Deligne crossed groupoids and descent data. 
The proof of this ``result'' was only sketched, and we never completed the
proof (it was supposed to have appeared in a separate paper). It is now
presented as Conjecture \ref{conj:521} at the end of this paper. 

Speaking of conjectures, we should mention Conjecture \ref{conj:520} regarding
the rationality of the $2$-dimensional multiplicative integral when the Lie
groups $G$ and $H$ are unipotent algebraic groups. 

Despite the fact that the original motivation for this paper no longer exists,
we believe that the constructions and results presented here have 
independent value. This can be seen for instance by the role our work has in
the recent paper \cite{BGNT}. See the next subsection for a discussion of 
\cite{BGNT} and other related papers.

\subsection{} \label{subsec:compare}
We end the introduction with a discussion of recent related work. 
Our $2$-dimensional nonabelian multiplicative integral is closely related to
the {\em differential geometry of gerbes}, as developed by Breen and Messing
\cite{BM}. Indeed, the notion of fake curvature, that goes into our definition
of connection-curvature pair, and the notion of $3$-curvature, are both taken
from \cite{BM}. 

There is a series of papers by Baez, Schreiber and Waldorf (\cite{BS},
\cite{SW1}, \cite{SW2} and others) on {\em nonabelian gauge theory}. They
develop a theory of nonabelian curvature that seems to be very similar to our 
$2$-dimensional nonabelian multiplicative integral, even though the flavor of
their work is very different (it is much more abstract). It seems that
\cite{SW2} contains a proof the $3$-dimensional Stokes Theorem in the special
case where the $3$-curvature is zero. 

Very recently the paper \cite{BGNT} by Bressler, Gorokhovsky, Nest and Tsygan
came out. The goal of this paper is to prove that several existing notions of
the nerve of a $2$-groupoid are equivalent (as simplicial sets). 
The comparison between the Hinich simplicial set and the Deligne-Getzler
simplicial set is done using  $2$-dimensional nonabelian multiplicative
integration. The method seems to be quite similar, at least morally, to our
original program (see previous subsection and Conjecture \ref{conj:521}). 
The $2$-dimensional nonabelian multiplicative integral that is developed in
Sections 4-5 of \cite{BGNT} is inspired by our work, and they obtain 
a version of our $3$-dimensional Stokes Theorem in their setting (see 
\cite[Theorems 5.1 and 5.2]{BGNT}). 

The difference between the construction in \cite{BGNT} and ours is that they
work in the special case of a Lie crossed module  $(G, H, \Psi, \Phi)$
consisting of unipotent algebraic groups $G$ and $H$, and an algebraic
connection-curvature pair $(\alpha, \beta)$ with vanishing $3$-curvature. The
construction goes like this: they write a particular partial differential
equation, and the multiplicative integral is its unique
solution. This construction is significantly shorter than ours; 
and presumably it coincides with our construction in this setup. 
Cf.\ Subsection \ref{subsec:PDEs} of our paper. 
See the end of Subsection \ref{subsec:ration} regarding the relation between
the work of \cite{BGNT} and Conjecture \ref{conj:520}.

\medskip \noindent
\subsection{Acknowledgments}
Work on this paper began together with Fredrick Leitner, and I wish to
thank him for his contributions, without which the paper could not
have been written. Thanks also to Michael Artin, Maxim Kontsevich, 
Damien Calaque, Lawrence Breen, Amos Nevo, Yair Glasner, Barak Weiss and Victor
Vinnikov for assistance on various aspects of the paper.

\cleardoublepage
\section{Polyhedra and Piecewise Smooth Geometry}
\label{sec:pws}
\numberwithin{equation}{subsection}

\subsection{Conventions}
In this paper we work over the field $\mbb{R}$ of
real numbers. We denote by $\mbf{A}^{n} = \mbf{A}^{n}(\R)$
the real $n$-dimensional affine
space, with the usual smooth (i.e.\ $\mrm{C}^{\infty}$
differentiable) manifold structure, and the standard euclidean metric. 
The symbol $\otimes$ stands for $\otimes_{\R}$.
The coordinate functions on
$\mbf{A}^{n}$ are $t_1, \ldots, t_n$.

\subsection{Embedded Polyhedra} \label{subsec:embpoly}
By linear map $f : \mbf{A}^{m} \to \mbf{A}^{n}$ we mean the composition of a
linear homomorphism and a translation. Thus the group of invertible linear
maps of $\mbf{A}^{n}$ is $\mrm{GL}_n(\R) \ltimes \R^{n}$. 
By linear subset $X$ of $\mbf{A}^{n}$ we mean the zero locus of some set of
linear functions (not necessarily homogeneous). In other words $X$ is the
image of some injective linear map
$f : \mbf{A}^{m} \to \mbf{A}^{n}$; and then we let $m$ be the dimension of $X$. 

Given a set $S \subset \mbf{A}^{n}$, its linear (resp.\ convex) hull is the 
smallest linear (resp.\ convex) subset of $\mbf{A}^{n}$ containing $S$.
In case $S$ is finite, say $S = \{ x_1, \ldots, x_m \}$, then a point $x$ is in
the linear (resp.\ convex) hull of $S$ if and only if 
$x = \sum_{i = 1}^m a_i x_i$ for some real numbers $a_1, \ldots, a_m$
satisfying $\sum a_i = 1$ (resp.\ and $a_i \geq 0$). 

By {\em embedded polyhedron} \index{Embedded polyhedron} 
we always mean a convex bounded finite polyhedron
in $\mbf{A}^n$, for some $n$. Namely a polyhedron $X$ in $\mbf{A}^n$ is the
convex hull of a finite subset of $\mbf{A}^n$.
A point $x \in X$ is called a {\em vertex} if it is not in the convex hull of
any finite subset of $X - \{ x \}$. 
If we denote by $S$ the set  of vertices of $X$, then $S$ is finite, and
$X$ is the convex hull of $S$. 
The dimension $\opn{dim} X$ is the dimension of the linear hull of
$X$, and we call $\mbf{A}^n$ the {\em ambient linear space} of $X$.
The standard euclidean metric on $\mbf{A}^{n}$ restricts to a global metric (a
distance function) $\opn{dist}_X$ on $X$.
We denote by $\opn{diam}(X)$ the diameter of $X$, namely the maximal distance
between any two points.
The linear functions on $\mbf{A}^{n}$ (spanned by $t_1, \ldots, t_n$) restrict
to functions $X \to \R$, that we also call linear functions. The $\R$-module 
$\mcal{O}_{\mrm{lin}}(X)$ of linear functions has rank equal to 
$m := \opn{dim} X$. An $\R$-basis $\{ s_i \}_{i = 1, \ldots, m}$
of $\mcal{O}_{\mrm{lin}}(X)$ is called a {\em linear
coordinate system} on $X$. 

We shall mostly encounter two kinds of embedded polyhedra. The first is the 
{\em $n$-dimensional cube} $\mbf{I}^n$. This is the subset of
$\mbf{A}^n$ defined by the inequalities
$0 \leq t_1, \ldots, t_n \leq 1$. 
As a convex set it is spanned by its $2^n$ vertices.
See Figure \ref{fig:75} for an illustration. 

\begin{figure}
\includegraphics[scale=0.25]{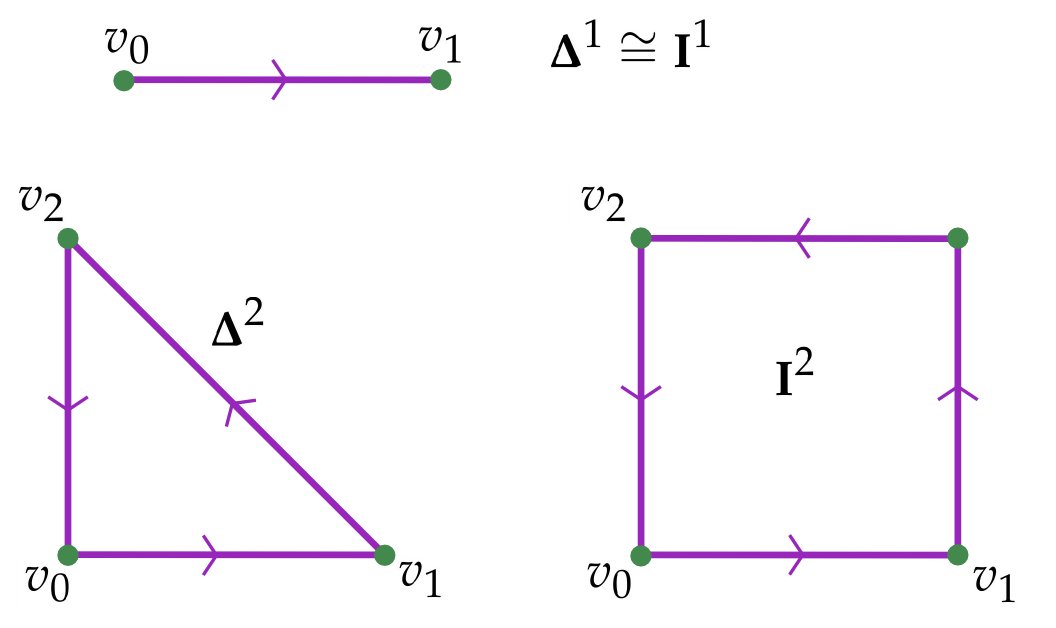}
\caption{The polyhedra $\bsym{\Delta}^1 \cong \mbf{I}^1$,
$\bsym{\Delta}^2$ and $\mbf{I}^2$, with their vertices and orientations.} 
\label{fig:75}
\end{figure}

The second kind of embedded polyhedron is the {\em $n$-dimensional simplex}
$\bsym{\Delta}^n$, embedded in $\mbf{A}^{n+1}$. We use 
barycentric coordinates $t_0, t_1 \ldots, t_n$ on $\mbf{A}^{n+1}$
when dealing with the simplex. In these coordinates
$\bsym{\Delta}^n$ is the subset of $\mbf{A}^{n+1}$ defined by the
conditions
$0 \leq t_0, \ldots, t_n$ and $\sum_{i=0}^n t_i = 1$.
As a convex set it is spanned by its $n+1$ vertices.

Suppose $X$ and $Y$ are both embedded polyhedra. A {\em linear map} 
$f : X \to Y$ is a function that extends to a linear map between the ambient
linear spaces. Such a map $f$ is determined by its value on a sequence
$(v_0, \ldots, v_n)$ of linearly independent vertices of $X$, 
where $n := \opn{dim} X$. Thus we shall often describe the linear map $f$ as 
\[ f(v_0, \ldots, v_n) = (w_0, \ldots, w_n) , \]
or just $f = (w_0, \ldots, w_n)$, where $w_i := f(v_i)$.

In case $X$ is the cube $\mbf{I}^n$, embedded in
$\mbf{A}^{n}$, we shall look at the following linearly independent vertices: 
$v_0$ is the origin $(0, \ldots, 0)$, and for $i = 1, \ldots, n$ we
take
\begin{equation} \label{eqn:1}
v_i := (0, \ldots, 1, \ldots, 0) ,
\end{equation}
with $1$ in the $i$-th position. In case of the simplex $\bsym{\Delta}^n$,
embedded in $\mbf{A}^{n+1}$ with barycentric coordinates, we look at all its
vertices, and use the notation \lb $(v_0, \ldots, v_n)$, where $v_i$ is as in
equation (\ref{eqn:1}).

\subsection{Smooth Maps and Forms}
The algebra of smooth real functions on
$\mbf{A}^n$ is denoted by $\mcal{O}(\mbf{A}^n)$.
The de Rham DG (differential graded) algebra of smooth differential forms on
$\mbf{A}^n$ is
\[ \Omega(\mbf{A}^n) = \bigoplus_{p = 0}^n \
\Omega^p(\mbf{A}^n) , \]
with differential $\d$ and wedge product $\wedge$.

Let $X$ be an $m$-dimensional polyhedron embedded in $\mbf{A}^n$. The set $X$ is
a compact topological space, its interior $\opn{Int} (X)$ is an $m$-dimensional
manifold, and the boundary $\partial X$ is a finite union of
$(m-1)$-dimensional polyhedra. But $X$ also has a differentiable structure.

\begin{dfn} \label{dfn:500}
Let $X \subset \mbf{A}^n$ be an embedded polyhedron, and let $U \subset X$ be
an open set.
\begin{enumerate}
\item A function $f : U \to \R$ is said to be {\em smooth} if 
it extends to a smooth function $\til{f} : \til{U} \to \R$
on some open neighborhood $\til{U}$ of $U$ in $\mbf{A}^n$.

\item A map $f : U \to \mbf{A}^p$ is called smooth if the functions 
$t_i \circ f : U \to \R$ are  smooth for all $i \in \{ 1, \ldots, p \}$.

\item Suppose $Y \subset \mbf{A}^p$ is an embedded polyhedron, and 
$V \subset Y$ is an open set. A map $f : U \to V$ is called smooth if the
composed map $f : U \to \mbf{A}^p$ is smooth. 

\item Let $V$ be an open subset of an embedded polyhedron. A map $f : U \to V$
is called a {\em diffeomorphism} of it is bijective, and both $f$ and $f^{-1}$
are smooth.
\end{enumerate}
\end{dfn}

We denote by $\mcal{O}(U)$ the $\R$-algebra of all smooth functions 
$U \to \R$.
A smooth map $f : U \to V$ between open subsets of embedded polyhedra induces an
$\R$-algebra homomorphism $f^* : \mcal{O}(V) \to \mcal{O}(U)$, namely pullback. 

Suppose $X$ is an embedded polyhedron and $U \subset X$ is an open set. 
We let $\opn{Int}(U) := U \cap \opn{Int}(X)$ and 
$\partial(U) := U \cap \partial(X)$. Note that $\opn{Int}(U)$ is a
manifold, and it is dense in $U$.

\begin{dfn}  \label{dfn:501}
Let $X \subset \mbf{A}^n$ be an embedded polyhedron,
and let $U \subset X$ be an open set.
A {\em smooth differential $p$-form} on $U$ is a differential form
$\alpha \in \Omega^p(\opn{Int} (U))$, that extends to a form
$\til{\alpha} \in \Omega^p(\til{U})$ on some open neighborhood $\til{U}$ of $U$
in $\mbf{A}^n$. 
\end{dfn}

We denote by $\Omega^p(U)$ the set of differential $p$-forms on
$U$, and we let 
\[ \Omega(U) := \bigoplus_{p \geq 0} \ \Omega^p(U) . \]
This is a DG subalgebra of $\Omega(\opn{Int}(U))$. 

Suppose $U$ and $V$ are open subsets of  embedded polyhedra. 
A smooth map $f : U \to V$  induces a DG algebra homomorphism 
$f^* : \Omega(V) \to \Omega(U)$.
If $f$ is an inclusion then we
usually write $\alpha|_U := f^*(\alpha)$ for $\alpha \in \Omega(V)$. 

Let $X$ be an $m$-dimensional embedded polyhedron, and $U \subset X$ a nonempty
open set. Then $\Omega^p(U)$ is a free $\mcal{O}(U)$-module of rank
$\binom{m}{p}$. Indeed, if we choose a linear coordinate system 
$\{ s_i \}_{i = 1, \ldots, m}$ on $X$, then the family
\begin{equation} \label{eqn:2}
\{ \d s_{i_1} \wedge \cdots \wedge \d s_{i_p} \}_{1 \leq i_1 < \cdots < i_p
\leq m}
\end{equation}
is a basis for $\Omega^p(U)$ as $\mcal{O}(U)$-module.

\begin{dfn} \label{dfn:27}
Let $\bsym{s} =  \{ s_i \}_{i = 1, \ldots, m}$ be a linear coordinate system
on the $m$-dimensional polyhedron $X$. Given a form 
$\alpha \in \Omega^p(X)$, it has a unique expansion as a sum
\[ \alpha = \sum_{\bsym{i}} \, 
\til{\alpha}_{\bsym{i}} \cdot \d s_{i_1} \wedge \cdots \wedge \d s_{i_p} , \]
where 
$\bsym{i} = (i_1, \ldots, i_p)$ runs over the set of strictly increasing
multi-indices in the range $\{ 1, \ldots, m \}$,  and 
$\til{\alpha}_{\bsym{i}} \in \mcal{O}(X)$.
The functions $\til{\alpha}_{\bsym{i}}$ are called the {\em coefficients}
of $\alpha$ with respect to the linear coordinate system 
$\bsym{s}$.
\end{dfn}

Let us denote by $\Omega_{\mrm{const}}(\mbf{A}^{n})$ the $\R$-subalgebra of
$\Omega(\mbf{A}^{n})$
generated by the elements $\d t_1, \ldots, \d t_n$. 
Now suppose $X$ is an $m$-dimensional polyhedron in $\mbf{A}^{n}$.
The image of
$\Omega_{\mrm{const}}(\mbf{A}^{n})$ under the canonical surjection
$\Omega(\mbf{A}^{n})\to \Omega(X)$
is denoted by $\Omega_{\mrm{const}}(X)$, and its elements are called
{\em constant differential forms}. There is an isomorphism of graded algebras
\[ \Omega(X) \cong \mcal{O}(X) \otimes \Omega_{\mrm{const}}(X) . \]
If we choose a linear coordinate system 
$\bsym{s} = \{ s_i \}_{i = 1, \ldots, m}$ on $X$, then (\ref{eqn:2}) is an
$\R$-basis for $\Omega^p_{\mrm{const}}(X)$.

\begin{dfn} \label{dfn:26}
Given a smooth form $\alpha \in \Omega^p(X)$ and a point 
$x \in X$, we define the {\em constant form} associated to $\alpha$ at $x$ to be
\[ \alpha(x) := \sum_{\bsym{i}} \, 
\til{\alpha}_{\bsym{i}}(x) \cdot \d s_{i_1} \wedge \cdots \wedge \d s_{i_p}
\in \Omega^p_{\mrm{const}}(X) , \]
where the functions $\til{\alpha}_{\bsym{i}}$ are the coefficients of $\alpha$
relative to some  linear coordinate system 
$\bsym{s} = \{ s_i \}$, as in 
Definition \ref{dfn:27}. 
\end{dfn}

Note that $\alpha(x)$ is independent of the linear coordinate system.

Let $f : X \to Y$ be a linear map between embedded polyhedra. If $g$ is a linear
function on $Y$, then 
$f^*(g) = g \circ f$ is a linear function on $X$; and if 
$\alpha \in \Omega_{\mrm{const}}(Y)$, then 
$f^*(\alpha) \in \Omega_{\mrm{const}}(X)$.

\subsection{Manifolds with Sharp Corners} \label{subsec:mfld-corners}
Manifolds with corners are modeled on 
$\R_{\geq 0}^n$; cf.\ \cite {Le} or \cite{Jo}. 
However not all polyhedra are manifolds with corners; so we introduce a variant
in which the local model is an embedded polyhedron.
See Remark \ref{rem:506} below. 

\begin{dfn} \label{dfn:510}
Let $X$ be a topological space. 
\begin{enumerate}
\item A {\em smooth chart with sharp corners} on $X$
is a pair $(U, \phi)$, where $U$ is an open subset of an embedded polyhedron,
and $\phi : U \to X$ is a map such that $\phi(U)$ is open in $X$ and 
$\phi : U \to \phi(U)$ is a homeomorphism. 

\item Two smooth charts with sharp corners $(U_1, \phi_1)$ and $(U_2, \phi_2)$
are compatible if 
\[ \phi_2^{-1} \circ \phi_1 : \phi_1^{-1} 
\bigl( \phi_1(U_1) \cap \phi_2(U_2) \bigr) \to 
\phi_2^{-1} \bigl( \phi_1(U_1) \cap \phi_2(U_2) \bigr) \]
is a diffeomorphism.

\item A {\em smooth atlas with sharp corners} on $X$ is a collection 
$\{ (U_i, \phi_i) \}$ of compatible charts with sharp corners, such that 
$X = \bigcup_i \phi_i(U_i)$.
\end{enumerate}
\end{dfn}

\begin{dfn} \label{dfn:505}
A {\em smooth manifold with sharp corners}
\index{Smooth manifold with sharp corners}
is a  Hausdorff second countable
topological space $X$ endowed with a smooth atlas with sharp corners
$\{ (U_i, \phi_i) \}$.
\end{dfn}

\begin{dfn} \label{dfn:506}
Let $X$ and $Y$ be smooth manifolds with sharp corners,
with atlases  $\{ (U_i, \phi_i) \}$ and $\{ (V_j, \psi_j) \}$ respectively. A
map $f : X \to Y$ is called {\em smooth} if for every $i,j$ the map 
\[ \psi_j^{-1} \circ f \circ \phi_i : 
(f \circ \phi_i)^{-1} (\psi_j(V_j)) \to V_j \]
is smooth. 
\end{dfn}

\begin{rem} \label{rem:506}
If we were to restrict the class of polyhedra in Definition \ref{dfn:505} to
cubes (or simplices), then we would recover the usual definition of manifold
with corners.  ``Sharp corners'' occur when too many faces of a polyhedron meet
at a vertex (to
be precise, more than the dimension of the manifold; e.g.\ at the top vertex of
a pyramid with square base). 

There are several notions of smooth maps between manifolds with corners in the
literature. Suppose $X$ and $Y$ are smooth manifolds with corners, and 
$f : X \to Y$ is a smooth map in the sense of Definition
\ref{dfn:506}. Then $f$ is called ``weakly smooth'' in \cite{Jo}. In order for
$f$ to be smooth according to \cite{Jo} it must satisfy a condition involving
the boundary strata of $X$ and $Y$; such a condition can't even be stated
when sharp corners are present.
\end{rem}

\begin{prop} \label{prop:500}
The category of smooth manifolds with sharp corners contains the following
categories as full subcategories:
\begin{enumerate}
\item The category of smooth manifolds (without boundary) and smooth maps
between them.
\item The category of embedded polyhedra and smooth maps
between them (Definition \tup{\ref{dfn:500}(3)}).
\end{enumerate}
\end{prop}

We leave out the easy proof. 

A function $f : X \to \R$ is called smooth if when viewed as a map 
$f : X \to \mbf{A}^1$ it is smooth. We denote by $\mcal{O}(X)$ the $\R$-algebra
of smooth functions.

If $U$ is an open subset of a smooth manifold with sharp corners $X$, then $U$
inherits a structure of smooth manifold with sharp corners.

A smooth manifold with sharp corners $X$  has an intrinsic
decomposition $X = \partial X \coprod \opn{Int} (X)$, where the interior
$\opn{Int} (X)$ is a manifold without boundary, and the
boundary $\partial X$ is a manifold with sharp corners.
The condition for a
point $x \in X$ to belong to $\opn{Int} (X)$ is that it has an open
neighborhood that's isomorphic, as manifold with sharp corners, to a manifold
without boundary.

We say that a manifold with sharp corners $X$ has dimension $n$ if 
the manifold $\opn{Int} (X)$ has dimension $n$. (Of course every connected
component of $X$ has a dimension.)  In this case $\partial X$ has dimension
$n-1$.

Let $X$ be a smooth manifold with sharp corners with atlas 
$\{ (U_i, \phi_i) \}$. For every $i$ we have a diffeomorphism
$\phi_i : \opn{Int}(U_i) \to  \opn{Int}(\phi_i(U_i))$
between manifolds without boundaries, and
$\opn{Int}(\phi_i(U_i)) = \phi_i(U_i) \cap \opn{Int} (X)$.
The module of smooth differentials $\Omega^p(U_i)$ was introduced in Definition 
\ref{dfn:501}.

\begin{dfn}  \label{dfn:511}
Let $X$ be a smooth manifold with sharp corners with atlas 
$\{ (U_i, \phi_i) \}$.
A {\em smooth differential $p$-form} on $X$ is a differential form
$\alpha \in \Omega^p(\opn{Int} (X))$, such that for every $i$ 
the form 
\[ \phi_i^* (\alpha|_{\opn{Int}(\phi_i(U_i))}) \in \Omega^p(\opn{Int}(U_i)) \]
extends to a form $\til{\alpha}_i \in \Omega^p(U_i)$.
\end{dfn}

We denote by $\Omega^p(X)$ the set of differential $p$-forms on
$X$, and we let 
\[ \Omega(X) := \bigoplus_{p \geq 0} \ \Omega^p(X) . \]
This is a DG subalgebra of $\Omega(\opn{Int}(X))$. 
Of course $\Omega^0(X) = \mcal{O}(X)$.

A smooth map $f : X \to Y$ between manifolds with sharp corners induces a DG
algebra homomorphism 
$f^* : \Omega(Y) \to \Omega(X)$.

\begin{conv} \label{conv:500}
For the rest of the paper, by ``manifold'' we always mean a smooth manifold
with sharp corners, as defined in Definition \ref{dfn:505}. 
Smooth maps between manifolds are in the sense of  Definition \ref{dfn:506}. 
\end{conv}

\subsection{Abstract Polyhedra} \label{subsec:abspoly}
It turns out that the geometric structure of a polyhedron is easy to
encode, and it is nicer to work
with polyhedra that are not embedded. 

By a {\em global metric} on a topological space $Z$ we mean a distance function 
$\opn{dist}_Z : Z^2 \to \R_{\geq 0}$ that induces the topology of $Z$. 
Recall that if $Z$ is an embedded polyhedron in $\mbf{A}^m$, then $Z$ inherits
a global metric from the ambient affine space, and also an $\R$-module 
of linear functions 
$\mcal{O}_{\mrm{lin}}(Z) \subset \mcal{O}(Z)$. 

\begin{dfn} \label{dfn:503}
A {\em polyhedron} \index{Polyhedron}
 is a smooth manifold with sharp corners
$X$, together with a global metric $\opn{dist}_X$ and an $\R$-submodule
$\mcal{O}_{\mrm{lin}}(X) \subset \mcal{O}(X)$.
The condition is that there exists an isomorphism 
of smooth manifolds with sharp corners
$f : X \iso Z$ for some embedded
polyhedron $Z \subset \mbf{A}^{m}$, with these two properties:
\begin{enumerate}
\rmitem{i} $f$ is an isometry; i.e.\ 
$f : (X, \opn{dist}_X) \to (Z, \opn{dist}_Z)$ is an isomorphism of metric
spaces.

\rmitem{ii} The $\R$-algebra isomorphism
$f^* : \mcal{O}(Z) \to \mcal{O}(X)$
satisfies \lb
$f^*(\mcal{O}_{\mrm{lin}}(Z)) = \mcal{O}_{\mrm{lin}}(X)$.
\end{enumerate}
The composed map $f : X \to \mbf{A}^{m}$ is called a {\em
linear metric embedding}.
If only property (ii) holds, then $f : X \to \mbf{A}^{m}$ is called a {\em
a linear embedding}.
\end{dfn}

In terms of this definition, an embedded polyhedron is a polyhedron $X$ \lb
equipped with a particular linear metric embedding $f : X \to \mbf{A}^{m}$.

\begin{dfn} \label{dfn:512}
Let $X$ and $Y$ be polyhedra, and let $f : X \to Y$ be a map. 
\begin{enumerate}
\item We say $f$ is {\em smooth} if it is a smooth map of smooth manifolds with
sharp corners (see Definition \ref{dfn:506}).

\item We say $f$ is {\em linear} if 
$f^*(\mcal{O}_{\mrm{lin}}(Y)) \subset \mcal{O}_{\mrm{lin}}(X)$.
\end{enumerate}
\end{dfn}

By {\em linear coordinate system} on $X$ we mean an $\R$-basis of
$\mcal{O}_{\mrm{lin}}(X)$.

\begin{prop} \label{prop:511}
Let $X$ be a polyhedron, and let $\bsym{s} = \{ s_i \}_{i = 1, \ldots, n}$
be linear coordinate system on $X$. Consider
the induced smooth map  $f : X \to \mbf{A}^{n}$.
Then the image $Z := f(X)$ is an $n$-dimensional embedded polyhedron in
$\mbf{A}^{n}$, $f : X \to Z$ is an isomorphism of 
smooth manifolds with sharp corners, and 
$f^*(\mcal{O}_{\mrm{lin}}(Z)) = \mcal{O}_{\mrm{lin}}(X)$.
I.e.\ the map $f : X \to \mbf{A}^{n}$ is a linear embedding. 
\end{prop}

The proof is left as an exercise. 

Suppose we are given linear embeddings $X \to \mbf{A}^{m}$ and 
$Y \to \mbf{A}^{n}$. It is easy to see that a map $f : X \to Y$ is linear in
the sense of Definition \ref{dfn:512} if and only if $f$ extends to a linear
map $\mbf{A}^{m} \to \mbf{A}^{n}$. 

Note that we said nothing about the metrics in Proposition \ref{prop:511}.

\begin{dfn}
An {\em orthonormal linear coordinate system} on $X$ is a linear coordinate
system $\bsym{s}$ such that the induced map 
$f : X \to \mbf{A}^{n}$ is a linear metric embedding.
\end{dfn}

Let $\bsym{s}$ be an orthonormal linear coordinate system on $X$.
Then for any $x, y \in X$ the distance between them is
\[ \opn{dist}_X(x, y) = \Bigl( \bosum_{i = 1}^n \bigl( s_i(x) - s_i(y) \bigr)^2
\Bigr)^{1/2} . \]

We can view our polyhedra as ``Riemannian manifolds'', in the
following sense. First note that we can talk about the DG algebra 
$\Omega_{\mrm{const}}(X)$ of constant forms: it is the DG algebra generated
over $\R$ by the family of forms $\{ \d s_i \}_{i = 1, \ldots, n}$, where 
$\bsym{s} = \{ s_i \}_{i = 1, \ldots, n}$ is a linear coordinate system.
If moreover $\bsym{s}$ is an orthonormal linear coordinate system, then
consider the constant symmetric $2$-form
\[ \omega_X := \sum_{i = 1}^n \, \d s_i \otimes \d s_i \in 
\Omega^1_{\mrm{const}}(X) \otimes_{\R} \Omega^1_{\mrm{const}}(X) \subset 
\Omega^1(X) \otimes_{\mcal{O}(X)} \Omega^1(X) . \]
The form $\omega_X$ is independent of the orthonormal linear coordinate system.
It encodes the global metric (the distance function $\opn{dist}_X$) in the
obvious way. 

Suppose $f : X \to Y$ is a linear map between polyhedra. Then
\[ f^*(\omega_Y) \in 
\Omega^1_{\mrm{const}}(X) \otimes_{\R} \Omega^1_{\mrm{const}}(X) . \]
The map $f$ is conformal if and only if
$f^*(\omega_Y) = a \omega_X$
for some positive real number $a$, which we call the scaling factor.
And $f$ is a linear metric embedding if and only if $a = 1$.

Let $X$ be a polyhedron. The linear functions on $X$, i.e.\ the elements of 
$\mcal{O}_{\mrm{lin}}(X)$, allow us to construct the convex hull of any
subset $S \subset X$ (cf.\ Proposition \ref{prop:511}). A sub-polyhedron $Z$ of
$X$ is by definition the convex
hull of a finite subset of $X$. $Z$ inherits from $X$ (by restriction) a
structure of a polyhedron. The inclusion map $Z \to X$ is also called a linear
metric embedding.

\subsection{Piecewise Smooth Forms and Maps} \label{subsec:pws}
Let $X$ be an $n$-dimen\-sional polyhedron (Definition \ref{dfn:503}). A
$p$-dimensional linear
simplex in $X$ is by definition the image of an injective linear map 
$\tau : \bsym{\Delta}^p \to X$. 
A {\em linear triangulation of $X$} is a finite collection
$\{ X_j \}_{j \in J}$ of linear simplices in $X$, such that 
$X_j \neq X_k$ if $j \neq k$; each face of
a simplex $X_j$ is the simplex $X_k$ for some $k \in J$; 
for any $j, k \in J$ the intersection $X_j \cap X_k$ is the simplex
$X_l$ for some $l \in J$; and $X = \bigcup_{j \in J} X_j$. 
We let 
$J_p := \{ j \in J \mid \opn{dim} X_j = p \}$,
so 
$J = \coprod_{0 \leq p \leq n} J_p$.
The topological space 
\[ \opn{sk}_p T := \bigcup_{q \leq p} \ \bigcup_{j \in J_{q}} X_j \] 
is called the {\em $p$-skeleton} of the triangulation $T$. 

We shall have to use piecewise smooth differential forms on polyhedra. 
Suppose $T = \{ X_j \}_{j \in J}$ is a linear triangulation of a
polyhedron $X$. A {\em piecewise smooth differential $p$-form on $X$, relative
to the triangulation $T$}, is a collection 
$\{ \alpha_j \}_{j \in J}$ of differential forms, with
$\alpha_j \in \Omega^p(X_j)$,
such that for any inclusion $X_j \subset X_k$ of simplices one has
$\alpha_j |_{X_k} = \alpha_k$.
Let us denote by $\Omega^p_{\mrm{pws}}(X; T)$ the set of 
piecewise smooth differential $p$-form on $X$ relative to $T$, and
\[ \Omega_{\mrm{pws}}(X; T) := 
\bigoplus_{p = 0}^n \ \Omega^p_{\mrm{pws}}(X; T) , \]
where $n := \opn{dim} X$. This is a DG algebra. 

Next suppose $S = \{ Y_k \}_{k \in K}$ is a linear triangulation of 
$X$ which is a subdivision of $T$. This means that each simplex
$Y_k$ is contained in some simplex $X_j$. Take a piecewise smooth differential
form 
$\alpha = \{ \alpha_j \} \in \Omega^p_{\mrm{pws}}(X; T)$.
Then there is a unique piecewise smooth differential form 
$\beta = \{ \beta_k \} \in \Omega^p_{\mrm{pws}}(X; S)$
satisfying 
$\beta_k = \alpha_j|_{Y_k}$
for any inclusion $Y_k \subset X_j$.
In this way we get a DG algebra homomorphism
$\Omega_{\mrm{pws}}(X; T) \to \Omega_{\mrm{pws}}(X; S)$, which is actually
injective. Since any two triangulations have a common subdivision,
the DG algebras $\Omega_{\mrm{pws}}(X; T)$ form a directed system. 

\begin{dfn} \label{dfn:25}
Let $X$ be a polyhedron. The 
{\em algebra of piecewise smooth differential forms} 
\index{Piecewise smooth differential form}
on $X$ is the DG algebra
\[ \Omega_{\mrm{pws}}(X) := \lim_{T \to}\, \Omega_{\mrm{pws}}(X; T) , \]
where $T$ runs over all linear triangulations of $X$.
\end{dfn}

This definition is very similar to Sullivan's PL forms; see \cite{FHT}. 

Note that any element of
$\Omega^p_{\mrm{pws}}(X)$ is represented by 
some $\alpha \in \lb \Omega^p_{\mrm{pws}}(X; T)$, where $T$ is a
linear triangulation of $X$.
Such a triangulation $T$ is called a {\em smoothing triangulation} for 
$\alpha$. There are injective DG algebra homomorphisms
\begin{equation} \label{eqn:17}
\Omega(X) \to \Omega_{\mrm{pws}}(X; T) \to \Omega_{\mrm{pws}}(X) .
\end{equation}

It is an exercise to show that the cohomology of 
$\Omega(X)$ vanishes in positive degrees, and 
$\mrm{H}^0 \, \Omega(X) = \R$ (recall that $X$ is convex).
Also the homomorphisms (\ref{eqn:17}) are quasi-isomorphisms. 
(We shall not use these facts.)

\begin{dfn} \label{dfn:504}
Let $X$ be a polyhedron, and let $Y$ be a manifold (see Convention
\ref{conv:500}). 
\begin{enumerate}
\item  Let $T = \{ X_j \}_{j \in J}$ be a linear triangulation of $X$.
A map $f : X \to Y$ is called a {\em piecewise smooth map relative to $T$}
if $f|_{X_j} : X_j \to Y$ is smooth for every $j \in J$. 

\item A map $f : X \to Y$ is a  {\em piecewise smooth map}
\index{Piecewise smooth map}
if it is piecewise smooth relative to some linear triangulation $T$; and then 
we say that $T$ is a {\em smoothing triangulation} for $f$.
\end{enumerate}
\end{dfn}

Note that a piecewise smooth map $f : X \to \mbf{A}^1$ is
the same as an element of the $\R$-algebra
$\mcal{O}_{\mrm{pws}}(X) := \Omega^0_{\mrm{pws}}(X)$.
Given a piecewise smooth map $f : X \to Y$, there is an induced DG algebra
homomorphism
\[ f^* : \Omega(Y) \to \Omega_{\mrm{pws}}(X) . \]

Next suppose $X$ and $Y$ are both polyhedra. 
Then we can talk about {\em piecewise linear maps}
$f : X \to Y$, using linear triangulations of $X$ as explained in the previous
paragraph. Given a piecewise linear map $f : X \to Y$, there is an induced DG
algebra homomorphism
\[ f^* : \Omega_{\mrm{pws}}(Y) \to \Omega_{\mrm{pws}}(X) . \]
As usual, if $f$ is an embedding, then we write
$\alpha|_X := f^*(\alpha)$ for 
$\alpha \in \Omega_{\mrm{pws}}(Y)$.

Let $\alpha \in \Omega^p_{\mrm{pws}}(X)$.
The form $\alpha$ can be presented as follows. Choose a smoothing triangulation 
$T = \{ X_j \}_{j \in J}$ for $\alpha$; so that 
$\alpha|_{X_j} \in \Omega^p(X_j)$ for every $j$. Let
$\bsym{s} = (s_1, \ldots, s_n)$ be a linear coordinate system on $X$. For any 
$j$ let 
$\til{\alpha}_{j, \bsym{i}} \in \mcal{O}(X_j)$ be the coefficients of 
$\alpha|_{X_j}$ with respect to $\bsym{s}$, as in Definition \ref{dfn:27}.
Then
\begin{equation} \label{eqn:232}
\alpha|_{X_j} = \sum_{\bsym{i}} \, 
\til{\alpha}_{j, \bsym{i}} \cdot \d s_{i_1} \wedge \cdots \wedge \d s_{i_p} 
\in \Omega^p(X_j) , 
\end{equation}
where the multi-index 
$\bsym{i} = (i_1, \ldots, i_p)$ runs through all strictly increasing elements in
$\{ 1, \ldots, n \}^p$. 

Let $X$ be an $n$-dimensional polyhedron. Given a form
$\alpha \in \Omega_{\mrm{pws}}(X)$ and a point $x \in X$, we say that 
{\em $x$ is a smooth point of $\alpha$} if there exists some $n$-dimensional
simplex $Y$ in $X$, such that $x \in \opn{Int} Y$ and
$\alpha|_{Y} \in \Omega(Y)$.
The smooth locus of $\alpha$, i.e.\ the set of all smooth points of $\alpha$,
is open and dense in $X$. Its complement, called
the singular locus of $\alpha$, is contained in a finite union of simplices of
dimensions $< n$. Indeed, given any triangulation $T$ that smooths $\alpha$, the
singular locus of $\alpha$ is contained in the $(n-1)$-skeleton of $T$. 

Let $f : X \to Y$ be a 
piecewise smooth map from a polyhedron to a manifold. 
As above we can talk about the smooth locus of $f$, and whether 
$f|_Z : Z \to Y$ is smooth
for some simplex $Z \subset X$.

\subsection{Sobolev Norm} \label{subsec:sob}
It shall be very convenient to have a bound for piecewise smooth forms
on polyhedra and some of their derivatives. 

Let $X$ be a polyhedron, with orthonormal linear coordinate
system $\bsym{s} =$ \lb  $(s_1, \ldots, s_n)$. Given a smooth function 
$f \in \mcal{O}(X)$ and a multi-index
\[ \bsym{i}  = (i_1, \ldots, i_q) \in \{ 1, \ldots, n \}^q , \]
we write
\[ \partial_{\bsym{i}} f :=
\smfrac{\partial^q}{\partial s_{i_1} \cdots \partial s_{i_q}} f \in  
\mcal{O}(X) . \]
In the case $q = 0$, so the only sequence is the empty sequence 
$\emptyset$, we write
$\partial_{\emptyset} f := f$. 
We refer to the operators $\partial_{\bsym{i}}$ as the partial derivatives
with respect to $\bsym{s}$.

\begin{dfn}
Let $X$ be a polyhedron of dimension $n$, and let $\bsym{s}$ be an
orthonormal linear coordinate system on $X$. 
\begin{enumerate}
\item Take a smooth differential form $\alpha \in \Omega^p(X)$. 
Consider the coefficients $\til{\alpha}_{\bsym{i}}$ of $\alpha$ with
respect to $\bsym{s}$, as in Definition \ref{dfn:27}. 
Let $\partial_{\bsym{j}}$ be the partial derivatives with respect to
$\bsym{s}$. For a point $x \in X$ and a natural number $q$ we define 
\[ \norm{\alpha}_{\mrm{Sob}; x, q} :=
\Bigl( \bosum_{\bsym{i}, \bsym{j}} \, 
(\partial_{\bsym{j}} \til{\alpha}_{\bsym{i}})(x)^2 \Bigr)^{1/2} \ , \]
where the sum is on all 
$\bsym{j} \in \{ 1, \ldots, n \}^q$ and on all strictly increasing 
$\bsym{i} \in \{ 1, \ldots, n \}^p$.

\item For $\alpha \in \Omega^p(X)$ we define 
\[ \norm{\alpha}_{\mrm{Sob}} :=
\sup_{q = 0, 1, 2} \ \sup_{x \in X} \ \norm{\alpha}_{\mrm{Sob}; x, q} \, . \]

\item Take a piecewise smooth differential form
$\alpha \in \Omega^p_{\mrm{pws}}(X)$. 
Let $\{ X_j \}_{j \in J}$ be some smoothing triangulation for $\alpha$,
and let $\alpha_j := \alpha|_{X_j} \in \Omega^p(X_j)$. Define
\[ \norm{\alpha}_{\mrm{Sob}} :=
\sup_{j \in J} \ \norm{ \alpha_j }_{\mrm{Sob}} \ . \]
\end{enumerate}

The number $\norm{-}_{\mrm{Sob}}$ is called the {\em Sobolev norm to order $2$}.
\end{dfn}

\begin{rem}
Even for a smooth function $f$, our Sobolev norm 
$\norm{f}_{\mrm{Sob}}$ is not the same as the
usual order $2$ Sobolev norm from functional analysis, namely
$\norm{f}_s$ with $s = 2$. But it is in the ``same spirit'', and hence we use
this name.
\end{rem}

\begin{prop}
Let $X$ be an $n$-dimensional polyhedron.
\begin{enumerate}
\item For $\alpha \in \Omega^p(X)$, $x \in X$ and $q \in \N$, the number 
$\norm{\alpha}_{\mrm{Sob}; x, q}$ is  independent of the orthonormal linear
coordinate system $\bsym{s}$ on $X$. Hence the number
$\norm{\alpha}_{\mrm{Sob}}$ is also coordinate independent.

\item Given $\alpha \in \Omega^p_{\mrm{pws}}(X)$, its Sobolev norm 
$\norm{\alpha}_{\mrm{Sob}}$ is independent of the  smoothing triangulation
$\{ X_j \}_{j \in J}$.

\item Let $Z$ be a sub-polyhedron of $X$ and let
$\alpha \in \Omega^p_{\mrm{pws}}(X)$. Then 
$\norm{\alpha|_Z}_{\mrm{Sob}} \leq \norm{\alpha}_{\mrm{Sob}}$.
\end{enumerate}
\end{prop}

\begin{proof}
(1) Let us denote by $V$ the vector space of constant vector fields on $X$.
This is a rank $n$ vector space, endowed with a canonical inner product. The
set $\{ \smfrac{\partial}{\partial s_i} \}_{i = 1, \ldots, n}$
is an orthonormal basis for $V$. 

For $p, q \in \N$ let us denote by $\opn{T}^q(V)$ and $\bwedge^p(V)$ the 
tensor  power and exterior power of $V$, respectively. The vector space
$\opn{T}^q(V) \otimes \bwedge^p(V)$ has an induced  inner product.

For a strictly increasing sequence 
\[ \bsym{i}   = (i_1, \ldots, i_p) \in \{ 1, \ldots, n \}^p \] 
let us write
\[ \pi_{\bsym{i}} := 
\smfrac{\partial}{\partial s_{i_1}} \wedge \cdots \wedge 
\smfrac{\partial}{\partial s_{i_p}} \in \bwedge^p(V) . \]
Then the set 
$\{ \partial_{\bsym{j}} \otimes \pi_{\bsym{i}} \}$,
where the indices run over all $\bsym{j} \in \{ 1, \ldots, n \}^q$ and  all
strictly increasing $\bsym{i} \in \{ 1, \ldots, n \}^p$,
is an orthonormal basis of $\opn{T}^q(V) \otimes \bwedge^p(V)$.

Now $V$ is the linear dual of $\Omega^1_{\mrm{const}}(X)$, so 
$\pi_{\bsym{i}}$ can be viewed as a linear map
$ \Omega^p(X) \to \mcal{O}(X)$.
Given $\alpha \in \Omega^p(X)$, its coefficients with respect to $\bsym{s}$
satisfy
$\til{\alpha}_{\bsym{i}} = \pi_{\bsym{i}}(\alpha)$.
Thus
\[ \partial_{\bsym{j}} \til{\alpha}_{\bsym{i}} = 
(\partial_{\bsym{j}} \circ \pi_{\bsym{i}})(\alpha) \]
in $\mcal{O}(X)$.
  
Suppose $\bsym{s}' = (s'_1, \ldots, s'_n)$ is another orthonormal
linear coordinate system. Let us denote by 
$\partial'_{\bsym{j}}$, $\til{\alpha}'_{\bsym{i}}$ and $\pi'_{\bsym{i}}$ the
new operators and coefficient, namely those with respect to $\bsym{s}'$. The
set 
$\{ \partial'_{\bsym{j}} \otimes \pi'_{\bsym{i}} \}$
is also an orthonormal basis of $\opn{T}^q(V) \otimes \bwedge^p(V)$.
So this set and the set 
$\{ \partial_{\bsym{j}} \otimes \pi_{\bsym{i}} \}$
are related by a constant orthogonal matrix (of size $n^q \cdot \binom{n}{p}$).
It follows that
\[ \sum_{\bsym{i}, \bsym{j}} \, 
(\partial'_{\bsym{j}} \til{\alpha}'_{\bsym{i}})(x)^2 =
\sum_{\bsym{i}, \bsym{j}} \, 
(\partial_{\bsym{j}} \til{\alpha}_{\bsym{i}})(x)^2 
 \]
for any $x \in X$. 

We see that $\norm{\alpha}_{\mrm{Sob}; x, q}$
is independent of coordinates. The assertion for 
$\norm{\alpha}_{\mrm{Sob}}$ is now clear.

\medskip  \noindent
(2) The supremum on $x$ can be restricted to a dense open set in each cell
$X_j$ of the triangulation $\{ X_j \}_{j \in J}$. Therefore we get the same
value by replacing the triangulation $\{ X_j \}_{j \in J}$ with a refinement. 

\medskip  \noindent
(3) Given $Z$, choose an orthonormal linear coordinate system on it, and extend
it to an orthonormal linear coordinate system on $X$. Now it is clear.
\end{proof}

\subsection{Orientation and Integration} \label{subsec:orient}
An orientation on a manifold means a choice of a volume form, up to
multiplication by a positive smooth function. However in the case of a
polyhedron we can normalize the volume form:

\begin{dfn}
Let $X$ be a polyhedron of dimension $n$.
An {\em orientation} on $X$ is a constant form 
$\opn{or}(X) \in \Omega_{\mrm{const}}^n(X)$, 
such that 
\[ \opn{or}(X) =  \d s_1 \wedge \cdots \wedge \d s_n \]
for some orthonormal linear coordinate system 
$(s_1, \ldots, s_n)$ on $X$. 
\end{dfn}

Note that if $(s'_1, \ldots, s'_n)$ is some other orthonormal linear coordinate
system, then 
\[ \opn{or}(X) = \pm \d s'_1 \wedge \cdots \wedge \d s'_n . \]
If the sign is $+$ then $(s'_1, \ldots, s'_n)$ is said to be {\em  positively
oriented}.

For the polyhedron $\mbf{I}^n$ there is a standard orientation, coming from
its embedding in $\mbf{A}^{n}$. It is 
\[ \opn{or}(\mbf{I}^n) := \d t_1 \wedge \cdots \wedge \d t_n , \]
where $(t_1, \ldots, t_n)$ is the standard coordinate system on $\mbf{A}^{n}$.
Likewise for the polyhedron $\bsym{\Delta}^n$:
we let
\[ \opn{or}(\bsym{\Delta}^n) := \d t_1 \wedge \cdots \wedge \d t_n , \]
where $(t_0, \ldots, t_n)$ is the barycentric coordinate system on 
$\mbf{A}^{n+1}$. 
For $n = 1, 2$ the orientations can be also described by arrows -- see Figure
\ref{fig:75}.

\begin{dfn} \label{dfn:24}
Let $X$ be an $n$-dimensional oriented polyhedron and let
$\alpha \in \Omega^n(X)$. 
The {\em coefficient} of $\alpha$ is the function 
$\til{\alpha} \in \mcal{O}(X)$ such that
$\alpha = \til{\alpha} \cdot \opn{or}(X)$.
\end{dfn}

In other words, $\til{\alpha}$ is the coefficient of $\alpha$ with respect to
any positively oriented orthonormal linear coordinate system, in the sense of 
Definition \ref{dfn:27}.

An orientation on $X$ induces an orientation on the $(n-1)$-dimensional 
polyhedra that make up the boundary $\partial X$ (by contracting the
orientation volume form of $X$ with 
a constant outer gradient to the face; cf.\  \cite[Section 4.8]{Wa}).

Let $X$ be an $n$-dimensional oriented polyhedron. Suppose 
$\alpha = \{ \alpha_j \}_{j \in J} \lb \in 
\Omega^n_{\mrm{pws}}(X; T)$,
where $T = \{ X_j \}_{j \in J}$ is some linear
triangulation of $X$. Any $n$-dimen\-sional simplex $X_j$ inherits an
orientation, and hence the integral 
$\int_{X_j} \alpha_j$ is  well-defined. We let 
\[ \int_{X} \alpha := \sum_{j \in J_n } \
\int_{X_j} \alpha_j \in \R \ .  \]
This integration is compatible with subdivisions, and thus we have a
well-defined function
\[ \int_{X} : \Omega^n_{\mrm{pws}}(X) \to \R . \]

\begin{thm}[Stokes Theorem] \label{thm:15}
Let $X$ be an oriented $n$-dimensional polyhedron, and let
$\alpha \in \Omega^{n-1}_{\mrm{pws}}(X)$.
Then
\[ \int_{X} \d \alpha = \int_{\partial X} \alpha \ . \]
\end{thm}

\begin{proof}
Choose a linear triangulation $T = \{ X_j \}_{j \in J}$ of $X$ that smooths
$\alpha$; so that 
$\alpha = \{ \alpha_i \}_{j \in J}$ with
$\alpha_j \in \Omega^{n-1}(X_j)$.
Then by definition of $\int_{\partial X} \alpha$, and by cancellation due to
opposite orientations of inner $(n-1)$-dimensional simplices, we get
\[ \int_{\partial X} \alpha = 
\sum_{j \in J_n} \ \int_{\partial X_j} \alpha_j . \]
And the usual Stokes Theorem tells us that
\[ \int_{\partial X_j} \alpha_j = \int_{X_j} \d \alpha_j \ . \]
\end{proof}

\subsection{Piecewise Continuous Forms} \label{subsec:pwc}
For a polyhedron $X$, let us denote by $\mcal{O}_{\mrm{cont}}(X)$ the set of
continuous $\R$-valued functions on $X$. This is a commutative $\R$-algebra
containing $\mcal{O}_{\mrm{pws}}(X)$.

\begin{dfn} \label{dfn:40}
Let $p$ be a natural number. By {\em piecewise continuous $p$-form} on $X$ we
mean an element 
$\gamma \in \Omega^p_{\mrm{pwc}}(X)$, where we define
\[ \Omega^p_{\mrm{pwc}}(X) :=
\mcal{O}_{\mrm{cont}}(X) \otimes_{\mcal{O}_{\mrm{pws}}(X)}
\Omega^p_{\mrm{pws}}(X) . \]
\end{dfn}

A piecewise continuous $p$-form $\gamma$ can be represented as follows:
there are continuous functions
$f_1, \ldots, f_m \in \mcal{O}_{\mrm{cont}}(X)$,
and piecewise smooth forms
$\alpha_1, \ldots, \alpha_m \in \Omega^p_{\mrm{pws}}(X)$,
such that 
\begin{equation} \label{eqn:220}
\gamma = \sum_{i = 1}^m \, f_i \cdot \alpha_i .
\end{equation}

For $p = 0$ we have 
$\Omega^p_{\mrm{pwc}}(X) = \mcal{O}_{\mrm{cont}}(X)$ of course. 
But for $p > 0$ these forms are much more complicated, as we shall see below.

Warning: the exterior derivative $\d (\gamma)$ is not defined for 
$\gamma \in \Omega^p_{\mrm{pwc}}(X)$; at least not as a piecewise continuous
form (it is a distribution). The problem of course is that the functions $f_i$ 
above cannot be derived.

However, given $\gamma \in \Omega^p_{\mrm{pwc}}(X)$ and a piecewise linear map
$g
: Y \to X$ between polyhedra, the pullback
\[ g^*(\gamma) \in \Omega^p_{\mrm{pwc}}(Y) \]
is well defined. In terms of the expansion (\ref{eqn:220}), the formula is
\[ g^*(\gamma) = \sum_{i = 1}^m \, g^*(f_i) \cdot g^*(\alpha_i) . \]
Here $g^*(f_i) \in \mcal{O}_{\mrm{cont}}(Y)$ and
$g^*(\alpha_i) \in \Omega^p_{\mrm{pws}}(Y)$. 
If $g$ is a closed embedding then we write
$\gamma|_Y := g^*(\gamma)$ as usual.

The presentation (\ref{eqn:220}) of 
$\gamma \in \Omega^p_{\mrm{pwc}}(X)$ can be expanded further. 
Choose a linear triangulation $\{ X_j \}_{j \in J}$ of $X$ that smooths all the
piecewise smooth forms $\alpha_1, \ldots, \alpha_m$. 
Also choose a linear coordinate system $\bsym{s} = (s_1, \ldots, s_n)$ on $X$.
For any $i \in \{ 1, \ldots, m \}$ and $j \in J$ let 
$e_{i, j, \bsym{k}} \in \mcal{O}(X_j)$ be the coefficients of the smooth form
$\alpha_i|_{X_j}$ with respect to $\bsym{s}$, as in Definition \ref{dfn:27}.
Then
\begin{equation} \label{eqn:234}
\gamma|_{X_j} = \sum_{i = 1}^m \ \sum_{\bsym{k}} \ 
f_i \cdot e_{i, j, \bsym{k}} \cdot 
\d s_{k_1} \wedge \cdots \wedge \d s_{k_p} 
\in \Omega^p_{\mrm{pwc}}(X_j) ,
\end{equation}
where as usual the multi-index 
$\bsym{k} = (k_1, \ldots, k_p)$ runs through all strictly increasing elements in
$\{ 1, \ldots, n \}^p$. 

For top degree forms one can say more. 
Let $\gamma \in \Omega^n_{\mrm{pwc}}(X)$. In this case the expansion 
(\ref{eqn:234}) can be simplified to 
\begin{equation} \label{eqn:235}
\gamma|_{X_j} = \sum_{i = 1}^m  \ 
f_i \cdot e_{i, j} \cdot \d s_1 \wedge \cdots \wedge \d s_n
\in \Omega^n_{\mrm{pwc}}(X_j) ,
\end{equation}
with $e_{i, j} \in \mcal{O}(X_j)$.

Now assume that the polyhedron $X$ is oriented. We can choose the coordinate
system $\bsym{s} = (s_1, \ldots, s_n)$ so that it is orthonormal and positively
oriented; and then 
\[ \opn{or}(X) = \d s_1 \wedge \cdots \wedge \d s_n . \]
For any $j \in J_n$ we can integrate the continuous function 
$\sum_{i = 1}^m \, f_i \cdot e_{i, j}$ on the oriented $n$-dimensional
simplex $X_j$, obtaining
\[ \int_{X_j} \gamma := \int_{X_j} \
\sum_{i = 1}^m \, f_i \cdot e_{i, j} \cdot 
\d s_1 \wedge \cdots \wedge \d s_n \in \R . \]
Using this formula we define
\begin{equation} \label{eqn:236}
\int_{X} \gamma := \sum_{j \in J_n}  \ \int_{X_j} \gamma
\in \R . 
\end{equation}

\begin{prop} \label{prop:26}
For an $n$-dimensional oriented polyhedron $X$ and a piecewise continuous form
$\gamma \in \Omega^n_{\mrm{pwc}}(X)$, the number
$\int_{X} \gamma$ is independent of the presentation 
\tup{(\ref{eqn:235})} of $\gamma$.
\end{prop}

We leave out the easy proof. 

More generally, suppose $X$ is an $n$-dimensional polyhedron (not necessarily
oriented), $Z$ is an oriented $p$-dimensional polyhedron, 
$\tau : Z \to X$ is a piecewise linear map, and 
$\gamma \in \Omega^p_{\mrm{pwc}}(X)$. Then 
$\tau^*(\gamma) \in \Omega^p_{\mrm{pwc}}(Z)$, and we can use 
(\ref{eqn:236}) to integrate $\gamma$ along $\tau$ as follows:
\begin{equation} \label{eqn:221}
\int_{\tau} \gamma := \int_Z \tau^*(\gamma) \in \R . 
\end{equation}

\cleardoublepage 
\section{Estimates for the Nonabelian Exponential Map}
\label{sec:expon}

\subsection{The Exponential Map}
\indent{Nonabelian exponential map}
We need some preliminary results on the exponential maps of Lie groups.
Let $G$ be a (finite dimensional) Lie group over $\R$, with Lie
algebra $\mfrak{g} = \opn{Lie}(G)$. The exponential map
\[ \opn{exp}_G : \mfrak{g} \to G  \]
is an analytic map, satisfying
$\opn{exp}_G(0) = 1$ and $\exp(-\alpha) = \exp(\alpha)^{-1}$. The
exponential map is a diffeomorphism near
$0 \in \mfrak{g}$. Namely there is an open neighborhood $U_0(\g)$ of $0$
in $\g$, such that $V_0(G) := \opn{exp}_G(U_0(\g))$ is open in $G$,
and $\opn{exp}_G : U_0(\g) \to V_0(G)$ is a diffeomorphism.
Let $\log_G : V_0(G) \to U_0(\g)$ be the inverse of $\exp_G$. We call such
$V_0(G)$ an {\em open neighborhood of $1$ in $G$ on which $\log_G$ is
defined}. See \cite[Sections 2.10]{Va} for details.

The exponential map is functorial. Namely given a map 
$\phi : G \to H$ of Lie groups, the diagram
\[ \UseTips \xymatrix @C=7ex @R=5ex {
\opn{Lie}(G)
\ar[r]^{\opn{\ \ exp}_G}
\ar[d]_{\opn{Lie}(\phi)}
& 
G
\ar[d]^{\phi}
\\
\opn{Lie}(H)
\ar[r]^{\opn{\ \ exp}_H}
& 
H
} \]
is commutative.

When there is no danger of confusion we write $\exp$ and $\log$ instead of 
$\opn{exp}_G$ and $\log_G$, respectively.

The product in $G$ is denoted by $\cdot$, and the Lie bracket in $\g$ is
denoted by $[-,-]$. Given a finite sequence
$(g_1, \ldots, g_m)$ of elements in the group $G$, we write
\begin{equation} \label{eqn:6}
\prod_{i = 1}^m \, g_i := g_1 \cdot g_2 \cdots g_m \in G \ .
\end{equation}

It is a basic fact that if 
 $\alpha_1, \ldots, \alpha_m \in \g$
are {\em commuting} elements, then 
\[ \prod_{i=1}^m\, \opn{exp}(\alpha_i) = 
\exp \Bigl( \bosum_{i=1}^m \, \alpha_i \Bigr) . \]
The next theorem lists several estimates for the discrepancy when the elements
do not necessarily commute. 

\begin{thm} \label{thm:6}
Let $G$ be a Lie group, with Lie algebra $\g$. Let $V_0(G)$ be an  open
neighborhood of $1$ in $G$ on which $\log$ is well-defined, and
let $\norm{-}$ be a euclidean norm on $\g$. 
Then there are real constants $\epsilon_0(G)$ and $c_0(G)$ with the following
properties:
\begin{enumerate}
\rmitem{i} $0 < \epsilon_0(G) \leq 1$ and $0 \leq c_0(G)$.
\rmitem{ii} Let $\alpha_1, \ldots, \alpha_m \in \g$ be such that
$\sum_{i = 1}^m\, \norm{\alpha_i} < \epsilon_0(G)$.
Then
\[ \prod_{i=1}^m\, \opn{exp}(\alpha_i) \in V_0(G) ,  \]
\[ \Norm{ \opn{log} 
\bigl( \prod\nolimits_{i=1}^m \opn{exp}(\alpha_i) \bigr) } 
\leq c_0(G) \cdot \sum\nolimits_{i=1}^m \norm{\alpha_i} \]
and
\[ \Norm{ \opn{log} 
\bigl( \prod\nolimits_{i=1}^m \opn{exp}(\alpha_i) \bigr) -
\sum\nolimits_{i=1}^m \alpha_i } 
\leq c_0(G) \cdot \bigl( \sum\nolimits_{i=1}^m \norm{\alpha_i} \bigr)^2 . \]
\rmitem{iii} Let $\alpha_1, \alpha_2 \in \g$ be such that
$\norm{\alpha_1} + \norm{\alpha_2} < \epsilon_0(G)$. 
Then
\[  \norm{ [\alpha_1, \alpha_2] } \leq
c_0(G) \cdot \norm{\alpha_1} \cdot \norm{\alpha_2} , \]
\[ \begin{aligned}
& \Norm{ \opn{log} \bigl( \opn{exp}(\alpha_1) \cdot 
\opn{exp}(\alpha_2) \bigr) -  
( \alpha_1 + \alpha_2 + \smfrac{1}{2}[\alpha_1, \alpha_2] ) } \\
& \qquad 
\leq c_0(G) \cdot \norm{\alpha_1} \cdot \norm{\alpha_2} \cdot 
\bigl( \norm{\alpha_1} + \norm{\alpha_2} \bigr) .
\end{aligned} \]
and
\[ \begin{aligned}
& \Norm{ \opn{log} \bigl( \opn{exp}(\alpha_1) \cdot 
\opn{exp}(\alpha_2) \cdot \opn{exp}(\alpha_1)^{-1} \cdot
\opn{exp}(\alpha_2)^{-1} \bigr) -  
[\alpha_1, \alpha_2] } \\
& \qquad 
\leq c_{0}(G) \cdot \norm{\alpha_1} \cdot \norm{\alpha_2} \cdot 
\bigl( \norm{\alpha_1} + \norm{\alpha_2} \bigr) .
\end{aligned} \]
\rmitem{iv} Let 
$\alpha_1, \ldots, \alpha_m, \beta_1, \ldots, \beta_m \in \g$ be
such that
\[ \sum_{i = 1}^m\, (\norm{\alpha_i} + \norm{\beta_i}) < \epsilon_0(G) \, . \]
Then
\[ \begin{aligned}
& \Norm{ \opn{log} 
\bigl( \prod\nolimits_{i=1}^m \opn{exp}(\alpha_i + \beta_i) \bigr) -
\opn{log} \bigl(   \prod\nolimits_{i=1}^m \opn{exp}(\alpha_i) \bigr) -
\sum\nolimits_{i=1}^m \beta_i } \\
& \qquad 
\leq c_0(G) \cdot \bigl( \sum\nolimits_{i=1}^m (\norm{\alpha_i}  +
\norm{\beta_i}) \bigr) \cdot 
\bigl( \sum\nolimits_{i=1}^m \norm{\beta_i} \bigr) .
\end{aligned} \]
and
\[ \begin{aligned}
& \Norm{ \opn{log} 
\bigl( \prod\nolimits_{i=1}^m \opn{exp}(\alpha_i + \beta_i) \bigr) -
\opn{log} 
\bigl( \prod\nolimits_{i=1}^m \opn{exp}(\alpha_i) \bigr) } \\
& \qquad 
\leq c_0(G) \cdot 
\bigl( \sum\nolimits_{i=1}^m \norm{\beta_i} \bigr) \, .
\end{aligned} \]
\rmitem{v} Let $\alpha_1, \ldots, \beta_m$ be as in property \tup{(iv)},
and assume moreover that $[\alpha_i, \alpha_j] \lb = 0$ for all $i, j$. 
Then
\[ \begin{aligned}
& \Norm{ \opn{log} 
\bigl( \prod\nolimits_{i=1}^m \opn{exp}(\alpha_i + \beta_i) \bigr) -
\sum\nolimits_{i=1}^m (\alpha_i + \beta_i) } \\
& \qquad 
\leq c_0(G) \cdot \bigl( \sum\nolimits_{i=1}^m (\norm{\alpha_i}  +
\norm{\beta_i}) \bigr) \cdot 
\bigl( \sum\nolimits_{i=1}^m \norm{\beta_i} \bigr) .
\end{aligned} \]
\end{enumerate}
\end{thm}

The theorem is proved at the end of Subsection \ref{subsec:touches}, after
some preparations. 

The constants $c_0(G)$ and $\epsilon_0(G)$ are called a {\em commutativity
constant} and a {\em convergence radius} for $G$ respectively. (If $G$ is
abelian one can take $c_0(G) = 0$.)

\subsection{The CBH Theorem}
\label{subsec:CBH}
\index{CBH Theorem}
There is an element $F(x, y)$ in the completed free Lie algebra over 
$\mbb{Q}$ in the variables $x, y$, called the {\em Hausdorff series}.  See
\cite[Sections II.6 and II.7]{Bo}, where the letter $H$ is used instead of $F$. 
For any $i,j \geq 0$ let us denote by $F_{i, j}(x, y)$ the homogeneous
component of $F(x, y)$ of degree $i$ in $x$ and degree $j$ in $y$. So 
\[ F(x, y) =  \sum_{i, j \geq 0} F_{i, j}(x, y) ץ \]

Now consider a Lie group $G$ as before, with Lie algebra $\g$. Choose a
euclidean norm $\norm{-}$ on the vector space $\g$. 
Given elements $\alpha, \beta \in \g$ we can evaluate
$F_{i, j}(x, y)$ on them, obtaining an element
$F_{i, j}(\alpha, \beta) \in \g$. 

The {\em Campbell-Baker-Hausdorff Theorem} says that 
there is an open neighborhood $U_1$ of $0$ in $\g$ 
(its size depending on the norm) such that the series
$F(\alpha, \beta)$ converges absolutely and uniformly for 
$\alpha, \beta \in U_1$, and moreover 
\begin{equation} \label{eqn:87}
\exp \bigl( F(\alpha, \beta) \bigr) = 
\exp(\alpha) \cdot \exp(\beta) .
\end{equation}

This assertion is not explicit in the book \cite{Bo}; it
requires combining various scattered results in Sections II.7 and III.4
of \cite{Bo}. An explicit statement of the CBH Theorem is 
\cite[Theorem 2.15.4]{Va}. However the treatment of the structure of the
Hausdorff series is not sufficiently detailed in \cite{Va}. 

For a positive number $r$ we denote by $U(r)$ the open ball of radius $r$ and
center $0$ in $\g$, and by $\ol{U}(r)$ its closure (the closed ball). 

\begin{lem} \label{lem:35}
There are constants $\epsilon_1$ and $c_1$, such that
$0 < \epsilon_1$,
$U(\epsilon_1 ) \subset U_1$, $0 \leq c_1$, and for every 
$\alpha, \beta \in U(\epsilon_1)$ the inequality
\[ \Norm{ F(\alpha, \beta) - 
(\alpha + \beta + \smfrac{1}{2} [x,y]) } 
\leq c_1 \cdot \norm{\alpha} \cdot 
\norm{\beta} \cdot (\norm{\alpha} + \norm{\beta})  \]
holds.
\end{lem}

\begin{proof}
Let us denote by 
$f_{i, j}$ the norm of the bi-homogeneous function 
\[ F_{i, j} : \g \times \g \to \g . \]
Namely
\[ f_{i, j} := \sup \, \{ \norm{F_{i, j}(\alpha, \beta)} \mid 
\norm{\alpha}, \norm{\beta} \leq 1 \} . \]
So for any $\alpha, \beta \in \g$ one has
\[ \norm{F_{i, j}(\alpha, \beta)} \leq 
f_{i, j} \cdot \norm{\alpha}^i \cdot \norm{\beta}^j . \]
In \cite[Section II.7.2]{Bo} it is shown that there is a positive
number $\epsilon_1$, such that the series 
$\sum_{i, j \geq 0} f_{i,j} \epsilon_1^{i+j}$
converges (to a finite sum). 
By shrinking $\epsilon_1$ if necessary, we may assume that
$\ol{U}(\epsilon_1 ) \subset U_1$.

Now in \cite[Section II.6.4]{Bo} it is proved that 
\[ F_{1,0}(x, y) = x, \ F_{0, 1}(x, y) = y, \
F_{1, 1}(x,y) = \smfrac{1}{2} [x,y] \]
and 
\[ F_{i, 0}(x,y) = F_{0, j}(x,y) = 0 \ \ \text{for} \ \ i, j \neq 1 . \]
Thus for $\alpha, \beta \in U(\epsilon_1)$ we have
\[  F(\alpha, \beta) - (\alpha + \beta + \smfrac{1}{2} [x,y]) = 
\sum\limits_{(i, j) \in I} F_{i, j}(\alpha, \beta) , \]
where
\[ I := \{ (i, j) \mid i, j \geq 1 \ \text{and} \ i + j \geq 3 \} . \]

Now if $(i, j) \in I$ and $0 \leq a, b \leq \epsilon_1$, then 
\[ a^i b^j \leq a b (a+b) \cdot \epsilon_1^{i + j -3} . \]
Therefore for $\alpha, \beta \in U(\epsilon_1)$, with
$a := \norm{\alpha}$ and $b := \norm{\beta}$, we have the estimate
\[ \begin{aligned}
& \Norm{ \sum\limits_{(i, j) \in I} F_{i, j}(\alpha, \beta) }
\leq \sum\limits_{(i, j) \in I} \Norm{ F_{i, j}(\alpha, \beta) } \\
& \qquad \leq \sum\limits_{(i, j) \in I} f_{i,j} \cdot a^i \cdot b^j
\leq a b (a+b) \cdot \epsilon_1^{-3} \cdot  
\sum\limits_{(i, j) \in I} f_{i,j} \cdot \epsilon_1^{i+j} \, .
\end{aligned} \]
Taking the number
\[ c_1 := \epsilon_1^{-3} \cdot  
\sum\limits_{(i, j) \in I} f_{i,j} \cdot \epsilon_1^{i+j} \]
does the trick.
\end{proof}

\subsection{Calculations with a Few Elements}
We continue with the chosen open set $V_0(G) \subset G$ and norm $\norm{-}$ on
the Lie algebra $\g$.
In this subsection we translate the estimate for the CBH series to estimates on
$\exp$ and $\log$.
Recall our notation: $U(r)$ is the open ball of positive radius $r$ in $\g$, and
$\ol{U}(r)$ is the closed ball.

\begin{lem} \label{lem11}
There exist constants $\epsilon_2, c_2 \in \R$ such that 
\[ 0< \epsilon_2 \leq 1 \ \ \text{and} \ \ 1 \leq c_2 , \]
and such the following formulas hold for every 
$\alpha, \beta \in U(\epsilon_2)$.
\begin{eqnarray} \label{eqn:82}
\opn{exp}(\alpha) \cdot \opn{exp}(\beta) \in V_0(G) , \\
\label{eqn:86}
\norm{\, [\alpha, \beta]\, } \leq
c_{2} \cdot \norm{\alpha} \cdot \norm{\beta} \, ,
\end{eqnarray}
and
\begin{equation} \label{eqn:83}
\begin{aligned}
& \Norm{ \opn{log} \bigl( \opn{exp}(\alpha) \cdot 
\opn{exp}(\beta) \bigr) -  
( \alpha + \beta + \smfrac{1}{2}[\alpha, \beta] ) } \\
& \qquad
\leq c_{2} \cdot \norm{\alpha} \cdot \norm{\beta} \cdot 
\bigl( \norm{\alpha} + \norm{\beta} \bigr) \, .
\end{aligned}
\end{equation}
\end{lem}

\begin{proof}
The existence of a positive $\epsilon_2$ making formula (\ref{eqn:82}) true
is easy (due to continuity). We can assume 
$\epsilon_2 \leq \min (1, \epsilon_1)$.

The existence of a number $c_2$ making inequality (\ref{eqn:86}) valid is also
easy, because the function
$\norm{\, [\alpha, \beta]\, }$ is bilinear and $\g$ is finite dimensional. 
We can assume that $c_2 \geq \max(1, c_1)$. 

Finally, the validity of formula (\ref{eqn:83}) comes from Lemma \ref{lem:35}
and the CBH formula (\ref{eqn:87}).
\end{proof}

\begin{lem} \label{lem:33}
There is a real number $c_3$ such that $c_3 \geq c_2$, and for any 
$\alpha_1, \alpha_2, \beta_1, \beta_2 \in \g$ satisfying
$\norm{\alpha_i} + \norm{\beta_i} < \epsilon_2$ the inequality 
\begin{equation} \label{eqn:84}
\begin{aligned}
& \Norm{ \opn{log} \bigl( \opn{exp}(\alpha_1 + \beta_1) \cdot 
\opn{exp}(\alpha_2 + \beta_2) \bigr) -  
\opn{log} \bigl( \opn{exp}(\alpha_1) \cdot 
\opn{exp}(\alpha_2) \bigr) } \\
& \qquad \leq c_{3} \cdot (\norm{\beta_1} + \norm{\beta_2} ) .
\end{aligned} 
\end{equation}
holds.
\end{lem}

\begin{proof}
Given $\alpha \in \ol{U}(\epsilon_2)$ we have a smooth
(in fact analytic) function
$f_{\alpha} : U(\epsilon_2) \to \g$ defined by
\[ f_{\alpha}(\beta) := \log 
\bigl( \opn{exp}(\alpha) \cdot \opn{exp}(\beta) \bigr) . \]
Let $\g^*$ be the dual space of $\g$, and let 
$\d_{\beta} f_{\alpha} \in \g^*$
be the derivative of $f_{\alpha}$ at $\beta$. 
Then Taylor expansion of $f_{\alpha}$ around $\beta$ gives us
\begin{equation} \label{eqn15}
f_{\alpha}(\beta + \gamma) = f_{\alpha}(\beta) +
(\d_{\beta} f_{\alpha})(\gamma) +
(\opn{R}_{\beta}^{\geq 2} f_{\alpha})(\gamma)
\end{equation}
for any $\gamma \in \ol{U}(\epsilon_2)$. 
Here $\opn{R}_{\beta}^{\geq 2} f_{\alpha}$ is just the remainder, in other
words the higher order part of the Taylor expansion. 
The linear term 
$(\d_{\beta} f_{\alpha})(\gamma)$ 
can be estimated by
\[ \norm{ (\d_{\beta} f_{\alpha})(\gamma)  } \leq
a_1 \cdot \norm{\gamma}  \]
for a suitable positive number $a_1$. 
The remainder $(\opn{R}_{\beta}^{\geq 2} f_{\alpha})(\gamma)$ can be
estimated by
\[ \norm{ (\opn{R}_{\beta}^{\geq 2} f_{\alpha})(\gamma) } 
\leq a_2 \cdot \norm{\gamma}^2  \]
for some positive number $a_2$. 

These inequalities tell us that if 
$\alpha, \beta, \gamma \in U(\epsilon_2)$ then 
\begin{equation} \label{eqn17}
\begin{aligned}
& \Norm{ \opn{log} \bigl( \opn{exp}(\alpha) \cdot 
\opn{exp}(\beta + \gamma) \bigr) -  
\opn{log} \bigl( \opn{exp}(\alpha) \cdot 
\opn{exp}(\beta) \bigr) } \\
& \qquad
= \Norm{ \opn{log} \bigl( f_{\alpha}(\beta + \gamma) \bigr) -  
\opn{log} \bigl( f_{\alpha}(\beta) \bigr) } \\
& \qquad
\leq a_1 \cdot \norm{\gamma} + a_2 \cdot \norm{\gamma}^2
\leq (a_1 + a_2) \cdot \norm{\gamma} . 
\end{aligned} 
\end{equation}

In a similar way we look at the function 
$g_{\beta}(\alpha) := f_{\alpha}(\beta)$, 
and we obtain constants $b_1, b_2$ such that
\begin{equation}  \label{eqn18}
\Norm{ \opn{log} \bigl( \opn{exp}(\alpha + \gamma) \cdot 
\opn{exp}(\beta) \bigr) -  
\opn{log} \bigl( \opn{exp}(\alpha) \cdot 
\opn{exp}(\beta) \bigr) } 
\leq (b_1 + b_2) \cdot \norm{\gamma} .
\end{equation}

Therefore by taking
\[ c_3 := \max \{ c_2, a_1 + a_1, b_1 + b_2  \} \]
all conditions are satisfied.
\end{proof}

\begin{lem} \label{lem16}
There exists a constant $c_4 \geq c_3$ such that the following inequalities
hold for any $\alpha, \beta \in \g$ satisfying
$\norm{\alpha}, \norm{\beta} < \epsilon_2$.
\[ \Norm{ \opn{log} \bigl( \opn{exp}(\alpha) \cdot 
\opn{exp}(\beta) \bigr) -  ( \alpha + \beta ) } \leq
c_{4} \cdot \norm{\alpha} \cdot \norm{\beta} \]
and
\[ \Norm{ \opn{log} \bigl( \opn{exp}(\alpha) \cdot \opn{exp}(\beta) \bigr) }
\leq \norm{\alpha} + \norm{\beta} +
c_{4} \cdot \norm{\alpha} \cdot \norm{\beta} \, . \]
\end{lem}

\begin{proof}
{}From Lemma \ref{lem11} and the triangle inequality we get
\[ \begin{aligned}
& \Norm{ \opn{log} \bigl( \opn{exp}(\alpha) \cdot 
\opn{exp}(\beta) \bigr) -  
( \alpha + \beta ) } \\
& \qquad \leq c_{2} \cdot \norm{\alpha} \cdot \norm{\beta} \cdot 
\bigl( \norm{\alpha} + \norm{\beta} \bigr)  +
\smfrac{1}{2} c_2 \cdot \norm{\alpha} \cdot \norm{\beta} \\
& \qquad \leq 3 c_2  \cdot \norm{\alpha} \cdot \norm{\beta} \, .  
\end{aligned} \]
This proves the first inequality, with
$c_4 := \max (c_3, 3 c_2)$.
The second inequality is an immediate consequence.
\end{proof}

\subsection{Calculations with Many Elements}
Let us define
\begin{equation} \label{eqn:101}
\epsilon_4 :=  \smfrac{1}{12} c_4^{-1} \cdot \epsilon_2 \, . 
\end{equation}

\begin{lem} \label{lem17}
Let $\alpha_1, \ldots, \alpha_m \in \g$ be such that
$\sum_{i = 1}^m\, \norm{\alpha_i} < \epsilon_4$.
Then the relation and inequalities below hold:
\begin{equation} \label{eqn:90}
\prod_{i=1}^m\, \opn{exp}(\alpha_i) \in V_0(G) ,
\end{equation}
\begin{equation} \label{eqn:91}
\Norm{ \opn{log} \bigl( \prod\nolimits_{i=1}^m
\opn{exp}(\alpha_i) \bigr) } \leq 
\smfrac{3}{2} \sum_{i = 1}^m\, \norm{\alpha_i} 
\end{equation}
and
\begin{equation} \label{eqn:92}
\Norm{ \opn{log} \bigl( \prod\nolimits_{i=1}^m
\opn{exp}(\alpha_i) \bigr) - \sum\nolimits_{i=1}^m \alpha_i } \leq
c_4 \cdot \bigl( \sum\nolimits_{i = 1}^m\, \norm{\alpha_i} \bigr)^2 \, .
\end{equation}
\end{lem}

\begin{proof}
Recall that $\epsilon_2 \leq 1$ and $c_4 \geq 1$.

For $m = 1$ all is clear. Take $m = 2$. Then since
$\norm{\alpha_i} < \epsilon_2$ we have 
\[ \opn{exp}(\alpha_1) \cdot \opn{exp}(\alpha_2) \in V_0(G) \]
by Lemma \ref{lem11}. Next, since
$\norm{\alpha_1} \leq \smfrac{1}{6} c_4^{-1}$, it follows that
\[ c_4 \cdot \norm{\alpha_1} \cdot \norm{\alpha_2} \leq 
\smfrac{1}{6}  \norm{\alpha_2} \, . \]
{}From this and Lemma \ref{lem16} we see that
\[ \Norm{ \opn{exp}(\alpha_1) \cdot \opn{exp}(\alpha_2) } \leq 
\norm{\alpha_1} + \norm{\alpha_2} + c_4 \cdot \norm{\alpha_1} \cdot
\norm{\alpha_2} 
\leq \smfrac{3}{2} \cdot (\norm{\alpha_1} + \norm{\alpha_2}) \, . \]
Hence (\ref{eqn:91}) holds for $m = 2$.
Again using Lemma \ref{lem16}  we have
\[ \begin{aligned}
& \Norm{ \opn{log} \bigl( \opn{exp}(\alpha_1) \cdot 
\opn{exp}(\alpha_2) \bigr) -  
( \alpha_1 + \alpha_2 ) } \\
& \qquad \leq
c_{4} \cdot \norm{\alpha_1} \cdot \norm{\alpha_2} 
\leq c_{4} \cdot (\norm{\alpha_1} + \norm{\alpha_2})^2 .
\end{aligned}  \]
This finishes the case $m = 2$.

Now assume the assertions are true for $m \geq 2$, and consider 
$\alpha_1, \ldots, \lb \alpha_{m+1} \in \g$ such that
$\sum_{i = 1}^{m+1}\, \norm{\alpha_i} < \epsilon_4$.
Define
\[ \beta := \opn{log} \bigl( \prod\nolimits_{i=1}^{m}
\opn{exp}(\alpha_i) \bigr)  , \]
so that
\[ \prod_{i=1}^{m+1} \opn{exp}(\alpha_i) =
\opn{exp}(\beta) \cdot \opn{exp}(\alpha_{m+1}) . \]
By the induction hypothesis, i.e.\ inequality (\ref{eqn:91}) for $m$,
we have
\begin{equation} \label{eqn:93}
\norm{\beta} \leq \smfrac{3}{2} \sum_{i = 1}^m\, \norm{\alpha_i}
\leq
\smfrac{3}{2} \sum_{i = 1}^{m+1}\, \norm{\alpha_i}
\leq \smfrac{3}{2} \cdot \smfrac{1}{6}  c_4^{-1} \cdot \epsilon_2 =
\smfrac{1}{4} c_4^{-1} \cdot \epsilon_2 \, .
\end{equation}
This implies 
$\norm{\beta} < \epsilon_2$. Since we also have 
$\norm{\alpha_{m+1}} < \epsilon_2$, Lemma  \ref{lem11} says that
\[ \opn{exp}(\beta) \cdot \opn{exp}(\alpha_{m+1}) \in V_0(G) . \]
This verifies (\ref{eqn:90}) for $m+1$.

Using Lemma \ref{lem11} again we see that 
\[ \Norm{ \opn{log} \bigl( \opn{exp}(\beta) \cdot
\opn{exp}(\alpha_{m+1})
\bigr) } \leq 
\norm{\beta} + \norm{\alpha_{m+1}} + 
c_4 \cdot \norm{\beta} \cdot \norm{\alpha_{m+1}} \, . \]
Because we also have the inequality
$c_4 \cdot \norm{\beta} \leq \smfrac{1}{2}$, 
we conclude that
\[ \norm{\beta} + \norm{\alpha_{m+1}} + 
c_4 \cdot \norm{\beta} \cdot \norm{\alpha_{m+1}} \leq
\smfrac{3}{2} \cdot \norm{\beta} + \smfrac{3}{2} \cdot \norm{\alpha_{m+1}} \leq
\smfrac{3}{2} \cdot \sum_{i = 1}^{m+1}\, \norm{\alpha_i} \, . \]
Thus inequality (\ref{eqn:91}) holds for $m+1$.

According to Lemma \ref{lem16} we have
\[ \Norm{ \opn{log} \bigl( \opn{exp}(\beta) \cdot
\opn{exp}(\alpha_{m+1}) 
\bigr) - (\beta + \alpha_{m+1}) } \leq 
c_4 \cdot \norm{\beta} \cdot \norm{\alpha_{m+1}} . \]
The induction assumption for inequality (\ref{eqn:92}) says that
\[ \Norm{ \beta - \sum\nolimits_{i=1}^m \alpha_i } \leq 
c_4 \cdot \bigl( \sum\nolimits_{i = 1}^m\, \norm{\alpha_i} \bigr)^2 \, . \]
Combining these with the inequality (\ref{eqn:93}) we obtain
\[ \begin{aligned}
& \Norm{ \opn{log} \bigl( \opn{exp}(\beta) \cdot
\opn{exp}(\alpha_{m+1}) 
\bigr) -  \sum\nolimits_{i=1}^{m+1} \alpha_i } \\
& \quad \leq 
\Norm{ \opn{log} \bigl( \opn{exp}(\beta) \cdot
\opn{exp}(\alpha_{m+1}) 
\bigr) - (\beta + \alpha_{m+1}) } \\
& \quad \quad +
\Norm{ (\beta + \alpha_{m+1}) - \sum\nolimits_{i=1}^{m+1} \alpha_i } \\
& \quad \leq c_4 \cdot \norm{\beta} \cdot \norm{\alpha_{m+1}} +
c_4 \cdot \bigl( \sum\nolimits_{i = 1}^m\, \norm{\alpha_i} \bigr)^2 \\
& \quad \leq c_4 \cdot \Bigl(
\bigl( \smfrac{3}{2} \sum\nolimits_{i = 1}^m\, \norm{\alpha_i} \bigr)
\cdot \norm{\alpha_{m+1}} + 
\bigl( \sum\nolimits_{i = 1}^m\, \norm{\alpha_i} \bigr)^2 \Bigr)  \\
& \quad \leq c_4 \cdot 
\bigl( \sum\nolimits_{i = 1}^{m+1}\, \norm{\alpha_i} \bigr)^2 \, .
\end{aligned} \]
This proves (\ref{eqn:92}).
\end{proof}

\begin{lem} \label{lem18}
Let 
$\alpha_1, \ldots, \alpha_m, \beta_1, \ldots, \beta_m \in \g$ be
such that
\[ \sum_{i = 1}^m\, \bigl( \norm{\alpha_i} + \norm{\beta_i} \bigr)
< \epsilon_4 \, . \]
Then
\[ \begin{aligned}
& \Norm{ \opn{log} \Bigl( \prod\nolimits_{i=1}^m
\bigl( \exp(\alpha_i) \cdot \exp(\beta_i) \bigr) \Bigr) \\
& \qquad \qquad -
\opn{log} \Bigl( \bigl( \prod\nolimits_{i=1}^m \opn{exp}(\alpha_i)
\bigr)
\cdot \bigl( \prod\nolimits_{i=1}^m \opn{exp}(\beta_i) \bigr) \Bigr) }
\\
& \qquad \leq 4 c_4^2 \cdot 
\bigl( \sum\nolimits_{i = 1}^m\, \norm{\alpha_i} \bigr) \cdot  
\bigl( \sum\nolimits_{i = 1}^m\, \norm{\beta_i} \bigr) 
\end{aligned} \]
and
\[ \begin{aligned}
& \Norm{ \opn{log} \bigl( \prod\nolimits_{i=1}^m
\opn{exp}(\alpha_i + \beta_i) \bigr) -
\opn{log} \Bigl( \prod\nolimits_{i=1}^m
\bigl( \exp(\alpha_i) \cdot \exp(\beta_i) \bigr) \Bigr) } \\
& \qquad \leq 2 c_4^2 \cdot \sum\nolimits_{i = 1}^m\, 
\norm{\alpha_i} \cdot \norm{\beta_i} .
\end{aligned} \]
\end{lem}

\begin{proof}
For the first inequality we note that for any $i,j$ one has
\[ \begin{aligned}
& \Norm{ \log \bigl( \exp(\beta_j) \cdot \exp(\alpha_i) \bigr) - 
\log \bigl( \exp(\alpha_i) \cdot \exp(\beta_j) \bigr) } \\
& \qquad \leq 2 c_2 \cdot \norm{\alpha_i} \cdot \norm{\beta_j} \cdot
\bigl( \norm{\alpha_i} + \norm{\beta_j} \bigr) +
c_2 \cdot \norm{\alpha_i} \cdot \norm{\beta_j} \\
& \qquad \leq 2 c_2 \cdot \norm{\alpha_i} \cdot \norm{\beta_j} .
\end{aligned} \]
This is due to Lemma \ref{lem11}.
Now for our inequality we have to move all the $\exp(\beta_j)$ in the
product across all the $\exp(\alpha_i)$. According Lemma \ref{lem:33} the
``cost'' of each such move is at most
$2 c_3 \cdot 2 c_2 \cdot \norm{\alpha_i} \cdot \norm{\beta_j}$.
So the total ``cost'' is at most
\[ \sum\nolimits_{i, j = 1}^m\, 2 c_3 \cdot 
2 c_2 \cdot \norm{\alpha_i} \cdot \norm{\beta_j} =
4 c_2 c_3 \cdot 
\bigl( \sum\nolimits_{i = 1}^m\, \norm{\alpha_i} \bigr) \cdot  
\bigl( \sum\nolimits_{j = 1}^m\, \norm{\beta_j} \bigr) . \]
Since $c_4 \geq c_2, c_3$ this proves the first inequality.

For the second inequality we note that, due to Lemma \ref{lem16}, we
have
\[ \Norm{ \log \bigl( \exp(\alpha_i + \beta_i) \bigr) -
\log \bigl( \exp(\alpha_i) \cdot \exp(\beta_i) \bigr) } 
\leq c_4 \cdot \norm{\alpha_i} \cdot \norm{\beta_i}  . \]
Since we have to make such a change for every $i$, according to
Lemma \ref{lem:33} the total ``cost'' is at most
$2 c_3 c_4 \cdot \sum\nolimits_{i = 1}^m\, 
\norm{\alpha_i} \cdot \norm{\beta_i}$.
\end{proof}

\begin{lem} \label{lem:36}
There exists a constant $c_5$ such that 
$c_5 \geq \max( c_4, \smfrac{3}{2} )$, 
and for every   
$\alpha_1, \ldots, \alpha_m, \beta_1, \ldots, \beta_m \in \g$ 
satisfying
\[ \sum_{i = 1}^m\, \bigl( \norm{\alpha_i} + \norm{\beta_i} \bigr)
< \epsilon_4 \]
one has:
\begin{equation} \label{eqn:94}
\begin{aligned}
& \Norm{ \opn{log} 
\bigl( \prod\nolimits_{i=1}^m \opn{exp}(\alpha_i + \beta_i) \bigr) -
\opn{log} \bigl(   \prod\nolimits_{i=1}^m \opn{exp}(\alpha_i) \bigr) -
\sum\nolimits_{i=1}^m \beta_i } \\
& \qquad 
\leq c_5 \cdot \bigl( \sum\nolimits_{i=1}^m (\norm{\alpha_i}  +
\norm{\beta_i}) \bigr) \cdot 
\bigl( \sum\nolimits_{i=1}^m \norm{\beta_i} \bigr) 
\end{aligned} 
\end{equation}
and
\begin{equation} \label{eqn:95}
\begin{aligned}
& \Norm{ \opn{log} 
\bigl( \prod\nolimits_{i=1}^m \opn{exp}(\alpha_i + \beta_i) \bigr) -
\opn{log} 
\bigl( \prod\nolimits_{i=1}^m \opn{exp}(\alpha_i) \bigr) } \\
& \qquad 
\leq c_5 \cdot 
\bigl( \sum\nolimits_{i=1}^m \norm{\beta_i} \bigr) \, .
\end{aligned} 
\end{equation}
If moreover $[\alpha_i, \alpha_j] = 0$ for all $i, j$, then
\begin{equation} \label{eqn:96}
\begin{aligned}
& \Norm{ \opn{log} 
\bigl( \prod\nolimits_{i=1}^m \opn{exp}(\alpha_i + \beta_i) \bigr) -
\sum\nolimits_{i=1}^m (\alpha_i + \beta_i) } \\
& \qquad 
\leq c_5 \cdot \bigl( \sum\nolimits_{i=1}^m (\norm{\alpha_i}  +
\norm{\beta_i}) \bigr) \cdot 
\bigl( \sum\nolimits_{i=1}^m \norm{\beta_i} \bigr) .
\end{aligned} 
\end{equation}
\end{lem}

\begin{proof}
According to Lemma \ref{lem18} we have
\begin{equation} \label{eqn:97}
\begin{aligned}
& \Norm{ \opn{log} \bigl( \prod\nolimits_{i=1}^m
\opn{exp}(\alpha_i + \beta_i) \bigr) -
\opn{log} \Bigl( \bigl( \prod\nolimits_{i=1}^m \opn{exp}(\alpha_i)
\bigr)
\cdot \bigl( \prod\nolimits_{i=1}^m \opn{exp}(\beta_i) \bigr) \Bigr) }
\\
& \qquad
\leq 2 c_4^2 \cdot 
\sum_{i = 1}^m\, \norm{\alpha_i} \cdot \norm{\beta_i} \, . 
\end{aligned}
\end{equation}
By Lemma \ref{lem17} we know that 
\begin{equation} \label{eqn:98}
\Norm{ \opn{log} \bigl( \prod\nolimits_{i=1}^m
\opn{exp}(\beta_i) \bigr) - \sum\nolimits_{i=1}^m \beta_i } \leq
c_4 \cdot \bigl( \sum\nolimits_{i = 1}^m\, \norm{\beta_i} \bigr)^2 \, .
\end{equation}
Combining Lemmas \ref{lem16} and \ref{lem17} we get
\begin{equation} \label{eqn:99}
\begin{aligned}
& \Norm{ 
\opn{log} \Bigl( \bigl( \prod\nolimits_{i=1}^m \opn{exp}(\alpha_i)
\bigr)
\cdot \bigl( \prod\nolimits_{i=1}^m \opn{exp}(\beta_i) \bigr) \Bigr) \\
& \qquad \qquad  -
\Bigl( \opn{log} \bigl( \prod\nolimits_{i=1}^m \opn{exp}(\alpha_i) \bigr) +
\opn{log} \bigl( \prod\nolimits_{i=1}^m \opn{exp}(\beta_i) \bigr)
\Bigr) } \\
& \qquad \leq c_4 \cdot 
\Norm{ \opn{log} \bigl( \prod\nolimits_{i=1}^m \opn{exp}(\alpha_i) \bigr) }
\cdot 
\Norm{ \opn{log} \bigl( \prod\nolimits_{i=1}^m \opn{exp}(\beta_i) \bigr) } \\
& \qquad \leq c_4 \cdot 
\bigl( \smfrac{3}{2} \, \sum\nolimits_{i = 1}^m\, \norm{\alpha_i} \bigr)
\cdot 
\bigl( \smfrac{3}{2} \, \sum\nolimits_{i = 1}^m\, \norm{\beta_i} \bigr) \\
& \qquad = \smfrac{9}{4} c_4 \cdot 
\bigl( \sum\nolimits_{i = 1}^m\, \norm{\alpha_i} \bigr)
\cdot 
\bigl( \sum\nolimits_{i = 1}^m\, \norm{\beta_i} \bigr) \, . 
\end{aligned}
\end{equation}
Putting all these together we obtain the inequality
\begin{equation} \label{eqn:100}
\begin{aligned}
& \Norm{ \opn{log} 
\bigl( \prod\nolimits_{i=1}^m \opn{exp}(\alpha_i + \beta_i) \bigr) -
\opn{log} \bigl(   \prod\nolimits_{i=1}^m \opn{exp}(\alpha_i) \bigr) -
\sum\nolimits_{i=1}^m \beta_i } \\
& \qquad 
\leq 2 c_4^2 \cdot 
\bigl( \sum\nolimits_{i = 1}^m\, \norm{\alpha_i} \bigr)
\cdot 
\bigl( \sum\nolimits_{i = 1}^m\, \norm{\beta_i} \bigr) \\
& \qquad \qquad + 
 \smfrac{9}{4} c_4^2 \cdot 
\bigl( \sum\nolimits_{i = 1}^m\, \norm{\alpha_i} \bigr)
\cdot 
\bigl( \sum\nolimits_{i = 1}^m\, \norm{\beta_i} \bigr) \\
& \qquad \qquad +
c_4 \cdot 
\bigl( \sum\nolimits_{i = 1}^m\, \norm{\beta_i} \bigr)^2 \\
& \qquad 
\leq 6 c_4^2  
\cdot \bigl( \sum\nolimits_{i=1}^m (\norm{\alpha_i}  +
\norm{\beta_i}) \bigr) \cdot 
\bigl( \sum\nolimits_{i=1}^m \norm{\beta_i} \bigr) \, . 
\end{aligned} 
\end{equation}
This proves inequality (\ref{eqn:94}) with
$c_5 := 7 c_4^2$.

Inequality (\ref{eqn:95}) follows easily from (\ref{eqn:100}), once we notice
that
\[ 6 c_4^2  
\cdot \bigl( \sum\nolimits_{i=1}^m (\norm{\alpha_i}  +
\norm{\beta_i}) \bigr) \leq 
6 c_4^2 \cdot 2 \epsilon_3 \leq c_4 \, . \]
When $[\alpha_i, \alpha_j] = 0$ for all $i, j$ we have
\[ \opn{log} \bigl(   \prod\nolimits_{i=1}^m \opn{exp}(\alpha_i) \bigr) =
\sum\nolimits_{i = 1}^m\, \alpha_i . \]
So inequality (\ref{eqn:95}) holds.
\end{proof}

\subsection{Final Touches} \label{subsec:touches}

\begin{lem} \label{lem:34}
There is a real number $c_6$ such that $c_6 \geq c_5$, and for any 
$\alpha, \beta \in \g$ with
$\norm{\alpha}, \norm{\beta} < \epsilon_4$ the inequality
\begin{equation} \label{eqn:85}
\begin{aligned}
& \Norm{ \opn{log} \bigl( \opn{exp}(\alpha) \cdot 
\opn{exp}(\beta) \cdot \opn{exp}(\alpha)^{-1} \cdot
\opn{exp}(\beta)^{-1} \bigr) -  
[\alpha, \beta] } \\
& \hspace{10ex}
\leq c_{6} \cdot \norm{\alpha} \cdot \norm{\beta} \cdot 
\bigl( \norm{\alpha} + \norm{\beta} \bigr) .
\end{aligned} 
\end{equation}
holds.
\end{lem}

\begin{proof}
Let's write
\[ a := \opn{exp}(\alpha), \quad b := \opn{exp}(\beta) , \]
\[ \gamma^+ := \alpha + \beta + \smfrac{1}{2} [\alpha, \beta], \quad
\gamma^- := -\alpha - \beta + \smfrac{1}{2} [\alpha, \beta], \]
\[ \delta^+ := \log(a \cdot b) - \gamma^+, \quad
\delta^- := \log(a^{-1} \cdot b^{-1})  - \gamma^- .  \]
According to Lemma \ref{lem11} we have
\[ \norm{\delta^+}, \norm{\delta^-} \leq 
c_{2} \cdot \norm{\alpha} \cdot \norm{\beta} \cdot 
\bigl( \norm{\alpha} + \norm{\beta} \bigr) \,  \]
and
\[ \norm{\gamma^+}, \norm{\gamma^-} \leq 
\norm{\alpha} + \norm{\beta} +
\smfrac{1}{2} c_{2} \cdot \norm{\alpha} \cdot \norm{\beta}
\leq (1 + c_2)(\norm{\alpha} + \norm{\beta}) \ . \]

Now
\[ a \cdot b = \exp (\gamma^+ + \delta^+) \]
and
\[ a^{-1} \cdot b^{-1} = \exp (\gamma^- + \delta^-) . \]
So by Lemma \ref{lem:33} we have
\[ \begin{aligned}
& \Norm{ \opn{log} \bigl( (a \cdot b) \cdot (a^{-1} \cdot b^{-1}) \bigr)
- \opn{log} \bigl( \exp (\gamma^+) \cdot \exp (\gamma^-) \bigr) } \\
& \qquad \leq c_{3} \cdot (\norm{\delta^+} + \norm{\delta^-} ) \\
& \qquad \leq 2 c_2 c_3 
\cdot \norm{\alpha} \cdot \norm{\beta}
\cdot \bigl( \norm{\alpha} + \norm{\beta} \bigr) \, .
\end{aligned}  \]

On the other hand, since $\alpha + \beta$ and $-\alpha - \beta$ commute,
by formula (\ref{eqn:96}) in Lemma \ref{lem:36} we have 
\[ \begin{aligned}
& \Norm{ \opn{log} \bigl( \exp (\gamma^+) \cdot \exp (\gamma^-) \bigr) -
[\alpha, \beta] } \\
& \qquad = 
\Norm{ \opn{log} 
\bigl( \exp((\alpha + \beta) + \smfrac{1}{2} [\alpha, \beta]) \cdot 
\exp((-\alpha - \beta) + \smfrac{1}{2} [\alpha, \beta]) \bigr) \\
& \qquad \qquad - \bigl( (\alpha + \beta) + \smfrac{1}{2} [\alpha, \beta] +
(-\alpha - \beta) + \smfrac{1}{2} [\alpha, \beta] \bigr) } \\
& \qquad \leq c_5 \cdot \bigl( 2 \cdot \norm{\alpha} + 
2 \cdot \norm{\beta} + \norm{ [\alpha, \beta] } \bigr) \cdot 
\bigl( \norm{ [\alpha, \beta] } \bigr) \\
& \qquad \leq c_5 \cdot (2 + c_2 \cdot \epsilon_4) \cdot c_2 
\cdot \norm{\alpha} \cdot \norm{\beta}
\cdot \bigl( \norm{\alpha} + \norm{\beta} \bigr) \, . 
\end{aligned} \]
In the last inequality we used
\[  \norm{ [\alpha, \beta] } \leq c_2 \cdot \norm{\alpha} \cdot \norm{\beta} 
\leq c_2 \cdot \epsilon_4 \cdot \norm{\beta} \, . \]
Therefore (\ref{eqn:85}) holds with 
\[ c_6 := 2 c_2 c_3 + c_5 \cdot (2 + c_2 \epsilon_4) \cdot c_2 \, . \]
\end{proof}

\begin{proof}[Proof of Theorem \tup{\ref{thm:6}}]
Take
$c_0(G) := c_6$ and $\epsilon_0(G) := \epsilon_4$.
Since $\epsilon_4 \leq \epsilon_2$ and
\[ c_6 \geq c_5 \geq \max( c_4, \smfrac{3}{2} ) 
\geq c_4 \geq c_3 \geq c_2 \, , \]
the assertions of the theorem 
are contained on Lemmas \ref{lem11}, \ref{lem17}, \ref{lem:36}
and \ref{lem:34}.  
\end{proof}

\cleardoublepage
\section{Multiplicative Integration in Dimension One}
\label{sec:dim1}

Nonabelian multiplicative integration in dimension $1$ is classical, dating back
to work of Volterra (cf.\ \cite{DF}).
In modern differential geometry it is
usually viewed as the holonomy of a connection along a path. We prefer
to do everything from scratch, for several reasons: this allows us to introduce
notation; it serves as a warm-up for the much more difficult
$2$-dimensional integration; and also to cover the case of piecewise smooth
differential forms. 

\subsection{Binary Tessellations}
Recall that $\mbf{I}^1$ is the unit line segment in 
$\mbf{A}^{1} = \mbf{A}^{1}(\R)$, which we view as an oriented
polyhedron (cf.\ Section \ref{sec:pws}). The vertices
(endpoints) of $\mbf{I}^1$ are $v_0 = 0$ and $v_1 = 1$, and the coordinate
function is $t_1$. 

Let $X$ be a polyhedron. 
Consider a linear map $\sigma : \mbf{I}^{1} \to X$.
The length of the line segment $Z := \sigma(\mbf{I}^{1}) \subset X$
(possibly zero) is denoted
by $\opn{len}(\sigma)$. If  $\sigma$ is not constant then it is a conformal map,
and the scaling factor is precisely $\opn{len}(\sigma)$. In this case $\sigma$ 
determines an orientation on the $1$-dimensional polyhedron $Z$. We may choose
an orthonormal linear coordinate
function $f_1$ on $Z$, such that $\d f_1$ is the orientation of $Z$. Then 
$\sigma^*(\d f_1) = \opn{len}(\sigma) \cdot \d t_1$.

Next consider a piecewise linear map $\sigma : \mbf{I}^{1} \to X$.
By definition (cf.\ Subsection \ref{subsec:pws})
there is a linear triangulation $\{ Y_j \}_{j \in J}$ of $\mbf{I}^{1}$ such
that 
$\sigma|_{Y_j} : Y_j \to X$ 
is a linear map for every $j \in J$. We define
\[ \opn{len}(\sigma) := \sum_{j \in J_1} \opn{len}(\sigma_j) . \]

It is convenient to have a composition operation for piecewise linear maps.
Suppose $\sigma : \mbf{I}^1 \to X$ and
$\rho: \mbf{I}^1 \to \mbf{I}^1$ are piecewise linear maps.
Then the set-theoretical composition
$\sigma \circ \rho$ is also a piecewise linear map $\mbf{I}^1 \to X$.
Given finite sequences 
$\bsym{\sigma} = ( \sigma_i )_{i = 1, \ldots, m}$ and 
$\bsym{\rho} = ( \rho_j )_{j = 1, \ldots, n}$ of piecewise linear maps
$\sigma_i : \mbf{I}^{1} \to X$ and $\rho_j : \mbf{I}^{1} \to \mbf{I}^{1}$,
we define the sequence of piecewise linear maps
\[ \bsym{\sigma} \circ \bsym{\rho} := 
( \sigma_i \circ \rho_j )_{(i, j) \in \{ 1, \ldots, m \} \times
\{ 1, \ldots, n \} } \]
in lexicographical order, i.e.\ 
\begin{equation} \label{eqn:19}
\bsym{\sigma} \circ \bsym{\rho} = 
(\sigma_1 \circ \rho_1,\, \sigma_1 \circ \rho_2,\, \cdots, 
\sigma_2 \circ \rho_1,\, \sigma_2 \circ \rho_2,\, \cdots,
\sigma_m \circ \rho_n ) \ .
\end{equation}

\begin{dfn} \label{dfn:41}
For any $k \geq 0$, the {\em $k$-th binary tessellation} of $\mbf{I}^1$ is the
sequence
\[ \opn{tes}^k \mbf{I}^1 = ( \sigma^k_1, \ldots, \sigma^k_{2^k} ) \]
of linear maps in $\sigma^k_i : \mbf{I}^{1} \to \mbf{I}^1$ defined recursively
as follows.
\begin{itemize}
\item For $k = 0$ we define $\sigma^0_1$ to be the identity map of $\mbf{I}^1$.
\item For $k = 1$ we take the linear maps
$\sigma^1_1, \sigma^1_2$ defined on vertices by
\[ \sigma^1_1(v_0, v_1) := (v_0, \smfrac{1}{2}) , \]
\[ \sigma^1_2(v_0, v_1) := (\smfrac{1}{2}, v_1) . \]
\item For $k \geq 1$ we define
\[ \opn{tes}^{k+1} \mbf{I}^1 := 
(\opn{tes}^1 \mbf{I}^1) \circ (\opn{tes}^k \mbf{I}^1) , \]
using the convention (\ref{eqn:19}).
\end{itemize}
\end{dfn}

We call $\sigma^0_1$ the {\em basic map}.

See Figure \ref{fig:40} for an illustration of $\opn{tes}^2 \mbf{I}^1$.

\begin{figure}
\includegraphics[scale=0.35]{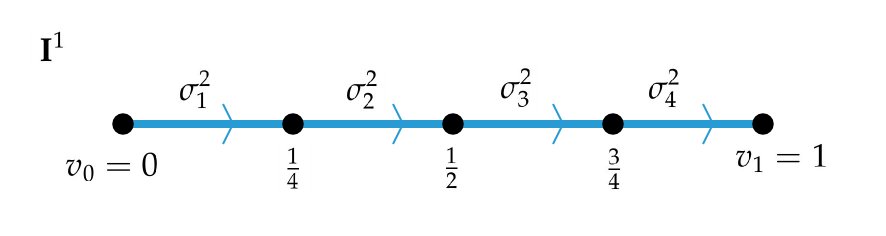}
\caption{The $2$-nd binary tessellation of $\mbf{I}^{1}$. The arrowheads 
indicate the orientations of the linear maps 
$\sigma^2_i : \mbf{I}^{1} \to \mbf{I}^{1}$.} 
\label{fig:40}
\end{figure}

\subsection{Riemann Products}
Let $\g$ be a finite dimensional Lie algebra over $\R$. 
For any $n$-dimensional polyhedron $X$ we then have the DG Lie algebra  
of piecewise smooth $\g$-valued differential forms
\[ \Omega_{\mrm{pws}}(X) \otimes \mfrak{g} = \bigoplus_{p=0}^n \,
\Omega^p_{\mrm{pws}}(X) \otimes \mfrak{g}  \]
(Subsection \ref{subsec:pws}). 
The operations are as follows: for 
$\alpha_i \in \Omega^{p_i}_{\mrm{pws}}(X)$ and
$\gamma_i \in \g$ one has
\[ \d(\alpha_1 \otimes \gamma_1) = \d(\alpha_1) \otimes \gamma_1 
\in \Omega_{\mrm{pws}}^{p_1+1}(X) \otimes \g \]
and
\[ [ \alpha_1 \otimes \gamma_1,\, \alpha_2 \otimes \gamma_2 ] =
(\alpha_1 \wedge \alpha_2) \otimes [\gamma_1, \gamma_2] \in 
\Omega_{\mrm{pws}}^{p_1+ p_2}(X) \otimes \g . \]
By definition, any particular piecewise smooth differential form $\alpha$
belongs to 
$\Omega_{\mrm{pws}}(X; T) \otimes \mfrak{g}$, 
for some linear triangulation $T$ of $X$. 
This construction is functorial in the following sense. Suppose
$f : X \to Y$ is a piecewise linear map of polyhedra, and $\phi : \g \to \h$ is
map of Lie algebras. Then there is a homomorphism of DG Lie algebras
\[ f^* \otimes \phi : \Omega_{\mrm{pws}}(Y) \otimes \mfrak{g} \to 
\Omega_{\mrm{pws}}(X) \otimes \mfrak{h} . \]
In case $\h = \g$ and $\phi$ is the identity map, we shall often write
$f^*$ instead of $f^* \otimes \phi$.

{}From now on in this section we consider a Lie group $G$, with Lie
algebra $\mfrak{g}$, and a polyhedron $X$.
We fix some euclidean norm $\norm{-}$ on the vector space $\g$. As in
Section \ref{sec:pws} we also fix an open neighborhood $V_0(G)$ of $1$ in
$G$ on which $\log_G$ is defined, a convergence radius
$\epsilon_0(G)$, and a commutativity constant $c_0(G)$.
The choices of $V_0(G)$, $\epsilon_0(G)$ and $c_0(G)$ are auxiliary only; they
are needed for the proofs, but do not effect the results.

\begin{dfn} \label{dfn:1}
Let
$\alpha \in \Omega^1_{\mrm{pws}}(\mbf{I}^1) \otimes \mfrak{g}$. 
The {\em basic Riemann product of $\alpha$ on $\mbf{I}^1$} 
is the element 
\[ \opn{RP}_0(\alpha \vert \mbf{I}^1) \in G \]
defined as follows. Let 
$w := \smfrac{1}{2}$, namely the midpoint of $\mbf{I}^{1}$. 
\begin{itemize}
\item Suppose $w$ is a smooth point of $\alpha$. Then there is
some $1$- \lb dimensional simplex $Y$ in $\mbf{I}^{1}$, such that 
$w \in \opn{Int} Y$ and $\alpha|_{Y}$ is smooth. 
Let $\til{\alpha} \in \mcal{O}(Y) \otimes \mfrak{g}$ be the coefficient of
$\alpha|_Y$, as in Definition \ref{dfn:24}; namely
\[ \alpha|_Y = \til{\alpha} \cdot \d t_1 . \]
We define 
\[ \opn{RP}_{0}(\alpha \vert \mbf{I}^{1}) :=
\opn{exp}_G \bigl(  \til{\alpha}(w) \bigr) . \]
\item If $w$ is a singular point of $\alpha$, then we let
$\opn{RP}_{0}(\alpha \vert \mbf{I}^{1}) := 1$.
\end{itemize}
\end{dfn}

Observe that the element $\til{\alpha}(w) \in \g$ in the definition above is
independent of
the simplex $Y$. 

\begin{dfn} \label{dfn:12}
Let 
$\alpha \in \Omega^1_{\mrm{pws}}(X) \otimes \mfrak{g}$,
and let $\sigma : \mbf{I}^{1} \to X$ be a piecewise linear map. Then
$\sigma^*(\alpha) \in \Omega^1_{\mrm{pws}}(\mbf{I}^1) \otimes \mfrak{g}$,
and we define the 
{\em basic Riemann product of $\alpha$ along $\sigma$} to be 
\[ \opn{RP}_{0}(\alpha \vert \sigma) :=
\opn{RP}_{0} \bigl( \sigma^*(\alpha) \vert \mbf{I}^{1} \bigr) \in G . \]
\end{dfn}

Note that if $\sigma$ is a constant map then 
$\sigma^*(\alpha) = 0$, so 
$\opn{RP}_{0}(\alpha \vert \bsym{\sigma}) = 1$.

\begin{dfn} \label{dfn:2}
Let 
$\alpha \in \Omega^1_{\mrm{pws}}(X) \otimes \mfrak{g}$,
and let $\sigma : \mbf{I}^{1} \to X$ be a piecewise linear map.
For $k \geq 0$ we define the {\em $k$-th refined Riemann product of $\alpha$
along $\sigma$} to be 
\[ \opn{RP}_{k}(\alpha \vert \sigma) :=
\prod_{i = 1}^{2^k} \,
\opn{RP}_{0} ( \alpha \vert \sigma \circ \sigma^k_i ) \in G , \]
using the $k$-th binary tessellation
$\opn{tes}^k \mbf{I}^1 = ( \sigma^k_1, \ldots, \sigma^k_{2^k})$
and the convention (\ref{eqn:6}).
\end{dfn}

\subsection{Convergence of Riemann Products}
As before we are given a form
$\alpha \in \Omega^1_{\mrm{pws}}(X) \otimes \g$. 
Recall $\norm{\alpha}_{\mrm{Sob}}$, the Sobolev norm to order $2$ of $\alpha$,
from Subsection \ref{subsec:sob}.

\begin{lem} \label{lem:1}
There are constants $c_1(\alpha)$ and $\epsilon_1(\alpha)$ with the
following properties.
\begin{enumerate}
\rmitem{i} These inequalities hold:
\begin{gather*}
0 < \epsilon_1(\alpha) \leq 1 \, , \\
c_1(\alpha) \geq 1 \, , \\
\epsilon_1(\alpha) \cdot c_1(\alpha) \leq \smfrac{1}{2} \cdot
\epsilon_0(G) \ .
\end{gather*}

\rmitem{ii} For any piecewise linear map 
$\sigma : \mbf{I}^1 \to X$ 
such that $\opn{len}(\sigma) < \epsilon_1(\alpha)$, and for any sufficiently
large $k$, one has
\[ \opn{RP}_{k}(\alpha \vert \sigma) \in V_0(G)  \]
and 
\[ \norm{\, \opn{log}_G \bigl( 
\opn{RP}_{k}(\alpha \vert \sigma) \bigr) \, }
\leq c_1(\alpha) \cdot \opn{len}(\sigma) . \]

\rmitem{iii} Moreover, if the map $\sigma$ in \tup{(ii)} is linear, then the
assertions there hold for any $k \geq 0$.
\end{enumerate} 
\end{lem}

\begin{proof}
We are given a piecewise linear map $\sigma : \mbf{I}^1 \to X$.
Let us write $\epsilon := \opn{len}(\sigma)$.
Excluding the trivial case, we may assume that $\epsilon > 0$.

Take $k \geq 0$. For any index
$i \in \{ 1, \ldots, 2^k \}$ let
$W_i := \sigma^k_i(\mbf{I}^1)$ and 
$w_i := \sigma^k_i(\smfrac{1}{2})$; so
$w_i$ is the midpoint of the segment ($1$-dimensional polyhedron) $W_i$. 
Define
$\epsilon_i := \opn{len}(\sigma \circ \sigma^k_i)$,
$Z_i := \sigma(W_i)$ and $z_i := \sigma(w_i)$. 
Note that 
$\sum_{i} \epsilon_i = \epsilon$.

We will say that an index $i$ is good if the map
$\sigma|_{W_i}$ is linear and injective. 
In this case $Z_i$ is a segment of length $\epsilon_i$ and midpoint $z_i$. 
Otherwise we will call $i$ a bad index.
The sets of good and bad indices are denoted by $\opn{good}(k)$ and 
$\opn{bad}(k)$ respectively. Let $m$ be the number of singular points of the map
$\sigma$. Then $\abs{ \opn{bad}(k) } \leq m$. In particular, if $\sigma$ is
linear then $\opn{bad}(k) = \emptyset$.

Let
\[ \alpha' := \sigma^*(\alpha) \in \Omega^1_{\mrm{pws}}(\mbf{I}^1) \otimes \g .
\]
For an index $i$ we define an element
$\lambda_i \in \g$ as follows. If $w_i$ is a smooth point of $\alpha'$, then let
$\til{\alpha}'_i$ be the coefficient of $\alpha'$ near $w_i$, and let
\[ \lambda_i := (\smfrac{1}{2})^k \cdot \til{\alpha}'_i(w_i) . \]
Otherwise we let 
$\lambda_i := 0$.
In any case we have
\[ \exp_G(\lambda_i) = \opn{RP}_{0}(\alpha' \vert \sigma^k_i) = 
\opn{RP}_{0}(\alpha \vert \sigma \circ \sigma^k_i) . \]
Therefore
\begin{equation} \label{eqn:174}
\prod_{i = 1}^{2^k} \, \exp_G(\lambda_i) = 
\opn{RP}_{k}(\alpha \vert \sigma) .
\end{equation}
Furthermore, if $i$ is a good index and $w_i$ is a smooth point of $\alpha'$,
then $z_i$ is a smooth point $\alpha|_{Z_i}$, and then 
\begin{equation} \label{eqn:175}
\lambda_i = \epsilon_i \cdot \til{\alpha}_i(z_i) , 
\end{equation}
where $\til{\alpha}_i$ is the coefficient of $\alpha|_{Z_i}$ near $z_i$. 

Now here are the estimates. For any index $i$  we have 
\[ \norm{ \lambda_i } \leq (\smfrac{1}{2})^k \cdot \norm{ \alpha' }_{\mrm{Sob}}
\, . \]
If $i$ is a good index then by (\ref{eqn:175}) we have
\[ \norm{ \lambda_i } \leq \epsilon_i \cdot \norm{ \alpha }_{\mrm{Sob}} \, . \]
Hence
\begin{equation} \label{eqn:176}
\begin{aligned}
& \sum_{i = 1}^{2^k} \, \norm{ \lambda_i } = 
\sum_{i \in \opn{good}(k)} \, \norm{ \lambda_i } + 
\sum_{i \in \opn{bad}(k)} \, \norm{ \lambda_i } \\
& \qquad \leq 
\epsilon \cdot \norm{ \alpha }_{\mrm{Sob}} +
m \cdot (\smfrac{1}{2})^k \cdot \norm{ \alpha' }_{\mrm{Sob}} \, . 
\end{aligned}
\end{equation}

Let
\[ \epsilon_1(\alpha) := \smfrac{1}{2} \cdot 
(1 + \norm{ \alpha }_{\mrm{Sob}})^{-1} \cdot \epsilon_0(G) \]
and
\[ c_1(\alpha) := 2 \cdot \bigl( 1 + c_0(G) \cdot \norm{ \alpha }_{\mrm{Sob}}
\bigr) \, . \]
Choose $k_0$ large enough so that 
\[ m \cdot (\smfrac{1}{2})^{k_0} \cdot \norm{ \alpha' }_{\mrm{Sob}} < 
\min \bigl( \smfrac{1}{2} \cdot \epsilon_0(G), \,
\smfrac{1}{2} \cdot c_1(\alpha) \cdot \epsilon \cdot 
(1 + c_0(G))^{-1} \bigr) \, . \]
If $\sigma$ is linear then $m = 0$, and we may take $k_0 := 0$.

Now suppose that $\epsilon < \epsilon_1(\alpha)$ and $k \geq k_0$.
According to inequality (\ref{eqn:176}) we have
\[ \sum_{i = 1}^{2^k} \, \norm{ \lambda_i } < \epsilon_0(G) \, . \]
Therefore by property (ii) of Theorem \ref{thm:6} and by formula
(\ref{eqn:174}) we get
\[ \opn{RP}_{k}(\alpha \vert \sigma) \in V_0(G)  \]
and 
\[ \norm{\, \opn{log}_G \bigl( 
\opn{RP}_{k}(\alpha \vert \sigma) \bigr) \, }
\leq c_0(G) \cdot \bigl( \bosum_{i = 1}^{2^k} \, \norm{ \lambda_i } \bigr)
\leq c_1(\alpha) \cdot \epsilon \, . \]
\end{proof}

\begin{rem} \label{rem:3}
Heuristically we think of $\epsilon$ in the proof of the lemma above as a
``tiny'' size. In the ``tiny scale'' we can measure things (i.e.\ 
$\norm{\log_G(g)}$ is defined), and we can use Taylor series and CBH series. 
\end{rem}

\begin{lem} \label{lem:30}
There are constants $c_2(\alpha)$ and $\epsilon_2(\alpha)$ with the
following properties.
\begin{enumerate}
\rmitem{i} These inequalities hold:
\begin{gather*}
0 < \epsilon_2(\alpha) \leq \epsilon_1(\alpha)  \, , \\
c_2(\alpha) \geq c_1(\alpha) \,  .
\end{gather*}

\rmitem{ii} Suppose $\sigma : \mbf{I}^{1} \to X$ is a linear map
such that $\opn{len}(\sigma) < \epsilon_2(\alpha)$, and 
$\alpha|_{\sigma(\mbf{I}^{1})}$ is smooth.
Then for any $k \geq 0$ one has 
\[ \Norm{ 
\opn{log}_G \bigl( \opn{RP}_{k}(\alpha \vert \sigma) \bigr) -
\opn{log}_G \bigl( \opn{RP}_{0}(\alpha \vert \sigma) \bigr) }
\leq c_2(\alpha) \cdot \opn{len}(\sigma)^3 . \]
\end{enumerate}
\end{lem}

\begin{proof}\item
Take
\[ \epsilon_2(\alpha) := \min \bigl( \epsilon_1(\alpha), \,
\smfrac{1}{4} \epsilon_0(G) \cdot (1 + \norm{\alpha}_{\mrm{Sob}})^{-1} 
\bigr) \, . \]

Let $\sigma$ be a linear map with $\epsilon := \opn{len}(\alpha)$ satisfying
$0 < \epsilon < \epsilon_2(\alpha)$.
Let $Z := \sigma(\mbf{I}^{1})$, which is a $1$-dimensional oriented polyhedron,
and $z := \sigma(\smfrac{1}{2})$, the midpoint of $Z$.
Choose a positively oriented orthonormal linear coordinate function $s_1$
on $Z$, such that $s_1(z) = 0$.

Let $\til{\alpha} \in \mcal{O}(Z) \otimes \g$ be the coefficient of
$\alpha|_Z$, i.e.\ 
$\alpha|_Z = \til{\alpha} \cdot \d t_1$. 
Consider the Taylor expansion to second order of the smooth function
$\til{\alpha} : Z \to \g$, around the point $z$:
\begin{equation} \label{eqn:67}
\til{\alpha}(x) =  a_0 + s_1(x) \cdot a_1 + s_1(x)^2 \cdot g(x) 
\end{equation}
for $x \in Z$, where
$a_0 := \til{\alpha}(z) \in \g$,
$a_1 := (\smfrac{\partial}{\partial s_1} \til{\alpha})(z) \in \g$,
and $g : Z \to \g$ is a continuous function.
We know that
\begin{equation*} 
\norm{a_0} , \,
\norm{a_1} , \,
\norm{g(x)} \leq \norm{\alpha}_{\mrm{Sob}}  \, .
\end{equation*}
And, as we have seen before, 
\begin{equation} \label{eqn:71}
 \opn{RP}_{0}(\alpha \vert \sigma) = 
\opn{exp}_G \bigl(  \epsilon a_0 \bigr) .
\end{equation}

Take $k \geq 0$. For $i \in \{ 1, \ldots, 2^k \}$ let
$\sigma_i := \sigma \circ \sigma^k_i$ and
$z_i := \sigma_i(\smfrac{1}{2})$. 
Since 
$\epsilon < \epsilon_2(\alpha)$ it follows that 
\begin{equation} \label{eqn:66}
\sum_{i = 1}^{2^k} \,
\norm{ (\smfrac{1}{2})^k \epsilon \cdot a_0 } 
\leq \epsilon \cdot \norm{\alpha}_{\mrm{Sob}} \leq \smfrac{1}{4} \epsilon_0(G)
\, .
\end{equation}
Because  
$\abs{ s_1(z_i) } \leq  \smfrac{1}{2} \epsilon$ we also have
\begin{equation} \label{eqn:10}
\sum_{i = 1}^{2^k} \,
\norm{ (\smfrac{1}{2})^k \epsilon \cdot s_1(z_i) \cdot a_1 } 
\leq \smfrac{1}{2} \epsilon^2 \cdot \norm{\alpha}_{\mrm{Sob}}
\leq \smfrac{1}{4} \epsilon_0(G) 
\end{equation}
and 
\begin{equation} \label{eqn:8}
\sum_{i = 1}^{2^k} \,
\norm{ (\smfrac{1}{2})^k \epsilon \cdot s_1(z_i)^2 \cdot g(z_i) } 
\leq \smfrac{1}{4} \epsilon^3  \cdot \norm{\alpha}_{\mrm{Sob}}
\leq \smfrac{1}{4} \epsilon_0(G) \, .
\end{equation}

Define
\[ \lambda_i := (\smfrac{1}{2})^k \epsilon \cdot \til{\alpha}(z_i) \in \g . \]
By the Taylor expansion (\ref{eqn:67}) we have
\begin{equation} \label{eqn:69}
\lambda_i =  (\smfrac{1}{2})^k \epsilon \cdot a_0 + 
(\smfrac{1}{2})^k \epsilon \cdot s_1(z_i) \cdot a_1 + 
(\smfrac{1}{2})^k \epsilon \cdot s_1(z_i)^2 \cdot g(z_i)  .
\end{equation}
Since $\sigma^*(\d s_1) = \epsilon \cdot \d t_1$, we see that 
\[ \opn{RP}_{0}(\alpha \vert \sigma_i) = 
\opn{exp}_G(\lambda_i) , \]
and hence 
\begin{equation} \label{eqn:70}
\opn{RP}_{k}(\alpha \vert \sigma) = 
\prod_{i = 1}^{2^k} \ 
\opn{exp}_G (\lambda_i) . 
\end{equation}

Because the constant terms in the Taylor expansions (\ref{eqn:69}) of the
$\lambda_i$ are all equal to 
$(\smfrac{1}{2})^k \epsilon \cdot  a_0$, 
we can use property (v) of Theorem \ref{thm:6}, together with the estimates 
(\ref{eqn:66}), (\ref{eqn:10}), (\ref{eqn:8}), to deduce
\begin{equation}  \label{eqn:15}
\begin{aligned}
& \Norm{ \log_G \Bigl( \boprod_{i = 1}^{2^k}  \, \exp_G (\lambda_i) \Bigr)
 - \bosum_{i = 1}^{2^k} \, \lambda_i } \\
& \qquad \leq c_0(G) \cdot \bigl( 
\epsilon \cdot \norm{\alpha}_{\mrm{Sob}} \cdot 
(1 + \smfrac{1}{2} \epsilon + \smfrac{1}{4} \epsilon^2 ) \bigr)
\cdot \bigl( \epsilon \cdot \norm{\alpha}_{\mrm{Sob}} \cdot 
(\smfrac{1}{2} \epsilon + \smfrac{1}{4} \epsilon^2 ) \bigr) \\
& \qquad \leq \epsilon^3 \cdot c_0(G) \cdot 
\norm{\alpha}_{\mrm{Sob}}^2 \cdot 2 \, . 
\end{aligned} 
\end{equation}
Trivially the sum of the constant terms of the Taylor expansions of the
$\lambda_i$ is
\[ \sum_{i=1}^{2^k}\,  (\smfrac{1}{2})^k \epsilon \cdot a_0 = 
\epsilon \cdot a_0 . \]
The linear terms satisfy
\[ s_1(z_i) = - s_1(z_{2^{k} - i}) \]
because of symmetry; and therefore they cancel out:
\[ \sum_{i=1}^{2^k}\,  (\smfrac{1}{2})^k \epsilon \cdot
s_1(z_i) \cdot a_1 = 0 . \]
Therefore, using the estimate (\ref{eqn:8}) to eliminate the quadratic terms of
the Taylor expansions, we conclude that
\begin{equation} \label{eqn:13}
\Norm{ \bosum_{i = 1}^{2^k} \,  \lambda_i - \epsilon a_0 } \leq 
\smfrac{1}{4} \epsilon^3  \cdot \norm{\alpha}_{\mrm{Sob}} . 
\end{equation}

Finally we define
\[ c_2(\alpha) := \max \bigl( c_1(\alpha), \, 
2 c_0(G) \cdot \norm{\alpha}_{\mrm{Sob}}^2 + 
\smfrac{1}{4} \cdot \norm{\alpha}_{\mrm{Sob}} \bigr) \, . \]
Combining equations (\ref{eqn:71}) and (\ref{eqn:70}), plus the estimates
(\ref{eqn:13}) and (\ref{eqn:15}), we obtain 
\[ \Norm{ \log_G \bigl( \opn{RP}_0 (\alpha \vert \sigma) \bigl) - 
\log_G \bigl( \opn{RP}_{k} (\alpha \vert \sigma) \bigl) } 
\leq c_2(\alpha) \cdot \epsilon^3  \, . \]
\end{proof}

\begin{dfn} \label{dfn:13}
Let us fix a constant $\epsilon_2(\alpha)$ as in Lemma
\ref{lem:30}. A piecewise linear map $\sigma : \mbf{I}^1 \to X$ with 
$\opn{len}(\sigma) < \epsilon_2(\alpha)$
will be called an {\em $\alpha$-tiny piecewise linear map} (in this section).
\end{dfn}

\begin{rem}
We shall use the term ``tiny'' several times in the paper, each time with a new
meaning, depending on context. The notion ``tiny'' should be considered as
``local to each section''. 
\end{rem}

\begin{lem} \label{lem:31}
Let $\sigma : \mbf{I}^{1} \to X$ be an $\alpha$-tiny linear map. Then
there is a constant $c_3(\alpha, \sigma) \geq 0$, 
such that for any integers $k' \geq k \geq 0$ one has
\[ \Norm{ \log_G \bigl( \opn{RP}_{k'} (\alpha \vert \sigma) \bigl) - 
\log_G \bigl( \opn{RP}_{k} (\alpha \vert \sigma) \bigl) }
\leq c_3(\alpha, \sigma) \cdot \opn{len}(\sigma) \cdot (\smfrac{1}{2})^k \, .
\]
\end{lem}

\begin{proof}
We may assume that $\opn{len}(\sigma) > 0$. 
Let $Z := \sigma(\mbf{I}^{1}) \subset X$, and let $m$ be the number of singular
points of the differential form $\alpha|_Z$. 

Take $k \geq 0$. For an index $i \in \{ 1, \ldots, 2^k \}$ define
$\sigma_i := \sigma \circ \sigma^k_i$
and $Z_i := \sigma_i(\mbf{I}^{1})$. We say that $i$ is good if 
$\alpha|_{Z_i}$ is smooth, and otherwise $i$ is bad. The set of good and
bad indices are denoted by $\opn{good}(k)$ and $\opn{bad}(k)$ respectively.
Since for any singular point $x$ of $\alpha|_Z$ there is at most one index $i$
such that $x \in \opn{Int} Z_i$, 
it follows that the cardinality of $\opn{bad}(k)$ is at most $m$. 

Next take $k' \geq k \geq 0$, and let $l := k' - k$. 
By the recursive definition of the binary tessellations we have
\[ \opn{RP}_{k}(\alpha \vert \sigma) =
\prod_{i = 1}^{2^k} \,  \opn{RP}_{0}(\alpha \vert \sigma_i)  \]
and
\[ \opn{RP}_{k'}(\alpha \vert \sigma) =
\prod_{i = 1}^{2^k} \,  \opn{RP}_{l}(\alpha \vert \sigma_i) . \]
If $i \in \opn{good}(k)$ then by Lemma \ref{lem:30} we know that 
\[ \Norm{ 
\opn{log}_G \bigl( \opn{RP}_{l}(\alpha \vert \sigma_i) \bigr) -
\opn{log}_G \bigl( \opn{RP}_{0}(\alpha \vert \sigma)_i \bigr) }
\leq c_2(\alpha) \cdot \opn{len}(\sigma)^3 \cdot (\smfrac{1}{2})^{3 k} . \]
On the other hand, $i \in \opn{bad}(k)$, then by Lemma \ref{lem:1} we know
that 
\[ \Norm{ 
\opn{log}_G \bigl( \opn{RP}_{l}(\alpha \vert \sigma_i) \bigr) -
\opn{log}_G \bigl( \opn{RP}_{0}(\alpha \vert \sigma)_i \bigr) }
\leq 
2 \cdot c_1(\alpha) \cdot \opn{len}(\sigma) \cdot (\smfrac{1}{2})^{k} . \]
Therefore by property (iv) of Theorem \ref{thm:6} we have
\[ \begin{aligned}
& \Norm{ 
\opn{log}_G \bigl( \opn{RP}_{k'}(\alpha \vert \sigma) \bigr) -
\opn{log}_G \bigl( \opn{RP}_{k}(\alpha \vert \sigma) \bigr) } \\
& \qquad = \Norm{ \opn{log}_G \Bigl( \boprod_{i = 1}^{2^k} \,
\opn{RP}_{l}(\alpha \vert \sigma_i) \Bigr) -
\opn{log}_G \bigl( \opn{RP}_{0}(\alpha \vert \sigma_i) \bigr) } \\
& \qquad \leq
c_0(G) \cdot \bigl( 
\abs{ \opn{bad}(k) } \cdot 2 \cdot c_1(\alpha) \cdot \opn{len}(\sigma) 
 \cdot (\smfrac{1}{2})^{k} \\
& \qquad \quad +
\abs{ \opn{good}(k) } \cdot
c_2(\alpha) \cdot \opn{len}(\sigma)^3 \cdot (\smfrac{1}{2})^{3 k}
\bigr) . 
\end{aligned} \]
Thus we may take
\[ c_3(\alpha, \sigma) := c_0(G) \cdot \bigl( 2 m \cdot c_1(\alpha) +
c_2(\alpha) \bigr) \, . \]
\end{proof}

\begin{thm} \label{thm:13}
Let $X$ be a polyhedron, let 
$\alpha \in \Omega^1_{\mrm{pws}}(X) \otimes \g$,
and let $\sigma : \mbf{I}^{1} \to X$ be a piecewise linear map. Then the limit
$\lim_{k \to \infty} \opn{RP}_{k}(\alpha \vert \sigma)$
exists.
\end{thm}

\begin{proof}
For any $k$ we have
\[ \begin{aligned}
& \opn{RP}_{k}(\alpha \vert \sigma) = 
\prod_{i = 1}^{2^k} \, \opn{RP}_{0}(\alpha \vert \sigma \circ \sigma^k_i) =
\prod_{i = 1}^{2^k} \, \opn{RP}_{0} \bigl( (\sigma \circ \sigma^k_i)^*(\alpha)
\vert \mbf{I}^1 \bigr) \\
& \qquad = 
\prod_{i = 1}^{2^k} \, \opn{RP}_{0} \bigl( \sigma^*(\alpha)
\vert \sigma^k_i \bigr) = 
\opn{RP}_{k} \bigl( \sigma^*(\alpha) \vert \sigma^0_1 \bigr) 
\end{aligned} \]
by definition. 
So after replacing $\alpha$ with $\sigma^*(\alpha)$, we can assume that 
$X = \mbf{I}^1$, and we have to prove that the limit 
$\lim_{k \to \infty} \opn{RP}_{k}(\alpha \vert \sigma^0_1)$ exists. 

Take $k$ large enough such that for each $i \in \{ 1, \ldots, 2^k \}$
the linear map $\sigma^k_i$ is $\alpha$-tiny. 
For any $k' \geq 0$ we have
\[ \opn{RP}_{k + k'}(\alpha \vert \sigma^0_1) = 
\prod_{i = 1}^{2^k}  \, \opn{RP}_{k'}(\alpha \vert \sigma^k_i) . \]
Thus it suffices to prove that for any $i$ the limit
$\lim_{k' \to \infty} \opn{RP}_{k'}(\alpha \vert \sigma^k_i)$
exists.

We have now reduced our problem to showing that for any $\alpha$-tiny
linear map $\sigma : \mbf{I}^1 \to X$ the limit
$\lim_{k \to \infty} \opn{RP}_{k}(\alpha \vert \sigma)$
exists. But this follows immediately from Lemma \ref{lem:31}.
\end{proof}

\begin{dfn} \label{dfn:28}
In the situation of Theorem \ref{thm:13}, the 
{\em multiplicative integral of $\alpha$ on $\sigma$} is 
\[ \opn{MI}(\alpha \vert \sigma) := 
\lim_{k \to \infty} \opn{RP}_{k}(\alpha \vert \sigma) \in G . \] 

If $X = \mbf{I}^{1}$ and $\sigma = \sigma^0_1$, the basic map, then we write
$\opn{MI}(\alpha \vert \mbf{I}^{1}) := \opn{MI}(\alpha \vert \sigma^0_1)$.
\end{dfn}

\begin{rem}
If the group $G$ is {\em abelian}, then 
$\opn{RP}_k(\alpha \vert \sigma)$
is the exponential of a Riemann sum, and therefore in the limit we get
\[ \opn{MI}(\alpha \vert \sigma) = 
\exp_G \bigl( \int_{\sigma} \alpha \bigr) . \]
\end{rem}

\begin{prop} \label{prop:5}
Let $\sigma : \mbf{I}^{1} \to X$ be a piecewise linear map.
\begin{enumerate}
\item If $\sigma$ is $\alpha$-tiny \tup{(}see Definition \tup{\ref{dfn:13})},
then 
$\opn{MI}(\alpha \vert \sigma) \in V_0(G)$, and 
\[ \Norm{ \opn{log}_G \bigl( \opn{MI}(\alpha \vert \sigma) \bigr) }
\leq c_1(\alpha) \cdot  \opn{len}(\sigma) \, . \]

\item If $\sigma$ is linear and $\alpha$-tiny, and if 
$\alpha|_{\sigma(\mbf{I}^{1})}$ is smooth, then for any $k \geq 0$ one has
\[ \Norm{ \opn{log}_G \bigl( \opn{MI}(\alpha \vert \sigma) \bigr) -
\opn{log}_G \bigl( \opn{RP}_0(\alpha \vert \sigma) \bigr) }
\leq c_2(\alpha) \cdot \opn{len}(\sigma)^3 \, . \]

\item For any $k \geq 0$ one has
\[ \opn{MI}(\alpha \vert \sigma)  =
\prod_{i = 1}^{2^k}  \, \opn{MI}(\alpha \vert \sigma \circ \sigma^k_i ) . \]
\end{enumerate}
\end{prop}

\begin{proof}
(1) By Lemma \ref{lem:1}, for sufficiently large $k$ we have
$\opn{RP}_k(\alpha \vert \sigma) \in V_0(G)$ and 
\[ \Norm{ \opn{log}_G \bigl( \opn{RP}_k(\alpha \vert \sigma) \bigr) } 
\leq c_1(\alpha) \cdot \opn{len}(\sigma) \leq 
\smfrac{1}{2} \cdot \epsilon_0(G) \ . \]
Let $B$ be the closed ball of radius $\smfrac{1}{2} \epsilon_0(G)$ in
$\g$, and let $Z := \opn{exp}_G(B)$, which is a compact subset of $G$.
Since for every $k$ one has
$\opn{RP}_k(\alpha \vert \sigma) \in Z$, it follows that in the limit
$\opn{MI}(\alpha \vert \sigma) \in Z \subset V_0(G)$.
The bound on 
$\Norm{ \opn{log}_G \bigl( \opn{MI}(\alpha \vert \sigma) \bigr) }$
is then obvious.

\medskip \noindent
(2) This is immediate from Lemma \ref{lem:30}.

\medskip \noindent
(3) For every $k' \geq 0$ we have by definition
\[ \opn{RP}_{k + k'}(\alpha \vert \sigma) = 
\prod_{i = 1}^{2^k}  \, \opn{RP}_{k'}(\alpha \vert \sigma \circ \sigma^k_i ) .
\]
Now pass to the limit $k' \to  \infty$.
\end{proof}

\subsection{Functoriality of the MI}
The next results are on the functoriality of the multiplicative integral with
respect to $G$ and $X$. 

\begin{prop}  \label{prop:6}
Let $\Phi : G \to H$ be a map of Lie groups, with induced Lie algebra map
$\phi := \opn{Lie}(\Phi) : \g \to \h$. Let $f : Y \to X$ be a piecewise linear
map between polyhedra, and let
$\alpha \in \Omega^1_{\mrm{pws}}(X) \otimes \g$.
Then for any piecewise linear map $\sigma : \mbf{I}^{1} \to Y$ one has
\[ \Phi \bigl( \opn{MI}(\alpha \vert f \circ \sigma) \bigr) = 
\opn{MI} \bigl( (f^* \otimes \phi)(\alpha) \vert \sigma \bigr) \]
in $H$.
\end{prop}

\begin{proof}
It suffices to consider $f$ and $\Phi$ separately; so we look at two cases.

\medskip \noindent
Case 1. $H = G$ and $\Phi$ is the identity map. Here for every $k \geq 0$
and $i \in \{ 1, \ldots, 2^k \}$ we have
\[ \opn{RP}_{0}(\alpha \vert f \circ \sigma \circ \sigma^k_i) 
\bigr) =
\opn{RP}_{0} \bigl( (f \circ \sigma)^*(\alpha) \vert \sigma^k_i \bigr) =
\opn{RP}_{0} \bigl( f^*(\alpha) \vert \sigma \circ \sigma^k_i \bigr) . \]
Hence
\[ \opn{RP}_{k}(\alpha \vert f \circ \sigma) =
\opn{RP}_{k} \bigl( (f \circ \sigma)^*(\alpha)  \vert \mbf{I}^{1} \bigr) =
\opn{RP}_{k} \bigl( f^*(\alpha) \vert \sigma \bigr) . \]
Going to the limit in $k$ we see that 
\[ \opn{MI}(\alpha \vert f \circ \sigma) =
\opn{MI} \bigl( (f \circ \sigma)^*(\alpha)  \vert \mbf{I}^{1} \bigr) =
\opn{MI} \bigl( f^*(\alpha) \vert \sigma \bigr) . \]

\medskip \noindent
Case 2. 
Here we assume that $Y = X$ and $f$ is the identity map. 
Since 
\[ \opn{MI}(\alpha \vert \sigma) = \opn{MI}(\sigma^*(\alpha) \vert \sigma^0_1)
\]
and 
\[ \opn{MI}(\phi(\alpha) \vert \sigma) = 
\opn{MI}((\sigma^* \otimes \phi)(\alpha) \vert \sigma^0_1) =
\opn{MI}(\phi(\sigma^*(\alpha)) \vert \sigma^0_1) , \]
we can replace $\sigma$ with $\sigma^0_1$ and $\alpha$ with
$\sigma^*(\alpha)$. Therefore we can assume that $\sigma$ is a linear map. 

Put a euclidean norm on $\h$ such that $\phi : \g \to \h$ has operator norm 
$\norm{\phi} \leq 1$. 
This implies that 
$\norm{ \phi(\alpha) }_{\mrm{Sob}} \leq \norm{\alpha}_{\mrm{Sob}}$.
So we can assume that
$\epsilon_1(\phi(\alpha)) \geq \epsilon_1(\alpha)$
and
$c_1(\phi(\alpha)) \leq c_1(\alpha)$
(cf.\ proof of Lemma \ref{lem:1}).

Take $k$ large enough such that 
$\sigma \circ \sigma^k_i$ is $\alpha$-tiny for every 
$i \in \{ 1, \ldots, 2^k \}$.
Then $\sigma \circ \sigma^k_i$ is also $\phi(\alpha)$-tiny.

By part (3) of Proposition \ref{prop:5} we have
\[ \opn{MI}(\alpha \vert \sigma )  =
\prod_{i = 1}^{2^k}  \, 
\opn{MI}(\alpha \vert \sigma \circ \sigma^k_i )  \]
and
\[ \opn{MI} \bigl( \phi(\alpha) \vert \sigma \bigr) =
\prod_{i = 1}^{2^k}  \, 
\opn{MI} \bigl( \phi(\alpha) \vert \sigma \circ \sigma^k_i \bigr) . \]
So it suffices to prove that 
\[ \Phi \bigl( \opn{MI}(\alpha \vert \sigma \circ \sigma^k_i) \bigr) = 
\opn{MI} \bigl( \phi(\alpha) \vert \sigma \circ \sigma^k_i \bigr) \]
for every $i$. 

By this reduction we can assume that 
$\sigma$ is $\alpha$-tiny and also $\phi(\alpha)$-tiny.
Take any $k \geq 0$, and for every $i \in \{ 1, \ldots, 2^k \}$ let 
$w_i := \sigma^k_i(\smfrac{1}{2}) \in \mbf{I}^{1}$. 
If $w_i$ is a smooth point of 
\[ \sigma^*(\alpha) \in \Omega^1_{\mrm{pws}}(\mbf{I}^{1}) \otimes 
\g , \]
then it is also a smooth point of 
\[ (\sigma^* \otimes \phi)(\alpha)  
\in \Omega^1_{\mrm{pws}}(\mbf{I}^{1}) \otimes \h .  \]
In this case we have
\begin{equation} \label{eqn:72}
\Phi \bigl( \opn{RP}_{0}(\alpha \vert \sigma \circ \sigma^k_i) 
\bigr) =
\opn{RP}_0 \bigl( \phi(\alpha) \vert \sigma \circ \sigma^k_i \bigr) .
\end{equation}

In case $w_i$ is a singular point of
$(\sigma^* \otimes \phi)(\alpha)$,
then it is also a singular point of 
$\sigma^*(\alpha)$.
Hence (\ref{eqn:72}) also holds (both sides are $1$). 

The only problem is when $w_i$ is a smooth point of
$(\sigma^* \otimes \phi)(\alpha)$
but a singular point of $\sigma^*(\alpha)$.

Now we use the estimate provided by Lemma \ref{lem:1}
(noting that $\phi \circ \log_G = \log_H \circ\, \Phi$):
\[ \begin{aligned}
& \Norm{ \log_H \bigl( \Phi \bigl( \opn{RP}_{0} 
(\alpha \vert \sigma \circ \sigma^k_i ) \bigl) \bigr) - 
\log_H \bigl( \opn{RP}_{0} \bigl( \phi(\alpha) \vert 
\sigma \circ \sigma^k_i \bigr) \bigr) } \\
& \qquad 
\leq 2 \cdot c_1(\alpha) \cdot \opn{len}(\sigma) \cdot (\smfrac{1}{2})^k \, .
\end{aligned} \]

Let $m$ be the number of singular points of $\sigma^*(\alpha)$. 
Then using property (iv) of Theorem \ref{thm:6} we get
\[ \begin{aligned}
& \Norm{ \log_H \bigl( \Phi \bigl( \opn{RP}_{k} 
(\alpha \vert \sigma) \bigl) \bigr) - 
\log_H \bigl( \opn{RP}_{k} \bigl( \phi(\alpha) \vert 
\sigma \bigr) \bigr) } \\
& \qquad 
\leq 2 m \cdot c_1(\alpha) \cdot \opn{len}(\sigma) \cdot (\smfrac{1}{2})^k \, .
\end{aligned} \]
Since $m$ is independent of $k$, in the limit $k \to \infty$ we get
\[ \Phi \bigl( \opn{MI} (\alpha \vert \sigma) \bigl) = 
\opn{MI} \bigl( \phi(\alpha) \vert \sigma \bigr) . \]
\end{proof}

A particular case of the above is when we are given a representation of $G$,
namely a map of Lie groups 
$\Phi : G \to \mrm{GL}_m(\R)$ for some $m$. 
The Lie algebra of $\mrm{GL}_m(\R)$ is
$\mfrak{gl}_m(\R) = \mrm{M}_m(\R)$, the algebra of $m \times m$ matrices.
For a matrix $a \in \mrm{M}_m(\R)$ let us denote by $\norm{a}$ its operator
norm, as linear operator $a : \R^m \to \R^m$, 
where $\R^m$ has the standard euclidean inner product.

\begin{prop} \label{prop:4}
Let $\Phi : G \to \mrm{GL}_m(\R)$ be a representation, and let
$\alpha \in \Omega^1_{\mrm{pws}}(X) \otimes \g$.
Then there is a constant $c_4(\alpha, \Phi) \geq 0$ such that the following
holds:
\begin{enumerate}
\item[($*$)] Given a piecewise linear map
$\sigma : \mbf{I}^{1} \to X$, let 
$g := \opn{MI}(\alpha \vert \sigma) \in G$.
Then
\[ \norm{ \Phi(g) } \leq \exp \bigl( c_4(\alpha, \Phi) \cdot
\opn{len}(\sigma)
\bigr) \, , \]
where $\exp$ the usual real exponential function.
\end{enumerate}
\end{prop}

\begin{proof}
Let's write $H := \mrm{GL}_m(\R)$,
$\h := \mfrak{gl}_m(\R)$ and
$\alpha' := \opn{Lie}(\Phi)(\alpha) \in$ \lb 
$\Omega^1_{\mrm{pws}}(X) \otimes \h$. 
We use the operator norm on $\h$ to determine the constants 
$\epsilon_1(\alpha')$ and $c_1(\alpha')$ in
Definition \ref{dfn:13}.

Take $k$ large enough so that $\sigma \circ \sigma^k_i$ is $\alpha'$-tiny
for every $i \in \{ 1, \ldots, 2^k \}$. 
Let $d := \opn{len}(\sigma)$ and $d_i := \opn{len}(\sigma \circ \sigma_i)$,
so $d = \sum_{i = 1}^{2^k}  \, d_i$.
Define $h_i := \opn{MI}(\alpha' \vert \sigma \circ \sigma^k_i) \in H$. 
Then $h_i \in V_0(H)$, and we can define 
$\gamma_i := \log_{H} (h_i) \in \h$.

Since the matrix $\gamma_i$ satisfies
\[ \norm{ \gamma_i } \leq c_1(\alpha') \cdot d_i , \]
it follows that its exponential $h_i$ satisfies
\[ \norm{ h_i } \leq \exp \bigl( c_1(\alpha') \cdot d_i \bigr) \, . \]
Now 
\[ \Phi(g) = \opn{MI}(\alpha' \vert \sigma) = \prod_{i = 1}^{2^k}  \, h_i , \]
so we obtain
\[ \norm{\Phi(g)} \leq \exp \bigl( c_1(\alpha') \cdot d \bigr) \, . \]
Therefore we can take 
$c_4(\alpha, \Phi) := c_1(\alpha')$.
\end{proof}

\subsection{Strings}
\begin{dfn} \label{dfn:3}
Let $X$ be a polyhedron. A {\em string} in $X$ is a sequence 
$\sigma = (\sigma_1, \ldots, \sigma_m)$ of piecewise
linear maps $\sigma_i : \mbf{I}^1 \to X$, such that 
$\sigma_i(v_1) = \sigma_{i+1}(v_0)$ for all $i$. 
The maps $\sigma_i : \mbf{I}^1 \to X$ are called the {\em pieces} of $\sigma$.  
We write $\sigma(v_0) := \sigma_1(v_0)$ and
$\sigma(v_1) := \sigma_m(v_1)$, and call these points the {\em initial} and
{\em terminal} points of $\sigma$, respectively. 
The {\em length} of $\sigma$ is
$\opn{len}(\sigma) := \sum_{i = 1}^m \opn{len}(\sigma_i)$.
\end{dfn}

Suppose $\alpha \in \Omega^1_{\mrm{pws}}(X) \otimes \g$.
We say that $\sigma$ is an {\em $\alpha$-tiny string} if
$\opn{len}(\sigma) < \epsilon_1(\alpha)$; cf.\ Definition \ref{dfn:13}.

Here are a few operations on strings. 
Suppose $\sigma = (\sigma_1, \ldots, \sigma_m)$ and
$\tau = (\tau_1, \ldots, \tau_l)$ are two strings in
$X$, with $\tau(v_0) = \sigma(v_1)$. 
Then we define the {\em concatenated string}
\begin{equation} \label{eqn12}
\sigma * \tau := (\sigma_1, \ldots, \sigma_m, \tau_1, \ldots, \tau_l) .
\end{equation}
See Figure \ref{fig:41} for an illustration.

\begin{figure}
\includegraphics[scale=0.27]{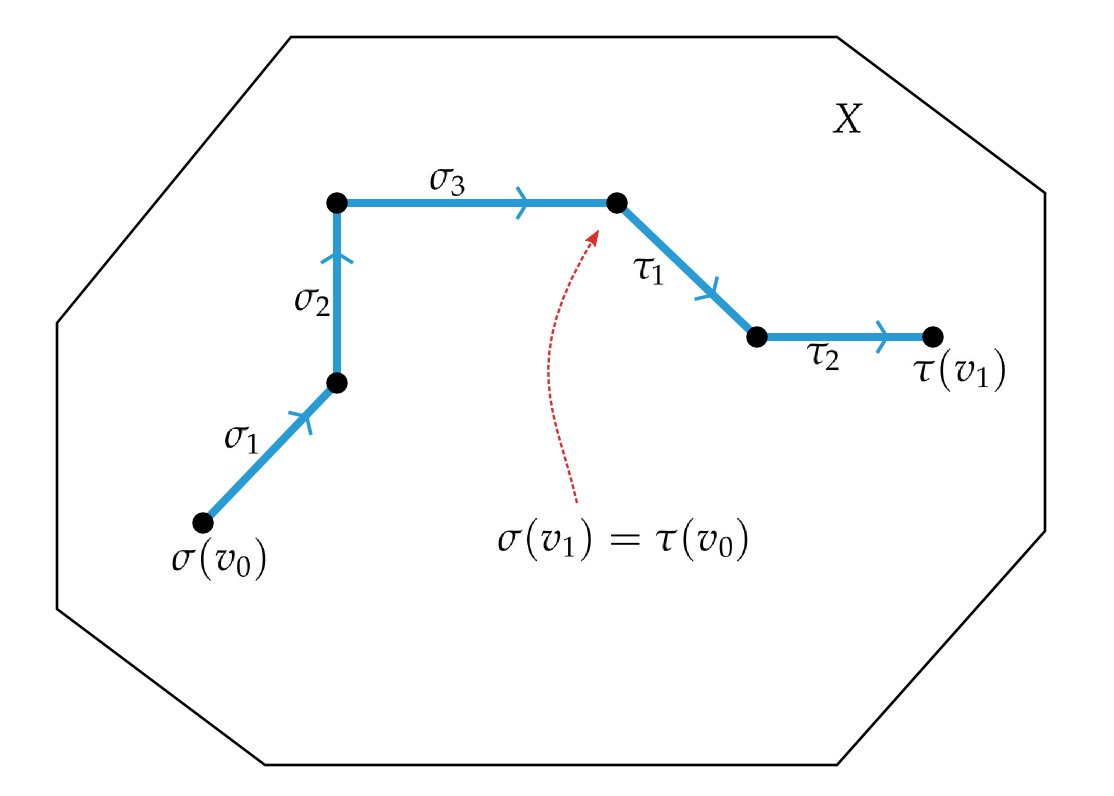}
\caption{The strings $\sigma = (\sigma_1, \sigma_2, \sigma_3)$,
$\tau = (\tau_1, \tau_2)$ and
$\sigma * \tau$ in the polyhedron $X$.} 
\label{fig:41}
\end{figure}

The {\em flip} of $\mbf{I}^{1}$ is the linear bijection
$\opn{flip} : \mbf{I}^{1} \to \mbf{I}^{1}$
defined on vertices by
\[ \opn{flip}(v_0, v_1) := (v_1, v_0) . \]
Given a piecewise linear map $\sigma : \mbf{I}^{1} \to X$ we let
\begin{equation}
\sigma^{-1} := \sigma \circ  \opn{flip} : \mbf{I}^{1} \to X .
\end{equation}
For a string $\sigma = (\sigma_1, \ldots, \sigma_m)$ in $X$ we define the
{\em inverse string}
\begin{equation} \label{eqn13}
\sigma^{-1} := (\sigma_m^{-1}, \ldots, \sigma_1^{-1}) .
\end{equation}

The empty string is the unique string of length $0$, and we denote it by
$\emptyset$. For any string $\sigma$ we let
$\sigma * \emptyset := \sigma$ and 
$\emptyset * \sigma := \sigma$.

Let $f : X \to Y$ be a piecewise linear map between polyhedra, and let
$\sigma = (\sigma_1, \ldots, \sigma_m)$ be a string in $X$.
We define the string $f \circ \sigma$ in
$Y$ to be
\begin{equation} \label{eqn14}
f \circ \sigma := (f \circ \sigma_1, \ldots, f \circ \sigma_m) .
\end{equation}

\begin{rem}
The reason for working with strings (rather than with paths, as is the custom
in algebraic topology) is that composition of strings, as defined above, is
associative, whereas composition of paths is only associative up to homotopy.
\end{rem}

It will be convenient to integrate along a string. 
As before $G$ is a Lie group with Lie algebra $\g$.

\begin{dfn} \label{dfn:10}
Suppose $\alpha \in \Omega^1_{\mrm{pws}}(X) \otimes \mfrak{g}$
and 
$\sigma = (\sigma_1, \ldots, \sigma_m)$
is a string in $X$. The multiplicative integral 
\index{Multiplicative integral on strings}
$\opn{MI}(\alpha \vert \sigma)$ of $\alpha$ on $\sigma$ is 
\[ \opn{MI}(\alpha \vert \sigma) :=
\prod_{i=1}^m\ \opn{MI}(\alpha \vert \sigma_i) \in G . \]
\end{dfn}

\begin{prop} \label{prop:1}
Let 
$\alpha \in \Omega_{\mrm{pws}}^1(X) \otimes \mfrak{g}$.
\begin{enumerate}
\item Given strings $\sigma$ and $\tau$ in $X$ such that
$\tau(v_0) = \sigma(v_1)$, one has
\[ \opn{MI}(\alpha \vert \sigma * \tau) = 
\opn{MI}(\alpha \vert \sigma) \cdot \opn{MI}(\alpha \vert \tau) . \]

\item Given a string $\sigma$ in $X$ , one has
\[ \opn{MI}(\alpha \vert \sigma^{-1}) = 
\opn{MI}(\alpha \vert \sigma)^{-1} . \]

\item If $\sigma$ is an $\alpha$-tiny string in $X$ 
\tup{(}i.e\ 
$\opn{len}(\sigma) < \epsilon_1(\alpha)$\tup{)}, then 
$\opn{MI}(\alpha \vert \sigma) \in V_0(G)$, and 
\[ \Norm{ \opn{log}_G \bigl( \opn{MI}(\alpha \vert \sigma) \bigr) }
\leq c_1(\alpha) \cdot \opn{len}(\sigma) \, . \]
\end{enumerate}
\end{prop}

\begin{proof}
(1) This is trivial.

\medskip \noindent
(2) By part (1) it suffices to consider a piecewise linear map 
$\sigma : \mbf{I}^{1} \to X$.
Since the flip reverses orientation on $\mbf{I}^1$, it follows that 
\[ \opn{RP}_0(\alpha \vert \sigma^{-1}) = 
\opn{RP}_0(\alpha \vert \sigma)^{-1} . \]
{}From this, and the symmetry of the binary tessellations, it follows that for
every $k \geq 0$ one has
\[ \opn{RP}_{k}(\alpha \vert \sigma^{-1}) = 
\opn{RP}_{k}(\alpha \vert \sigma)^{-1} . \]
In the limit we get
$\opn{MI}(\alpha \vert \sigma^{-1}) = 
\opn{MI}(\alpha \vert \sigma)^{-1}$.

\medskip \noindent
(3) Say $\sigma  = (\sigma_1, \ldots, \sigma_m)$.
Take $k$ large enough so that $2^k \geq m$. Let
$\sigma' : \mbf{I}^1 \to X$ be the unique piecewise linear map satisfying
$\sigma' \circ \sigma^k_i = \sigma_i$ for $i \leq m$, and
$\sigma' \circ \sigma^k_i$ is the constant map $\sigma(v_1)$ for $i > m$.
Note that $\opn{len}(\sigma') = \opn{len}(\sigma)$.
By Definition \ref{dfn:10} and Proposition \ref{prop:5}(3) we have
$\opn{MI}(\alpha \vert \sigma') = 
\opn{MI}(\alpha \vert \sigma)$.
Now use Proposition \ref{prop:5}(1).
\end{proof}

\begin{prop} \label{prop:12}
Let $\Phi : G \to H$ be a map of Lie groups, let $f : Y \to X$ be a piecewise
linear map between polyhedra, and let
$\alpha \in \Omega^1_{\mrm{pws}}(X) \otimes \g$.
Then for any string $\sigma$ in $Y$ one has
\[ \Phi \bigl( \opn{MI}(\alpha \vert f \circ \sigma) \bigr) = 
\opn{MI} \bigl( (f^* \otimes \opn{Lie}(\Phi))(\alpha) \vert \sigma \bigr) 
\in H . \]
\end{prop}

\begin{proof}
This is an immediate consequence of Proposition \ref{prop:1}(1) and \lb
Proposition \ref{prop:6}.
\end{proof}

\begin{prop} \label{prop:13}
Let $\Phi : G \to \opn{GL}_m(\R)$ be a representation, and let
$\alpha \in \Omega_{\mrm{pws}}^1(X) \otimes \mfrak{g}$.
Given a string $\sigma$ in $X$ let $g := \opn{MI}(\alpha \vert \sigma)$.
Then the norm of the operator $\Phi(g)$ on $\R^m$ satisfies
\[ \norm{ \Phi(g) } \leq  \exp \bigl( c_4(\alpha, \Phi) \cdot
\opn{len}(\sigma) \bigr) \, , \]
where $c_4(\alpha, \Phi)$ is the constant from Proposition \tup{\ref{prop:4}}.
\end{prop}

\begin{proof}
By part (1) of Proposition \ref{prop:1} 
the left side of this inequality is multiplicative with respect
to the operation $*$. And clearly the right side is also multiplicative with
respect to $*$. So it suffices to consider a piecewise linear map 
$\sigma : \mbf{I}^{1} \to X$. Now we can use Proposition \ref{prop:4}.
\end{proof}

For a string $\sigma = (\sigma_1, \ldots, \sigma_m)$ 
and a form 
$\alpha \in \Omega_{\mrm{pws}}^1(X) \otimes \mfrak{g}$
we write
\[ \int_{\sigma} \alpha := \sum_{i = 1}^m \int_{\sigma_i} \alpha \in \g , \]
where 
\[ \int_{\sigma_i} \alpha = \int_{\mbf{I}^1} \sigma_i^*(\alpha) \] 
is the usual integral of this $\g$-valued piecewise smooth differential form. 

\begin{prop} \label{prop:16}
Let $\alpha \in \Omega^1_{\mrm{pws}}(X) \otimes \g$, and let $\sigma$ be an
$\alpha$-tiny string in $X$. Then 
\[ \Norm{ \log_G \bigl( \opn{MI}(\alpha \vert \sigma) \bigr) - 
\int_{\sigma} \alpha } \leq 
c_0(G) \cdot c_1(\alpha)^2 \cdot \opn{len}(\sigma)^2 \ . \]
\end{prop}

What this result says, is that in the tiny scale the nonabelian 
integral is very close to the abelian integral.

\begin{proof}
Step 1. 
We begin the proof with a reduction to the case $m = 1$, and
$\sigma = \sigma_1$ is a single linear map $\mbf{I}^1 \to X$. 
First we append a few empty strings at the
end of $\sigma$, so that the number of linear pieces becomes 
$m = 2^k$ for some $k$. This does not change $\opn{len}(\sigma)$ nor
$\opn{MI}(\alpha \vert \sigma)$.
Let $Z$ be an oriented $1$-dimensional polyhedron (a line segment) of length 
$\opn{len}(\sigma)$, partitioned into segments
$Z_1, \ldots, Z_m$, with $\opn{len}(Z_i) = \opn{len}(\sigma_i)$.
Let $\sigma' : \mbf{I}^1 \to Z$ be the unique oriented linear bijection.
There is a unique piecewise linear map $f : Z \to X$, such that 
$\sigma_i = f \circ \sigma' \circ \sigma^k_i$ as piecewise linear maps
$\mbf{I}^1 \to X$ for every $i$.
Let $\alpha' := f^*(\alpha) \in \Omega^1_{\mrm{pws}}(Z) \otimes \g$.
According to Propositions \ref{prop:5}(3), \ref{dfn:10} and \ref{prop:12} we
have
\[ \opn{MI}(\alpha \vert \sigma) = \opn{MI}(\alpha' \vert \sigma') . \]
And clearly 
\[ \int_{\sigma} \alpha = \int_{\sigma'} \alpha' . \]
Because the piecewise linear map $f : Z \to X$ is a linear metric
embedding on each of its linear pieces, it follows that 
$\norm{\alpha'}_{\mrm{Sob}} \leq \norm{\alpha}_{\mrm{Sob}}$; and hence we can
choose
$\epsilon_1(\alpha') \geq \epsilon_1(\alpha)$ and 
$c_1(\alpha') \leq c_1(\alpha)$.

Note that 
\[ \opn{len}(\sigma') = \opn{len}(Z) = \opn{len}(\sigma) < \epsilon_1(\alpha)
\, . \]
So we can replace $X$ with $Z$, $\sigma$ with $\sigma'$ and $\alpha$ with
$\alpha'$. Doing so, we can now assume that $m = 1$ and $\sigma$ is a single
linear map.

\medskip \noindent
Step 2. Here we assume that $\sigma$ is an $\alpha$-tiny linear map, and we let
$\epsilon := \opn{len}(\sigma)$.
Take any $k \geq 0$. We know that
\[ \opn{RP}_{k}(\alpha \vert \sigma) = 
\prod_{i = 1}^{2^k} \,  \opn{RP}_{0}(\alpha \vert \sigma \circ \sigma^k_i) . \]
Also (by Lemma \ref{lem:1}) we have
\[ \opn{RP}_{k}(\alpha \vert \sigma), 
\opn{RP}_{0}(\alpha \vert \sigma \circ \sigma^k_i) \in V_0(G) . \]
For any $i$ let
\[ \lambda_i := \log_G \bigl( \opn{RP}_{0}(\alpha \vert \sigma \circ \sigma^k_i)
\bigr) \in \g , \]
so 
\[ \opn{RP}_{k}(\alpha \vert \sigma) = 
\prod_{i = 1}^{2^k} \, \exp_G(\lambda_i) . \]
Let us write
\[ \opn{RS}_k(\alpha \vert \sigma) := \sum_{i = 1}^{2^k} \,  \lambda_i . \]
This is a Riemann sum for the usual integral.

By Lemma \ref{lem:1} we have
\[ \norm{ \lambda_i } \leq 
c_1(\alpha) \cdot (\smfrac{1}{2})^k \cdot \epsilon  \]
for every $i$. Using property (ii) of Theorem \ref{thm:6} we see that
\[ \begin{aligned}
& \Norm{ 
\log_G \bigl( \opn{RP}_{k}(\alpha \vert \sigma) \bigr) -
\opn{RS}_{k}(\alpha \vert \sigma) } \\
& \qquad \leq 
c_0(G) \cdot \bigl( \sum\nolimits_{i = 1}^{2^k} \,  \norm{\lambda_i} \bigr)^2 
\leq c_0(G) \cdot c_1(\alpha)^2 \cdot \epsilon^2 \ . 
\end{aligned} \]

In the limit $k \to \infty$ we have
\[ \lim_{k \to \infty} \opn{RS}_{k}(\alpha \vert \sigma) = 
\int_{\sigma} \alpha \]
and 
\[ \lim_{k \to \infty} \opn{RP}_{k}(\alpha \vert \sigma) = 
\opn{MI}_{}(\alpha \vert \sigma) , \]
so the proof is done.
\end{proof}

\begin{cor} \label{cor:7}
Let $\alpha \in \Omega^1_{\mrm{pws}}(X) \otimes \g$, and let $\sigma$ be an
$\alpha$-tiny closed string in $X$ which bounds a polygon $Z$. Then
\[ \Norm{ \log_G \bigl( \opn{MI}(\alpha \vert \sigma) \bigr) } \leq 
c_0(G) \cdot c_1(\alpha)^2 \cdot \opn{len}(\sigma)^2 +
\opn{area}(Z) \cdot \norm{\alpha}_{\mrm{Sob}} \, . \]
\end{cor}

\begin{proof}
Choose an orientation on $Z$.
By the abelian Stokes Theorem (Theorem \ref{thm:15}) we have
\[ \Norm{ \int_{\sigma} \alpha } = \Norm{ \int_Z \d(\alpha) }
\leq \opn{area}(Z) \cdot \norm{\alpha}_{\mrm{Sob}} \, . \]
Now combine this estimate with the Proposition above.
\end{proof}

\begin{prop} \label{prop:21}
Let $\Phi : G \to \opn{GL}_m(\R)$ be a representation,
and let $\alpha \in \Omega^1_{\mrm{pws}}(X) \otimes \g$.
Then there are constants $\epsilon_5(\alpha, \Phi)$ and
$c_5(\alpha, \Phi)$ such that 
\[ 0 < \epsilon_5(\alpha, \Phi) \leq 1 \quad \text{and} \quad
c_5(\alpha, \Phi) \geq 1 \, , \]
and such that conditions \tup{(i)-(iii)} below  hold for every 
string $\sigma$ in $X$ satisfying 
$\opn{len}(\sigma) < \epsilon_5(\alpha, \Phi)$.
Let us write
\[ \alpha' := \opn{Lie}(\Phi)(\alpha) \in 
\Omega^1_{\mrm{pws}}(X) \otimes \mfrak{gl}_n(\R) , \]
and
\[ g' := \opn{MI}(\alpha' \vert \sigma) = \Phi 
\bigl( \opn{MI}(\alpha \vert \sigma) \bigr) \in \opn{GL}_m(\R) . \]
Let $\bsym{1}$ be the identity operator on $\R^m$, and let $\norm{-}$ denote
the operator norm on $\R^m$. 
The conditions are:
\begin{itemize}
\rmitem{i} 
\[ \norm{ g' - \bsym{1} } \leq  c_5(\alpha, \Phi) \cdot \opn{len}(\sigma) \, .
\]
\rmitem{ii}
\[ \Norm{ g' - \bigl( \bsym{1} + \int_{\sigma} \alpha' \bigr) } 
\leq c_5(\alpha, \Phi) \cdot \opn{len}(\sigma)^2 \, . \]
\rmitem{iii} 
Assume $\alpha'$ is smooth. Let $x_0$ be the initial point of $\sigma$, and let
$\alpha'(x_0)$ be the constant form defined in Definition 
\tup{\ref{dfn:26}}. Then
\[ \Norm{ g' - \bigl( \bsym{1} + \int_{\sigma} \alpha'(x_0) \bigr) } 
\leq c_5(\alpha, \Phi) \cdot \opn{len}(\sigma)^2 \, . \]
\end{itemize}
\end{prop}

\begin{proof}
Let us write $H := \opn{GL}_m(\R)$ and
$\h := \mfrak{gl}_n(\R)$. 
Define $\epsilon_5(\alpha, \Phi) := \epsilon_1(\alpha')$, in the sense of
Definition \ref{dfn:13}, and let
$d := c_1(\alpha') \cdot \epsilon_1(\alpha')$.

By reasons of convergence of analytic functions on compact domains,
there is a constant $c$ such that 
for every matrix $\lambda \in \h$ with 
$\norm{\lambda} \leq d$ the inequalities
\begin{equation} \label{eqn:144}
\Norm{ \exp_H(\lambda) - \bsym{1} } \leq c \cdot \norm{\lambda}
\end{equation}
and
\begin{equation} \label{eqn:145}
\Norm{ \exp_H(\lambda) - (\bsym{1} + \lambda) } \leq c \cdot \norm{\lambda}^2 
\end{equation}
hold. 

Take a string $\sigma$ with 
$\epsilon := \opn{len}(\sigma) < \epsilon_5(\alpha, \Phi)$.
By Proposition \ref{prop:1}(3) we have
$g' \in V_0(H)$ and 
\[ \norm{\lambda} \leq c_1(\alpha') \cdot \epsilon \leq  d \, , \]
for the elements
$g' := \opn{MI}(\alpha' \vert \sigma)$ and
$\lambda := \log_H(g') \in \h$.
Inequality (\ref{eqn:144}) gives
\begin{equation} \label{eqn:146}
\norm{ g' - \bsym{1} } \leq c \cdot  c_1(\alpha') \cdot \epsilon 
\, . 
\end{equation}
Next,  using Proposition \ref{prop:16} and
inequality (\ref{eqn:145}) we have 
\begin{equation} \label{eqn:147}
\begin{aligned}
& \Norm{ g' - \bigl( \bsym{1} + \int_{\sigma} \alpha' \bigr) } \leq 
\Norm{ g' - (\bsym{1} + \lambda) } +
\norm{ \lambda - \int_{\sigma} \alpha' } \\
& \qquad \leq c \cdot c_1(\alpha')^2  \cdot \epsilon^2
+  c_1(\alpha')^2 \cdot \epsilon^2 \, . 
\end{aligned}
\end{equation}

Finally, assume that $\alpha'$ is smooth. By Taylor expansion of the
coefficients of $\alpha'$ (cf.\ (\ref{eqn:67})) we have the estimate
\[ \Norm{ \alpha'(x) - \alpha'(x_0) } \leq \norm{ \alpha' }_{\mrm{Sob}} \cdot 
\epsilon \]
for every point $x$ in the image of the string $\sigma$. 
Therefore
\begin{equation} \label{eqn:148}
\Norm{ \int_{\sigma} \alpha' - \int_{\sigma} \alpha'(x_0) } \leq
\norm{ \alpha' }_{\mrm{Sob}} \cdot \epsilon^2 \, .
\end{equation}

{}From the inequalities (\ref{eqn:146}), (\ref{eqn:147}) and (\ref{eqn:148})
we can now easily extract a constant $c_5(\alpha, \Phi)$ for which all three
conditions hold.
\end{proof}

\cleardoublepage
\section{Multiplicative Integration in Dimension Two}
\label{sec.MI}

We pass up to dimension $2$. Here it turns out that things are
really much more complicated, for geometrical reasons.

A rudimentary multiplicative integration on surfaces was already
introduced by Schlesinger in 1928, in his
nonabelian $2$-dimensional Stokes Theorem
(see \cite[Appendix A.II,9]{DF}).
However we need stronger results, for which a more complicated
multiplicative integration procedure is required.

It turns out (this was already in Schlesinger's work) that the correct
multiplicative integral is twisted: there is a $2$-form, say $\beta$,
that is integrated, but this integration is twisted by a $1$-form
$\alpha$. The geometric cycles on which integration is performed are
the kites, to be defined now.

\subsection{Kites} \label{subsec:kites}
By a {\em pointed polyhedron} $(X, x_0)$ we mean a polyhedron $X$
(see Section \ref{sec:pws}), together with a base point $x_0 \in X$.
As base point for $\mbf{I}^2$ we always take the vertex
$v_0$; cf.\ (\ref{eqn:1}). 

\begin{dfn}
Let $(X, x_0)$ be a pointed polyhedron.
\index{Quadrangular kite}
\begin{enumerate}
\item A {\em quadrangular kite} in $(X, x_0)$ is a pair $(\sigma, \tau)$, where
$\sigma$ is a string in $X$ (see Definition \ref{dfn:3}), and 
$\tau : \mbf{I}^2 \to X$
is a linear map. The conditions are that
$\sigma(v_0) = x_0$ and
$\sigma(v_1) = \tau(v_0)$.

\item If $\tau(\mbf{I}^{2})$ is $2$-dimensional then we call $(\sigma, \tau)$
a {\em nondegenerate kite}. 

\item If $\tau(\mbf{I}^{2})$ is a square in $X$ (of positive size), then we
call $(\sigma, \tau)$ a {\em square kite}.
\end{enumerate} 
\end{dfn}

See Figure \ref{fig:1} for illustration. 

Until Section \ref{sec:simpl}, where triangular kites are introduced, we shall
only encounter quadrangular kites. Hence is Sections
\ref{sec.MI}-\ref{sec:stokes-3} a kite shall always mean a quadrangular kite.

Consider a kite $(\sigma, \tau)$. The image $\tau(\mbf{I}^{2})$ 
is a parallelogram in $X$. We denote the area of 
$\tau(\mbf{I}^{2})$ by $\opn{area}(\tau)$. 
If $\tau(\mbf{I}^{2})$ is a square, then we denote the side of this square by
$\opn{side}(\tau)$. 

We view $(\mbf{I}^2, v_0)$ as an oriented pointed polyhedron.
Suppose $(\sigma, \tau)$ is a nondegenerate kite in $(\mbf{I}^2, v_0)$.
If the orientation of $\tau$ is positive, then we say that the kite
$(\sigma, \tau)$ is positively oriented.

We shall need the following composition operation on kites.
Suppose $(\sigma_1, \tau_1)$ is a kite in $(X, x_0)$, and 
$(\sigma_2, \tau_2)$ is  a kite in $(\mbf{I}^2, v_0)$. 
Then 
$\tau_1 \circ \tau_2 : \mbf{I}^{2} \to X$ 
is a linear map, and 
$\sigma_1 * (\tau_1 \circ \sigma_2)$
is a string in $X$. (See (\ref{eqn12}) for the concatenation
operation $\ast$.) We define
\begin{equation} \label{eqn:18}
(\sigma_1, \tau_1) \circ (\sigma_2, \tau_2) :=
\bigl( \sigma_1 \ast (\tau_1 \circ \sigma_2), \tau_1 \circ \tau_2 \bigr) ,
\end{equation}
which is also a kite in $(X, x_0)$.  
Note that this composition operation is associative.
For an illustration see Figures \ref{fig:1} and \ref{fig:2}.

\begin{figure} 
\includegraphics[scale=0.25]{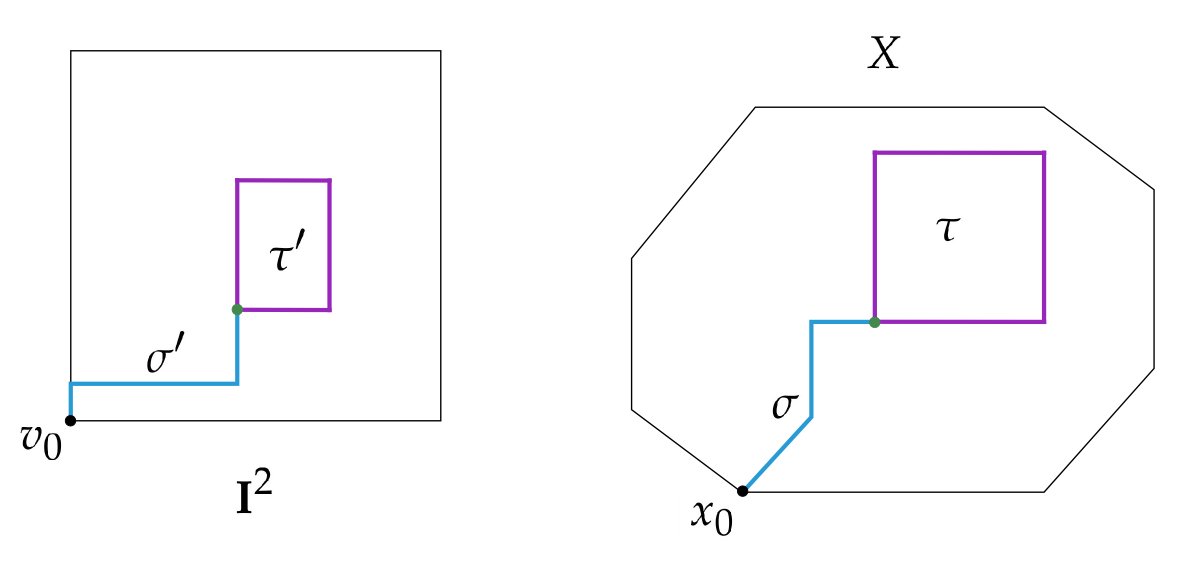}
\caption{A linear quadrangular kite $(\sigma, \tau)$ in the pointed polyhedron
$(X, x_0)$, and a linear quadrangular kite $(\sigma', \tau')$ in the pointed
polyhedron $(\mbf{I}^{2}, v_0)$.} 
\label{fig:1}
\end{figure}

\begin{figure} 
\includegraphics[scale=0.23]{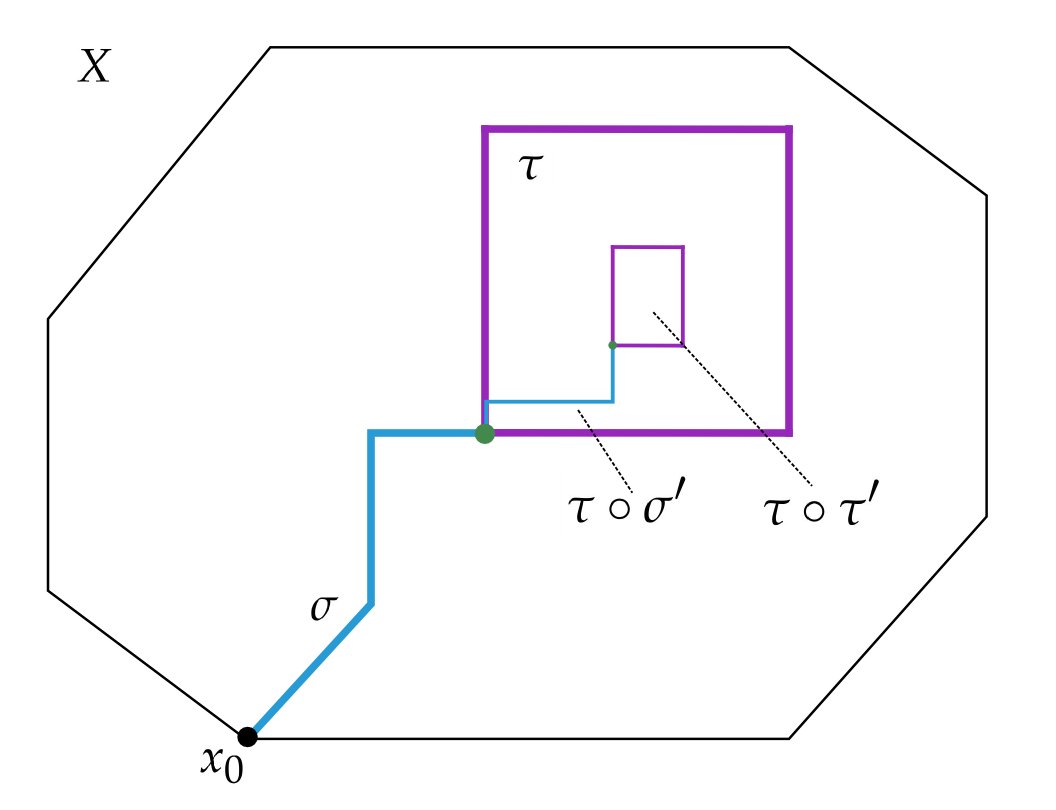}
\caption{Continued from Figure \ref{fig:1}: the linear quadrangular kite 
$(\sigma, \tau) \circ (\sigma', \tau')$ 
in the pointed polyhedron $(X, x_0)$.} 
\label{fig:2}
\end{figure}

Let $(\sigma, \tau)$ be a kite in $(X, x_0)$, and let 
$f : (X, x_0) \to (Y, y_0)$ be a piecewise linear map 
between pointed polyhedra. Assume that the restriction of $f$ to the 
subpolyhedron $\tau(\mbf{I}^2) \subset X$ is linear. 
As in formula (\ref{eqn14}) we have a string 
$f \circ \sigma$ in $Y$. We define 
\begin{equation} \label{eqn:116}
f \circ (\sigma, \tau) := (f \circ \sigma, f \circ \tau) ,
\end{equation}
which is a kite in $(Y, y_0)$.

Given a kite $(\sigma, \tau)$ in $(X, x_0)$, its {\em boundary} 
is the closed string $\partial (\sigma, \tau)$ defined as follows.
First we define
\begin{equation} \label{eqn:126}
\partial \mbf{I}^2 := 
(v_0, v_1) \ast (v_1, (1, 1)) \ast 
((1, 1), v_2) \ast (v_2, v_0) \, . 
\end{equation}
This is a closed string in $\mbf{I}^2$, based at $v_0$. Next we let
$\partial \tau := \tau \circ (\partial \mbf{I}^2)$,
where composition is in the sense of (\ref{eqn14}). 
So $\partial \tau$ is a closed string in $X$. Finally  we define
\begin{equation}
\partial (\sigma, \tau) :=
\sigma \ast (\partial \tau) \ast \sigma^{-1} , 
\end{equation}
where $\sigma^{-1}$ is the inverse string from (\ref{eqn13}).
See Figure \ref{fig:23}.

\begin{figure}
\includegraphics[scale=0.23]{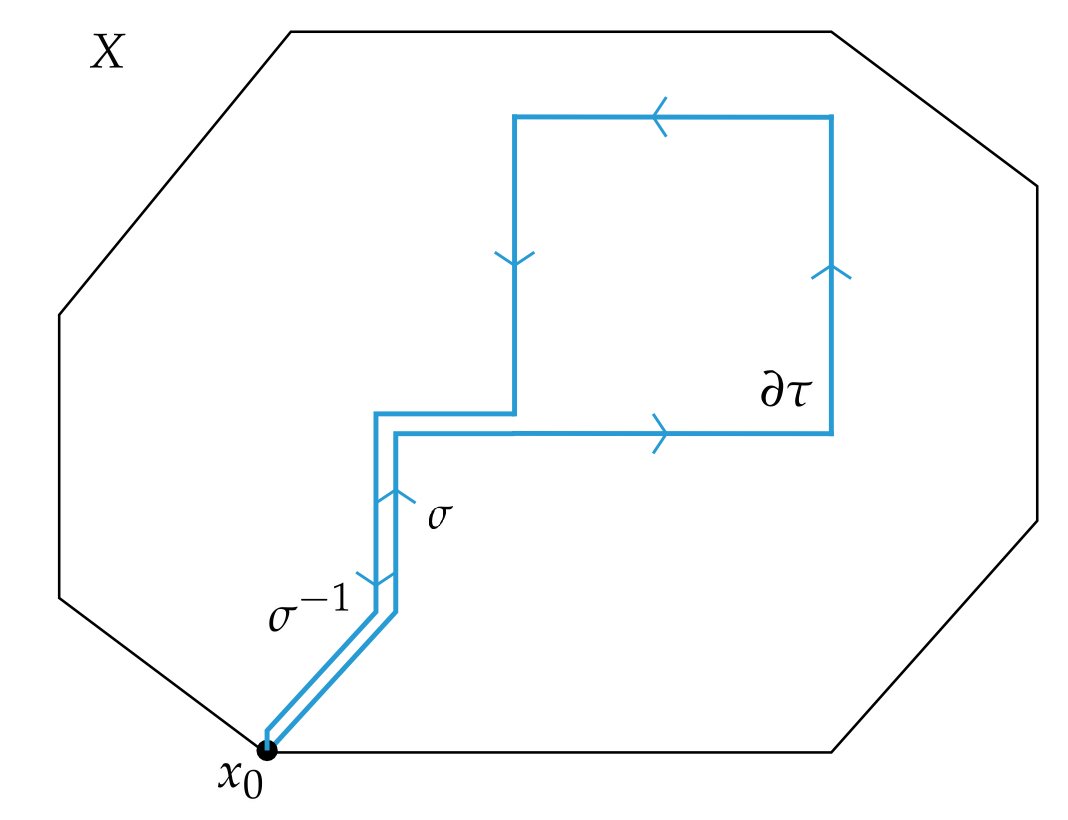}
\caption{The boundary 
$\partial (\sigma, \tau) = \sigma * \partial \tau * \sigma^{-1}$ 
of the kite
$(\sigma, \tau)$ from Figure \ref{fig:1}.} 
\label{fig:23}
\end{figure}

Here is a useful fact about the geometry of kites.

\begin{prop} \label{prop:14}
Let $(\sigma, \tau)$ be a kite in $(X, x_0)$. Then there is a square kite
$(\sigma', \tau')$ in $(\mbf{I}^{2}, v_0)$, and a piecewise linear map of
pointed polyhedra
$f : (\mbf{I}^{2}, v_0) \to (X, x_0)$, 
such that $\opn{len}(\sigma') \leq 1$, 
$f$ is linear on $\tau'(\mbf{I}^{2})$, and 
\[ (\sigma, \tau) = f \circ (\sigma', \tau') \]
as kites in $(X, x_0)$.
\end{prop}

\begin{proof}
Say $\sigma = (\sigma_1, \ldots, \sigma_m)$ is the decomposition of $\sigma$
into pieces. 
If $m = 0$ (i.e.\ $\sigma$ is the empty string), then
we let $\sigma'$ also be the empty string, and we take
$\tau' := \opn{id}_{\mbf{I}^{2}}$ and $f := \tau$.

Otherwise, if $m > 0$, then we define the square kite 
$(\sigma', \tau')$ in $(\mbf{I}^{2}, v_0)$ as follows. 
The square map $\tau' : \mbf{I}^{2} \to \mbf{I}^{2}$ is defined on vertices by
the formula
\[ \tau'(v_0, v_1, v_2) := \bigl( (\smfrac{1}{2}, \smfrac{1}{2}), 
(1, \smfrac{1}{2}), (\smfrac{1}{2}, 1) \bigr) . \]
And we let
$Y := \tau'(\mbf{I}^{2})$. 

Next consider the oriented line segment $Z$ going from $v_0$ to 
$(\smfrac{1}{2}, \smfrac{1}{2})$. We divide $Z$ into $m$ equal pieces,
labeled $Z_1, \ldots, Z_m$. We let $\sigma'_i : \mbf{I}^{1} \to Z$
be the positively oriented linear map with image $Z_i$. 
And we let $\sigma'$ be the string 
$\sigma' := (\sigma'_1, \ldots, \sigma'_m)$.

The map $f|_{Y}$ is defined to be the unique linear map $Y \to X$ such that 
$f \circ \tau' = \tau$. And for every $i$ the map $f|_{Z_i}$ defined to be the
unique piecewise linear map $Z_i \to X$ such that 
$f \circ \sigma'_i = \sigma_i$. 
We thus have a map 
$f|_{Y \cup Z} : Y \cup Z \to X$; see Figure \ref{fig:67} for an illustration.

Finally let $g : \mbf{I}^{2} \to Y \cup Z$ be any piecewise linear retraction;
for instance as suggested by Figure \ref{fig:43}.
We define 
$f := f|_{Y \cup Z} \circ g : \mbf{I}^{2} \to X$. 
See Figure \ref{fig:42}.
\end{proof}

\begin{figure}
\includegraphics[scale=0.35]{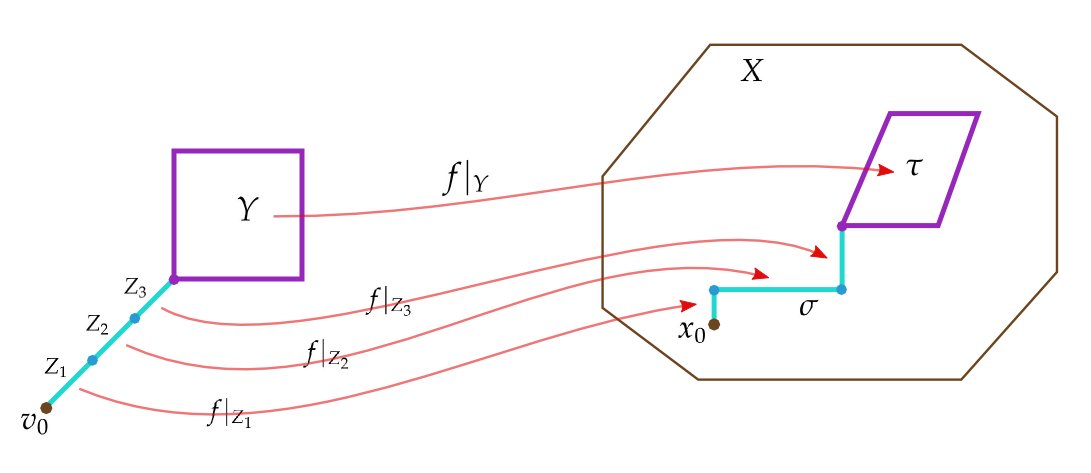}
\caption{The map
$f|_{Y \cup Z} : Y \cup Z \to X$, where
$Z = Z_1 \cup Z_2 \cup Z_3$.} 
\label{fig:67}
\end{figure}

\begin{figure}
\includegraphics[scale=0.34]{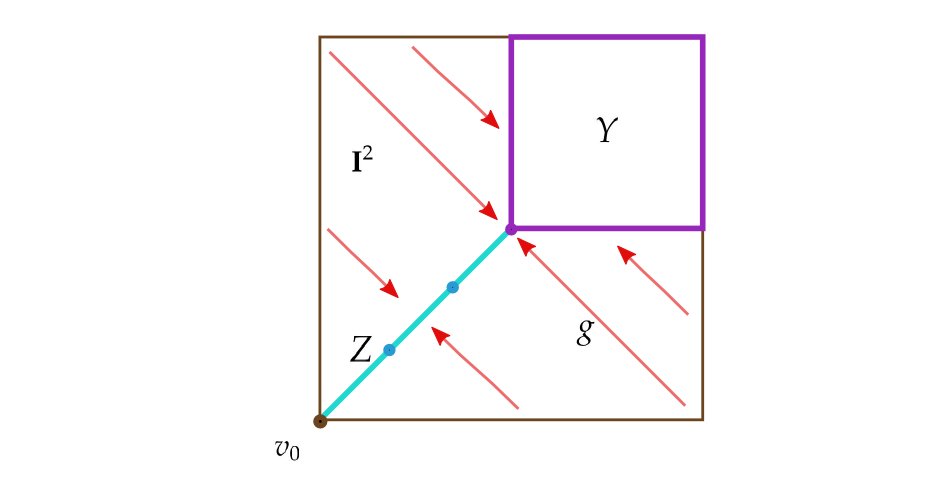}
\caption{A piecewise linear retraction
$g : \mbf{I}^{2} \to Y \cup Z$.} 
\label{fig:43}
\end{figure}

\begin{figure}
\includegraphics[scale=0.35]{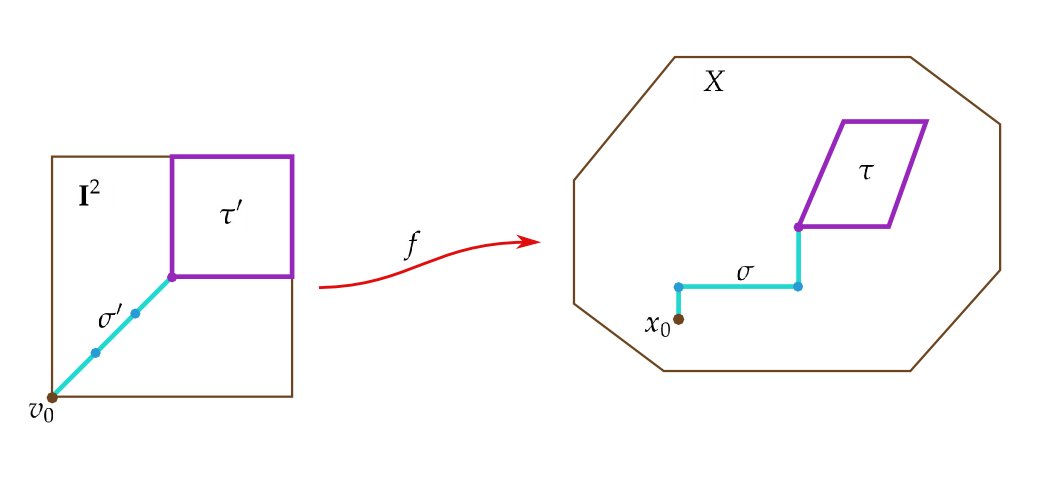}
\caption{The piecewise linear map
$f := f|_{Y \cup Z} \circ g$, and the kite $(\sigma', \tau')$ such that 
$f \circ (\sigma', \tau') = (\sigma, \tau)$.} 
\label{fig:42}
\end{figure}

\newpage
\subsection{Binary Tessellations of $\mbf{I}^2$} \label{subsec:bintes}
For $k \geq 0$ the {\em $k$-th binary subdivision of $\mbf{I}^2$} is the
cellular subdivision $\opn{sd}^k \mbf{I}^2$ of $\mbf{I}^2$ into $4^k$ squares,
each of side $(\smfrac{1}{2})^k$.
The $1$-skeleton of $\opn{sd}^k \mbf{I}^2$ is the set 
$\opn{sk}_1 \opn{sd}^k \mbf{I}^2$ 
consisting of the union of all edges (i.e.\
$1$-cells) in $\opn{sd}^k \mbf{I}^2$. Thus $\opn{sk}_1 \opn{sd}^k \mbf{I}^2$
is a grid.
See Figure \ref{fig:22}. 

\begin{figure}
\includegraphics[scale=0.22]{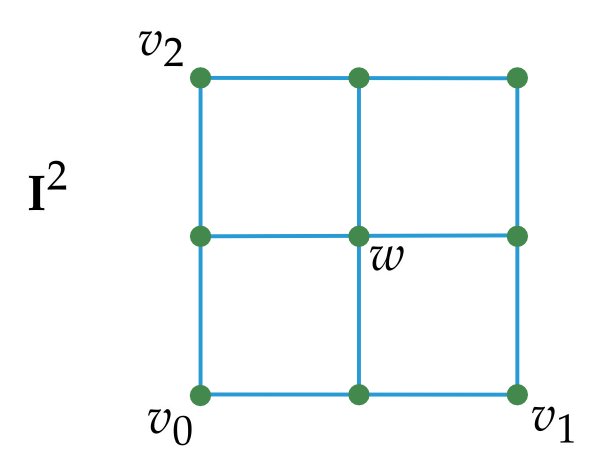}
\caption{The $1$-st binary subdivision $\opn{sd}^1 \mbf{I}^2$ of
$\mbf{I}^2$. Here $w := (\smfrac{1}{2} , \smfrac{1}{2})$.}
\label{fig:22}
\end{figure}

\begin{dfn}
Let $k$ be a natural number.
\begin{enumerate}
\item Let $p = 0, 1, 2$. 
A linear map $\sigma : \mbf{I}^{p} \to \mbf{I}^{2}$
is said to be {\em patterned on $\opn{sd}^k \mbf{I}^2$}
if the image $\sigma(\mbf{I}^{p})$ is  a $p$-cell of  
$\opn{sd}^k \mbf{I}^2$.

\item  A string $\sigma = (\sigma_1, \ldots, \sigma_m)$ in $\mbf{I}^{2}$
is said to be patterned on $\opn{sd}^k \mbf{I}^2$
if each piece $\sigma_i : \mbf{I}^{1} \to \mbf{I}^{2}$ is a linear map patterned
on $\opn{sd}^k \mbf{I}^2$.

\item A square kite $(\sigma, \tau)$ in $(\mbf{I}^{2}, v_0)$ is said to be 
patterned on $\opn{sd}^k \mbf{I}^2$ if both the linear map $\tau$ and the
string $\sigma$ are patterned on $\opn{sd}^k \mbf{I}^2$.
\end{enumerate}
\end{dfn}

The fundamental group 
of the topological space $\opn{sk}_1 \opn{sd}^k \mbf{I}^2$, based at
$v_0$, is denoted by
$\bsym{\pi}_1 (\opn{sk}_1 \opn{sd}^k \mbf{I}^2)$.
It is a free group on $4^k$ generators.
Given a closed string $\sigma$ patterned on 
$\opn{sd}^k \mbf{I}^2$ and based at
$v_0$, we denote by $[\sigma]$ the corresponding element of
$\bsym{\pi}_1 (\opn{sk}_1 \opn{sd}^k \mbf{I}^2)$.
In particular, if $(\sigma, \tau)$ is a kite patterned on 
$\opn{sd}^k \mbf{I}^2$, then the
boundary $\partial (\sigma, \tau)$ represents an element 
\[ [\partial (\sigma, \tau)] \in 
\bsym{\pi}_1 (\opn{sk}_1 \opn{sd}^k \mbf{I}^2) . \]

Recall the boundary $\partial \mbf{I}^2$ from equation (\ref{eqn:126}).

\begin{dfn} \label{dfn:4}
Let $k$ be a natural number.
A {\em tessellation of $\mbf{I}^{2}$ patterned on $\opn{sd}^k \mbf{I}^2$} is a
sequence
\[ \bsym{\rho} = \bigl( (\sigma_1, \tau_1), \ldots, 
(\sigma_{4^k}, \tau_{4^k}) \bigr) \]
of kites in $(\mbf{I}^{2}, v_0)$, satisfying these two
conditions:
\begin{enumerate}
\rmitem{i} Each kite  $(\sigma_i, \tau_i)$ is patterned on 
$\opn{sd}^k \mbf{I}^2$.
\rmitem{ii} One has
\[ \prod_{i = 1}^{4^k} \ [\partial (\sigma_i, \tau_i)] = 
[\partial \mbf{I}^{2}] \]
in the group $\bsym{\pi}_1 (\opn{sk}_1 \opn{sd}^k \mbf{I}^2)$.
The product is according to the convention (\ref{eqn:6}). 
\end{enumerate}
\end{dfn}

\begin{rem}
Suppose $\bsym{\rho}$ is a tessellation of $\mbf{I}^{2}$ patterned on 
$\opn{sd}^k \mbf{I}^2$. Then, in the notation of the definition, each kite 
$(\sigma_i, \tau_i)$ is positively oriented, and 
each $2$-cell of $\opn{sd}^k \mbf{I}^2$ occurs as
$\tau_i(\mbf{I}^{2})$ for exactly one index $i$.

This assertion (that we will not use in the paper) can be proved directly, by a
topological argument. But it also follows from
Corollary \ref{cor:2}, by taking the abelian Lie groups $G := 1$ and $H :=
\mrm{GL}_1(\R)$, and the differential forms
$\alpha := 0$ and 
$\beta := f \cdot \d t_1 \wedge \d t_1$, where 
$f : \mbf{I}^{2} \to \R$ is a smooth nonnegative bump
function supported in the interior of a given $2$-cell of 
$\opn{sd}^k \mbf{I}^2$.
\end{rem}

The {\em probe} is the string 
\begin{equation} \label{eqn:237}
\sigma_{\mrm{pr}} := \bigl( v_0, (0, \smfrac{1}{2}) \bigr) * 
\bigl( (0, \smfrac{1}{2}), (\smfrac{1}{2}, \smfrac{1}{2}) \bigr) 
\end{equation}
(with two pieces) in $\mbf{I}^2$. 
See Figure \ref{fig:73} for an illustration and Remark \ref{rem:4} for an
explanation.

\begin{figure}
\includegraphics[scale=0.20]{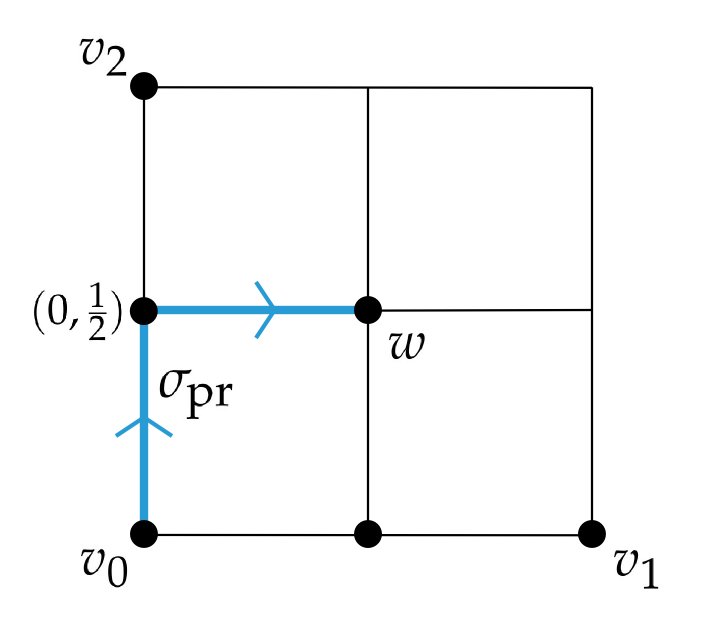}
\caption{The probe $\sigma_{\mrm{pr}}$ in $\mbf{I}^2$. Here 
$w := (\smfrac{1}{2}, \smfrac{1}{2})$.} 
\label{fig:73}
\end{figure}

\begin{dfn} \label{dfn:14}
\index{Binary tessellation of the square}
Let $k$ be a natural number. The {\em $k$-th binary tessellation of
$\mbf{I}^{2}$} is the sequence
\[ \opn{tes}^k \mbf{I}^2 = 
\bigl( \opn{tes}^k_1 \mbf{I}^2, \ldots, \opn{tes}^k_{4^k} \mbf{I}^2  \bigr)
= \bigl( (\sigma^k_1, \tau^k_1), \ldots, (\sigma^k_{4^k}, \tau^k_{4^k}) \bigr)
\] 
of kites in $(\mbf{I}^{2}, v_0)$, patterned on $\opn{sd}^k \mbf{I}^2$, that is
 defined recursively as follows. 
\begin{enumerate}
\item For $k = 0$ let $\sigma^0_1$ be the empty string, and let
$\tau^0_1$ be the identity map of $\mbf{I}^{2}$.

\item For $k = 1$ all four strings $\sigma^1_i$ are the same; they are
$\sigma^1_i := \sigma_{\mrm{pr}}$. The four linear maps 
$\tau^1_i : \mbf{I}^{2} \to \mbf{I}^{2}$ 
are patterned on $\opn{sd}^1 \mbf{I}^2$, positively oriented, and have
$\tau^1_i(v_0) = (\smfrac{1}{2}, \smfrac{1}{2})$. 
It remains to specify the points $\tau^1_i(v_1)$:
\[ \tau^1_1(v_1) = (\smfrac{1}{2}, 0) , \
\tau^1_2(v_1) = (1, \smfrac{1}{2}) , \
\tau^1_3(v_1) = (\smfrac{1}{2}, 1) , \
\tau^1_4(v_1) = (0, \smfrac{1}{2}) . \]

\item For $k \geq 2$ we define
\[ \opn{tes}^k \mbf{I}^2 :=
(\opn{tes}^1 \mbf{I}^2) \circ (\opn{tes}^{k-1} \mbf{I}^2) . \]
Here composition of sequences of kites is using the
operations (\ref{eqn:18}) and (\ref{eqn:19}).
\end{enumerate}
\end{dfn}

See Figures \ref{fig:24} and \ref{fig:25} for an illustration.

\begin{figure}
\includegraphics[scale=0.36]{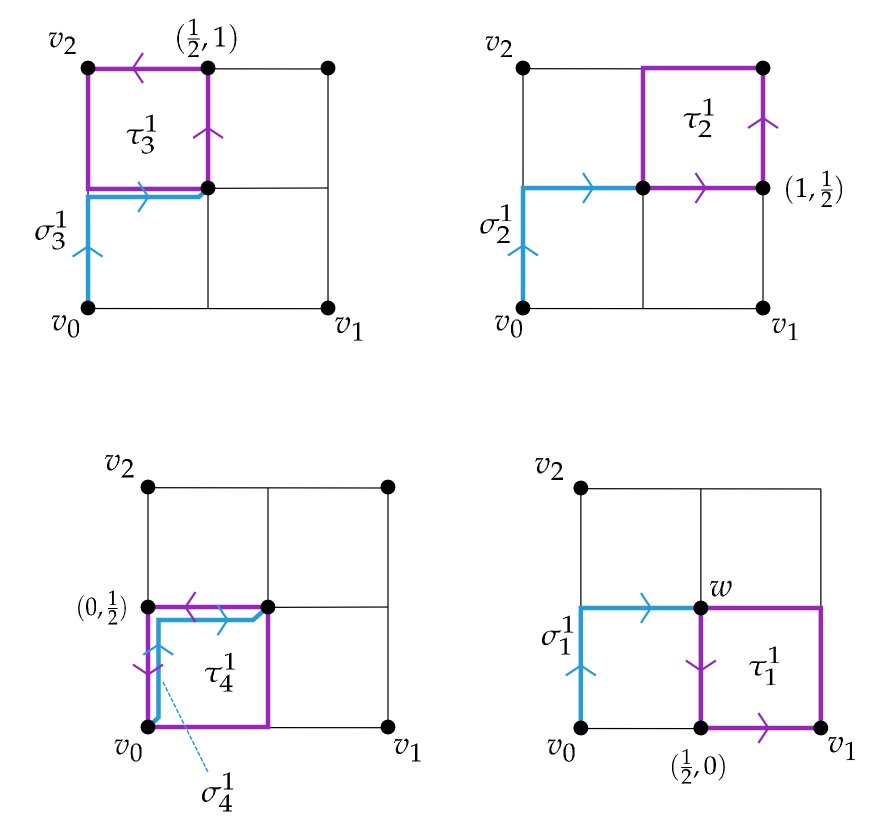}
\caption{The $1$-st binary tessellation $\opn{tes}^{1} \mbf{I}^{2}$.
The arrowheads indicate the orientation of the linear maps.} 
\label{fig:24}
\end{figure}

\begin{figure}
\includegraphics[scale=0.22]{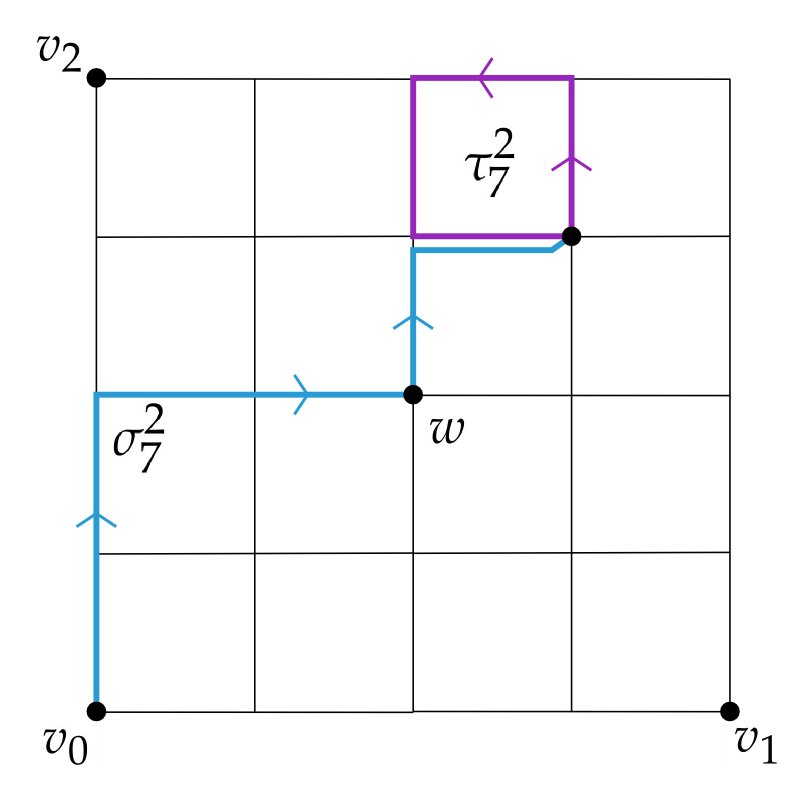}
\caption{The kite $(\sigma^{2}_7, \tau^{2}_7)$ in 
$\opn{tes}^{2} \mbf{I}^{2}$.} 
\label{fig:25}
\end{figure}

It is clear that sequence $\opn{tes}^k \mbf{I}^2$ is a tessellation of
$\mbf{I}^{2}$ patterned on $\opn{sd}^k \mbf{I}^2$, in the sense of Definition 
\tup{\ref{dfn:4}}.
We call $(\sigma^0_1, \tau^0_1)$ the {\em basic kite}.

Hopefully there will be no confusion between the linear map 
$\sigma^k_i$ belonging $\opn{tes}^k \mbf{I}^1$, and the string 
$\sigma^k_i$ belonging $\opn{tes}^k \mbf{I}^2$; these are distinct objects that
share the same notation. 

An easy calculation shows that
\begin{equation} \label{eqn:127}
\opn{len}(\sigma^k_i) \leq 2
\end{equation}
for all $k$ and $i$.

\begin{dfn} \label{dfn:44}
Let $(\sigma, \tau)$ be a kite in $(X, x_0)$. For $k \in \N$ and 
$i \in \{ 1, \ldots, 4^k \}$ let
\[ \opn{tes}^k_i (\sigma, \tau) := (\sigma, \tau) \circ \opn{tes}^k_i \mbf{I}^2
= (\sigma, \tau) \circ (\sigma^k_i, \tau^k_i) , \]
which is also a kite in $(X, x_0)$.
The sequence of kites
\[ \opn{tes}^k (\sigma, \tau) := 
\bigl( \opn{tes}^k_1 (\sigma, \tau), \ldots, 
\opn{tes}^k_{4^k} (\sigma, \tau) \bigr) \]
is called the {\em $k$-th binary tessellation} of $(\sigma, \tau)$.
\end{dfn}

\begin{rem}
The choice of strings for the kites in the binary tessellations (and
the ordering of the kites) is clearly artificial, and also very
asymmetrical. As we shall see later, in favorable situations this
will not matter at all -- any other tessellation works! See
Corollary \ref{cor:2}.
\end{rem}

\subsection{Additive Twisting and Riemann Products} \label{subsec:rp}
Let $\h$ be a finite dimensional vector space (over $\R$). We denote by
$\opn{GL}(\h)$ the group of linear automorphisms of $\h$, which is a Lie group
(noncanonically isomorphic to $\opn{GL}_d(\R)$ for $d := \opn{dim} \h$).

\begin{dfn} \label{dfn:5}
\index{Twisting setup}
A {\em twisting setup} is the data 
$\mbf{C} = (G, H, \Psi_{\h})$, consisting of:
\begin{enumerate}
\item Lie groups $G$ and $H$, with Lie algebras $\g$ and $\h$ respectively.
\item A map of Lie groups
$\Psi_{\h} : G \to \opn{GL}(\h)$, called an {\em additive twisting}.
\end{enumerate}
\end{dfn}

Warning: we do not assume that the map
$\Psi_{\h}(g) : \h \to \h$, for $g \in G$,
is an automorphism of Lie algebras!

\begin{exa}
Let $G$ be any Lie group. Take $H := G$ and
$\Psi_{\h} := \opn{Ad}_{\h}$, the adjoint action of $H = G$ on its Lie algebra.
Then $(G, H, \Psi_{\h})$ is a twisting setup. Here $\Psi_{\h}(g)$ is in
fact an automorphism of Lie groups.
\end{exa}

Let us fix, for the rest of this section, a twisting setup
$\mbf{C} = (G, H, \Psi_{\h})$. 
The Lie algebras of $G$ and $H$ are $\g$ and $\h$ respectively. 

We choose some euclidean norm $\norm{-}_{\g}$ on the vector space $\g$.
As in Section \ref{sec:expon} we also choose an open neighborhood $V_0(G)$ of
$1$ in $G$ on which $\log_G$ is well-defined, a convergence radius
$\epsilon_0(G)$, and a commutativity constant $c_0(G)$. Likewise we choose
$\norm{-}_{\h}$, $V_0(H)$, $\epsilon_0(H)$ and $c_0(H)$.
Given $g \in G$, the linear operator $\Psi_{\h}(g) \in \opn{End}(\h)$ is given
the operator norm $\norm{\Psi_{\h}(g)}$.
It should be noted that these choices are auxiliary only, and do not effect
the definition of the multiplicative integration.

Piecewise smooth differential forms were discussed in Subsection
\ref{subsec:pws}. The string $\sigma_{\mrm{pr}}$ (the probe) was introduced in 
(\ref{eqn:237}).

\begin{dfn}[Basic Riemann Product] \label{dfn:11}
Let $(X, x_0)$ be a pointed polyhedron, let
\[ \alpha \in \Omega^1_{\mrm{pws}}(X) \otimes \g , \] 
let
\[ \beta \in \Omega^2_{\mrm{pws}}(X) \otimes \h ,  \]
and let $(\sigma, \tau)$ be a kite in $(X, x_0)$.
We define an element
\[ \opn{RP}_0(\alpha, \beta \vert \sigma, \tau) \in H , \]
called the {\em basic Riemann product of $(\alpha, \beta)$ on 
$(\sigma, \tau)$}, as follows. Write
$Z := \tau(\mbf{I}^2)$ and $z := \tau(\smfrac{1}{2}, \smfrac{1}{2}) \in Z$. 
There are two cases to consider:
\begin{enumerate}
\item Assume $\opn{dim} Z = 2$ and $z$ is a smooth point of the form 
$\beta|_Z$. Put on $Z$ the orientation compatible with $\tau$. 
Choose a triangle $Y$ in $Z$ such that $z \in \opn{Int} Y$ and 
$\beta|_Y$ is smooth, and let 
$\til{\beta} \in \mcal{O}(Y) \otimes \h$
be the coefficient of $\beta|_Y$ with respect to the orientation form of $Y$
(see Definition \ref{dfn:24}). Also let 
\[ g := \opn{MI}(\alpha \vert \sigma * (\tau \circ \sigma_{\mrm{pr}})) \in G .
\]
We define
\[ \opn{RP}_0(\alpha, \beta \vert \sigma, \tau) := 
\exp_H \bigl( \opn{area}(Z) \cdot \Psi_{\h}(g)(\til{\beta}(z)) \bigr)  . \]

\item If $\opn{dim} Z < 2$, or $\opn{dim} Z = 2$ and $z$ is a singular point of
$\beta|_Z$, we define
\[ \opn{RP}_0(\alpha, \beta \vert \sigma, \tau) := 1   . \]
\end{enumerate}
\end{dfn}

It is obvious that case (1) of the definition is independent of the triangle
$Y$. See Figure \ref{fig:74} for an illustration. 

\begin{figure}
\includegraphics[scale=0.31]{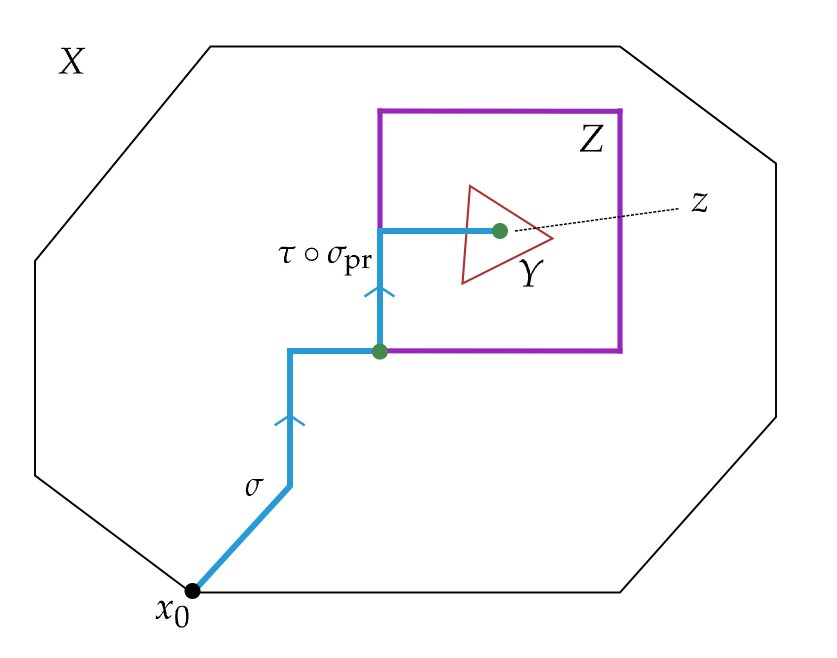}
\caption{Calculating $\opn{RP}_0(\alpha, \beta \vert \sigma, \tau)$ in the
smooth case. Here $Z = \tau(\mbf{I}^2)$, 
$z = \tau(\smfrac{1}{2}, \smfrac{1}{2})$, and $Y$ is a triangle
in $Z$ such that $\beta|_Y$ is smooth.} 
\label{fig:74}
\end{figure}

\begin{rem} \label{rem:4}
We call the string $\sigma_{\mrm{pr}}$ ``the probe'' because it reaches into
the middle of the square $\mbf{I}^2$. The ``reading'' it gives, namely the
formula for $\opn{RP}_0(\alpha, \beta \vert \sigma, \tau)$ in part (1) of the
definition above, is better than (\ref{eqn:238}), because it converges to the
limit faster: order of $\opn{side}(\tau)^4$ versus $\opn{side}(\tau)^3$.
Cf.\ Lemma \ref{lem:5} below.
\end{rem}

Suppose we are given a finite sequence of kites
\[ \bsym{\rho} =  
\bigl( (\sigma_1, \tau_1), \ldots, (\sigma_m, \tau_m) \bigr) \]
in $(X, x_0)$. We write
\begin{equation} \label{eqn:128}
\opn{RP}_{0} (\alpha, \beta \vert \bsym{\rho}) :=
\prod_{i = 1}^{m} \,
\opn{RP}_{0} \bigl( \alpha , \beta \vert \sigma_i, \tau_i \bigr) . 
\end{equation}

Recall the binary tessellation $\opn{tes}^k (\sigma, \tau)$ from Definition
\ref{dfn:44}.

\begin{dfn}[Refined Riemann Product]
\index{Refined Riemann product}
Let $(X, x_0)$ be a pointed polyhedron, let
$\alpha \in \Omega^1_{\mrm{pws}}(X) \otimes \g$, let
$\beta \in \Omega^2_{\mrm{pws}}(X) \otimes \h$, and let
$(\sigma, \tau)$ be a kite in $(X, x_0)$.
For $k \geq 0$ we define 
\[ \opn{RP}_{k}(\alpha, \beta \vert \sigma, \tau) :=
\opn{RP}_{0} \bigl( \alpha, \beta \vert \opn{tes}^k (\sigma , \tau) \bigr) = 
\prod_{i = 1}^{4^k} \,
\opn{RP}_{0} \bigl( \alpha, \beta \vert (\sigma, \tau) \circ 
(\sigma^k_i, \tau^k_i) \bigr) . \]
\end{dfn}

\begin{lem} \label{lem:38}
Let $f : (X', x'_0) \to (X, x_0)$ be a piecewise linear map of pointed
polyhedra,
and let $(\sigma', \tau')$ be a kite in $(X', x'_0)$. Assume that $f$ is linear
on $\tau'(\mbf{I}^2)$, and let
\[ (\sigma, \tau) := f \circ (\sigma', \tau') , \] 
which is a kite in $(X, x_0)$.
Let $\alpha' := f^*(\alpha)$ and 
$\beta' := f^*(\beta)$. Then
\[ \opn{RP}_k(\alpha, \beta \vert \sigma, \tau) = 
\opn{RP}_k(\alpha', \beta' \vert \sigma', \tau') \]
for any $k \geq 0$.
\end{lem}

\begin{proof}
Take $k \geq 0$ and $i \in \{ 1, \ldots, 4^k \}$.
We will prove that 
\begin{equation} \label{eqn:149}
 \opn{RP}_0(\alpha, \beta \vert (\sigma, \tau) \circ 
(\sigma^k_i, \tau^k_i)) = 
\opn{RP}_0(\alpha', \beta' \vert (\sigma', \tau') \circ 
(\sigma^k_i, \tau^k_i)) .
\end{equation}

Let $Z := \tau(\mbf{I}^2) \subset X$
and $Z' := \tau'(\mbf{I}^2) \subset X'$.
If $\opn{dim} Z < 2$ then 
$\beta|_{Z} = 0$ and $\beta'|_{Z'} = 0$, and hence both sides of
(\ref{eqn:149}) equal $1$. 

If $\opn{dim} Z = 2$ then the linear map 
$f|_{Z'} : Z' \to Z$ is bijective. Let
$w := (\smfrac{1}{2}, \smfrac{1}{2}) \in \mbf{I}^2$,
$z_i := (\tau \circ \tau^k_i)(w) \in Z$ and 
$z'_i := (\tau' \circ \tau^k_i)(w) \in Z'$.
Then $z_i$ is a smooth point of $\beta|_Z$ if and only if 
$z'_i$ is a smooth point of $\beta'|_{Z'}$. In the singular case again 
both sides of (\ref{eqn:149}) equal $1$. 

In the smooth case we know (by Proposition \ref{prop:12}) that
\[ \opn{MI}(\alpha \vert \sigma * (\tau \circ \sigma^k_i) *  
(\tau \circ \tau^k_i \circ \sigma_{\mrm{pr}})) = 
\opn{MI}(\alpha' \vert \sigma' * (\tau' \circ \sigma^k_i) *  
(\tau' \circ \tau^k_i \circ \sigma_{\mrm{pr}})) . \]
This says the twistings are the same. 
Let $\til{\beta}_i$ be the coefficient of $\beta$ near $z_i$, as in
case (1) of Definition \ref{dfn:11}; and let $\til{\beta}'_i$ be 
the coefficient of $\beta'$ near $z'_i$. Then
\[ \opn{area}(Z_i) \cdot \til{\beta}_i(z_i) =  
\opn{area}(Z'_i) \cdot \til{\beta}'_i(z'_i) . \]
So in this case we also get equality in (\ref{eqn:149}).
\end{proof}

\subsection{Convergence of Riemann Products}
We continue with the setup of Subsection \ref{subsec:rp}.
Fix differential forms
$\alpha \in \Omega^1_{\mrm{pws}}(X) \otimes \g$ and
$\beta \in \Omega^2_{\mrm{pws}}(X) \otimes \h$.

\begin{lem} \label{lem:4}
There are constants $c_1(\alpha, \beta)$ and
$\epsilon_1(\alpha, \beta)$ with the following properties.
\begin{enumerate}
\rmitem{i} The inequalities below hold:
\begin{gather*}
1 \leq c_1(\alpha, \beta) \\
0 < \epsilon_1(\alpha, \beta) \leq 1 \\
\epsilon_1(\alpha, \beta) \cdot c_1(\alpha, \beta)
\leq \smfrac{1}{4} \epsilon_0(H) \, .
\end{gather*}

\rmitem{ii} Suppose $(\sigma, \tau)$ is a square kite in $(X, x_0)$
such that 
$\opn{side}(\tau) < \epsilon_1(\alpha, \beta)$
and
$\opn{len}(\sigma) \leq 4 \cdot \opn{diam}(X)$.
Then for any $k \geq 0$ one has  
\[ \opn{RP}_{k}(\alpha, \beta \vert \sigma, \tau) \in V_0(H) \]
and
\[ \Norm{ \opn{log}_H \bigl( \opn{RP}_{k}(\alpha, \beta \vert \sigma, \tau)
\bigr) } \leq 
c_1(\alpha, \beta) \cdot \opn{side}(\tau)^2 \, . \]
\end{enumerate}
\end{lem}

\begin{proof}
Let $w := (\smfrac{1}{2}, \smfrac{1}{2}) \in \mbf{I}^2$.
Given a square kite $(\sigma, \tau)$, let 
$Z := \tau(\mbf{I}^2)$, $z := \tau(w)$ and $\epsilon := \opn{side}(\tau)$.

For $i \in \{ 1 ,\ldots, 4^k \}$ define 
\begin{equation} \label{eqn:150}
(\sigma_i, \tau_i) := (\sigma, \tau) \circ (\sigma^k_i, \tau^k_i) . 
\end{equation}
This is a square kite in $(X, x_0)$ satisfying
$\opn{side}(\tau_i) = (\smfrac{1}{2})^k  \epsilon$
and
$\opn{area}(\tau_i) = (\smfrac{1}{4})^k \epsilon^2$.
Let 
$Z_i := \tau_i(\mbf{I}^2)$ and $z_i := \tau_i(w)$.  
see Figures \ref{fig:60} and \ref{fig:61} for illustration. 
Since $\opn{len}(\sigma^k_i) \leq 2$ and 
$\epsilon \leq \opn{diam}(X)$, we have 
\begin{equation} \label{eqn:151}
\opn{len}(\sigma_i * (\tau_i \circ \sigma_{\mrm{pr}})) = 
\opn{len}(\sigma) + \epsilon \cdot \opn{len}(\sigma^k_i) 
+ (\smfrac{1}{2})^k  \epsilon 
\leq 7 \cdot \opn{diam}(X) .
\end{equation}

\begin{figure}
\includegraphics[scale=0.22]{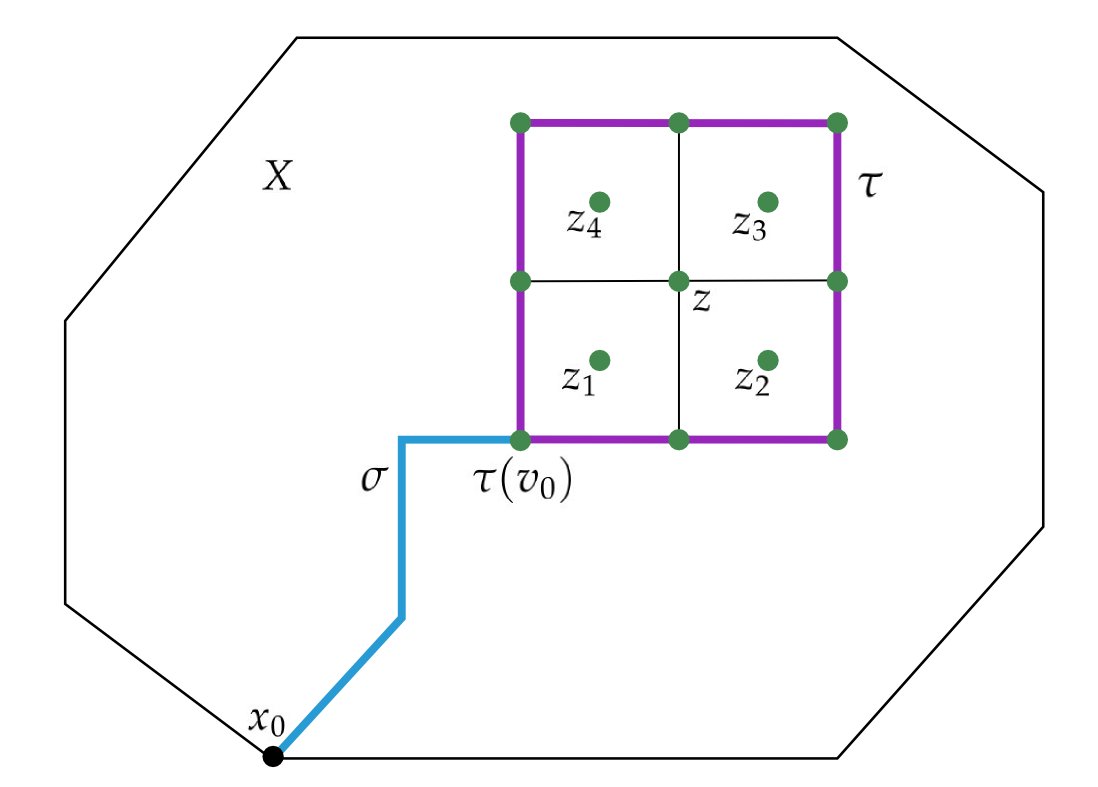}
\caption{The kite $(\sigma, \tau)$ in $(X, x_0)$ and the 
 points $z_1, \ldots, z_{4^k}$ in $\tau(\mbf{I}^{2})$. Here 
$k = 1$.} 
\label{fig:60}
\end{figure}

\begin{figure}
\includegraphics[scale=0.22]{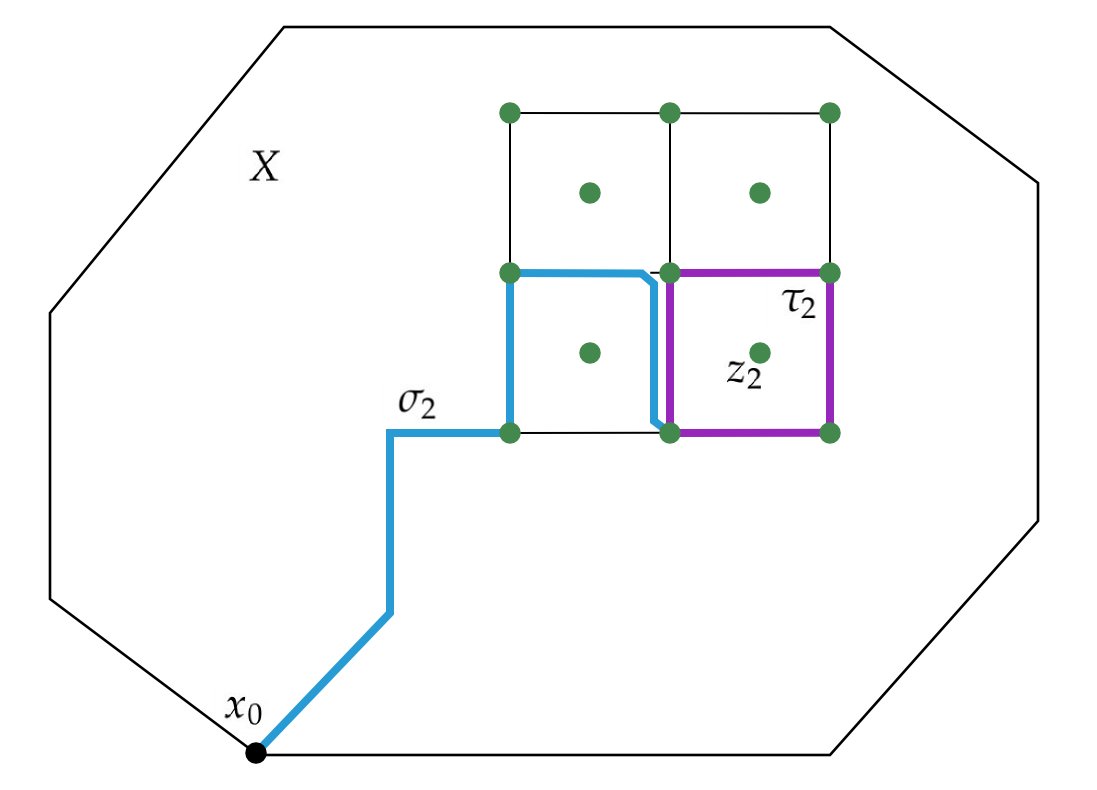}
\caption{The kite $(\sigma_2, \tau_2)$ in $(X, x_0)$.} 
\label{fig:61}
\end{figure}

Consider the group elements
\[ g_i := \opn{MI}(\alpha \vert \sigma_i * (\tau_i \circ \sigma_{\mrm{pr}})) \in
G \]
and
\[ h_i := \opn{RP}_0(\alpha, \beta \vert \sigma_i, \tau_i) \in H . \]
By definition of the Riemann product we have
\[ \opn{RP}_k(\alpha, \beta \vert \sigma, \tau) = 
\prod_{i = 1}^{4^k} \, h_i . \]

Take $i \in \{ 1 ,\ldots, 4^k \}$. If $z_i$ is a smooth point of 
$\beta|_{Z}$ (this is the good case), then let $\til{\beta}_i$ be the
coefficient of $\beta|_Z$ near $z_i$, and let
\[ \lambda_i := (\smfrac{1}{4})^k \cdot \epsilon^2 \cdot
\Psi_{\h}(g_i)(\til{\beta}_i(z_i)) \in \h . \]
Otherwise, if  $z_i$ is a singular point of $\beta|_{Z}$ (this is the bad
case), then we let $\lambda_i := 0$. In any case, by definition we have
$h_i = \exp_H(\lambda_i)$.

According to Proposition \ref{prop:4} and the inequality (\ref{eqn:151})
we have 
\[ \norm{\Psi_{\h}(g_i)} \leq 
\exp \, \bigl( c_4(\alpha, \Psi_{\h}) \cdot
7 \cdot \opn{diam}(X) \bigr)  \]
for some constant $c_4(\alpha, \Psi_{\h})$. Note that this bound is
independent of $k$ and $(\sigma, \tau)$.

Let
\[ c := \exp \, \bigl( c_4(\alpha, \Psi_{\h}) \cdot
7 \cdot \opn{diam}(X) \bigr) \cdot  \norm{\beta}_{\mrm{Sob}} + 1 , \]
\[ \epsilon_1(\alpha, \beta) := \epsilon_0(H)^{\smfrac{1}{2}}  
\cdot c^{-\smfrac{1}{2}}  \]
and
\[ c_1(\alpha, \beta) := c \cdot \bigl( c_0(H) + 1 \bigr) . \]
Now assume our kite satisfies $\epsilon < \epsilon_1(\alpha, \beta)$. 
Because 
$\norm{\til{\beta}_i(z_i)} \leq \norm{\beta}_{\mrm{Sob}}$
in the good case, and $\lambda_i = 0$ in the bad case, we see that 
\begin{equation*} \label{eqn:104}
\norm{\lambda_i} \leq (\smfrac{1}{4})^k \cdot \epsilon^2 \cdot c
< (\smfrac{1}{4})^k \cdot \epsilon_0(H) .
\end{equation*}
Hence 
$\sum_{i = 1}^{4^k} \norm{\lambda_i} < \epsilon_0(H)$, 
and by property (ii) of Theorem \ref{thm:6} we can deduce that
\[ \prod\nolimits_{i = 1}^{4^k} \, h_i \in V_0(H) \]
and
\[ \Norm{ \log_H \bigl( \prod\nolimits_{i = 1}^{4^k} \, h_i \bigr) }
\leq c_0(H) \cdot \bigl( \sum\nolimits_{i = 1}^{4^k} \, \norm{\lambda_i} \bigr)
\leq c_0(H) \cdot \epsilon^2 \cdot c 
\leq c_1(\alpha, \beta) \cdot \epsilon^2 \, . \]
\end{proof}

\begin{lem} \label{lem:5}
There are constant $\epsilon_2(\alpha, \beta)$ and
$c_2(\alpha, \beta)$ with the following properties:
\begin{itemize}
\rmitem{i} $c_2(\alpha, \beta) \geq c_1(\alpha, \beta)$ and
$0 < \epsilon_2(\alpha, \beta) \leq \epsilon_1(\alpha, \beta)$.
\rmitem{ii} Suppose $(\sigma, \tau)$ is a square kite in $(X, x_0)$ such that 
$\beta|_{\tau(\mbf{I}^{2})}$ is smooth,
$\opn{side}(\tau) < \epsilon_2(\alpha, \beta)$ and
$\opn{len}(\sigma) \leq 4 \cdot \opn{diam}(X)$. Then 
\[ \Norm{ \opn{log}_H \bigl( \opn{RP}_{k}(\alpha, \beta \vert \sigma, \tau)
\bigr) -
\opn{log}_H \bigl( \opn{RP}_{0}(\alpha, \beta \vert \sigma, \tau)
\bigr) } \leq 
c_2(\alpha, \beta) \cdot \opn{side}(\tau)^4  \]
for every $k \geq 0$.
\end{itemize}
\end{lem}

The exponent $4$ in ``$\opn{side}(\tau)^4$'' in the inequality above will be of
utmost importance later on. On the other hand, the factor $4$ appearing in
``$4 \cdot \opn{diam}(X)$'' is quite arbitrary (any number bigger than $2$ would
probably do just as well). 

\begin{proof}
Let $w := (\smfrac{1}{2}, \smfrac{1}{2}) \in \mbf{I}^2$.
Suppose $(\sigma, \tau)$ is some square kite in $(X, x_0)$. Write
$Z:= \tau(\mbf{I}^2)$, $z_0 := \tau(w) \in Z$ and 
$\epsilon := \opn{side}(Z)$. Assume that $\epsilon > 0$ and 
$\beta|_Z$ is smooth (otherwise there is nothing to prove). 
Put on $Z$ the orientation compatible with $\tau$. 

Let $\til{\beta} \in \mcal{O}(Z) \otimes \h$ be the coefficient of $\beta|_Z$,
as in Definition \ref{dfn:24}. Let
\[ g_0 := \opn{MI}(\alpha \vert \sigma * (\tau \circ \sigma_{\mrm{pr}})) \in G ,
\]
and 
\begin{equation*} \label{eqn:157}
\lambda_0 := \epsilon^2 \cdot \Psi_{\h}(g_0)(\til{\beta}(z_0)) \in \h .
\end{equation*}
So by definition
\[ \exp_H(\lambda_0) = \opn{RP}_0(\alpha, \beta \vert \sigma, \tau) . \]

Suppose $\pi$ is some string in $Z$, with initial point $z_0$. Let
$z$ be the terminal point of $\pi$.  
See Figure \ref{fig:76}.

\begin{figure}
\includegraphics[scale=0.18]{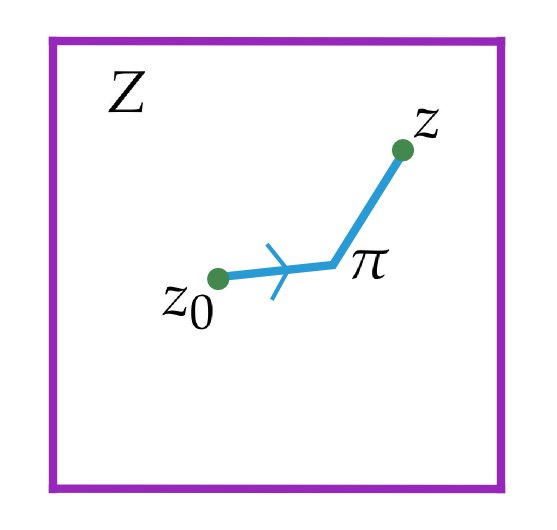}
\caption{A string $\pi$ in $Z = \tau(\mbf{I}^2)$, with initial point 
$z_0$ and terminal point $z$.} 
\label{fig:76}
\end{figure}

We shall need the following variant of
the Taylor expansion of $\til{\beta}(z)$ to second order around $z_0$:
\begin{equation*} \label{eqn:152}
\til{\beta}(z) = \til{\beta}(z_0) +  
\int_{\pi} (\d \til{\beta})(z_0) + \opn{R}^2(\til{\beta}, \pi) .
\end{equation*}
Here 
$\d \til{\beta} \in \Omega^1(Z) \otimes \h$, 
and 
$(\d \til{\beta})(z_0) \in \Omega^1_{\mrm{const}}(Z) \otimes \h$
is the associated constant form (see Definition \ref{dfn:26}). Thus 
$\int_{\pi} (\d \til{\beta})(z_0)$ is the linear term in the expansion -- it
depends linearly on $\pi$. And the quadratic remainder term 
$\opn{R}^2(\til{\beta}, \pi) \in \h$ has this bound:
\[ \norm{ \opn{R}^2(\til{\beta}, \pi) } \leq
\norm{ \beta }_{\mrm{Sob}} \cdot \opn{len}(\pi)^2 \, . \]
Therefore we get the estimate
\begin{equation} \label{eqn:154}
\Norm{ \til{\beta}(z) - \bigl( \til{\beta}(z_0) +  
\int_{\pi} (\d \til{\beta})(z_0) \bigr) } \leq
\norm{ \beta }_{\mrm{Sob}} \cdot \opn{len}(\pi)^2 \, . 
\end{equation}

Next let
\[ g := \opn{MI}(\alpha \vert \pi) \in G  \]
and 
\[ \alpha' := \opn{Lie}(\Psi_{\h})(\alpha) \in 
\Omega^1_{\mrm{pws}}(X) \otimes \opn{End}(\h) . \]
Assume that
$\opn{len}(\pi) < \epsilon_5(\alpha, \Psi_{\h})$. 
According to Proposition \ref{prop:21} we have this estimate for the operator
$\Psi_{\h}(g) \in \opn{End}(\h)$:
\begin{equation} \label{eqn:153}
\Norm{ \Psi_{\h}(g) - \bigl( \bsym{1} + \int_{\pi} \alpha'(z_0) \bigr) }
\leq c_5(\alpha, \Psi_{\h}) \cdot \opn{len}(\pi)^2 \, . 
\end{equation}
And by Proposition \ref{prop:13} we have the bound
\begin{equation} \label{eqn:155}
\Norm{ \Psi_{\h}(g) } \leq \exp( c_4(\alpha, \Psi_{\h}) \cdot 
\epsilon_5(\alpha, \Psi_{\h})) \, . 
\end{equation}

By combining inequalities (\ref{eqn:154}), (\ref{eqn:153}) and (\ref{eqn:155}),
we see that there exists a constant 
$c(\alpha, \beta)$ such that 
\begin{equation} \label{eqn:156}
\begin{aligned}
& \Norm{ \Psi_{\h}(g)(\til{\beta}(z)) - 
\Bigl( \til{\beta}(z_0) + \bigl( \int_{\pi} \alpha'(z_0) \bigl)
(\til{\beta}(z_0)) +  
\int_{\pi} (\d \til{\beta})(z_0) \Bigr) } \\
& \qquad \leq 
c(\alpha, \beta) \cdot \opn{len}(\pi)^2 \, . 
\end{aligned} 
\end{equation}
This holds for every string $\pi$ in $Z$ with 
$\opn{len}(\pi) < \epsilon_5(\alpha, \Psi_{\h})$,
$\pi(v_0) = z_0$ and $\pi(v_1) = z$.

Now take $k \geq 0$. For any index 
$i \in \{ 1, \ldots, 4^k \}$ 
let $\pi_i$ be the unique string in $Z$ such that
\[ (\tau \circ \sigma_{\mrm{pr}}) * \pi_i = 
(\tau \circ \sigma^k_i) * (\tau \circ \tau^k_i \circ \sigma_{\mrm{pr}}) \]
as strings. This is a string with initial point $z_0$, and with 
terminal point
\[ z_i := (\tau \circ \tau^k_i)(w) . \]
Notice that for $i \leq 2^k$ the strings $\pi_i$ and $\pi_{2^k + i}$
are reflections of each other relative to the point $z_0$. 
See Figure \ref{fig:77}.

\begin{figure}
\includegraphics[scale=0.28]{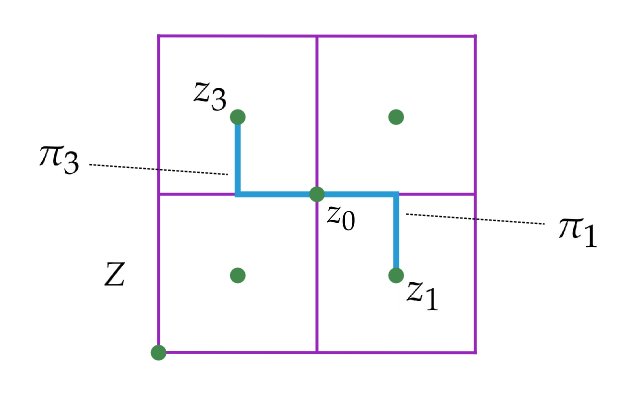}
\caption{The strings $\pi_1$ and $\pi_3$ in $Z$. Here $k = 1$.} 
\label{fig:77}
\end{figure}

Define
\[  g_i := \opn{MI}(\alpha \vert \pi_i) \in G  \]
and
\[ \lambda_i := (\smfrac{1}{4})^k \cdot \epsilon^2 \cdot
\Psi_{\h}(g_0 \cdot g_i)(\til{\beta}(z_i)) \in \h . \]
So we have
\[ \exp_H(\lambda_i) = 
\opn{RP}_0(\alpha, \beta \vert (\sigma, \tau) \circ (\sigma^k_i, \tau^k_i) ) .
\]
And the bound for $\lambda_i$ is 
\begin{equation} \label{eqn:159}
\norm{ \lambda_i } \leq (\smfrac{1}{4})^k \cdot \epsilon^2 \cdot 
c'(\alpha, \beta) \, ,
\end{equation}
where we write
\[ c'(\alpha, \beta) :=  \exp( c_4(\alpha, \Psi_{\h}) \cdot 
7 \cdot \opn{diam}(X) ) \cdot \norm{ \beta }_{\mrm{Sob}} + 1  \, . \]
Using the abbreviation
\[ \gamma := \epsilon^2 \cdot \Psi_{\h}(g_0)(\d \til{\beta}) \in 
\Omega^1_{\mrm{pws}}(Z) \otimes \h , \]
formula (\ref{eqn:156}), and the inequality
$\opn{len}(\pi_i) \leq \epsilon$, we obtain
\begin{equation} \label{eqn:158}
\begin{aligned}
& \Norm{ \lambda_i - (\smfrac{1}{4})^k \cdot 
\Bigl( \lambda_0 + \bigl( \int_{\pi_i} \alpha'(z_0) \bigl)
(\lambda_0) +  \int_{\pi_i} \gamma(z_0) \Bigr) } \\
& \qquad \leq 
c(\alpha, \beta) \cdot (\smfrac{1}{4})^k \cdot \epsilon^4 \, . 
\end{aligned} 
\end{equation}
For this to be true we should assume that 
$\epsilon < \epsilon_5(\alpha, \Psi_{\h})$.

Let us set
\[ \epsilon_2(\alpha, \beta) := 
\min \bigl( \epsilon_1(\alpha, \beta), \, 
c'(\alpha, \beta)^{- 1/2} \cdot \epsilon_0(H)^{1/2} , \,
\epsilon_5(\alpha, \Psi_{\h}) \bigl) \, , \]
We now assume furthermore that 
$\epsilon < \epsilon_2(\alpha, \beta)$.
In particular, from (\ref{eqn:159}) we obtain 
\[ \sum_{i = 1}^{4^k} \, \norm{ \lambda_i } \leq \epsilon_0(H) \, . \]
We know that
\[ \prod_{i = 1}^{4^k} \, \exp_H(\lambda_0) = 
\opn{RP}_k(\alpha, \beta \vert \sigma, \tau) . \]
Therefore we can use property (ii) of Theorem \ref{thm:6} to deduce that
\[ \Norm{ \log_H \bigl( \opn{RP}_k(\alpha, \beta \vert \sigma, \tau) \bigr) 
- \sum_{i = 1}^{4^k} \, \lambda_i } \leq
c_0(H) \cdot \epsilon^4 \cdot c'(\alpha, \beta)^2 \, . \]
The geometric symmetry of the sequence of strings $\pi_1. \ldots, \pi_{4^k}$
implies that
\[ \bigl( \int_{\pi_i} \alpha'(z_0) \bigl)(\lambda_0) = 
- \bigl( \int_{\pi_{2^k + i}} \alpha'(z_0) \bigl)(\lambda_0) \]
for $i \leq 2^k$; and therefore 
\[ \sum_{i = 1}^{4^k} \, \bigl( \int_{\pi_i} \alpha'(z_0) \bigl)(\lambda_0) 
= 0 . \]
Similarly 
\[ \sum_{i = 1}^{4^k} \, \int_{\pi_i} \gamma(z_0) = 0 . \]
Plugging in the estimate (\ref{eqn:158}) we obtain 
\[  \Norm{ \sum_{i = 1}^{4^k} \, \lambda_i - \lambda_0 } \leq 
c(\alpha, \beta) \cdot \epsilon^4 \, . \]
We see that the constant 
\[ c_2(\alpha, \beta) := c(\alpha, \beta) + 
c_0(H) \cdot c'(\alpha, \beta)^2 + c_1(\alpha, \beta) \, . \]
works.
\end{proof}

\begin{dfn}
Let us fix constants $\epsilon_2(\alpha, \beta)$ and
$c_2(\alpha, \beta)$ as in Lemma \ref{lem:5}.
A square kite $(\sigma, \tau)$ in $(X, x_0)$
will be called {\em $(\alpha, \beta)$-tiny} in this section if 
\[ \opn{side}(\tau) < \epsilon_2(\alpha, \beta) \]
and 
\[ \opn{len}(\sigma) \leq 4 \cdot \opn{diam}(X) \,  . \]
\end{dfn}

\begin{dfn} \label{dfn:45}
Let $(\sigma, \tau)$ be a nondegenerate kite in $(X, x_0)$. 
For $k \in \N$ and $i \in \{ 1, \ldots, 4^k \}$ let 
$Z_i := (\tau \circ \tau^k_i)(\mbf{I}^2) \subset X$. 
An index $i$ is called {\em good} if the forms
$\alpha|_{Z_i}$ and $\beta|_{Z_i}$ are smooth. Otherwise $i$ is called {\em
bad}. The sets of good and bad indices are denoted by 
$\opn{good}(\tau, k)$ and $\opn{bad}(\tau, k)$ respectively.
\end{dfn}

\begin{lem} \label{lem:51}
Let $Z$ be a $2$-dimensional subpolyhedron of $X$. There exist constants
$a_0(\alpha, \beta, Z)$ and $a_1(\alpha, \beta, Z)$ such that for any 
nondegenerate kite $(\sigma, \tau)$ in $(X, x_0)$ with 
$\tau(\mbf{I}^2) \subset Z$, one has
\[ \abs{ \opn{bad}(\tau, k) } \leq 
a_0(\alpha, \beta, Z) + a_1(\alpha, \beta, Z) \cdot 2^k \, . \]
\end{lem}

\begin{proof} 
Let $Z_1, \ldots, Z_m$ be line segments in $Z$ such that 
the singular locus of $\alpha$ and the singular locus of $\beta$ are
contained in $\bigcup_{j = 1}^m Z_j$. Take
$a_0(\alpha, \beta, Z) := 2 m$ and 
$a_1(\alpha, \beta, Z) := 2 \cdot \sum_{j = 1}^m \, \opn{len} (Z_j)$.
\end{proof}

\begin{lem} \label{lem:6}
Let $Z$ be a $2$-dimensional subpolyhedron of $X$.  Then there is a constant
$c_3(\alpha, \beta, Z) \geq 0$, such that for any 
$(\alpha, \beta)$-tiny kite $(\sigma, \tau)$ satisfying
$\tau(\mbf{I}^2) \subset Z$, and any $k' \geq k \geq 0$, one has
\[ \begin{aligned}
& \Norm{ \opn{log}_H \bigl( \opn{RP}_{k'}(\alpha, \beta \vert 
\sigma, \tau) \bigr) -
\opn{log}_H \bigl( \opn{RP}_{k}(\alpha, \beta \vert \sigma, \tau) \bigr) } \\
& \qquad \leq (\smfrac{1}{4})^k \cdot c_3(\alpha, \beta, Z) \cdot 
\opn{side}(\tau)^2  \, .
\end{aligned}  \]
\end{lem}

\begin{proof}
Let $(\sigma, \tau)$ an $(\alpha, \beta)$-tiny kite such that 
$\tau(\mbf{I}^2) \subset Z$. Write $\epsilon := \opn{side}(\tau)$. For 
$i \in \{ 1 ,\ldots, 4^k \}$ let $(\sigma_i, \tau_i)$ be as in equation
(\ref{eqn:150}), and let $Z_i := \tau_i(\mbf{I}^2) \subset Z$. 
Note that 
$\opn{side}(Z_i) = (\smfrac{1}{2})^k \cdot \epsilon$.

Let $l := k' - k$. If $i \in \opn{good}(\tau, k)$, then by Lemma 
\ref{lem:5} we know that
\[ \Norm{ \opn{log}_H \bigl( \opn{RP}_{l}(\alpha, \beta \vert 
\sigma_i, \tau_i) \bigr) -
\opn{log}_H \bigl( \opn{RP}_{0}(\alpha, \beta \vert \sigma_i, \tau_i)
\bigr) } \leq 
c_2(\alpha, \beta) \cdot ((\smfrac{1}{2})^k \cdot \epsilon)^4 \, . \]
If $i \in \opn{bad}(\tau, k)$, then by Lemma \ref{lem:4} we know that
\[ \Norm{ \opn{log}_H \bigl( \opn{RP}_{l}(\alpha, \beta \vert 
\sigma_i, \tau_i) \bigr) -
\opn{log}_H \bigl( \opn{RP}_{0}(\alpha, \beta \vert \sigma_i, \tau_i)
\bigr) } \leq 
2 c_1(\alpha, \beta) \cdot ((\smfrac{1}{2})^k \cdot \epsilon)^2 \, . \]
Therefore by property (iv) of Theorem \ref{thm:6}
and Lemma \ref{lem:51}  we have
\[ \begin{aligned}
& \Norm{ \opn{log}_H \bigl( \opn{RP}_{k'}(\alpha, \beta \vert 
\sigma, \tau) \bigr) -
\opn{log}_H \bigl( \opn{RP}_{k}(\alpha, \beta \vert \sigma, \tau) \bigr) } \\
& \qquad =
\Norm{ \opn{log}_H \bigl( 
\prod\nolimits_{i = 1}^{4^k} \, 
\opn{RP}_{l}(\alpha, \beta \vert \sigma_i, \tau_i \bigr) -
\opn{log}_H \bigl( 
\prod\nolimits_{i = 1}^{4^k} \, 
\opn{RP}_{0}(\alpha, \beta \vert \sigma_i, \tau_i \bigr) } \\
& \qquad \leq c_0(H) \cdot \bigl(
\abs{ \opn{good}(\tau, k) } \cdot c_2(\alpha, \beta) \cdot 
(\smfrac{1}{2})^{4 k} \cdot \epsilon^4 \\
& \qquad \qquad +
\abs{ \opn{bad}(\tau, k) } \cdot 2 c_1(\alpha, \beta) \cdot 
(\smfrac{1}{2})^{2 k} \cdot \epsilon^2 \bigr) \\
& \qquad \leq c_3(\alpha, \beta, Z) \cdot (\smfrac{1}{2})^k \cdot \epsilon^2 
\end{aligned} \]

where we take (very generously)
\[ c_3(\alpha, \beta, Z) := c_0(H) \cdot \bigl( c_2(\alpha, \beta) +
2 (a_0(\alpha, \beta, Z) + a_1(\alpha, \beta, Z)) \cdot c_1(\alpha, \beta)
\bigr) . \]
\end{proof}

\begin{thm} \label{thm:12}
Let $(X, x_0)$ be a pointed polyhedron, let
$\alpha \in \Omega^1_{\mrm{pws}}(X) \otimes \g$, let
$\beta \in \Omega^2_{\mrm{pws}}(X) \otimes \h$, and let
$(\sigma, \tau)$ be a kite in $(X, x_0)$. 
Then the limit
\[ \lim_{k \to \infty} \, \opn{RP}_{k}(\alpha, \beta \vert  \sigma, \tau)
\]
exists in $H$.
\end{thm}

\begin{proof}
According to Proposition \ref{prop:14} and Lemma \ref{lem:38} we can assume
that $(X, x_0) = (\mbf{I}^2, v_0)$ and 
$\opn{len}(\sigma) \leq 1$. 

Take $k$ large enough such that for each $i \in \{ 1, \ldots, 4^k \}$
the kite $(\sigma_i, \tau_i)$, in the notation of (\ref{eqn:150}), 
is $(\alpha, \beta)$-tiny. For any $k' \geq 0$ we have
\[ \opn{RP}_{k + k'} (\alpha, \beta \vert  \sigma, \tau) = 
\prod_{i = 1}^{4^k}  \, \opn{RP}_{k'}
(\alpha, \beta \vert  \sigma_i, \tau_i) . \]
Thus it suffices to prove that for any $i$ the limit
\[ \lim_{k' \to \infty} \opn{RP}_{k'}(\alpha, \beta \vert  \sigma_i, \tau_i) \]
exists. Now Lemma \ref{lem:6} says that the sequence 
\[ \bigl( \opn{RP}_{k'}(\alpha, \beta \vert  \sigma_i, \tau_i) 
\bigr)_{k' \geq 0} \]
is a Cauchy sequence in $H$; and therefore it converges.
\end{proof}

\begin{dfn}[Multiplicative Integral] \label{dfn:16}
\index{Multiplicative integral on kites}
Let $(G, H, \Psi_{\h})$ be a twisting setup, let
$(X, x_0)$ be a pointed polyhedron, let
$\alpha \in \Omega^1_{\mrm{pws}}(X) \otimes \g$, let
$\beta \in \Omega^2_{\mrm{pws}}(X) \otimes \h$, and let
$(\sigma, \tau)$ be a kite in $(X, x_0)$. We define the
{\em multiplicative integral of $\beta$ twisted by
$\alpha$ on $(\sigma, \tau)$} to be
\[ \opn{MI}(\alpha, \beta \vert \sigma, \tau) :=
 \lim_{k \to \infty} \, \opn{RP}_{k}(\alpha, \beta \vert  \sigma, \tau)
\in H . \]

If $(X, x_0) = (\mbf{I}^{2}, v_0)$ then we write
\[  \opn{MI}(\alpha, \beta \vert \mbf{I}^{2}) :=
 \opn{MI}(\alpha, \beta \vert \sigma^0_1, \tau^0_1) , \]
where $(\sigma^0_1, \tau^0_1)$ is the basic kite.
\end{dfn}

\subsection{Some Properties of $\opn{MI}$}
We continue with the setup of the previous subsections.

\begin{prop} \label{prop:3}
In the situation of Definition \tup{\ref{dfn:16}}, for any $k \geq 0$ one
has
\[ \opn{MI}(\alpha, \beta \vert \sigma, \tau) =
\prod_{i = 1}^{4^k} \,
\opn{MI} \bigl( \alpha, \beta \vert (\sigma, \tau) \circ 
(\sigma^k_i, \tau^k_i) \bigr) . \]
\end{prop}

\begin{proof}
For any $k' \geq 0$ we have
\[ \opn{RP}_{k + k'}(\alpha, \beta \vert \sigma, \tau) =
\prod_{i = 1}^{4^k} \,
\opn{RP}_{k'} \bigl( \alpha, \beta \vert (\sigma, \tau) \circ 
(\sigma^k_i, \tau^k_i) \bigr) . \]
Now take the limit $\lim_{k' \to \infty}$.
\end{proof}

\begin{prop} \label{prop:8}
Consider the situation of Definition \tup{\ref{dfn:16}}, and assume that 
$(\sigma, \tau)$ is an $(\alpha, \beta)$-tiny kite.
\begin{enumerate}
\item One has
\[ \opn{MI}(\alpha, \beta \vert \sigma, \tau) \in V_0(H) , \]
and
\[ \Norm{ \opn{log}_H \bigl( \opn{MI}(\alpha, \beta \vert \sigma, \tau)
\bigr) } \leq 
c_1(\alpha, \beta) \cdot \opn{side}(\tau)^2 \, . \]

\item If $\beta|_{\tau(\mbf{I}^{2})}$ is smooth then 
\[ \Norm{ \opn{log}_H \bigl( \opn{MI}(\alpha, \beta \vert 
\sigma, \tau) \bigr) -
\opn{log}_H \bigl( \opn{RP}_{0}(\alpha, \beta \vert \sigma, \tau) \bigr) } 
\leq c_2(\alpha, \beta) \cdot \opn{side}(\tau)^4  \, .
\]

\item Let $Z$ be a $2$-dimensional subpolyhedron of $X$ containing 
$\tau(\mbf{I}^2)$. For any $k \geq 0$ one has
\[ \begin{aligned}
& \Norm{ \opn{log}_H \bigl( \opn{MI}(\alpha, \beta \vert 
\sigma, \tau) \bigr) -
\opn{log}_H \bigl( \opn{RP}_{k}(\alpha, \beta \vert \sigma, \tau) \bigr) } \\
& \quad \leq (\smfrac{1}{2})^k \cdot c_3(\alpha, \beta, Z) \cdot 
\opn{side}(\tau)^2  \, .
\end{aligned} \]
\end{enumerate}
\end{prop}

\begin{proof}
Immediate from Lemmas \ref{lem:4}, \ref{lem:5} and \ref {lem:6}.
\end{proof}

\begin{prop}[Functoriality in $X$] \label{prop:10}
Let $f : (Y, y_0) \to (X, x_0)$ be a piecewise linear map between pointed
polyhedra, let
$\alpha \in \Omega^1_{\mrm{pws}}(X) \otimes \g$, let
$\beta \in \Omega^2_{\mrm{pws}}(X) \otimes \h$, 
and let $(\sigma, \tau)$ be a kite in $(Y, y_0)$.
Assume that $f$ is linear on $\tau(\mbf{I}^2)$. Then
\[ \opn{MI}(\alpha, \beta \vert f \circ \sigma, f \circ \tau) = 
\opn{MI} \bigl( f^*(\alpha), f^*(\beta) \vert  \sigma, \tau \bigl) . \]
\end{prop}

\begin{proof}
Immediate from Lemma \ref{lem:38}.
\end{proof}

The next proposition says that ``in the tiny scale the $2$-dimensional MI is
abelian''. 

\begin{prop} \label{prop:22}
There are constant $\epsilon_{2'}(\alpha, \beta)$ and
$c_{2'}(\alpha, \beta)$ with the following properties:
\begin{itemize}
\rmitem{i} $c_{2'}(\alpha, \beta) \geq c_2(\alpha, \beta)$ and
$0 < \epsilon_{2'}(\alpha, \beta) \leq \epsilon_2(\alpha, \beta)$.
\rmitem{ii} Suppose $(\sigma, \tau)$ is a square kite in $(X, x_0)$ such that 
$\opn{side}(\tau) < \epsilon_{2'}(\alpha, \beta)$
and
$\opn{len}(\sigma) \leq 4 \cdot \opn{diam}(X)$. 
Let
$g := \opn{MI}(\alpha \vert \sigma) \in G$.
Then 
\[ \Norm{ \opn{log}_H \bigl( \opn{MI}(\alpha, \beta \vert \sigma, \tau)
\bigr) -
\Psi_{\h}(g)\bigl( \int_{\tau} \beta \bigr)  } \leq 
c_{2'}(\alpha, \beta) \cdot \opn{side}(\tau)^3 \, .  \]
\end{itemize}
\end{prop}

Actually with more effort we can get a better estimate (order 
$\opn{side}(\tau)^4$) in property (ii) above. 

\begin{proof}
This is very similar to Proposition \ref{prop:16}.
Take 
\[ \epsilon_{2'}(\alpha, \beta) := \min \bigl( 
\epsilon_{2}(\alpha, \beta), \, \epsilon_{5}(\alpha, \Psi_{\h}) \bigr) \, , \]
where $\epsilon_{5}(\alpha, \Psi_{\h})$ is the constant from Proposition
\ref{prop:21}. 
Let's write $\epsilon := \opn{side}(\tau)$
and $Z := \tau(\mbf{I}^2)$. Assume that 
$\epsilon < \epsilon_{2'}(\alpha, \beta)$. 

Take $k \geq 0$. For $i \in \{ 1, \ldots, 4^k \}$ let
\[ (\sigma_i, \tau_i) := 
(\sigma, \tau) \circ (\sigma^k_i, \tau^k_i) = 
\opn{tes}^k_i (\sigma, \tau) . \]
Write
$Z_i := \tau_i(\mbf{I}^2)$ and $z_i := \tau_i(w)$, where
$w := (\smfrac{1}{2}, \smfrac{1}{2}) \in \mbf{I}^2$ as usual. 
The sets $\opn{good}(\tau, k)$ and $\opn{bad}(\tau, k)$ were defined in
Definition \ref{dfn:45}. 

Define 
\[ g'_0 := \opn{MI}(\alpha \vert \tau \circ \sigma_{\mrm{pr}}) \in G , \]
\[ g'_i := \opn{MI}(\alpha \vert 
(\tau \circ \sigma^k_i) * (\tau \circ \tau^k_i \circ \sigma_{\mrm{pr}})) \]
and 
\[ \lambda_i :=  \opn{log}_H \bigl( 
\opn{RP}_0(\alpha, \beta \vert \sigma_i, \tau_i) \bigr) \in \h  \]
for $i \in \{ 1, \ldots, 4^k \}$.
{}From Definition \ref{dfn:11} and Proposition \ref{prop:13} we know that
\begin{equation} \label{eqn:195}
\norm{ \lambda_i } \leq (\smfrac{1}{4})^k \cdot \epsilon^2 \cdot 
\norm{\Psi_{\h}(g \cdot g'_i)} \cdot \norm{\beta}_{\mrm{Sob}} 
\leq (\smfrac{1}{4})^k \cdot \epsilon^2 \cdot c  \, ,
\end{equation}
where we let
\[ c := \exp \bigl( c_4(\alpha, \phi) \cdot 6 \opn{diam}(X) \bigl) \cdot 
\norm{\beta}_{\mrm{Sob}} \, . \]
According to property (ii) of Theorem \ref{thm:6} we have
\begin{equation} \label{eqn:193}
\begin{aligned}
& \Norm{ \log_H \bigl( \opn{RP}_k(\alpha, \beta \vert \sigma, \tau) \bigr)
- \bosum_{i = 1}^{4^k} \, \lambda_i } \\
& \qquad \leq c_0(H) \cdot \bigl( \bosum_{i = 1}^{4^k} \norm{ \lambda_i }
\bigr)^2 \leq 
c_0(H) \cdot (\smfrac{1}{4})^{2 k} \cdot \epsilon^4 \cdot c^2  \, .
\end{aligned}
\end{equation}

For $i \in \opn{good}(\tau, k)$ let $\til{\beta}_i$ be the coefficient of 
$\beta|_{Z_i}$. Define 
\[ \mu_i := 
\begin{cases}
(\smfrac{1}{4})^k \cdot \epsilon^2 \cdot 
\Psi_{\h}(g) \bigl( \til{\beta}_i(z_i) \bigr) &
\text{ if } i \text{ is good }, \\
0 & \text{ otherwise} 
\end{cases} \]
and
\[ \opn{RS}_k(\alpha, \beta \vert \sigma, \tau) := 
\sum_{i = 1}^{4^k}\, \mu_i . \]

Now let us compare $\mu_i$ to $\lambda_i$. 
If $i$ is a bad index, then
\[ \norm{ \mu_i - \lambda_i } = \norm{\lambda_i} \leq 
(\smfrac{1}{4})^k \cdot \epsilon^2 \cdot c  \]
by (\ref{eqn:195}). On the other hand, if $i$ is a good index then 
\[ \mu_i = \Psi_{\h}(g \cdot g'_0 \cdot g'^{-1}_i \cdot g^{-1}) (\lambda_i) . \]
By Proposition \ref{prop:13} the operators $\Psi_{\h}(g)$ and
$\Psi_{\h}(g^{-1})$ have known \lb bounds (since the length of the string 
$\sigma$ is bounded by $4 \cdot \opn{diam}(X)$).
Hence there is a bound for the conjugation operator
$\opn{Ad}(\Psi_{\h}(g))$ on $\opn{End}(\h)$. 
And by Proposition \ref{prop:21} there is an estimate for the norm of the
operator 
$\Psi_{\h}(g'_0 \cdot g'^{-1}_i) - \bsym{1}$,
where $\bsym{1}$ denotes the identity operator of $\h$.
Since $\opn{Ad}(\Psi_{\h}(g))$ fixes $\bsym{1}$, we can conclude that 
there is a constant $c'$
(independent of of $(\sigma, \tau)$ or $k$) such that
\[ \begin{aligned}
& \Norm{ \Psi_{\h}(g \cdot g'_0 \cdot g'^{-1}_i \cdot g^{-1}) - \bsym{1} } \\
& \qquad = \Norm{ \opn{Ad}(\Psi_{\h}(g)) \bigl( 
\Psi_{\h}(g'_0 \cdot g'^{-1}_i) - \bsym{1} \bigr) } 
\leq c' \cdot \epsilon \, . 
\end{aligned} \]
Hence using (\ref{eqn:195}) we get
\[ \norm{ \mu_i - \lambda_i } \leq c' \cdot c \cdot
(\smfrac{1}{4})^k \cdot \epsilon^3 \, . \]
Summing over all $i$ we see that
\[ \begin{aligned}
& \Norm{ \bosum_{i = 1}^{4^k} \, \lambda_i - 
\bosum_{i = 1}^{4^k} \, \mu_i } 
\leq \sum_{i = 1}^{4^k} \, \norm{ \mu_i - \lambda_i } \\
& \qquad \leq \abs{\opn{good}(\tau, k)} \cdot 
c' \cdot c \cdot (\smfrac{1}{4})^k \cdot \epsilon^3 +
\abs{\opn{bad}(\tau, k)} \cdot (\smfrac{1}{4})^k \cdot \epsilon^2 \cdot c \\
& \qquad \leq 4^k \cdot c' \cdot c \cdot (\smfrac{1}{4})^k \cdot \epsilon^3 +
(a_0 + a_1 \cdot 2^k)  \cdot
(\smfrac{1}{4})^k \cdot \epsilon^2 \cdot c \, .  
\end{aligned} \]
Here $a_0 := a_0(\alpha, \beta, Z)$ and $a_1 := a_1(\alpha, \beta, Z)$ are the
constants from Lemma \ref{lem:51}. 
Combining this estimate with (\ref{eqn:193}) we obtain
\begin{equation} \label{eqn:196}
\begin{aligned}
& \Norm{ \log_H \bigl( \opn{RP}_k(\alpha, \beta \vert \sigma, \tau) \bigr) - 
\opn{RS}_k(\alpha, \beta \vert \sigma, \tau) } \\
& \qquad \leq 
c_0(H) \cdot (\smfrac{1}{4})^{2 k} \cdot \epsilon^4 \cdot c^2 +
4^k \cdot c' \cdot c \cdot (\smfrac{1}{4})^k \cdot \epsilon^3 \\
& \qquad \qquad +
(a_0 + a_1 \cdot 2^k)  \cdot (\smfrac{1}{4})^k \cdot \epsilon^2 \cdot c \, . 
\end{aligned}
\end{equation}

Finally, by properties of the usual Riemann integration we have
\begin{equation*} \label{eqn:194}
\lim_{k \to  \infty} \opn{RS}_k(\alpha, \beta \vert \sigma, \tau) = 
\Psi_{\h}(g)\bigl( \int_{\tau} \beta \bigr) .
\end{equation*}
Hence in the limit $k \to \infty$ we get
\[ \Norm{ \opn{log}_H \bigl( \opn{MI}(\alpha, \beta \vert \sigma, \tau)
\bigr) -
\Psi_{\h}(g)\bigl( \int_{\tau} \beta \bigr)  } \leq 
c' \cdot c \cdot  \epsilon^3 \, , \]
and we can take
$c_{2'}(\alpha, \beta) := c' \cdot c$.
\end{proof}

\subsection{Transfer of Twisting Setups}

Suppose $H'$ is another Lie group, with Lie algebra $\h'$. 
The vector space of $\R$-linear maps $\h \to \h'$ is denoted by
$\opn{Hom}(\h, \h')$. 
Consider the $\mcal{O}_{\mrm{pws}}(X)$-module 
$\mcal{O}_{\mrm{pws}}(X) \otimes \opn{Hom}(\h, \h')$.
An element 
\[ \phi \in \mcal{O}_{\mrm{pws}}(X) \otimes \opn{Hom}_{\R}(\h, \h') \]
is called a {\em piecewise smooth family of linear maps} from $\h$ to $\h'$.
Indeed, we may view $\phi$ as a piecewise smooth map
\[ \phi : X \to \opn{Hom}(\h, \h') . \]
For any point $x \in X$ there is a linear map
$\phi(x) : \h \to \h'$.

\begin{dfn} \label{dfn:7}
Suppose $\mbf{C} = (G, H, \Psi_{\h})$ and 
$\mbf{C}' = (G', H', \Psi'_{\h'})$ are two twisting
setups. A {\em transfer of twisting setups} from $\mbf{C}$
to $\mbf{C}'$, parametrized by $(X, x_0)$, is the data
\[ \bsym{\Theta}_X = (\Theta_G, \Theta_H, \Theta_{\h, X}) , \] 
consisting of:
\begin{enumerate}
\item Maps of Lie groups
$\Theta_G : G \to G'$ and 
$\Theta_H : H \to H'$.

\item An element
$\Theta_{\h, X} \in \mcal{O}_{\mrm{pws}}(X) \otimes \opn{Hom}(\h, \h')$.
\end{enumerate}

The following condition is required: 
\begin{itemize}
\rmitem{$*$} The equality 
\[ \opn{Lie}(\Theta_H) = \Theta_{\h, X}(x_0) \]
holds in $\opn{Hom}(\h, \h')$.
\end{itemize}

We denote this transfer by
$\bsym{\Theta}_X : \mbf{C} \to \mbf{C}'$.
\end{dfn}

Note that for $x \neq x_0$ the linear map
$\phi(x) : \h \to \h'$ might fail to be a Lie algebra homomorphism.

\begin{dfn} \label{dfn:17}
Let $\bsym{\Theta}_X : \mbf{C} \to \mbf{C}'$ be a transfer of twisting
setups as in Definition \ref{dfn:7}, and let 
$\alpha \in \Omega^1_{\mrm{pws}}(X) \otimes \g$. 
We say that $\alpha$ is a {\em connection compatible with 
$\bsym{\Theta}_X$} if the following condition, called the
{\em holonomy condition}, holds.
\begin{itemize}
\rmitem{$\Diamond$} Let $\sigma$ be a string in $X$, with
$x_0 = \sigma(v_0)$ and $x_1 := \sigma(v_1)$. Define
\[ g := \opn{MI}(\alpha \vert \sigma) \in G  \]
and $g' := \Theta_G(g) \in G'$.
Then the diagram 
\[ \UseTips \xymatrix @C=9ex @R=5ex {
\h
\ar[d]_{\Theta_{\h, X}(x_0)}
& \h
\ar[d]^{\Theta_{\h, X}(x_1)} 
\ar[l]_{\Psi_{\h}(g)}
\\
\h'
& \h'
\ar[l]^{\Psi'_{\h'}(g')}
} \]
is commutative.
\end{itemize}
\end{dfn}

\begin{rem}
The holonomy condition for $\alpha$ can be stated as a differential equation.
We shall not need this equation in our treatment. 
\end{rem}

Consider a transfer of twisting setups $\bsym{\Theta}_X$ as above. 
The family of linear maps
$\Theta_{\h, X}$ induces, by tensoring, a homomorphism of graded
$\Omega_{\mrm{pws}}(X)$-modules
\begin{equation} \label{eqn:140}
\Theta_{\h, X} : \Omega_{\mrm{pws}}(X) \otimes \h \to 
\Omega_{\mrm{pws}}(X) \otimes \h' .
\end{equation}
Warning: usually $\Theta_{\h, X}$ does not commute with the de Rham operator
$\d$.

\begin{prop}[Functoriality in $\mbf{C}$] \label{prop:2}
Let $\bsym{\Theta}_X : \mbf{C} \to \mbf{C}'$ be a transfer of twisting
setups, let 
$\alpha \in \Omega^1_{\mrm{pws}}(X) \otimes \g$,
and let 
$\beta \in \Omega^2_{\mrm{pws}}(X) \otimes \h$.
Assume that $\alpha$ is a connection compatible with $\bsym{\Theta}_X$. 
We write
\[ \alpha' := \opn{Lie}(\Theta_G)(\alpha) \in  
\Omega^1_{\mrm{pws}}(X) \otimes \g' \]
and 
\[ \beta' := \Theta_{\h, X}(\beta)  \in  
\Omega^2_{\mrm{pws}}(X) \otimes \h' . \]
Then for every kite $(\sigma, \tau)$ in $(X, x_0)$ one has
\[ \Theta_H \bigl( \opn{MI}(\alpha, \beta \vert \sigma, \tau) \bigr) = 
\opn{MI} \bigl( \alpha', \beta' \vert \sigma, \tau \bigr)  \]
in $H'$.
\end{prop}

\begin{proof}
Consider a kite $(\sigma', \tau')$ in $(\mbf{I}^2, v_0)$ and a piecewise linear
map $f$ like in Proposition \ref{prop:14}. By Proposition \ref{prop:10} we have
\[ \opn{MI}(\alpha, \beta \vert \sigma, \tau) =
\opn{MI} \bigl( f^*(\alpha) , f^*(\beta) \vert \sigma', \tau' \bigr) . \]
On the other hand
$\bigl( \Theta_G, \Theta_H, f^*(\Theta_{\h, X}) \bigr)$ 
is a transfer of twisting setups \lb parametrized by 
$(\mbf{I}^2, v_0)$, and by Proposition \ref{prop:10} we have
\[ \begin{aligned}
& \opn{MI} \bigl( \alpha', \beta' \vert \sigma, \tau \bigr) =
\opn{MI} \bigl( f^*(\alpha'), f^*(\beta') \vert 
\sigma', \tau' \bigr) \\
& \qquad =
\opn{MI} \bigl( \opn{Lie}(\Theta^1)(f^*(\alpha)), f^*(\Theta_{\h,
X})(f^*(\beta)) \vert \sigma', \tau' \bigr) .
\end{aligned} \]
Therefore we can assume that 
$(X, x_0) = (\mbf{I}^2, v_0)$ and $\opn{len}(\sigma) \leq 1$.
Using Proposition \ref{prop:3} we can further assume that 
$(\sigma, \tau)$ is $(\alpha, \beta)$-tiny and $(\alpha', \beta')$-tiny.

Fix $k \geq 0$. 
For $i \in \{ 1, \ldots, 4^k \}$ let
\[ (\sigma_i, \tau_i) := 
(\sigma, \tau) \circ (\sigma^k_i, \tau^k_i) . \]
Define 
\[ g_i := \opn{MI}(\alpha \vert \sigma_i * (\tau_i \circ \sigma_{\mrm{pr}})) \in
G \]
and
\[ g'_i := \opn{MI}(\alpha' \vert \sigma_i * (\tau_i \circ \sigma_{\mrm{pr}}))
\in G' . \]
By Proposition \ref{prop:12} we have
$g'_i = \Theta_G(g_i)$. 
Also define $Z_i := \tau_i(\mbf{I}^2)$ and 
$z_i := \tau_i(\smfrac{1}{2}, \smfrac{1}{2})$.

Suppose that $i \in \opn{good}(\tau, k)$, with notation as in 
Definition \ref{dfn:45}. Then there is a function 
$\til{\beta}_i \in \mcal{O}(Z_i) \otimes \h$,
called the coefficient of $\beta|_{Z_i}$, satisfying
\[ \beta|_{Z_i} = \til{\beta}_i \cdot \d t_1 \wedge \d t_2 . \]
The function 
\[ \til{\beta}'_i := \Theta_{\h, X}(\til{\beta}_i) \in \mcal{O}(Z_i) \otimes
\h' , \]
is then the coefficient of $\beta'|_{Z_i}$.
Note that 
\[  \til{\beta}'_i(z_i) = \Theta_{\h, X}(z_i) \bigl( \til{\beta}(z_i) \bigr) \]
in $\h'$. By condition ($*$) of Definition \ref{dfn:7} and condition
($\Diamond$) of Definition \ref{dfn:17} we have
\[ \begin{aligned}
& \Psi'_{\h'}(g'_i) \bigl( \til{\beta}'_i(z_i) \bigr) = 
\Psi'_{\h'}(g'_i) \bigl(  \Theta_{\h, X}(z_i) \bigl( \til{\beta}_i(z_i) \bigr)
\bigr) \\
& \qquad = \Theta_{\h, X}(x_0) \bigl( \Psi_{\h}(g_i)
\bigl( \til{\beta}_i(z_i) \bigr) \bigr) \\
& \qquad = \opn{Lie}(\Theta_H)
\bigl( \Psi_{\h}(g_i) \bigl( \til{\beta}_i(z_i) \bigr) \bigr) . 
\end{aligned} \]
By definition we have
\[ \opn{RP}_0(\alpha, \beta \vert \sigma_i, \tau_i) = 
\exp_H \bigl( (\smfrac{1}{4})^k \cdot \opn{area}(\tau) \cdot 
\Psi_{\h}(g_i)(\til{\beta}_i(z_i)) \bigr) \]
and
\[ \opn{RP}_0(\alpha', \beta' \vert \sigma_i, \tau_i) = 
\exp_{H'} \bigl( (\smfrac{1}{4})^k \cdot \opn{area}(\tau) \cdot 
\Psi_{\h}(g'_i)(\til{\beta}'_i(z_i)) \bigr) . \]
Since
\[ \exp_{H'} \circ \opn{Lie}(\Theta_H) = \Theta_H \circ \exp_H \]
we conclude that 
\begin{equation} \label{eqn:132}
\opn{RP}_0(\alpha', \beta' \vert \sigma_i, \tau_i) =
\Theta_H \bigl( \opn{RP}_0(\alpha, \beta \vert \sigma_i, \tau_i) \bigr) . 
\end{equation}

Like in Lemma \ref{lem:51} we can find a bound for
$\abs{ \opn{bad}(\tau, k) }$, and like in the proof of Lemma \ref{lem:6}
we can estimate 
\[ \Norm{ \log_{H'} \bigl( \Theta_H \bigl(  
\opn{RP}_0 ( \alpha, \beta \vert \sigma_i, \tau_i)  \bigr) \bigr) -
\log_{H'} \bigl( 
\opn{RP}_0 \bigl( \alpha', \beta' \vert \sigma_i, \tau_i) \bigr) }  \]
when $i \in \opn{bad}(\tau, k)$. 
From these estimates and from (\ref{eqn:132}) 
we conclude that there is a constant $c$, independent of
$k$, such that
\[  \Norm{ \log_{H'} \bigl( \Theta_H \bigl(
\opn{RP}_k ( \alpha, \beta \vert \sigma, \tau)  \bigr) \bigr) -
\log_{H'} \bigl( 
\opn{RP}_k \bigl( \alpha', \beta' \vert \sigma, \tau) \bigr) } 
\leq c \cdot (\smfrac{1}{2})^k \, . \]
In the limit $k \to \infty$ we see that
\[ \Theta_H \bigl( \opn{MI} (\alpha, \beta \vert \sigma, \tau) \bigr) =
\opn{MI} (\alpha', \beta' \vert \sigma, \tau). \]
\end{proof}

\cleardoublepage
\section{Quasi Crossed Modules and Additive Feedback}
\label{sec:LQC}

The full strength of multiplicative integration requires a more
elaborate setup than the twisting setup of Definition \ref{dfn:5}.

\subsection{Quasi Crossed Modules}
Let $(Y, y_0)$ be a pointed analytic manifold. By {\em automorphism of pointed
analytic manifolds} we mean an analytic
diffeomorphism $f : Y \to Y$ such that $f(y_0) = y_0$.
We denote by $\opn{Aut}(Y, y_0)$ the group of all such automorphisms.

Let $G$ be a Lie group. An analytic action of $G$ on $(Y, y_0)$ by
automorphisms of pointed manifolds is an analytic map
$\Psi : G \times Y \to Y$ having the following properties.
First, for any $g \in G$ the map
$\Psi(g) : Y \to Y$, $\Psi(g)(y) := \Psi(g,y)$, is an automorphism of
pointed analytic manifolds. Second, the function
$\Psi : G \to \opn{Aut}(Y, y_0)$,
$g \mapsto \Psi(g)$, is a group homomorphism.

Given an analytic action of $G$ on $(Y, y_0)$, and an element $g \in G$,
the differential
\[ \d_{y_0} \bigl( \Psi(g) \bigr) : \mrm{T}_{y_0} Y \to \mrm{T}_{y_0} Y \]
is an $\R$-linear automorphism of the tangent space 
$\mrm{T}_{y_0} Y$. In this way we get a map of Lie groups
$G \to \mrm{GL}(\mrm{T}_{y_0} Y)$, which we call the {\em linear action
induced by $\Psi$}.

Let $H$ be a Lie group, with unit element $1$. We view it as a
pointed analytic manifold $(H, 1)$.

\begin{dfn} \label{dfn:22}
A {\em Lie quasi crossed module} \index{Lie quasi crossed module}
is the data 
\[ \mbf{C} = (G, H, \Psi, \Phi_0) \]
consisting of:
\begin{enumerate}
\item Lie groups $G$ and $H$.
\item An analytic action $\Psi$ of $G$ on $H$ by automorphisms of
pointed manifolds, called the {\em multiplicative twisting}.
\item A map of Lie groups $\Phi_0 : H \to G$, called the {\em
multiplicative feedback}.
\end{enumerate}
The condition is:
\begin{enumerate}
\rmitem{$*$} Consider $\Psi$ as a group homomorphism
$\Psi : G \to \opn{Aut}(H, 1)$. Then there is equality
\[ \Psi \circ\, \Phi_0 = \opn{Ad}_H \]
as group homomorphisms $H \to \opn{Aut}(H, 1)$.
\end{enumerate}
\end{dfn}

\begin{rem} \label{rem:1}
Let $(G, H, \Psi, \Phi_0)$ be a Lie quasi crossed module.
Suppose $G_0$ is a closed Lie subgroup of $G$ such that the following hold:
$\Phi_0(H) \subset G_0$; $\Psi(g)$ is a group automorphism of $H$ for any
$g \in G_0$; and $\Phi_0 : H \to G_0$ is $G_0$-equivariant 
(relative to $\Psi$ and $\opn{Ad}_{G_0}$). 
Then $(G_0, H, \Psi, \Phi_0)$ is called a {\em Lie crossed module}.
See \cite{BM, BS}.
In this situation condition ($*$) is called the {\em Pfeiffer condition} in
the literature. 

Note that we can always find such a subgroup $G_0$: just take
$G_0$ to be the closure of $\Phi_0(H)$ in $G$. 
An easy calculation shows that this subgroup has the required properties.
\end{rem}

Here are three of examples of Lie quasi crossed modules. 

\begin{exa} \label{exa:4}
Suppose 
\[ 1 \to N \to H \xar{\Phi_0} G \to 1 \]
is a central extension of Lie groups. Since $\opn{Ad}_H(h)$ is trivial for
$h \in N$, the action $\opn{Ad}_H$ induces an action of $G$ on $H$, which we
denote by $\Psi$. We  get a Lie crossed module $(G, H, \Psi, \Phi_0)$.
\end{exa}

\begin{exa} \label{exa:2}
A very special case of Example \ref{exa:4} is when $H = G$ and 
$\Phi_0 = \opn{id}_G$. Namely
\[ (G, H, \Psi, \Phi_0) = (G, G, \opn{Ad}_G, \opn{id}_G) . \]
This is the situation dealt with in the classical work of Schlesinger. 
\end{exa}

\begin{exa} \label{exa:5}
Let $H$ be a {\em unipotent} Lie group, namely $H$ is nilpotent and simply
connected, and let $\h := \opn{Lie}(H)$. The map $\exp_H : \h \to H$ is then an
analytic diffeomorphism. 
Take $G := \opn{GL}(\h)$. The canonical action of $G$ on $\h$ becomes, via
$\exp_H$, and action of $G$ on $H$ by automorphisms of pointed manifolds, which
we denote by $\Psi$. The adjoint action $\opn{Ad}_{\h}$ of $H$ on $\h$ 
is a map of Lie groups $\Phi_0 : H \to G$. Then 
$(G, H, \Psi, \Phi_0)$ is a Lie quasi crossed module. 

Next let $G_0 \subset G$ be the group of Lie algebra automorphisms of $\h$. 
Then $(G_0, H, \Psi, \Phi_0)$ is a Lie crossed module. 
\end{exa}

\subsection{Additive Feedback and Compatible Connections}
Let $(G, H, \Psi, \Phi_0)$ be a Lie quasi crossed module. 
We write $\g := \opn{Lie}(G)$ and $\h := \opn{Lie}(H)$.
Recall that the Lie algebra $\h$ is the tangent space to $H$ at the element
$1$. Hence the multiplicative twisting $\Psi$ induces a linear action 
\[ \Psi_{\h} : G \to \opn{GL}(\h) . \]
We see that from the Lie quasi crossed module 
$(G, H, \Psi, \Phi_0)$
we obtain a twisting setup 
$(G, H, \Psi_{\h})$. As in Section \ref{sec.MI}, we call 
$\Psi_{\h}$ the {\em additive twisting}. 

Recall that $\opn{Hom}(\h, \g)$ is the space of $\R$-linear maps
$\h \to \g$, and an element 
\[ \phi \in \mcal{O}_{\mrm{pws}}(X) \otimes \opn{Hom}_{\R}(\h, \g) \]
is called a piecewise smooth family of linear maps from $\h$ to $\g$.

\begin{dfn} \label{dfn:21}
Let $(G, H, \Psi, \Phi_0)$ be a Lie quasi crossed module, and let
$(X, x_0)$ be a pointed polyhedron. 
An {\em additive feedback  for $(G, H, \Psi, \Phi_0)$ over $(X, x_0)$}
\index{Additive feedback}
is an element
\[ \Phi_X \in \mcal{O}_{\mrm{pws}}(X) \otimes \opn{Hom}(\h, \g)  \]
satisfying this condition:
\begin{enumerate}
\rmitem{$**$} There is equality 
\[ \opn{Lie}(\Phi_0) = \Phi_X(x_0) \]
in $\opn{Hom}(\h, \g)$.
\end{enumerate}
\end{dfn}

\begin{dfn} \label{dfn:6}
Let $(X, x_0)$ be a pointed polyhedron. 
A {\em Lie quasi crossed module with additive feedback over $(X, x_0)$}
is the data 
\[  \mbf{C} / X = (G, H, \Psi, \Phi_0, \Phi_X) \]
consisting of:
\begin{itemize}
\item A Lie quasi crossed module 
$\mbf{C} = (G, H, \Psi, \Phi_0)$. 
\item An additive feedback $\Phi_X$ for $\mbf{C}$ over $(X, x_0)$.
\end{itemize}
\end{dfn}

When we talk about a Lie quasi crossed module with additive feedback 
$\mbf{C} / X$, by default we use the notation of Definitions \ref{dfn:21} and
\ref{dfn:6}, and we write
$\g := \opn{Lie}(G)$ and $\h := \opn{Lie}(H)$.

Let $\mbf{C} / X$ be a Lie quasi crossed
module with additive feedback
over $(X, x_0)$. Given a piecewise linear map
$f : (Y, y_0) \to (X, x_0)$ between pointed polyhedra, consider 
\[ f^*(\Phi_X) \in \mcal{O}_{\mrm{pws}}(Y) \otimes \opn{Hom}_{\R}(\h, \g) . \]
Then 
\begin{equation} \label{eqn:240}
f^* (\mbf{C} / X) := (G, H, \Psi, \Phi_0, f^*(\Phi_X)) 
\end{equation}
is a Lie quasi crossed module  with additive feedback
over $(Y, y_0)$.

\begin{dfn} \label{dfn:8}
Let $\mbf{C} / X$ be a Lie quasi crossed module with additive feedback over 
$(X, x_0)$. A {\em connection compatible with $\mbf{C} / X$} 
\index{Compatible connection}
is a differential form
\[ \alpha \in \Omega^1_{\mrm{pws}}(X) \otimes \g \]
satisfying the {\em holonomy condition}:
\begin{itemize}
\rmitem{$\Diamond$} Let $\sigma$ be a string in $X$, with
$x_0 = \sigma(v_0)$ and $x_1 := \sigma(v_1)$. Define
\[ g := \opn{MI}(\alpha \vert \sigma) \in G . \]
Then the diagram 
\[ \UseTips \xymatrix @C=9ex @R=5ex {
\h
\ar[d]_{\Phi_X(x_0)}
& \h
\ar[l]_{\Psi_{\h}(g)}
\ar[d]^{\Phi_X(x_1)} \\
\g
& \g
\ar[l]_{\opn{Ad}_{\g}(g)}
} \]
is commutative.
\end{itemize}
\end{dfn}

It could happen that $\mbf{C} / X$ does not admit any compatible connection. 

\begin{exa}  \label{exa:1}
Suppose $(G, H, \Psi, \Phi_0)$ is a Lie crossed module (see Remark \ref{rem:1})
and $(X, x_0)$ is a pointed polyhedron. Define 
$\Phi_X := \opn{Lie}(\Phi_0)$;
this is a $G$-equi\-variant Lie algebra map $\h \to \g$, which we view as a
constant element of 
$\mcal{O}_{\mrm{pws}}(X) \otimes \opn{Hom}(\h, \g)$.
In this way we obtain a 
Lie quasi crossed module with additive feedback 
$\mbf{C} / X := (G, H, \Psi, \Phi_0, \Phi_X)$ over $(X, x_0)$.
Since $\Phi_X$ is $G$-equivariant, it follows that any 
$\alpha \in \Omega^1_{\mrm{pws}}(X) \otimes \g$
is a connection compatible with $\mbf{C} / X$.
\end{exa}

An example of an additive feedback, and of a compatible connection, is given in
Subsection \ref{subsec:QTDGLie}.

To a Lie quasi crossed module with additive feedback $\mbf{C} / X$
there are two naturally associated twisting setups, namely
$(G, H, \Psi_{\h})$ and \lb $(G, G, \opn{Ad}_{\g})$.

\begin{prop} \label{prop:9}
Let  
\[ \mbf{C} / X = (G, H, \Psi, \Phi_0, \Phi_X) \]
be a Lie quasi crossed module with additive feedback
over $(X, x_0)$. Then:
\begin{enumerate}
\item The data $\bsym{\Theta}_X := (\opn{id}_G, \Phi_0, \Phi_X)$ is a
transfer of twisting setups
\[  (G, H, \Psi_{\h}) \to (G, G, \opn{Ad}_{\g}) \]
parametrized by $(X, x_0)$, in the sense of Definition 
\tup{\ref{dfn:7}}.

\item A form
$\alpha \in \Omega_{\mrm{pws}}^1(X) \otimes \g$
is a connection compatible with $\mbf{C} / X$ if and only if it is a 
connection compatible with $\bsym{\Theta}_X$, in the sense of 
Definition \tup{\ref{dfn:17}}.
\end{enumerate}
\end{prop}

\begin{proof}
Immediate from the definitions.
\end{proof}

\begin{prop} \label{prop:15}
Let $\mbf{C} / X$ be a Lie quasi crossed module with additive feedback over
$(X, x_0)$, and let $\alpha$ be a compatible connection for $\mbf{C} / X$. 
Suppose $f : (Y, y_0) \to (X, x_0)$ is a piecewise linear map between pointed
polyhedra. Then $f^*(\alpha)$ is a compatible connection for
$f^* (\mbf{C} / X)$.
\end{prop}

\begin{proof}
Take a string $\sigma$ in $Y$ with $\sigma(v_0) = y_0$, and let
$y_1 := \sigma(v_1)$. Then $f \circ \sigma$ is a string in $X$. According to
Proposition \ref{prop:12} we have 
\[ \opn{MI}(\alpha \vert f \circ \sigma) = 
\opn{MI}( f^*(\alpha) \vert \sigma)  . \]
Let's call this element $g$. Let $x_1 := f(y_1)$. Then 
\[ \Phi_X(x_i) = f^*(\Phi_X)(y_i) \] 
as homomorphisms $\h \to \g$, for $i = 0, 1$. 
We see that the holonomy condition is satisfied for 
$f^*(\alpha)$, relative to $f^* (\mbf{C} / X)$.
\end{proof}

\begin{dfn} \label{dfn:18}
Suppose 
\[ \mbf{C} / X = (G, H, \Psi, \Phi_0, \Phi_X) \] 
and 
\[ \mbf{C}' / X = (G', H', \Psi', \Phi_0', \Phi'_X) \]
are two Lie quasi crossed
modules with additive feedbacks over $(X, x_0)$. 
A {\em transfer} between them is a transfer of twisting setups
\[ \bsym{\Theta}_X = (\Theta_G, \Theta_H, \Theta_{\h, X}) :
(G, H, \Psi_{\h}) \to (G', H', \Psi'_{\h'})  \]
parametrized by $(X, x_0)$,
in the sense of Definition \ref{dfn:7}, satisfying this condition:
\[ \Phi'_X \circ \Theta_{\h, X} = \opn{Lie}(\Theta_G) \circ \Phi_X \]
in 
$\mcal{O}_{\mrm{pws}}(X) \otimes \opn{Hom}(\h, \g')$.
\end{dfn}

\subsection{Connection-Curvature Pairs} \label{subsec:ccpair}

\begin{dfn} \label{dfn:19}
Let $\mbf{C} / X = (G, H, \Psi, \Phi_0, \Phi_X)$ be a 
Lie quasi crossed module with additive feedback over a pointed 
polyhedron $(X, x_0)$; see Definition \ref{dfn:6}. A 
{\em connection-curvature pair} 
\index{Connection-curvature pair}
for $\mbf{C} / X$ is a pair $(\alpha, \beta)$, consisting of Lie
algebra valued differential forms
\[ \alpha \in \Omega_{\mrm{pws}}^1(X) \otimes \g \]
and
\[ \beta \in \Omega_{\mrm{pws}}^2(X) \otimes \h , \]
satisfying the conditions below.
\begin{itemize}
\rmitem{i} $\alpha$ is a connection compatible with $\mbf{C} / X$
(Definition \ref{dfn:8}).

\rmitem{ii} The equation
\[ \Phi_X(\beta) = \d (\alpha) + \smfrac{1}{2} [\alpha, \alpha] \]
holds in $\Omega_{\mrm{pws}}^2(X) \otimes \g$.
\end{itemize}
\end{dfn}

\begin{rem}
Condition (ii) above is often referred to as 
{\em vanishing of the fake curvature}. See \cite{BM, BS}.
\end{rem}

\begin{prop} \label{prop:11}
Let $\mbf{C} / X$ be a Lie quasi crossed module with additive feedback over
$(X, x_0)$, and let $(\alpha, \beta)$ be a connection-curvature pair for
$\mbf{C} / X$. 
\begin{enumerate}
\item Let $f : (Y, y_0) \to (X, x_0)$ be a piecewise linear map between pointed
polyhedra. Then $\bigl( f^*(\alpha), f^*(\beta) \bigr)$ 
is a connection-curvature pair in \lb
$f^* (\mbf{C} / X)$.
\item Let $\mbf{C}' / X$ be another Lie quasi crossed module with additive
feedback over $(X, x_0)$, and let
\[ \bsym{\Theta}_X = (\Theta_G, \Theta_H, \Theta_{\h, X}) : \mbf{C} / X 
\to \mbf{C}' / X \]
be a morphism of Lie quasi crossed module with additive
feedback. Assume that $\alpha$ is
compatible with $\Theta_{\h, X}$. See Definitions \tup{\ref{dfn:18}} and
\tup{\ref{dfn:17}}. Then 
\[ \bigl( \opn{Lie}(\Theta_G)(\alpha), \Theta_{\h, X}(\beta) \bigr) \]
is a connection-curvature pair for $\mbf{C}' / X$.
\end{enumerate}
\end{prop}

\begin{proof}
(1) By Proposition \ref{prop:15} the form $\alpha' := f^*(\alpha)$ is a
connection compatible with $f^* (\mbf{C} / X)$.
Next let us write $\beta' := f^*(\beta)$ and 
$\Phi_X' := f^*(\Phi_X)$. . Since  
\[ f^* : \Omega^1_{\mrm{pws}}(X) \otimes \g \to 
\Omega^1_{\mrm{pws}}(Y) \otimes \g \]
is a DG Lie algebra homomorphism, we have
\[ \Phi_X'(\beta') = f^*(\Phi_X(\beta)) 
= f^* \bigl( \d (\alpha) + \smfrac{1}{2} [\alpha, \alpha] \bigr) 
= \d (\alpha') + \smfrac{1}{2} [\alpha', \alpha'] . \]

\medskip \noindent
(2) The fact that $\alpha$ is compatible with $\Theta_{\h, X}$, 
and that $\bsym{\Theta}_X$ satisfies the condition ($\Diamond$) in Definition
\ref{dfn:17}, imply directly that $\alpha' := \opn{Lie}(\Theta_G)(\alpha)$ is a
connection compatible with $\mbf{C}' / X$.

Let $\beta' := \Theta_{\h, X}(\beta)$. Then
\[ \Phi'_X(\beta') = \opn{Lie}(\Theta_G)(\Phi_X(\beta)) = 
\opn{Lie}(\Theta_G) \bigl( \d (\alpha) + \smfrac{1}{2} [\alpha, \alpha] \bigr)
= \d (\alpha') + \smfrac{1}{2} [\alpha', \alpha'] . \]
We see that condition (ii) of Definition \ref{dfn:19} holds.
\end{proof}

\begin{dfn}[Tame Connection] \label{dfn:46}
Let $\mbf{C} / X$ be a Lie quasi crossed module with additive feedback over
$(X, x_0)$. A form
$\alpha \in \Omega^1_{\mrm{pws}}(X) \otimes \g$
is called a {\em tame connection} \index{Tame connection}
for $\mbf{C} / X$ if there exists a form
$\beta \in \Omega^2_{\mrm{pws}}(X) \otimes \h$ 
such that $(\alpha, \beta)$ is a connection-curvature pair 
(as in Definition \ref{dfn:19}). 
\end{dfn}

In other words, $\alpha$ is a tame connection if it is a compatible connection
(Definition \ref{dfn:8}), and its curvature
$\d (\alpha) + \smfrac{1}{2} [\alpha, \alpha]$
comes from $\Omega^2_{\mrm{pws}}(X) \otimes \h$.

\begin{cor} \label{cor:8}
In the situation of Proposition \tup{\ref{prop:11}}, the forms 
$f^*(\alpha)$ and $\opn{Lie}(\Theta_G)(\alpha)$ are tame connections.
\end{cor}

The proof is trivial.

\subsection{Moving the Base Point}
\label{subsec:moving}
In this subsection we consider the following setup: 
\[ \mbf{C} / X = (G, H, \Psi, \Phi_0, \Phi_X) \]
is a Lie quasi crossed module with additive feedback over a pointed 
polyhedron $(X, x_0)$. We are given a form
$\alpha \in \Omega^1_{\mrm{pws}}(X) \otimes \g$, which is a connection
compatible with $\mbf{C} / X$. And we are given a string $\rho$ in $X$, with
initial point $\rho(v_0) = x_0$ and terminal point $x_1 := \rho(v_1)$.
Let 
\[ g := \opn{MI}(\alpha \vert \rho) \in G . \]

Recall that $\Psi(g)$ is an automorphism of the pointed analytic manifold
$(H, 1)$. We define a new multiplication on the manifold $H$, by the formula
\[ h_1 \cdot^g h_2 := \Psi(g)^{-1} \bigl( \Psi(g)(h_1) \cdot 
\Psi(g)(h_2) \bigr) \]
for $h_1, h_2 \in H$. In this way we obtain a new Lie group, that is denoted by
$H^g$, and a Lie group isomorphism
\[ \Psi(g) : H^g \to H . \]
See Figure \ref{fig:70}.

\begin{figure} 
\includegraphics[scale=0.18]{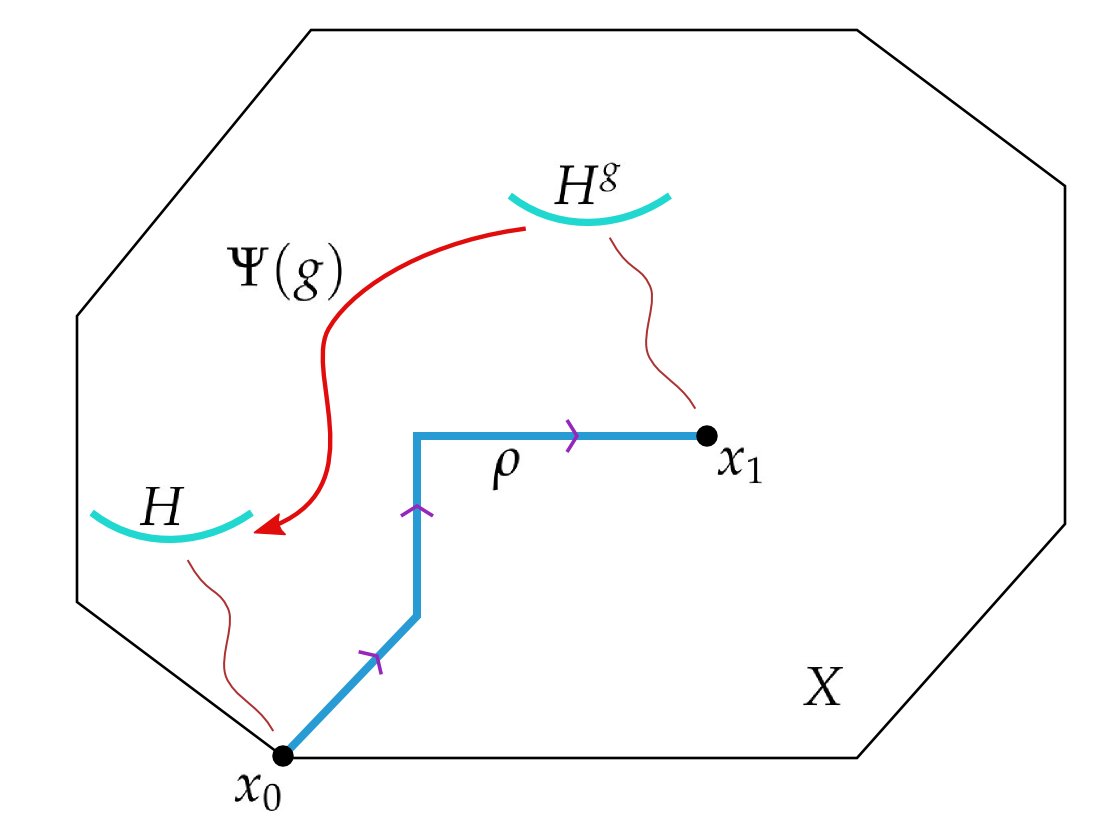}
\caption{The string $\rho$ from $x_0$ to $x_1$, the group element
$g := \opn{MI}(\alpha \vert \rho) \in G$, and the group isomorphism
$\Psi(g) : H^g \to H$.} 
\label{fig:70}
\end{figure}

The Lie algebra of $H^g$ is $\h^g$. So 
$\h^g = \h$ as vector spaces (this is the tangent space to $H$ at $1$), and 
\[ \Psi_{\h}(g) : \h^g \to \h \]
is a Lie algebra isomorphism. There is a commutative diagram of maps
\begin{equation} \label{eqn:133}
\UseTips \xymatrix @C=7ex @R=5ex {
\h
\ar[d]_{\exp_H}
& \h^g
\ar[l]_{\Psi_{\h}(g)}
\ar[d]^{\exp_{H^g}} \\
H
& H^g
\ar[l]_{\Psi(g)}
} 
\end{equation}
Note that if $\Psi(g)$ is not a group automorphism of $H$, then $H^g$ is not
equal to $H$ as groups. In this case the maps of manifolds
\[ \exp_H, \exp_{H^g} : \h \to H \]
could be distinct. 

The data $(G, H^g, \Psi_{\h})$ is a twisting setup (as in Definition
\ref{dfn:5}), which in general is distinct from the twisting setup
$(G, H, \Psi_{\h})$, because of the possibly distinct exponential maps. Given
a form $\beta \in \Omega^2_{\mrm{pws}}(X) \otimes \h$ and a kite
$(\sigma, \tau)$ in  $(X, x_1)$, let us denote by 
\begin{equation} \label{eqn:123}
\opn{MI}^g(\alpha, \beta \vert \sigma, \tau) \in H^g
\end{equation}
the multiplicative integral with respect to the twisting setup
$(G, H^g, \Psi_{\h})$. 

We define a map of Lie groups
$\Phi_0^g : H^g \to G$ by the commutative diagram
\[ \UseTips \xymatrix @C=7ex @R=5ex {
H
\ar[d]_{\Phi_0}
& H^g
\ar[l]_{\Psi(g)}
\ar[d]^{\Phi_0^g} \\
G
& G 
\ar[l]_{\opn{Ad}_G(g)}
} \]

\begin{prop}
The data
\[ \mbf{C}^g := (G, H^g, \Psi, \Phi_0^g) \]
is a Lie quasi crossed module.
\end{prop}

\begin{proof}
Let's write
$H' := H^g$, $G' := G$, $u := \opn{Ad}_G(g) : G' \to G$,
$v := \Psi(g) : H' \to H$,
$\Psi' := \Psi$ and $\Phi_0' := \Phi_0^g$. 
There are commutative diagrams of group homomorphisms
\begin{equation} \label{eqn:120}
\UseTips \xymatrix @C=7ex @R=5ex {
H
\ar[d]_{\Phi_0}
& H'
\ar[l]_{v}
\ar[d]^{\Phi_0'} \\
G
& G'
\ar[l]_{u}
} 
\qquad 
\UseTips \xymatrix @C=7ex @R=5ex {
G
\ar[d]_{\Psi}
& G'
\ar[l]_{u}
\ar[d]^{\Psi'} \\
\opn{Aut}(H)
& \opn{Aut}(H') 
\ar[l]_(0.5){\opn{Ad}(v)}
}
\end{equation}
Here $\opn{Aut}(H) = \opn{Aut}(H')$ is the group of automorphisms of the pointed
manifold $H = H'$, and $v$ is seen as an element of this group. 
The first diagram is just the definition of $\Phi_0'$.
The second diagram is commutative since for every $g' \in G'$ we have 
\[ \begin{aligned}
& \Psi(u(g')) = \Psi(g \cdot g' \cdot g^{-1}) = 
\Psi(g) \cdot \Psi(g') \cdot \Psi(g^{-1}) \\
& \qquad = v \cdot \Psi(g') \cdot v^{-1} = 
(\opn{Ad}(v) \circ \Psi')(g') . 
\end{aligned} \]
And by general group theory we have a commutative diagram
\begin{equation} \label{eqn:122}
\UseTips \xymatrix @C=7ex @R=5ex {
H
\ar[d]_{\opn{Ad}_{H}}
& H'
\ar[l]_{v}
\ar[d]^{\opn{Ad}_{H'}} 
\\
\opn{Aut}(H)
& \opn{Aut}(H')
\ar[l]^(0.5){\opn{Ad}(v)}
}
\end{equation}

We are given that  
$\Psi \circ \Phi_0 = \opn{Ad}_H$ 
(this is condition ($*$) of Definition \ref{dfn:22} for the Lie quasi crossed
module $\mbf{C}$). Therefore by combining the three commutative diagrams we
see that
$\Psi' \circ \Phi'_0 = \opn{Ad}_{H'}$ 
\end{proof}

\begin{prop}
The element $\Phi_X$ is an additive feedback for the Lie quasi crossed module
$\mbf{C}^g$ over the pointed polyhedron $(X, x_1)$. Thus
\[ \mbf{C}^g / X := (G, H^g, \Psi, \Phi_0^g, \Phi_X) \]
is a Lie quasi crossed module with additive feedback over 
$(X, x_1)$.
\end{prop}

\begin{proof}
In the notation used in the proof of the previous proposition, and
with $\Phi'_X := \Phi_X$, we have to show that
$\Phi'_X(x_1) = \opn{Lie}(\Phi'_0)$, as linear maps
$\h' \to \g'$. 
We have commutative diagrams of linear maps
\[ \UseTips \xymatrix @C=7ex @R=5ex {
\h
\ar[d]_{\opn{Lie}(\Phi_0)}
& \h'
\ar[l]_{\opn{Lie}(v)}
\ar[d]^{\opn{Lie}(\Phi_0')} \\
\g
& \g' 
\ar[l]_{\opn{Lie}(u)}
} 
\qquad 
\UseTips \xymatrix @C=8ex @R=5ex {
\h
\ar[d]_{\Phi_X(x_0)}
& \h
\ar[l]_{\Psi_{\h}(g)}
\ar[d]^{\Phi_X(x_1)} \\
\g
& \g
\ar[l]_(0.5){\opn{Ad}_{\g}(g)}
} \]
The first diagram is the differential of the first diagram in (\ref{eqn:120}),
and the second diagram is the holonomy condition for $\alpha$ relative to 
$\mbf{C} / X$. Since $\Phi_X$ is an additive feedback for $\mbf{C}$
over $(X, x_0)$ we have
$\Phi_X(x_0) = \opn{Lie}(\Phi_0)$.
We know that
$\Psi_{\h}(g) = \opn{Lie}(v)$ and
$\opn{Ad}_{\g}(g) = \opn{Lie}(u)$.
It follows that
$\Phi'_X(x_1) = \Phi_X(x_1) = \opn{Lie}(\Phi'_0)$.
\end{proof}

\begin{prop}
The form $\alpha$ is a connection compatible with $\mbf{C}^g / X$.
\end{prop}

\begin{proof}
We continue with the notation of the previous proofs. 
Let $\sigma$ be a string in $X$ initial point $\sigma(v_0) = x_1$ and terminal
point $x_2 := \sigma(v_1)$. Define
$g' := \opn{MI}(\alpha \vert \sigma) \in G' = G$. 
Because $\alpha$ is a connection compatible with $\mbf{C} / X$, and 
because
$g \cdot g' = \opn{MI}(\alpha \vert \rho * \sigma)$,
we have a commutative diagram
\[ \UseTips \xymatrix @C=10ex @R=6ex {
\h
\ar[r]^{\Psi_{\h}(g)}
\ar[d]_{\Phi_X(x_1)}
& \h
\ar[d]^{\Phi_X(x_0)} 
&
\h
\ar[d]^{\Phi_X(x_2)} 
\ar[l]_{\Psi_{\h}(g \cdot g')}
\\
\g
\ar[r]_(0.5){\opn{Ad}_{\g}(g)}
& \g
& \g
\ar[l]^(0.5){\opn{Ad}_{\g}(g \cdot g')}
} \]
Since
\[ \Psi_{\h}(g)^{-1} \circ \Psi_{\h}(g \cdot g') = \Psi_{\h}(g') \]
and
\[ \opn{Ad}_{\g}(g)^{-1} \circ \opn{Ad}_{\g}(g \cdot g') = 
\opn{Ad}_{\g}(g') \]
we see that the diagram 
\[ \UseTips \xymatrix @C=8ex @R=5ex {
\h'
\ar[d]_{\Phi'_X(x_1)}
& \h'
\ar[l]_{\Psi'_{\h'}(g')}
\ar[d]^{\Phi'_X(x_2)} \\
\g'
& \g'
\ar[l]_(0.5){\opn{Ad}_{\g'}(g')}
} \]
is commutative.
\end{proof}

\begin{thm} \label{thm:14}
Let 
\[ \mbf{C} / X = (G, H, \Psi, \Phi_0, \Phi_X) \]
be a Lie quasi crossed module with additive feedback over a pointed 
polyhedron $(X, x_0)$, let 
$\alpha \in \Omega^1_{\mrm{pws}}(X) \otimes \g$
be a connection compatible with $\mbf{C} / X$, and let
$\rho$ be a string in $X$, with initial point $\rho(v_0) = x_0$.
Define
$g := \opn{MI}(\alpha \vert \rho)$, 
and let
\[ \mbf{C}^g / X = (G, H^g, \Psi, \Phi_0^g, \Phi_X) \]
be the Lie quasi crossed module with additive feedback over $(X, x_1)$
constructed above. 

Given a form 
$\beta \in \Omega^2_{\mrm{pws}}(X) \otimes \h$ and a kite
$(\sigma, \tau)$ in  $(X, x_1)$, consider the element
\[ \opn{MI}^g(\alpha, \beta \vert \sigma, \tau) \in H^g  \]
from \tup{(\ref{eqn:123})}. Then 
\[ \Psi(g) \bigl( \opn{MI}^g(\alpha, \beta \vert \sigma, \tau) \bigr)  =
\opn{MI}(\alpha, \beta \vert \rho * \sigma, \tau)  \]
in $H$. 
\end{thm}

See Figure \ref{fig:71}.

\begin{figure} 
\includegraphics[scale=0.18]{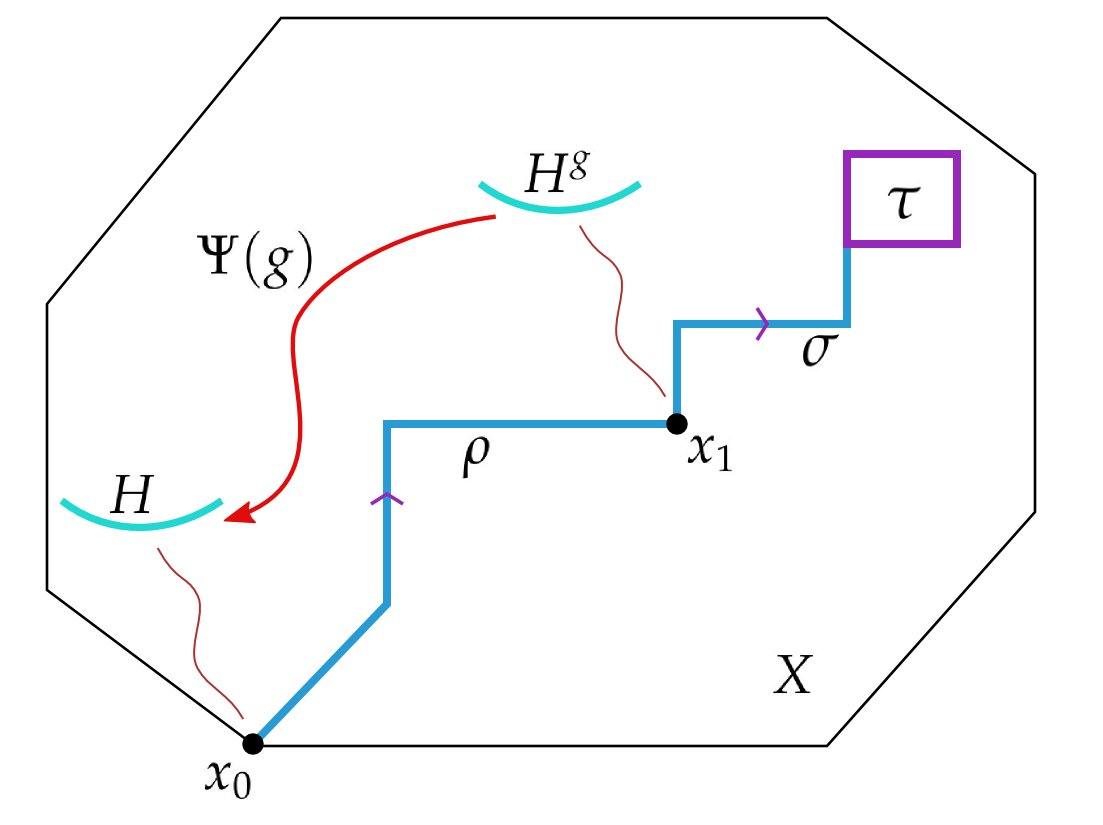}
\caption{The string $\rho$ from $x_0$ to $x_1$, and the kite
$(\sigma, \tau)$ in $(X, x_1)$.} 
\label{fig:71}
\end{figure}

\begin{proof}
We can assume that the kite $(\sigma, \tau)$ is nondegenerate. Define
$Z := \tau(\mbf{I}^2)$ and $g' := \opn{MI}^g(\alpha \vert \sigma)$.
Choose a positively oriented orthonormal linear coordinate system 
$(s_1, s_2)$ on $Z$.

Take $k \geq 0$ and $i \in \{ 1, \ldots, 4^k \}$, and 
define
$z_i := (\tau \circ \tau^k_i)(\smfrac{1}{2}, \smfrac{1}{2}) \in Z$
and
\[ g'_i := \opn{MI} \bigl( (\alpha \vert (\tau \circ \sigma^k_i) * 
(\tau \circ \tau^k_i \circ \sigma_{\opn{pr}}) \bigr) . \]
So
\[ \opn{MI} \bigl( \alpha \vert
\rho * \sigma * (\tau \circ \sigma^k_i) * 
(\tau \circ \tau^k_i \circ \sigma_{\opn{pr}})  \bigr) = 
g \cdot g' \cdot g'_i . \]

Assume that $z_i$ is a smooth point of $\beta|_Z$. Let $Y_i$ be a triangle
in $Z$, and let $\til{\beta}_i \in \mcal{O}(Y_i) \otimes \h$, such that
$z \in \opn{Int} Y_i$ and 
$\beta|_{Y_i} = \til{\beta}_i \cdot \d s_1 \wedge \d s_2$. According to 
Definition \ref{dfn:11}  and the commutative diagram (\ref{eqn:133}) we have
\[ \begin{aligned}
& \opn{RP_0} \bigl( \alpha, \beta \vert (\rho * \sigma, \tau) \circ 
(\sigma^k_i, \tau^k_i) \bigr) \\
& \qquad = 
\exp_H \bigl( (\smfrac{1}{4})^k \cdot \opn{area}(Z) \cdot 
\Psi_{\h}(g \cdot g' \cdot g'_i) (\til{\beta}_i(z_i)) \bigr) \\
& \qquad =
(\Psi(g) \circ \exp_{H^g}) \bigl( (\smfrac{1}{4})^k \cdot \opn{area}(Z) \cdot 
\Psi_{\h}(g' \cdot g'_i) (\til{\beta}_i(z_i)) \bigr) \\
& \qquad =
\Psi(g) \bigl( \opn{RP}_0^g \bigl( \alpha, \beta \vert 
(\sigma, \tau) \circ (\sigma^k_i, \tau^k_i) \bigr) \bigr) .
\end{aligned} \]

On the other hand, if $z_i$ is a singular point of $\beta|_Z$ then 
\[ \opn{RP}_0 \bigl( \alpha, \beta \vert (\rho * \sigma, \tau) \circ 
(\sigma^k_i, \tau^k_i) \bigr) = 1 =
\Psi(g) \bigl( \opn{RP}_0^g \bigl( \alpha, \beta \vert 
(\sigma, \tau) \circ (\sigma^k_i, \tau^k_i) \bigr) \bigr) . \]

We see that
\[ \opn{RP}_k (\alpha, \beta \vert \rho * \sigma, \tau) =
\Psi(g) \bigl( \opn{RP}_k^g ( \alpha, \beta \vert \sigma, \tau) \bigr) \]
for every $k$.
Passing to the limit $k \to  \infty$ finishes the proof.
\end{proof}

\subsection{Partial Differential Equations} \label{subsec:PDEs}
Let $\alpha \in \Omega^1(\mbf{I}^1) \otimes \g$. 
For $x \in \mbf{I}^1$ let $\sigma_x : \mbf{I}^1 \to \mbf{I}^1$
be the linear map defined on vertices by 
$\sigma_x(v_0, v_1) := (v_0, x)$.
Let
\begin{equation} \label{eqn:252}
g(x) := \opn{MI}(\alpha \vert \sigma_x) \in G .
\end{equation}
It is well known  that the function 
$g : \mbf{I}^1 \to G$ defined in this way is smooth, and moreover it satisfies
the differential equation
\begin{equation} \label{eqn:225}
\opn{dlog}(g) = \alpha
\end{equation}
with initial condition $g(0) = 1$.
See \cite{DF} for the case $G = \opn{GL}_n(\R)$, in which 
\[ \opn{dlog}(g) = g^{-1} \cdot \d g \]
as matrices. This ODE determines the function $g$. 

Now consider a Lie quasi-crossed module $\mbf{C}$, and a smooth
connection-cur\-vature
pair $(\alpha, \beta)$ in $\mbf{C} / \mbf{I}^2$. 
For a point $x \in \mbf{I}^2$ let 
$\tau_{x} : \mbf{I}^2 \to \mbf{I}^2$ be the linear map
defined on vertices by 
\[ \tau_x(v_0, v_1, v_1) := \bigl( v_0, (t_1(x), 0), (0, t_2(x)) \bigr) . \]
And let $\sigma$ be the empty string, so that 
$(\sigma, \tau_x)$ is a kite in $(\mbf{I}^2, v_0)$. Define
\[ h(x) := \opn{MI}(\alpha, \beta \vert \sigma, \tau_x) \in H . \]
Presumably the function 
$h : \mbf{I}^2 \to H$
is smooth, and it satisfies a partial differential equation generalizing 
(\ref{eqn:225}). We did not check this assertion.

In the very recent paper \cite{BGNT} the authors consider a special case:
the groups $G$ and $H$ are unipotent, $\mbf{C}$ is a crossed module, and the
pair $(\alpha, \beta)$ is algebraic. They write down a partial
differential equation, whose unique solution is declared to be the
multiplicative integral. Presumably this multiplicative integral coincides with
ours. For more in this direction see Subsection \ref{subsec:ration}.

\subsection{Quantum Type DG Lie Algebras} \label{subsec:QTDGLie}
Here we explain how Lie quasi crossed modules arise
from deformation theory. This is continued in the following subsection.

A {\em differential graded} (DG) Lie algebra is a graded $\R$-module
$\f = \boplus_{i \in \Z} \f^i$ \lb
equipped with a graded Lie bracket $[-,-]$ and a differential $\d$ (of degree
$1$) that satisfy the graded Leibniz rule. 
For instance, if $X$ is a manifold and $\g$ is a Lie algebra, then 
$\f := \Omega(X) \otimes \g$ is a DG Lie algebra. 
We say that $\f$ is a {\em quantum type} DG Lie algebra if
$\f^i = 0$ for all $i < -1$; i.e.\
$\f = \boplus_{i \geq -1} \f^i$.

Let us fix a quantum type DG Lie algebra $\f$, and assume that $\f$ is
nilpotent, and finite dimensional in each degree. 
(The assumptions of finiteness and nilpotence are for the sake of presentation;
the construction works even when $\f$ is infinite
dimensional  and pronilpotent.)

The Lie bracket $[-,-]$ of $\f$ makes the vector space $\g := \f^0$ into a
nilpotent Lie algebra. We denote by $G = \exp(\g)$ the corresponding unipotent
group. In order to be concrete, we take $G$ to be the analytic manifold $\g$, 
made into a Lie group by the CBH formula (cf.\ Section \ref{sec:expon}); in
particular $1_G := 0_{\g}$. The exponential map $\exp_G : \g \to G$ is just the
identity map here.

The Lie algebra $\g$ acts on the vector space $\f^i$ (any $i$) by the adjoint
action $\opn{ad}$, namely
\[ \opn{ad}(\alpha)(\beta) := [\alpha, \beta] \]
for $\alpha \in \g$ and $\beta \in \f^i$. Clearly $\opn{ad}(\alpha)$ is
$\R$-linear. 
There is a second action of $\g$ on $\f^1$, which we call the {\em affine
action}, and it is
\[ \opn{af}(\alpha)(\beta) := \d(\alpha) - [\alpha, \beta] \]
for $\beta \in \f^1$. The action $\opn{af}$ is usually not linear -- it is an
action by affine transformations. Both these actions integrate (or rather
exponentiate) to actions $\opn{Ad}$ and $\opn{Af}$ of the group $G$ on the
vector spaces $\f^i$ and $\f^1$ respectively. 

An element $\omega \in \f^1$ is called an {\em MC element} if it satisfies the
Maurer-Cartan equation
\[ \d(\omega)+ \smfrac{1}{2}[\omega, \omega] = 0 . \]
We denote by $\opn{MC}(\f)$ the set of all MC elements. 
It turns out that the action $\opn{Af}$ of $G$ on $\f^1$ preserves the subset 
$\opn{MC}(\f)$. Namely if $\omega \in \opn{MC}(\f)$ and $g \in G$, then 
$\omega' := \opn{Af}(g)(\omega)$ is also in $\opn{MC}(\f)$.
The quotient set by this action is denoted by $\ol{\opn{MC}}(\f)$.

Let us write $\h := \f^{-1}$. 
Say an MC element $\omega$ is given. Define an $\R$-linear map
$\d_{\omega} : \h \to \g$ by the formula
\begin{equation} \label{eqn:134}
\d_{\omega}(\beta) := \d(\beta) + [\omega, \beta]
\end{equation}
for $\beta \in \h$. And define an $\R$-bilinear operation 
$[-,-]_{\omega}$ on $\h$ by the formula
\begin{equation} \label{eqn:124}
[\beta_1, \beta_2]_{\omega} := [\d_{\omega}(\beta_1), \beta_2] . 
\end{equation}
An elementary calculation shows that the vector space $\h$ 
is a nilpotent Lie algebra with
respect to the bracket $[-,-]_{\omega}$, and we denote this Lie algebra by
$\h_{\omega}$. A similar calculation shows that 
$\d_{\omega} : \h_{\omega} \to \g$ is a Lie algebra map. 
Let us define
$H_{\omega} := \exp(\h_{\omega})$, the corresponding unipotent group, with
multiplication 
$h_1 \cdot_{\omega} h_2$ and inverse
$h_1^{-1_{\omega}}$, for $h_i \in H_{\omega}$.
So as analytic manifolds we have 
$H_{\omega} = \h_{\omega} = \h$, and 
$\exp_{H_{\omega}} : \h_{\omega} \to H_{\omega}$ is the identity map. 
There is a Lie group map
$\exp(\d_{\omega}) : H_{\omega} \to G$. 
It is easy to see from (\ref{eqn:124}) that for elements $h_1, h_2 \in
H_{\omega}$ one has
\[ \opn{Ad} \bigl( \exp(\d_{\omega})(h_1) \bigr)(h_2) = 
\opn{Ad}_{H_{\omega}}(h_1)(h_2) = 
h_1 \cdot_{\omega} h_2 \cdot_{\omega} h_1^{-1_{\omega}} .  \]

Suppose we are  given
$\omega \in \opn{MC}(\f)$ and $g \in G$. Define 
$\omega' := \opn{Af}(g)(\omega)$. A more difficult calculation shows that 
there is a commutative diagram of Lie algebra maps
\[ \UseTips \xymatrix @C=7ex @R=5ex {
\h_{\omega}
\ar[d]_{\d_{\omega}}
\ar[r]^{\opn{Ad}(g)}
& \h_{\omega'}
\ar[d]^{\d_{\omega'}}
\\
\g
\ar[r]^{\opn{Ad}(g)}
& \g
}  \]
See \cite[Section 1]{Ye5}. 
Hence there is also a commutative diagram of Lie group maps
\[ \UseTips \xymatrix @C=10ex @R=6ex {
H_{\omega}
\ar[d]_{\exp(\d_{\omega})}
\ar[r]^{\exp(\opn{Ad}(g))}
& H_{\omega'}
\ar[d]^{\exp(\d_{\omega'})}
\\
G
\ar[r]^{\exp(\opn{Ad}(g))}
& G
} \]
Observe that if 
$[-,-]_{\omega} \neq [-,-]_{\omega'}$ as Lie brackets, then 
$\h_{\omega} \neq \h_{\omega'}$ as Lie algebras, 
$H_{\omega} \neq H_{\omega'}$ as Lie groups, and
$\Psi(g) := \exp(\opn{Ad}(g))$ is not a Lie group automorphism of $H_{\omega}$;
it is only an automorphism of pointed analytic manifolds. 
At least we obtain in this way an action $\Psi$ of $G$ on $H_{\omega}$ by
automorphisms of pointed analytic manifolds. 

Take an MC element $\omega$, and define the map of Lie groups
\[ \Phi_{\omega} := \exp(\d_{\omega}) : H_{\omega} \to G . \]
What we have said so far implies that the data
\begin{equation} \label{eqn:125}
\mbf{C}_{\omega} := \bigl( G, H_{\omega}, \Psi, \Phi_{\omega} \bigr)
\end{equation}
is a Lie quasi crossed module. 

\subsection{Cosimplicial DG Lie Algebras} \label{subsec:CosDGLie}
Here we show how deformation quantization gives rise to Lie quasi \lb crossed
modules with additive feedback and con\-nection-curvature pairs. 

Suppose we are given a {\em cosimplicial nilpotent 
quantum type DG Lie algebra} $\f$. So 
$\f = \{ \f^{p, \bdot} \}_{p \in \N}$,
where each $\f^{p, \bdot} = \boplus_{i \geq -1} \f^{p,i}$
is a nilpotent quantum type DG Lie algebra.

{}From the cosimplicial object $\f$ one constructs its
Thom-Sullivan normalization $\til{\mrm{N}}(\f)$, which is a nilpotent 
quantum type DG Lie algebra, involving algebraic differential forms on all the
simplices $\bsym{\Delta}^p$, $p \geq 0$. For details see \cite[Section 2]{Ye2}.

Suppose $\til{\omega}$ is an MC element of
$\til{\mrm{N}}(\f)$. By definition 
$\til{\omega} = \{ \til{\omega}^{p, q} \}$, where
\[ \til{\omega}^{p, q} \in \Omega^q(\bsym{\Delta}^p) 
\otimes \f^{p, 1 - q} , \quad p \in \N,\ q \in \{ 0, 1, 2 \} .  \]
Let us fix a number $p$. We look at the components of $\til{\omega}$ with index
$p$:
\[ \omega := \til{\omega}^{p, 0} \in 
\mcal{O}(\bsym{\Delta}^p) \otimes \f^{p, 1} , \]
\[ \alpha := \til{\omega}^{p, 1} \in 
\Omega^1(\bsym{\Delta}^p) \otimes \f^{p, 0}  \]
and
\[ \beta := \til{\omega}^{p, 2} \in 
\Omega^2(\bsym{\Delta}^p) \otimes \f^{p, -1} . \]

Now let us write
$\g := \f^{p, 0}$, $\h := \f^{p, -1}$ and $X := \bsym{\Delta}^p$.
With this notation 
$\omega \in \mcal{O}(X) \otimes \f^{p, 1}$,
$\alpha \in \Omega^1(X) \otimes \g$
and
$\beta \in \Omega^2(X) \otimes \h$.

Choose some point $x_0 \in X$. The element 
$\omega(x_0) \in \f^{p, 1}$ turns out to be an MC element of the DG
Lie algebra $\f^{p, \bdot}$.
As in (\ref{eqn:125}), we have a Lie quasi crossed module 
\[  \mbf{C}_{\omega(x_0)}  =  
\bigl( G, H_{\omega(x_0)}, \Psi, \Phi_{\omega(x_0)} \bigr)  . \]
Next we define
\[ \Phi_{X} := \d_{\omega} \in 
\mcal{O}(X) \otimes \opn{Hom}(\h, \g) ; \]
see (\ref{eqn:134}). This is an additive feedback for $\mbf{C}_{\omega(x_0)}$
over the pointed polyhedron $(X, x_0)$, so we obtain a 
Lie quasi crossed module with additive feedback
\begin{equation} \label{eqn:522}
 \mbf{C}_{\omega(x_0)} / X  =  
\bigl( G, H_{\omega(x_0)}, \Psi, \Phi_{\omega(x_0)}, \Phi_{X} \bigr)  .
\end{equation}
Furthermore, it can be shown that the pair $(\alpha, \beta)$ is a
connection-curvature pair in 
$\mbf{C}_{\omega(x_0)} / X$. 
Taking any kite $(\sigma, \tau)$ in $(X, x_0)$ we thus obtain a group element 
\[ \opn{MI}(\alpha, \beta \vert \sigma, \tau) \in 
H_{\omega(x_0)} . \]

In Subsection \ref{subsec:conj.desc} we continue with this setup, and state
Conjecture \ref{conj:521}.

\cleardoublepage
\section{Stokes Theorem in Dimension Two}
\label{sec:stokes2}

The purpose of this section is to prove Theorem \ref{thm:1}.
When $H = G$ (see Example \ref{exa:3}) this is just a fancy
version of Schlesinger's theorem. 

Let  $(X, x_0)$ be a pointed polyhedron.
Recall the notions of Lie quasi crossed module with additive feedback
\begin{equation}
\mbf{C} / X = (G, H, \Psi, \Phi_0, \Phi_X)
\end{equation}
over $(X, x_0)$, connection-curvature pair $(\alpha, \beta)$, and 
multiplicative integral \lb
$\opn{MI}(\alpha, \beta \vert \sigma, \tau)$;
see Definitions \ref{dfn:6}, \ref{dfn:19} and \ref{dfn:16}
respectively.

\subsection{Some Estimates}
In this subsection we assume that $(X, x_0) = (\mbf{I}^2, v_0)$.
We fix a Lie quasi crossed module with additive feedback
$\mbf{C} / {X}$ over $(\mbf{I}^2, v_0)$, in which $G = H$, as in Example
\ref{exa:2}. We also fix a connection-curvature pair $(\alpha, \beta)$ for
$\mbf{C} / {X}$. Note that the equality
\begin{equation} \label{eqn:75}
\beta = \d (\alpha) + \smfrac{1}{2} [\alpha, \alpha] 
\end{equation}
holds in $\Omega^2_{\mrm{pws}}(\mbf{I}^2)  \otimes \mfrak{g}$.
As in Section \ref{sec:pws} we choose a euclidean norm $\norm{-}_{\g}$
on the vector space $\g$; an open neighborhood $V_0(G)$ of $1$
in $G$ on which $\log_G$ is well-defined; a convergence radius
$\epsilon_0(G)$; and a commutativity constant $c_0(G)$.

In Proposition \ref{prop:22} we established certain constants 
$c_{2'}(\alpha, \beta)$ and \lb $\epsilon_{2'}(\alpha, \beta)$.

\begin{lem} \label{lem:9}
There are constants $c_3(\alpha, \beta)$ and
$\epsilon_3(\alpha, \beta)$ with these properties:
\begin{enumerate}
\item $c_3(\alpha, \beta) \geq c_{2'}(\alpha, \beta)$ and
$0 < \epsilon_3(\alpha, \beta) \leq \epsilon_{2'}(\alpha, \beta)$.
\item Suppose $(\sigma, \tau)$ is a square kite in $(\mbf{I}^2, v_0)$
such that 
$\opn{side}(\tau) < \epsilon_3(\alpha, \beta)$ and
$\opn{len}(\sigma) \leq 5$. Then
\[ \opn{MI} (\alpha, \beta \vert \sigma, \tau) , \,
\opn{MI} (\alpha \vert \partial(\sigma, \tau)) \in V_0(G)  \]
and
\[ \Norm{ \log_G \bigl( \opn{MI} (\alpha, \beta \vert \sigma, \tau) \bigr) }
\, , \, 
\Norm{ \log_G \bigl( \opn{MI} (\alpha \vert \partial(\sigma, \tau)) \bigr) }
\leq c_3(\alpha, \beta) \cdot \opn{side}(\tau)^2  \, . \]
\end{enumerate}
\end{lem}

\begin{proof}
For $\opn{MI} (\alpha, \beta \vert \sigma, \tau)$ we can use the estimate from
Proposition \ref{prop:8}.

For the boundary we have to do some work. Let's write
$\epsilon := \opn{side}(\tau)$, and suppose that 
$\epsilon < \smfrac{1}{4} \epsilon_1(\alpha)$, where 
$\epsilon_1(\alpha)$ is the constant from Definition \ref{dfn:13}.
Consider the closed string $\partial \tau$ of length $4 \epsilon$. 
Write $g_1 := \opn{MI}(\alpha \vert \partial \tau) \in G$. 
According to Proposition \ref{prop:16} we know that 
$g_1 \in V_0(G)$, and 
\[ \Norm{ \log_G (g_1) - \int_{\partial \tau} \alpha } \leq 
c_0(G) \cdot c_1(\alpha)^2 \cdot (4 \epsilon)^2 \ . \]
By the abelian Stokes Theorem (Theorem \ref{thm:15}) we have
\[ \int_{\partial \tau} \alpha =  \int_{\tau} \d(\alpha) . \]
Now
\[ \Norm{ \int_{\tau} \d(\alpha) } \leq 
\opn{area}(\tau) \cdot \norm{\alpha}_{\mrm{Sob}} 
= \epsilon^2 \cdot  \norm{\alpha}_{\mrm{Sob}} \ . \]
We conclude that
\[ \norm{ \log_G (g_1) } \leq 
\bigl( c_0(G) \cdot c_1(\alpha)^2 \cdot 16 + \norm{\alpha}_{\mrm{Sob}} \bigr)
\cdot \epsilon^2 \ . \]

Next let 
$g_2 := \opn{MI}(\alpha \vert \sigma) \in G$.
Consider the representation 
$\opn{Ad}_{\g} : G \to \opn{GL}(\g)$. 
By Proposition \ref{prop:13} the norm of the operator 
$\opn{Ad}_{\g}(g_2)$ satisfies 
\[ \norm{ \opn{Ad}_{\g}(g_2) } \leq 
\exp( c_4(\alpha, \opn{Ad}_{\g}) \cdot 5) \ . \]

Finally we look at 
$g := \opn{MI} (\alpha \vert \partial(\sigma, \tau))$. 
By definition 
\[ g = g_2 \cdot g_1 \cdot g_2^{-1} = \opn{Ad}_{G}(g_2)(g_1) . \]
The logarithm is 
\[ \log_G(g) = \opn{Ad}_{\g}(g_2)(\log_G(g_1)) . \]
By combining the estimates above we get
\[ \norm{ \log_G(g) } \leq 
\exp( c_4(\alpha, \opn{Ad}_{\g}) \cdot 5) \cdot 
\bigl( c_0(G) \cdot c_1(\alpha)^2 \cdot 16 + \norm{\alpha}_{\mrm{Sob}} \bigr)
\cdot \epsilon^2 \ . \]
We see that we can take the following constants:
\[ \epsilon_3(\alpha, \beta) := \min \bigl( \epsilon_2(\alpha, \beta), 
\smfrac{1}{4} \epsilon_1(\alpha) \bigr) \] 
and 
\[ \begin{aligned}
& c_4(\alpha, \beta) := \\
& \quad \max \bigl( c_2(\alpha, \beta),
\exp( c_4(\alpha, \opn{Ad}_{\g}) \cdot 5) \cdot 
( c_0(G) \cdot c_1(\alpha)^2 \cdot 16 + \norm{\alpha}_{\mrm{Sob}} )
\bigr) \ . 
\end{aligned} \]
\end{proof}

\begin{lem} \label{lem:10}
There are constants $c_4(\alpha, \beta)$ and
$\epsilon_4(\alpha, \beta)$ with these properties:
\begin{enumerate}
\item $c_3(\alpha, \beta) \leq c_4(\alpha, \beta)$ and
$0 < \epsilon_4(\alpha, \beta) \leq \epsilon_3(\alpha, \beta)$.
\item Suppose $(\sigma, \tau)$ is a square kite in $(\mbf{I}^2, v_0)$
such that 
$\opn{side}(\tau) < \epsilon_4(\alpha, \beta)$, \linebreak
$\opn{len}(\sigma) \leq 5$, and $\alpha|_{\tau(\mbf{I}^2)}$ is smooth.
Then
\[ \Norm{ \log_G \bigl( \opn{MI} (\alpha, \beta \vert \sigma, \tau) \bigr) -
\log_G \bigl( \opn{MI} (\alpha \vert \partial(\sigma, \tau)) \bigr) }
\leq c_4(\alpha, \beta) \cdot \opn{side}(\tau)^3  \, . \]
\end{enumerate}
\end{lem}

\begin{proof}
Let $\epsilon := \opn{side}(\tau)$, $Z := \tau(\mbf{I}^{2})$, and
$z := \tau(\smfrac{1}{2}, \smfrac{1}{2})$, which is the midpoint of the square
$Z$. Denote by 
$(\rho_1, \ldots, \rho_4)$ the closed string $\partial \tau$. For any 
$i \in \{ 1, \ldots, 4 \}$ let
$z_i$ be the midpoint of the edge $\rho_i(\mbf{I}^{1})$. 
See Figure \ref{fig:26} for an illustration.
Let $(s_1, s_2)$ be the positively oriented orthonormal linear
coordinate system on $Z$, such that 
\[ \tau^*(s_i) = \epsilon \cdot (t_i - \smfrac{1}{2}) . \]
So in particular $s_i(z) = 0$. 

\begin{figure}
\includegraphics[scale=0.28]{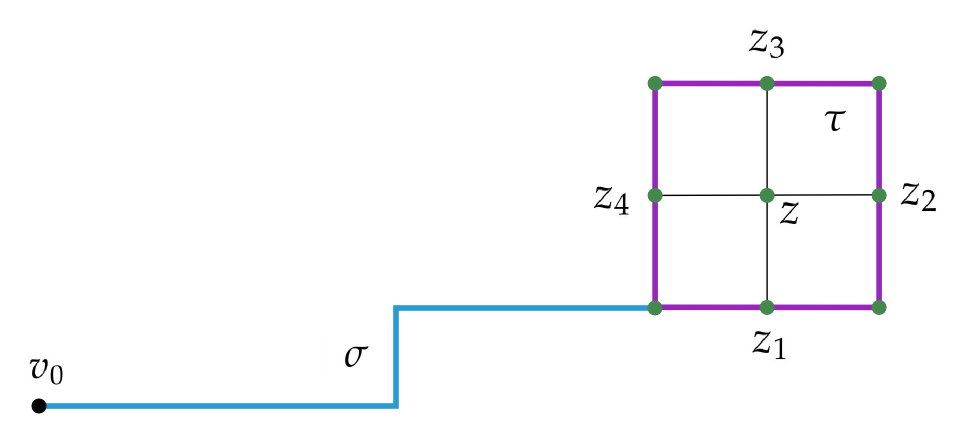}
\caption{Computing $\opn{RP}_{0}(\alpha \vert \partial \tau)$
for a tiny square kite $(\sigma, \tau)$ in  $(\mbf{I}^{2}, v_0)$.} 
\label{fig:26}
\end{figure}

Since $\alpha|_Z$ is smooth, there are functions
$\til{\alpha}_1, \til{\alpha}_2 \in \mcal{O}(Z) \otimes \g$ such that
\[ \alpha|_{Z} = \til{\alpha}_1 \cdot \d s_1 + \til{\alpha}_2 \cdot \d s_2 . \]
The Taylor expansion of $\til{\alpha}_j$ around $z$ to second order looks like
this:
\begin{equation} \label{eqn:241}
\til{\alpha}_j(x) = \til{\alpha}_j(z) + 
\sum_{1 \leq k \leq 2} 
(\smfrac{\partial}{\partial s_k} \til{\alpha}_j)(z) \cdot s_k(x) +
\sum_{1 \leq k, l \leq 2} 
g_{j, k, l}(x) \cdot s_k(x) \cdot s_l(x) 
\end{equation}
for $x \in Z$. Here 
$g_{j, k, l} : Z \to \g$ are continuous functions satisfying
\begin{equation} \label{eqn:242}
\norm{ g_{j, k, l} } \leq \norm{\alpha}_{\mrm{Sob}} \, . 
\end{equation}

We define elements $\lambda_i \in \g$ as follows:
\begin{equation}
\begin{aligned}
& \lambda_1 :=  \epsilon \cdot \til{\alpha}_1(z_1) , \\
& \lambda_2 :=  \epsilon \cdot \til{\alpha}_2(z_2) , \\
& \lambda_3 := -  \epsilon \cdot \til{\alpha}_1(z_3) , \\
& \lambda_4 := -  \epsilon \cdot \til{\alpha}_2(z_4) .
\end{aligned}
\end{equation}
Then, almost by definition,
\[ \opn{RP}_{0}(\alpha \vert \rho_i) = \exp_G(\lambda_i)  \]
and 
\begin{equation} \label{eqn:245}
\opn{RP}_{0}(\alpha \vert \partial \tau) =  
\prod_{i = 1}^{4} \, \exp_G(\lambda_i) .
\end{equation}

Now 
\[ s_k(z_i) \in \{ 0, \smfrac{1}{2} \epsilon, 
-  \smfrac{1}{2} \epsilon \} , \]
and the value $0$ occurs half the time. So the Taylor expansion 
(\ref{eqn:241}) for the point $z_i$ has only one summand of order $1$ in 
$\epsilon$ (instead of two). Let us define
\[ \mu_1 := \epsilon \cdot \til{\alpha}_1(z) , \quad
\mu'_1 := - \smfrac{1}{2} \epsilon^2 \cdot
(\smfrac{\partial}{\partial s_2} \til{\alpha}_1)(z)  ,  \]
\[ \mu_2 := \epsilon \cdot \til{\alpha}_2(z) , \quad
\mu'_2 := \smfrac{1}{2} \epsilon^2 \cdot 
(\smfrac{\partial}{\partial s_1} \til{\alpha}_2) (z) , \]
\[ \mu_3 := - \epsilon \cdot \til{\alpha}_1(z) , \quad
\mu'_3 := - \smfrac{1}{2} \epsilon^2 \cdot 
(\smfrac{\partial}{\partial s_2} \til{\alpha}_1)(z) , \]
\[ \mu_4 := - \epsilon \cdot \til{\alpha}_2(z) , \quad
\mu'_4 := \smfrac{1}{2} \epsilon^2 \cdot 
(\smfrac{\partial}{\partial s_1} \til{\alpha}_2)(z) . \]
Then using the estimate (\ref{eqn:242}) for the quadratic terms 
in (\ref{eqn:241}) we obtain
\begin{equation} \label{eqn:243}
\Norm{ \lambda_i - (\mu_i + \mu'_i) } \leq \epsilon^3 \cdot
\norm{\alpha}_{\mrm{Sob}} \, . 
\end{equation}
We also have these bounds:
\begin{equation} \label{eqn:246}
\norm{ \lambda_i } \leq \epsilon \cdot \norm{\alpha}_{\mrm{Sob}} \,
\end{equation}
and
\begin{equation} \label{eqn:247}
\norm{ \mu_i } \leq \epsilon \cdot \norm{\alpha}_{\mrm{Sob}} \ , \
\norm{ \mu'_i } \leq  \smfrac{1}{2} \epsilon^2 \cdot \norm{\alpha}_{\mrm{Sob}}
\, .
\end{equation}

According to property (iv) of Theorem \ref{thm:6}, the estimates
(\ref{eqn:243}) and (\ref{eqn:246}) yield 
\begin{equation} \label{eqn:22}
\begin{aligned}
& \Norm{ \log_G \Bigl( \boprod_{i = 1}^{4} \, \exp_G(\lambda_i) \Bigr)
- \log_G \Bigl( \boprod_{i = 1}^{4} \, \exp_G(\mu_i + \mu'_i) \Bigr) } \\
& \quad \quad \leq 4 \epsilon^3 \cdot c_0(G) \cdot \norm{\alpha}_{\mrm{Sob}} \,
. 
\end{aligned}
\end{equation}
for sufficiently small $\epsilon$. Similarly the estimates (\ref{eqn:247}) give
us
\begin{equation} \label{eqn:244}
\begin{aligned}
& \Norm{ \log_G \Bigl( \boprod_{i = 1}^{4} \, \exp_G(\mu_i + \mu'_i) \Bigr)
- \log_G \Bigl( \boprod_{i = 1}^{4} \, \exp_G(\mu_i) \Bigr)
- \bosum_{i = 1}^{4} \, \mu'_i } \\
& \quad \quad \leq 2 \epsilon^3 \cdot c_0(G) \cdot \norm{\alpha}_{\mrm{Sob}}^2
\, . 
\end{aligned}
\end{equation}

Now
$\mu_3 = -\mu_1$ and $\mu_4 = - \mu_2$, and hence
\[ \prod_{i = 1}^{4} \, \exp_G(\mu_i)  = 
\exp_G(\mu_1) \cdot \exp_G(\mu_2) \cdot \exp_G(- \mu_1) \cdot 
\exp_G(- \mu_2) . \]
According to property (iii) of Theorem \ref{thm:6} we see that
\begin{equation} \label{eqn:23}
\Norm{ \log_G \Bigl( \boprod_{i = 1}^{4} \, \exp_G(\mu_i) \Bigr)
- [\mu_1, \mu_2] }
\leq 2 \epsilon^3 \cdot c_0(G) \cdot \norm{\alpha}_{\mrm{Sob}}^3 \, .
\end{equation}
For the terms $\mu'_i$ we have
\begin{equation} \label{eqn:24}
\sum_{i = 1}^{4} \, \mu'_i = 
- \epsilon^2 \cdot (\smfrac{\partial}{\partial s_2} \til{\alpha}_1)(z) +
\epsilon^2 \cdot (\smfrac{\partial}{\partial s_1} \til{\alpha}_2) (z) \, . 
\end{equation}

Putting together equations (\ref{eqn:245}), (\ref{eqn:22}),
(\ref{eqn:244}), (\ref{eqn:23}) and (\ref{eqn:24})  
we conclude that for some 
$c \geq 1$ (depending on $(\alpha, \beta)$) the estimate
\[ \begin{aligned}
& \Norm{ \log_G \bigl( \opn{RP}_{0}(\alpha \vert \partial \tau) \bigr)
- \epsilon^2 \cdot \bigl(  [\til{\alpha}_1(z), \til{\alpha}_2(z)]  
- (\smfrac{\partial}{\partial s_2} \til{\alpha}_1)(z) +
(\smfrac{\partial}{\partial s_1} \til{\alpha}_2)(z) \bigr) } \\
& \qquad \leq \epsilon^3 \cdot c  
\end{aligned} \]
holds for sufficiently small $\epsilon$. 
Using Proposition \ref{prop:5}(2) we get
\begin{equation} \label{eqn:26}
\begin{aligned}
& \Norm{ \log_G \bigl( \opn{MI}(\alpha \vert \partial \tau) \bigr)
- \epsilon^2 \cdot \bigl(  [\til{\alpha}_1(z), \til{\alpha}_2(z)]  
- (\smfrac{\partial}{\partial s_2} \til{\alpha}_1)(z) +
(\smfrac{\partial}{\partial s_1} \til{\alpha}_2)(z) \bigr)  } \\
& \qquad \leq \epsilon^3 \cdot c' 
\end{aligned}
\end{equation}
for a suitable constant $c'$. 

Since equation (\ref{eqn:75}) holds, we know that $\beta$ is smooth on
$Z$. Let $\til{\beta} \in \mcal{O}(Z) \otimes \g$ be such that
$\beta|_{Z} = \til{\beta} \cdot \d s_1 \wedge \d s_2$.
Then (\ref{eqn:75}) becomes
\[ \til{\beta} = [\til{\alpha}_1, \til{\alpha}_2] - 
(\smfrac{\partial}{\partial s_2} \til{\alpha}_1) +
(\smfrac{\partial}{\partial s_1} \til{\alpha}_2) , \]
as smooth functions $Z \to \g$.
Thus we can rewrite (\ref{eqn:26}) as
\begin{equation} \label{eqn:27}
\Norm{ \log_G \bigl( \opn{MI}(\alpha \vert \partial \tau) \bigr)
- \epsilon^2 \cdot \til{\beta}(z) \bigr) } \leq \epsilon^3 \cdot c' \, .
\end{equation}
Letting
$g := \opn{MI}(\alpha \vert \sigma)$,
we have (by Proposition \ref{prop:1}):
\[ \opn{MI} \bigl( \alpha \vert \partial (\sigma, \tau) \bigr) =
g \cdot \opn{MI}(\alpha \vert \partial \tau) \cdot g^{-1} 
= \opn{Ad}_G(g) \bigl( \opn{MI}(\alpha \vert \partial \tau) \bigr) . \]
The map $\log_G$ sends $\opn{Ad}_G(g)$ to $\opn{Ad}_{\g}(g)$, and therefore
\[ \log_G \Bigl( \opn{MI} \bigl( \alpha \vert \partial (\sigma, \tau) \bigr) 
\Bigr) = \opn{Ad}_{\g}(g) \Bigl(
\log_G \bigl( \opn{MI}(\alpha \vert \partial \tau) \bigr) \Bigr) . \]
Recall that the length of $\sigma$ is bounded by $5$, so by 
Proposition \ref{prop:4} the norm of the operator 
$\opn{Ad}_{\g}(g)$ on $\g$ is also bounded. 
Plugging in the estimate (\ref{eqn:27})
we now arrive at
\begin{equation} \label{eqn:28}
\Norm{ \log_G \bigl( 
\opn{MI} \bigl( \alpha \vert \partial (\sigma, \tau) \bigr) \bigr)
- \epsilon^2 \cdot \opn{Ad}_{\g}(g) \bigl( \til{\beta}(z) \bigr) } \leq
\epsilon^3
\cdot c''' 
\end{equation}
for a suitable bound $c'''$, again depending on $(\alpha, \beta)$.

Finally, by definition we have
\[ \log_G \bigl( \opn{RP}_{0}(\alpha, \beta \vert \sigma, \tau) \bigr) = 
\epsilon^2 \cdot \opn{Ad}_{\g}(g) \bigl( \til{\beta}(z) \bigr) . \]
According to Proposition \ref{prop:8}(2) we know that
\[ \Norm{ \log_G \bigl(
\opn{RP}_{0}(\alpha, \beta \vert \sigma, \tau) \bigr)
- \log_G \bigl( \opn{MI}(\alpha, \beta \vert \sigma, \tau) \bigr) }
\leq c_2(\alpha, \beta) \cdot \epsilon^4 \, . \]
Combining this with (\ref{eqn:28}) we get
\[ \Norm{ \log_G \bigl( \opn{MI} (\alpha, \beta \vert \sigma, \tau) \bigr) -
\log_G \bigl( \opn{MI} (\alpha \vert \partial(\sigma, \tau)) \bigr) }
\leq c_4(\alpha, \beta) \cdot \epsilon^3   \]
for a sufficiently large constant $c_4(\alpha, \beta)$
and for all sufficiently small $\epsilon$. This gives us a value for
$\epsilon_4(\alpha, \beta)$.
\end{proof}

\begin{dfn} \label{dfn:34}
Let us fix constants $c_4(\alpha, \beta)$ and $\epsilon_4(\alpha, \beta)$ as in
Lemma \ref{lem:10}. 
A square kite $(\sigma, \tau)$ in $(\mbf{I}^2, v_0)$ will be
called {\em $(\alpha, \beta)$-tiny} in this section if 
$\opn{side}(\tau) < \epsilon_4(\alpha, \beta)$ and
$\opn{len}(\sigma) \leq 5$.
\end{dfn}

\begin{lem} \label{lem:7}
Let $(\sigma, \tau)$ be a kite in $(\mbf{I}^2, v_0)$. 
Take some $k \geq 0$. For $i \in \{ 1, \ldots, 4^k \}$ let 
\[ (\sigma_i, \tau_i) := (\sigma, \tau) \circ (\sigma^k_i, \tau^k_i)
= \opn{tes}^k_i (\sigma, \tau) . \]
Then
\[ \prod_{i = 1}^{4^k} \, 
\opn{MI} \bigl( \alpha \vert \partial(\sigma_i, \tau_i) \bigr) =
\opn{MI} \bigl( \alpha \vert \partial  (\sigma, \tau) \bigr) . \]
\end{lem}

\begin{proof}
By Proposition \ref{prop:1} we get cancellation of the contribution of all
inner edges. 
\end{proof}

\subsection{Stokes Theorem}

\begin{thm}[Nonabelian Stokes Theorem in Dimension $2$] \label{thm:1}
\index{Nonabelian Stokes Theorem for squares}
Let $(X, x_0)$ be a pointed polyhedron, let 	
$\mbf{C} / X$ be a Lie quasi crossed module with additive feedback over
$(X, x_0)$, and let $(\alpha, \beta)$ be a connection-curvature pair for
$\mbf{C} / X$.
Then for any kite $(\sigma, \tau)$ in $(X, x_0)$ one has
\[ \Phi_0 \bigl( \opn{MI} (\alpha, \beta \vert \sigma, \tau) \bigr) =
\opn{MI} \bigl( \alpha \vert \partial (\sigma, \tau) \bigr)  \]
in $G$.
\end{thm}

\begin{proof}
According to Proposition \ref{prop:9}(1) there is a transfer of
twisting setups
\[ (\opn{id}_G, \Phi_0, \Phi_X) :
(G, H, \Psi_{\h}) \to (G, G, \opn{Ad}_{\g}) \]
parametrized by $(X, x_0)$. Hence by Propositions \ref{prop:2} and
\ref{prop:11}(2) we can assume that $G = H$. 
Next, using Propositions \ref{prop:14}, \ref{prop:10} and \ref{prop:11}(1)
we can further assume that 
$(X, x_0) = (\mbf{I}^2, v_0)$, $(\sigma, \tau)$ is a square kite, and
$\opn{len}(\sigma) \leq 1$. We need to prove that 
\begin{equation}
\opn{MI} (\alpha, \beta \vert \sigma, \tau) = 
\opn{MI} (\alpha \vert \partial (\sigma, \tau))  
\end{equation}
in $G$.

Take $k$ large enough such that all the kites 
\begin{equation} \label{eqn:248}
(\sigma_i, \tau_i) := \opn{tes}^k_i (\sigma, \tau)
\end{equation}
in the $k$-th binary tessellation of $(\sigma, \tau)$ are 
$(\alpha, \beta)$-tiny. By Proposition \ref{prop:3} we have
\[ \opn{MI} (\alpha, \beta \vert \sigma, \tau) = 
\prod_{i = 1}^{4^k} \, 
\opn{MI} \bigl( \alpha, \beta \vert \sigma_i, \tau_i) . \]
Using Lemma \ref{lem:7} we see that it suffices to prove that 
\[ \opn{MI} \bigl( \alpha, \beta \vert \sigma_i, \tau_i) =
\opn{MI} \bigl( \alpha \vert \partial(\sigma_i, \tau_i) \bigr) \]
for every $i$. In this way we have reduced the problem to proving that 
\begin{equation} \label{eqn:29}
\opn{MI} (\alpha, \beta \vert \sigma, \tau) =
\opn{MI} \bigl( \alpha \vert \partial(\sigma, \tau) \bigr) 
\end{equation}
for any $(\alpha, \beta)$-tiny kite $(\sigma, \tau)$ in 
$(\mbf{I}^2, v_0)$.

So assume $(\sigma, \tau)$ is $(\alpha, \beta)$-tiny, with
$\epsilon := \opn{side}(\tau)$. Take some $k \geq 0$,
and let $(\sigma_i, \tau_i)$ be like in (\ref{eqn:248}). Then 
$\opn{side}(\tau_i) = (\smfrac{1}{2})^k \cdot \epsilon$.
We know that
\begin{equation} \label{eqn:30} 
\opn{MI} (\alpha, \beta \vert \sigma, \tau) =
\prod_{i = 1}^{4^k} \, 
\opn{MI} ( \alpha, \beta \vert \sigma_i, \tau_i)  
\end{equation}
and
\begin{equation} \label{eqn:31} 
\opn{MI} \bigl( \alpha \vert \partial  (\sigma, \tau) \bigr) = 
\prod_{i = 1}^{4^k} \, 
\opn{MI} \bigl( \alpha \vert \partial(\sigma_i, \tau_i) \bigr) . 
\end{equation}

Let us do some estimates now. If $i$ is a bad index (in the sense of 
Definition \ref{dfn:45}), then by Lemma \ref{lem:9} we know that
\[ \Norm{ \log_G \bigl( \opn{MI} (\alpha, \beta \vert \sigma_i, \tau_i) \bigr) 
} \leq c_3(\alpha, \beta) \cdot (\smfrac{1}{2})^{2 k} \epsilon^{2} \]
and
\[ \Norm{ \log_G \bigl( \opn{MI} (\alpha \vert \partial(\sigma_i, \tau_i))
\bigr) } \leq c_3(\alpha, \beta) \cdot (\smfrac{1}{2})^{2 k} \epsilon^{2} \, .
\]
On the other hand, if $i$ is a index kite, then Lemma
\ref{lem:10} says that 
\[ \Norm{ \log_G \bigl( \opn{MI} (\alpha, \beta \vert \sigma_i, \tau_i) \bigr) -
\log_G \bigl( \opn{MI} (\alpha \vert \partial(\sigma_i, \tau_i)) \bigr) }
\leq c_4(\alpha, \beta) \cdot  (\smfrac{1}{2})^{3 k} \epsilon^{3} \, . \]
Using these estimates, equations (\ref{eqn:30})-(\ref{eqn:31}),
Lemma \ref{lem:51} and property (iv) of Theorem \ref{thm:6}, we arrive at
\[ \begin{aligned}
& \Norm{ \log_G \bigl( \opn{MI} (\alpha, \beta \vert \sigma, \tau) \bigr)
- \log_G \bigl( \opn{MI} (\alpha \vert \partial(\sigma, \tau))
\bigr) } \\
& \qquad \leq 
\underset{\tup{good indices}}{\underbrace{4^k \cdot c_4(\alpha, \beta) \cdot 
(\smfrac{1}{2})^{3 k} \epsilon^{3}}} +
\underset{\tup{bad indices}}{\underbrace{
\bigl( a_0 + a_1 \cdot 2^k \bigr) \cdot 
c_3(\alpha, \beta) \cdot (\smfrac{1}{2})^{2 k} \epsilon^{2}}} \, .
\end{aligned} \]
Here 
$a_i := a_i(\alpha, \beta, \mbf{I}^2)$ are the constants from Lemma 
\ref{lem:51}. 
Since the right hand side of this inequality tends to $0$ as $k \to \infty$, 
we conclude that (\ref{eqn:29}) holds.
\end{proof}

\begin{exa} \label{exa:3}
A special case of the corollary is the situation of Example \ref{exa:2}. Take
any differential form 
$\alpha \in \Omega^1_{\mrm{pws}}(X) \otimes \g$, 
and let 
$\beta := \d (\alpha) + \smfrac{1}{2} [\alpha, \alpha]$. 
Then $(\alpha, \beta)$ is a connection-curvature pair, 
and the corollary says that
\[ \opn{MI} (\alpha, \beta \vert \sigma, \tau) =
\opn{MI} \bigl( \alpha \vert \partial (\sigma, \tau) \bigr) . \]
This is just Schlesinger's theorem \cite{DF}. 
\end{exa}

\subsection{The Fundamental Relation}

Here is an important consequence of Theorem \ref{thm:1}.

\begin{thm}[The Fundamental Relation] \label{thm:16}
\index{Fundamental Relation Theorem}
Let $(X, x_0)$ be a pointed polyhedron, let 	
$\mbf{C} / X$ be a Lie quasi crossed module with additive feedback over
$(X, x_0)$, let $(\alpha, \beta)$ be a connection-curvature pair for
$\mbf{C} / X$, and let $(\sigma, \tau)$ be a kite  in $(X, x_0)$.
Let us write 
\[ g := \opn{MI} \bigl( \alpha \vert \partial (\sigma, \tau) \bigr) \in G \]
and
\[ h :=  \opn{MI} (\alpha, \beta \vert \sigma, \tau) \in H . \]
Then
\[ \Psi(g) = \opn{Ad}_H(h) \]
as automorphisms of the pointed manifold $(H, 1)$. 
In particular, $\Psi(g)$ is a group automorphism of $H$.
\end{thm}

\begin{proof}
According to Theorem \ref{thm:1} we have
$g = \Phi_0(h)$, and by condition ($*$) of Definition \ref{dfn:22} we know that
\[ \Psi(g) = \Psi(\Phi_0(h)) = \opn{Ad}_{H}(h) . \]
\end{proof}

\begin{cor} \label{cor:6}
In the situation of Theorem \tup{\ref{thm:16}}, suppose 
$(\sigma', \tau')$ is
another kite in $(X, x_0)$. We get a closed string 
$\partial (\sigma', \tau')$ based at $x_0$, and a kite
$\bigl( \partial (\sigma', \tau') * \sigma, \tau \bigr)$. 

Then, writing
\[ h' :=  \opn{MI} (\alpha, \beta \vert \sigma', \tau') \in H , \]
one has
\[ \opn{MI} \bigl( \alpha, \beta \vert \partial (\sigma', \tau') * \sigma, \tau
\bigr) =
\opn{Ad}_H(h') (h)  \]
in $H$.
\end{cor}

See Figure \ref{fig:72}.

\begin{figure}
\includegraphics[scale=0.27]{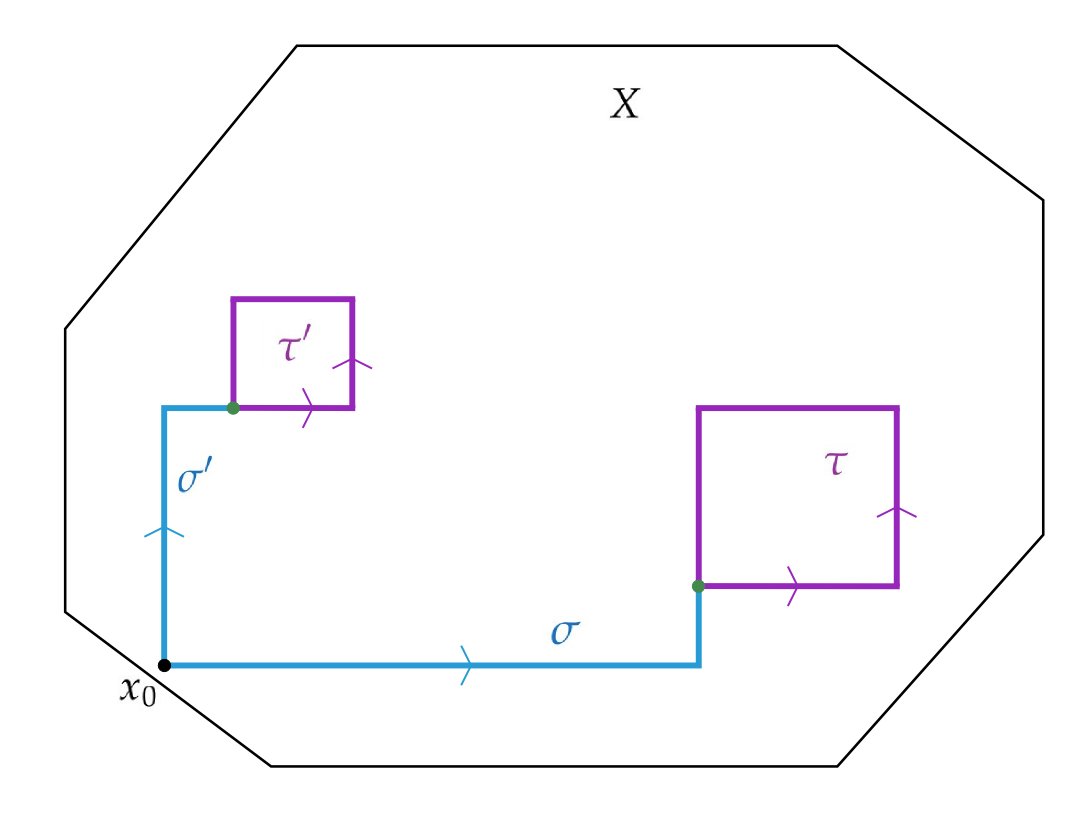}
\caption{Illustration for Corollary \ref{cor:6}.} 
\label{fig:72}
\end{figure}

\begin{proof}
Let $\rho := \partial (\sigma', \tau')$, which is a closed string  based at
$x_0$, and let $g' :=  \opn{MI}(\alpha \vert \rho) \in G$.
By Theorem \ref{thm:16} we have $\Psi(g') = \opn{Ad}_H(h')$, so this is a group
automorphism of $H$. Consider the ``moving of the base point'' corresponding to
$\rho$. Since $\Psi(g')$ is a group
automorphism of $H$, it follows that $H^{g'} = H$, and 
\[ \opn{MI}^{g'}(\alpha, \beta \vert \sigma, \tau)  =
\opn{MI}(\alpha, \beta \vert \sigma, \tau) = h . \]
But from Theorem \ref{thm:14} we get
\[ \Psi(g') \bigl( \opn{MI}^{g'}(\alpha, \beta \vert \sigma, \tau) \bigr)  =
\opn{MI}(\alpha, \beta \vert \rho * \sigma, \tau)  . \]
\end{proof}

\cleardoublepage
\section{Square Puzzles} \label{sec:puzzles}
\numberwithin{equation}{subsection}

In this section we work with  the pointed polyhedron
$(X, x_0) := (\mbf{I}^2, v_0)$. Let us fix some Lie quasi crossed module
with additive feedback 
\[ \mbf{C} / \mbf{I}^2 = (G, H, \Psi, \Phi_0, \Phi_{X}) \]
over $(\mbf{I}^2, v_0)$, and a piecewise smooth connection-curvature pair 
$(\alpha, \beta)$ in $\mbf{C} / \mbf{I}^2$.
See Definitions \ref{dfn:6} and \ref{dfn:19}.
We will show that under certain homotopical restrictions,
moving little square kites around inside $\mbf{I}^2$
doesn't alter the multiplicative surface integral.

\subsection{The Free Monoid with Involution} \label{subec:freemon}
It will be helpful for us to have some terminology for abstract words and
cancellation. 

Recall that a {\em monoid} is a unital semigroup. Suppose $M, N$ are monoids.
By homomorphism of monoids we mean a function
$\phi : M \to N$ that preserves the multiplications and the units. 

Let $S$ be some set, possibly infinite, whose elements we consider as symbols.
A {\em word} in $S$ is by definition a finite sequence $w = (s_1, \ldots, s_n)$
of elements of $S$. Thus $w$ is a function
\[ w : \{ 1, \ldots, n \} \to S . \] 
The natural number $n$ is the {\em length} of $w$.
We denote the set of all these words by 
$\opn{Wrd}(S)$. This is a monoid under the operation of
concatenation, which we denote by $*$, and the unit is the empty word 
$1 := ()$. We consider $S$ as a subset of $\opn{Wrd}(S)$, namely the words of
length $1$.
In fact $\opn{Wrd}(S)$ a free monoid: any function $S \to M$, where $M$
is a monoid, extends uniquely to a homomorphism of monoids
$\opn{Wrd}(S) \to M$. 

Next suppose $\bsym{s} = (s_1, \ldots, s_n)$ is a sequence of distinct elements
of $S$. We refer to such a sequence of distinct elements as an {\em alphabet}.
We denote by 
\[ \opn{Wrd}(\bsym{s}) = \opn{Wrd}(s_1, \ldots, s_n) \]
the subset of $\opn{Wrd}(S)$ consisting of the words in the alphabet $\bsym{s}$
only. More precisely, a word $w$ of length $m$ belongs to $\opn{Wrd}(\bsym{s})$
if and only if the function
$w : \{ 1, \ldots, m \} \to S$ factors as
$w = \bsym{s} \circ \til{w}$, for a function (necessarily unique)
\[ \til{w} : \{ 1, \ldots, m \} \to \{ 1, \ldots, n \} . \]
The elements of $\opn{Wrd}(\bsym{s})$ are denoted by $w(\bsym{s})$, and are
called {\em words in $\bsym{s}$}. Clearly $\opn{Wrd}(\bsym{s})$ is a sub-monoid
of $\opn{Wrd}(S)$; and
it is also a free monoid, on the $n$ symbols $s_1, \ldots, s_n$.
Given a homomorphism of monoids 
$\phi :  \opn{Wrd}(S) \to M$, and a word $w(\bsym{s}) \in \opn{Wrd}(\bsym{s})$, 
we shall use the following ``substitution notation'':
\[ w(\bsym{\gamma}) = w(\gamma_1, \ldots, \gamma_n) := \phi(w(\bsym{s})) \in M 
, \]
where $\bsym{\gamma}$ is the sequence
\[ \bsym{\gamma} = (\gamma_1, \ldots, \gamma_n) := 
\bigl( \phi(s_1), \ldots, \phi(s_n) \bigr) \]
in $M$. 

Let $M$ be a monoid. By an anti-automorphism of $M$ we mean a bijection 
$\psi : M \to  M$ that reverses the order of multiplication, and preserves the
unit. An {\em involution} of the monoid $M$ is an anti-automorphism $\psi$ such
that $\psi \circ \psi = \opn{id}$. 
The pair $(M, \psi)$ is then called a {\em monoid with involution}.

Let $S$ be a set. Let $S^{-1}$ be a new copy of $S$; i.e.\ $S^{-1}$
is a set disjoint from $S$, equipped with a bijection
$\psi_S : S \iso S^{-1}$. This bijection extends to an involution 
$\psi_S$ of the set $S \cup S^{-1}$. 
We define 
\[ \opn{Wrd}^{\pm 1}(S) := \opn{Wrd}(S \cup S^{-1}) . \]
This monoid has a canonical involution $\psi_S$, extending the involution on
the set $S \cup S^{-1}$. We sometimes write 
$w^{-1} := \psi_S(w)$ for a word $w \in \opn{Wrd}^{\pm 1}(S)$.
Note however that the product $w * w^{-1}$ is not $1$
(unless $w = 1$, the empty word). Indeed, the length of 
$w * w^{-1}$ is twice the length of $w$ (as words in $S \cup S^{-1}$).

The monoid $\opn{Wrd}^{\pm 1}(S)$ is a {\em free monoid with involution}. 
Here is what this means: let $(M, \psi)$ be any monoid with involution, and let
$f : S \to M$ be a function. Then there is a unique homomorphism of monoids
$\phi : \opn{Wrd}^{\pm 1}(S) \to M$ that commutes with the involutions and
extends
$f$. 

Given a sequence $\bsym{s} = (s_1, \ldots, s_n)$ of distinct elements of $S$,
we write 
\[ \begin{aligned}
&\opn{Wrd}^{\pm 1}(\bsym{s}) = \opn{Wrd}^{\pm 1}(s_1, \ldots, s_n) \\
& \qquad := \opn{Wrd}(s_1, \ldots, s_n, s_1^{-1}, \ldots, s_n^{-1}) \subset 
\opn{Wrd}^{\pm 1}(S) .
\end{aligned} \]
This is also a free monoid with involution.

Suppose the words 
$w$ and $w'$ in $\opn{Wrd}^{\pm 1}(S)$ satisfy this condition: 
\[ w = v_1 * u * v_2 \quad \text{and} \quad w' = v_1  * v_2 , \]
where $v_1$ and $v_2$ are some elements of $\opn{Wrd}^{\pm 1}(S)$, and
$u$ is either $s * s^{-1}$ or $s^{-1} * s$ for some $s \in S$. 
Then we say that $w'$ is {\em gotten from $w$ by cancellation}.
The equivalence relation generated by this condition is called {\em cancellation
equivalence}, and is denoted by $\sim_{\mrm{can}}$. Thus words 
$w, w' \in \opn{Wrd}(S)$ satisfy $w \sim_{\mrm{can}} w'$ 
if and only if there are words
\[ w = w_0, w_1, \ldots, w_r = w' \]
such that for each $i$ either $w_{i+1}$ is gotten from $w_i$
by cancellation, or vice versa. Note that the set of equivalence classes
$\opn{Wrd}^{\pm 1}(S) / \sim_{\mrm{can}}$ is a free group, with basis the image
of $S$. 

When we talk about $\opn{Wrd}^{\pm 1}(s_1, \ldots, s_n)$, we always
mean implicitly that it is the free monoid with involution on some sequence
$(s_1, \ldots, s_n)$ of distinct elements in 
some set $S$ (possibly $S = \{ s_1, \ldots, s_n \}$). 

If $\Gamma$ is a group, then by default we put on it the involution
$\psi(\gamma) := \gamma^{-1}$. In this way $\Gamma$ becomes a monoid with
involution. 

Let $(M, \psi)$ be a monoid with involution. Given an element
$m \in M$, define the operation
$\opn{Ad}(m) : M \to M$ by
\begin{equation} \label{eqn:249}
\opn{Ad}(m)(m') := m * m' * \psi(m) .
\end{equation}
Warning: $\opn{Ad}(m)$ is not a homomorphism of monoids (unless $M$ is a
group). 

Let $Y$ be a finite graph (i.e.\ a finite $1$-dimensional cellular complex),
with base point $y_0$ that's a vertex, and $n$ edges $\tau_1, \ldots, \tau_n$.
Choose an orientation on each $\tau_i$ (i.e.\ a homeomorphism 
$\mbf{I}^{1} \cong \tau_i$). 
The reversely oriented cell is denoted by $\tau_i^{-1}$. 
Take an alphabet $\bsym{s} = (s_1, \ldots, s_n)$.
Any word $w(\bsym{s}) \in \opn{Wrd}^{\pm 1}(\bsym{s})$ gives rise, by the
evaluation
$s_i \mapsto \tau_i$ and $s_i^{-1} \mapsto \tau_i^{-1}$, a sequence of oriented
cells $w(\tau_1, \ldots, \tau_n)$, which might or might
not be a path in $Y$. 

\begin{lem} \label{lem:15}
In the situation described above, suppose $w(\tau_1, \ldots, \tau_n)$ is a
closed path in $Y$ based at $y_0$, such that 
$[w(\tau_1, \ldots, \tau_n)] = 1$ in the fundamental group
$\bsym{\pi}_1(Y, y_0)$. Then $w(\bsym{s}) \sim_{\mrm{can}} 1$ 
in $\opn{Wrd}^{\pm 1}(\bsym{s})$. 
\end{lem}

\begin{proof}
We learned this proof from Y. Glasner.
First consider a pointed tree $(\til{Y}, \til{y}_0)$, and a sequence 
$(\til{\rho}_1, \ldots, \til{\rho}_l)$ of edges in $\til{Y}$ which is a path
starting at $\til{y}_0$.
If this path is closed, then it is cancellation equivalent to
a point. This can be seen by induction on $l$. 
Indeed, let $i$ be an index such that the endpoint of the path
$(\til{\rho}_1, \ldots, \til{\rho}_i)$ is at maximal distance from the
base vertex $\til{y}_0$. Then we must have
$\til{\rho}_{i+1} = \til{\rho}_i^{-1}$; so we can cancel these two edges,
yielding a shorter closed path. 

Getting back to our problem, write
$(\rho_1, \ldots, \rho_l) := [w(\tau_1, \ldots, \tau_n)]$, where 
$\rho_i \in \{ \tau_1^{\pm 1}, \ldots, \tau_n^{\pm 1} \}$.
Then $w(\bsym{s}) \sim_{\mrm{can}} 1$ if and only if 
$(\rho_1, \ldots, \rho_l)$ is cancellation equivalent to a point. 
Let $p : (\til{Y}, \til{y}_0) \to (Y, y_o)$ be the universal covering map; so 
$(\til{Y}, \til{y}_0)$ is a pointed tree.
Since the closed path 
$(\rho_1, \ldots, \rho_l)$ is trivial in the fundamental group of $Y$, it lifts
to a closed path $(\til{\rho}_1, \ldots, \til{\rho}_l)$
based at $\til{y}_0$ in the tree $\til{Y}$. By the first paragraph the 
path $(\til{\rho}_1, \ldots, \til{\rho}_l)$ is cancellation equivalent to a
point. Hence so is $(\rho_1, \ldots, \rho_l)$.
\end{proof}

\subsection{Generating Sequences}
Recall that for $k \geq 0$, the $k$-th binary subdivision 
$\opn{sd}^k \mbf{I}^2$ of $\mbf{I}^2$ is the cellular decomposition of 
$\mbf{I}^2$ into $4^k$ squares, each of side $(\smfrac{1}{2})^k$. 
The $1$-skeleton of $\opn{sd}^k \mbf{I}^2$ is the topological space
$\opn{sk}_1 \opn{sd}^k \mbf{I}^2$, and its fundamental group, based at $v_0$, is
$\bsym{\pi}_1(\opn{sk}_1 \opn{sd}^k \mbf{I}^2)$.
For a closed string $\sigma$ based at $v_0$ and patterned on 
$\opn{sd}^k \mbf{I}^2$ we denote by 
$[\sigma]$ the corresponding element of 
$\bsym{\pi}_1(\opn{sk}_1 \opn{sd}^k \mbf{I}^2)$.

\begin{dfn} \label{dfn:33}
Let $k \in \N$, and let
\[ \bsym{\rho}^{\natural} = \bigl( (\sigma^{\natural}_1, \tau^{\natural}_1),
\ldots, (\sigma^{\natural}_{4^k}, \tau^{\natural}_{4^k}) \bigr) \]
be a sequence of kites in $(\mbf{I}^2, v_0)$, all patterned on 
$\opn{sd}^k \mbf{I}^2$.
For any $i$ let
\[ a^{\natural}_i := [\partial (\sigma^{\natural}_i, \tau^{\natural}_i)]
\in \bsym{\pi}_1(\opn{sk}_1 \opn{sd}^k \mbf{I}^2) . \]
\begin{enumerate}
\item If the sequence $(a^{\natural}_1, \ldots, a^{\natural}_{4^k})$
is a basis of the group 
$\bsym{\pi}_1(\opn{sk}_1 \opn{sd}^k \mbf{I}^2)$,
then we say that $\bsym{\rho}^{\natural}$ is a {\em generating sequence for
$\opn{sd}^k \mbf{I}^2$}.

\item Let $(\sigma, \tau)$ be a kite in $(\mbf{I}^2, v_0)$ patterned on 
$\opn{sd}^k \mbf{I}^2$. If $\tau = \tau^{\natural}_i$ for some $i$ then 
we say that $(\sigma, \tau)$ is {\em aligned with $\bsym{\rho}^{\natural}$}.
\end{enumerate}
\end{dfn}

The $k$-th binary tessellation $\opn{tes}^k \mbf{I}^2$ 
(see Definition \ref{dfn:14}) is an example of a 
generating sequence for $\opn{sd}^k \mbf{I}^2$.

\begin{lem} \label{lem:17}
Let $\bsym{\rho}^{\natural}$ be a generating sequence for 
$\opn{sd}^k \mbf{I}^2$, and let 
$a^{\natural}_i = \lb [\partial (\sigma^{\natural}_i, \tau^{\natural}_i)]$
as in the definition above.
Let 
$w(\bsym{s}) \in \opn{Wrd}^{\pm 1}(\bsym{s}) = \opn{Wrd}^{\pm 1}(s_1, \ldots,
s_{4^k})$ 
be a word such that 
\[ w(a^{\natural}_1, \ldots, a^{\natural}_{4^k}) = 1  \]
in the group $\bsym{\pi}_1(\opn{sk}_1 \opn{sd}^k \mbf{I}^2)$.
Then 
\[ w(\bsym{s}) \sim_{\mrm{can}} 1 \]
in $\opn{Wrd}^{\pm 1}(\bsym{s})$.
\end{lem}

\begin{proof}
This is because the sequence 
$(a^{\natural}_1, \ldots, a^{\natural}_{4^k})$
is a basis of the free group \linebreak
$\bsym{\pi}_1(\opn{sk}_1 \opn{sd}^k \mbf{I}^2)$.
\end{proof}

\begin{lem} \label{lem:11}
Let $\sigma_1, \sigma_2$ be linear strings in $\mbf{I}^2$, and let
$\tau : \mbf{I}^2 \to \mbf{I}^2$ be a linear map, all patterned on 
$\opn{sd}^k \mbf{I}^2$ for some $k \geq 0$. 
\begin{enumerate}
\item Suppose $\sigma_1, \sigma_2$ are closed strings based at $v_0$, and
$[\sigma_1] = [\sigma_2]$ in \lb 
$\bsym{\pi}_1(\opn{sk}_1 \opn{sd}^k \mbf{I}^2)$. 
Then 
\[ \opn{MI} (\alpha \vert \sigma_1) = \opn{MI} (\alpha \vert \sigma_1) \]
in $G$.
\item Suppose that $\sigma_i(v_0) = v_0$ and 
$\sigma_i(v_1) = \tau(v_0)$ for $i = 1, 2$; so that 
$(\sigma_i, \tau)$ are kites in $\mbf{I}^2$ and 
$\sigma_2 * \sigma_1^{-1}$ is a closed string based at $v_0$. 
If 
$[\sigma_2 * \sigma_1^{-1}] = 1$ in 
$\bsym{\pi}_1(\opn{sk}_1 \opn{sd}^k \mbf{I}^2)$, then 
\[ \opn{MI} (\alpha, \beta \vert \sigma_1, \tau) = 
\opn{MI} (\alpha, \beta \vert \sigma_2, \tau) \]
in $H$. 
\end{enumerate}
\end{lem}

\begin{proof}
(1) Take 
$n := 2^{k+1} \cdot (2^k + 1)$, the number of $1$-cells in 
$\opn{sd}^k \mbf{I}^2$. Let us denote these $1$-cells by
$\tau_1, \ldots, \tau_n$, and let's choose an orientation for these cells, as
in Lemma \ref{lem:15}. There are unique words 
\[ w_1(\bsym{s}), w_2(\bsym{s}) \in \opn{Wrd}^{\pm 1}(\bsym{s}) = 
\opn{Wrd}^{\pm 1}(s_1, \ldots, s_n) \]
such that 
$w_i(\tau_1, \ldots, \tau_n) = \sigma_i$ as strings. 
Lemma \ref{lem:15} implies that \lb
$w_1(\bsym{s}) \sim_{\tup{can}} w_2(\bsym{s})$ in 
$\opn{Wrd}^{\pm 1}(\bsym{s})$. Now Definition \ref{dfn:10} and Proposition
\ref{prop:1}
tell us that 
\[ \opn{MI} \bigl( \alpha \vert w_1(\tau_1, \ldots, \tau_n) \bigr) = 
\opn{MI} \bigl( \alpha \vert w_2(\tau_1, \ldots, \tau_n) \bigr) . \]

\medskip \noindent
(2) Let
$g_i := \opn{MI} (\alpha \vert \sigma_i) \in G$.
By part (1) we know that 
\[ g_2 \cdot g_1^{-1} = \opn{MI} (\alpha \vert \sigma_2) \cdot
\opn{MI} (\alpha \vert \sigma_1^{-1}) = 
\opn{MI} (\alpha \vert \sigma_2 * \sigma_1^{-1}) = 1 ; \]
so that $g_1 = g_2$. Let $g := g_1 = g_2$.

Let $\sigma_0$ denote the empty string. So
$\sigma_i * \sigma_0 = \sigma_i$ for $i = 1, 2$.
Now according to Theorem \ref{thm:14} we have 
\[ \opn{MI}(\alpha, \beta \vert \sigma_i, \tau) =
\Psi(g) \bigl( \opn{MI}^g (\alpha, \beta \vert \sigma_0, \tau) \bigr) . \]
for $i = 1, 2$. But the right hand side is independent of $i$.
\end{proof}

\begin{lem} \label{lem:12}
Let $\bsym{\rho}^{\natural}$ be a generating sequence for 
$\opn{sd}^k \mbf{I}^2$, let 
$a^{\natural}_i$ be as in definition \tup{\ref{dfn:33}}, and let
\[ h^{\natural}_i := \opn{MI} (\alpha, \beta \vert 
\sigma^{\natural}_i, \tau^{\natural}_i) \in H . \]
Suppose $(\sigma, \tau)$ is a kite in $(\mbf{I}^2, v_0)$ patterned on
$\opn{sd}^k \mbf{I}^2$ and aligned with $\bsym{\rho}^{\natural}$. Then there
is a word 
\[ w(\bsym{s}) \in \opn{Wrd}^{\pm 1}(\bsym{s}) = 
\opn{Wrd}^{\pm 1}(s_1, \ldots, s_{4^k}) \]
such that
\[ [\partial (\sigma, \tau)] = 
w(a^{\natural}_1, \ldots, a^{\natural}_{4^k} ) \]
in $\bsym{\pi}_1(\opn{sk}_1 \opn{sd}^k \mbf{I}^2)$, and 
\[ \opn{MI} (\alpha, \beta \vert \sigma, \tau) = 
w(h^{\natural}_1, \ldots, h^{\natural}_{4^k}) \]
in $H$. 
\end{lem}

\begin{proof}
Let $i$ be an index such that 
$\tau = \tau^{\natural}_i$. Then 
$\sigma *  (\sigma^{\natural}_i)^{-1}$ is a closed string based at $v_0$ and
patterned on $\opn{sd}^k \mbf{I}^2$. 
(See Figure \ref{fig:33} for an illustration where $k = 1$
$\bsym{\rho}^{\natural} = \opn{tes}^1 \mbf{I}^2$.)
There is a  word
$u(\bsym{s}) \in \opn{Wrd}^{\pm 1}(\bsym{s})$ such that
\[ [ \sigma *  (\sigma^{\natural}_i)^{-1} ] = 
u(a^{\natural}_1, \ldots, a^{\natural}_{4^k}) \]
in the group $\bsym{\pi}_1(\opn{sk}_1 \opn{sd}^k \mbf{I}^2)$.

\begin{figure}
\includegraphics[scale=0.36]{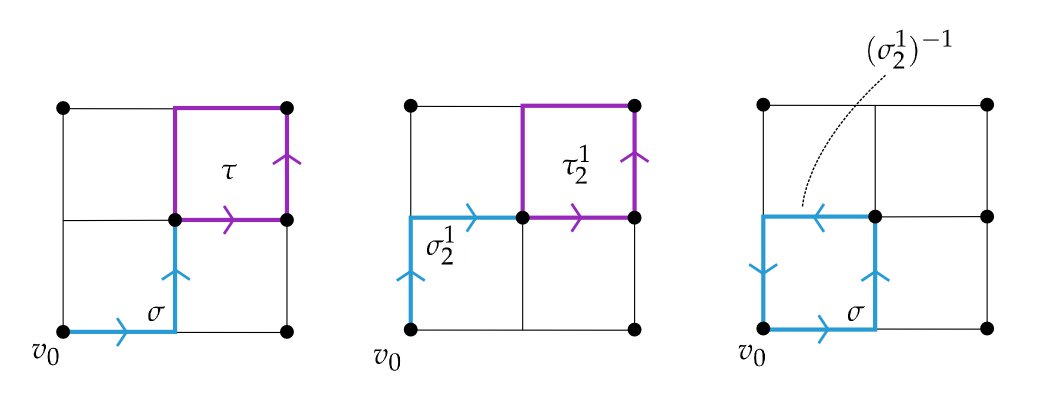}
\caption{A kite $(\sigma, \tau)$ patterned on 
$\opn{sd}^1 \mbf{I}^2$ and aligned with $\opn{tes}^1 \mbf{I}^2$, the
corresponding kite $(\sigma^{1}_i, \tau^{1}_i)$, and the closed string 
$\sigma * (\sigma^{1}_i)^{-1}$. Here $i = 2$.} 
\label{fig:33}
\end{figure}

We can also consider the monoid $M$ of finite sequences of closed strings
patterned on $\opn{sd}^k \mbf{I}^2$
and based at $v_0$, where composition is concatenation of sequences, and the
involution is given by reversal of order and (\ref{eqn13}). 
Consider the sequence
\[ \partial \bsym{\rho}^{\natural} :=
\bigl( \partial (\sigma^{\natural}_1, \tau^{\natural}_1),  \ldots,
\partial (\sigma^{\natural}_{4^k}, \tau^{\natural}_{4^k}) \bigr) . \]
The evaluation 
$u(\partial \bsym{\rho}^{\natural})$ of the word $u(\bsym{s})$
on the sequence
$\partial \bsym{\rho}^{\natural}$ is an element of $M$. And we have
\begin{equation} \label{eqn:35}
\bigl[ u(\partial \bsym{\rho}^{\natural}) \bigr] = 
u(a^{\natural}_1, \ldots, a^{\natural}_{4^k})
\end{equation}
in the group $\bsym{\pi}_1(\opn{sk}_1 \opn{sd}^k \mbf{I}^2)$.

Take
\[ w(\bsym{s}) := u(\bsym{s}) * s_i * u(\bsym{s})^{-1} \in \opn{Wrd}^{\pm
1}(\bsym{s}) .
\]
It is easy to check that 
\[ [\partial (\sigma, \tau)] = w(a_1^{\natural}, \ldots, a^{\natural}_{4^k}) \]
holds in $\bsym{\pi}_1(\opn{sk}_1 \opn{sd}^k \mbf{I}^2)$. Next, using Corollary
\ref{cor:6} recursively, equation (\ref{eqn:35}) 
and Lemma \ref{lem:11}(2), we obtain
\[ \begin{aligned}
w(h^{\natural}_1, \ldots, h^{\natural}_{4^k}) & =
u(h^{\natural}_1, \ldots, h^{\natural}_{4^k}) \cdot h^{\natural}_i \cdot 
u(h^{\natural}_1, \ldots, h^{\natural}_{4^k})^{-1}  \\
& =
\opn{MI} \bigl( \alpha, \beta \vert 
u(\partial \bsym{\rho}^{\natural}) * \sigma^{\natural}_i, \tau^{\natural}_i
\bigr)  \\
& =
\opn{MI} \bigl( \alpha, \beta \vert \sigma , \tau \bigr) .
\end{aligned} \]
\end{proof}

\begin{lem} \label{lem:13}
Let $\bsym{\rho}^{\natural}$ be a generating sequence for 
$\opn{sd}^k \mbf{I}^2$, and let 
\[ \bsym{\rho} = \bigl( (\sigma_1, \tau_1), \, \ldots, \, (\sigma_m, \tau_m) 
\bigr) \]
be some sequence of kites in $(\mbf{I}^2, v_0)$ patterned on 
$\opn{sd}^k \mbf{I}^2$ and aligned with $\bsym{\rho}^{\natural}$.
Write
\[ a_i := [\partial (\sigma_i, \tau_i)]  \in
\bsym{\pi}_1(\opn{sk}_1 \opn{sd}^k \mbf{I}^2) \]
and
\[ h_i := \opn{MI} (\alpha, \beta \vert \sigma_i, \tau_i)\in H . \]
Let 
\[ u(\bsym{s}) \in \opn{Wrd}^{\pm 1}(\bsym{s}) = 
\opn{Wrd}^{\pm 1}(s_1, \ldots, s_{m}) \]
be a word such that
\[ u(a_1, \ldots, a_m) = 1 . \]
Then 
\[ u(h_1, \ldots, h_m) = 1 . \]
\end{lem}

\begin{proof}
Using Lemma \ref{lem:12}, for each $i \in \{ 1, \ldots, m \}$ we can find
a word 
\[ w_i(\bsym{t}) \in \opn{Wrd}^{\pm 1}(\bsym{t}) = 
\opn{Wrd}^{\pm 1}(t_1, \ldots, t_{4^k}) \]
such that
\[ a_i = w_i(a^{\natural}_1, \ldots, a^{\natural}_{4^k}) \]
and
\[ h_i = w_i(h^{\natural}_1, \ldots, h^{\natural}_{4^k})  \]
(with notation as in the previous lemmas). Define
\[ w(\bsym{t}) := u \bigl( w_1(\bsym{t}), \ldots, w_m(\bsym{t}) \bigr) \in
\opn{Wrd}^{\pm 1}(\bsym{t}) . \]
Then
\[ w(a^{\natural}_1, \ldots, a^{\natural}_{4^k}) = 1 \] 
in 
$\bsym{\pi}_1(\opn{sk}_1 \opn{sd}^k \mbf{I}^2)$, 
and 
\[ w(h^{\natural}_1, \ldots, h^{\natural}_{4^k}) = u(h_1, \ldots, h_m)  \]
in $H$.

Now according to Lemma \ref{lem:17} we have
$w(\bsym{t}) \sim_{\mrm{can}} 1$ in $\opn{Wrd}^{\pm 1}(\bsym{t})$. 
Hence by cancellation in the group $H$ we get
$w(h^{\natural}_1, \ldots, h^{\natural}_{4^k}) = 1$.
\end{proof}

\subsection{The Flip}
The flip of $\mbf{I}^2$ is the linear automorphism whose action on vertices is
\[ \opn{flip}(v_0, v_1, v_2) := (v_0, v_2, v_1) . \]
Given a kite $(\sigma, \tau)$ in $(\mbf{I}^2, v_0)$, its flip is the kite 
\[ \opn{flip}(\sigma, \tau) := (\sigma, \tau \circ \opn{flip} ) . \]
See Figure \ref{fig:34} for an illustration.

\begin{figure}
\includegraphics[scale=0.36]{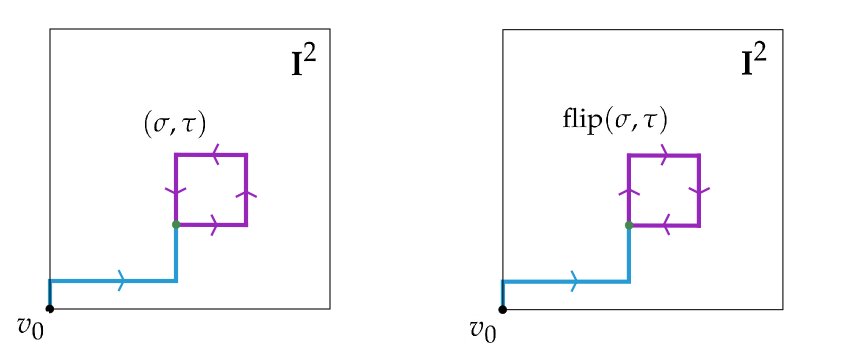}
\caption{A kite $(\sigma, \tau)$ in $(\mbf{I}^2, v_0)$, and the kite
$\opn{flip} (\sigma, \tau)$.} 
\label{fig:34}
\end{figure}

Note that 
\[ \opn{flip} (\sigma, \tau) 
= (\sigma, \tau) \circ \opn{flip} (\sigma^{0}_1, \tau^{0}_1) , \] 
where $(\sigma^{0}_1, \tau^{0}_1)$ is the basic kite. 
By associativity of kite composition it follows that given any two kites
$(\sigma_1, \tau_1)$ and $(\sigma_2, \tau_3)$ in $\mbf{I}^{2}$, one has
\begin{equation} \label{eqn:55}
\opn{flip} \bigl( (\sigma_1, \tau_1) \circ (\sigma_2, \tau_2) \bigr)
= (\sigma_1, \tau_1) \circ \opn{flip} (\sigma_2, \tau_2)
\end{equation}

If  $(\sigma, \tau)$ is a kite patterned on $\opn{sd}^k \mbf{I}^2$ for
some $k$, then the effect of flipping is:
\begin{equation} \label{eqn:177}
[ \partial \opn{flip}(\sigma, \tau) ] =  [ \partial (\sigma, \tau) ]^{-1}
\end{equation}
in $\bsym{\pi}_1(\opn{sk}_1 \opn{sd}^k \mbf{I}^2)$.

\begin{lem} \label{lem:41}
Let $(\sigma, \tau)$ be a kite in $(\mbf{I}^2, v_0)$, and let 
$k \geq 0$. Then 
\begin{equation} \label{eqn:178}
\opn{MI} \bigl( \alpha, \beta \vert \opn{flip}(\sigma, \tau) \bigr) = 
\prod_{i = 4^k}^1 \, 
\opn{MI} \bigl( \alpha, \beta \vert \opn{flip}( (\sigma, \tau) \circ 
(\sigma^k_i, \tau^k_i) ) \bigr) .
\end{equation}
\end{lem}

Note that the order of the product is reversed!

\begin{proof}
By moving the base point from $v_0$ to $x_0 := \tau(v_0)$ along the string
$\sigma$, and using Theorem \ref{thm:14}, we now have to prove that 
(\ref{eqn:178}) holds for a kite $(\sigma, \tau)$ in the pointed polyhedron
$(\mbf{I}^2, x_0)$, and moreover $\sigma$ is the empty string. 
Next we use Proposition \ref{prop:10} for the map of pointed polyhedra
$\tau : (\mbf{I}^2, v_0) \to (\mbf{I}^2, x_0)$ to reduce to proving 
(\ref{eqn:178}) for $(\sigma, \tau) = (\sigma^0_1, \tau^0_1)$,
the basic kite  in $(\mbf{I}^2, v_0)$.

So we have to prove that 
\begin{equation} \label{eqn:181}
\opn{MI} \bigl( \alpha, \beta \vert \opn{flip}(\sigma^0_1, \tau^0_1) \bigr) = 
\prod_{i = 4^k}^1 \, 
\opn{MI} \bigl( \alpha, \beta \vert \opn{flip}(\sigma^k_i, \tau^k_i) \bigr) .
\end{equation}
holds. 

For any $i \in \{ 1, \ldots, 4^k \}$ define
\[ (\sigma^{\natural}_i, \tau^{\natural}_i) :=
\opn{flip} (\sigma^0_1, \tau^0_1) \circ (\sigma^k_i, \tau^k_i) . \]
Then 
\[ \bsym{\rho}^{\natural} := \bigl( (\sigma^{\natural}_1, \tau^{\natural}_1),
\ldots, (\sigma^{\natural}_{4^k}, \tau^{\natural}_{4^k}) \bigr) \]
is a generating sequence for $\opn{sd}^k \mbf{I}^2$. By Proposition
\ref{prop:3} we know that
\begin{equation} \label{eqn:179}
\opn{MI} \bigl( \alpha, \beta \vert \opn{flip}(\sigma^0_1, \tau^0_1) \bigr) = 
\prod_{i = 1}^{4^k} \, h^{\natural}_i ,
\end{equation}
where $h^{\natural}_i$ are as in Lemma \ref{lem:12}.
And it is clear from Definition \ref{dfn:4}(2) and formula (\ref{eqn:177}) that 
\begin{equation} \label{eqn:180}
\prod_{i = 1}^{4^k} \, a^{\natural}_i = [ \partial \mbf{I}^2 ] ^{-1} ,
\end{equation}
where $a^{\natural}_i$ are as in Definition \ref{dfn:33}.

Consider another sequence of kites
\[ \bsym{\rho} := \bigl( (\sigma^{}_1, \tau^{}_1),
\ldots, (\sigma^{}_{4^k}, \tau^{}_{4^k}) \bigr) , \]
where we define
\[ (\sigma_i, \tau_i) := \opn{flip}(\sigma^k_i, \tau^k_i) . \]
It is not hard to show, by induction on $k$, that these kites are aligned with
the generating sequence $\bsym{\rho}^{\natural}$.
Let $a_i$ and $h_i$ be as in Lemma \ref{lem:13}, where we take 
$m := 4^k$ of course. Again by induction on $k$ one shows that 
\[ \prod_{i = 4^k}^1 \, a_i = [ \partial \mbf{I}^2 ] ^{-1} . \]
Combining this with (\ref{eqn:180}) we deduce
\[ a_{4^k} \cdots a_2 \cdot a_1 \cdot 
(a^{\natural}_{4^k})^{-1} \cdots (a^{\natural}_2)^{-1} \cdot
(a^{\natural}_1)^{-1} = 1  \]
in $\bsym{\pi}_1(\opn{sk}_1 \opn{sd}^k \mbf{I}^2)$.
Now Lemma \ref{lem:13} says that 
\[  h_{4^k} \cdots h_2 \cdot h_1 \cdot 
(h^{\natural}_{4^k})^{-1} \cdots (h^{\natural}_2)^{-1} \cdot
(h^{\natural}_1)^{-1} = 1 \]
in $H$. Using (\ref{eqn:179}) we get
\[ \prod_{i = 4^k}^1 \, h_i =
\opn{MI} \bigl( \alpha, \beta \vert \opn{flip}(\sigma^0_1, \tau^0_1) \bigr) . \]
This is precisely (\ref{eqn:181}).
\end{proof}

\begin{dfn}
In this section, a kite $(\sigma, \tau)$ in $(\mbf{I}^2, v_0)$ will be called 
{\em $(\alpha, \beta)$-tiny} if it is a square kite, 
$\opn{len}(\sigma) \leq 5$,
and
\[ \opn{side}(\tau) < \min \bigl( \epsilon_4(\alpha, \beta), \,
\epsilon_5(\alpha, \Psi_{\h}) \bigr) \, . \]
See Lemma \ref{lem:10} and Proposition \ref{prop:21} regarding the constants 
$\epsilon_4(\alpha, \beta)$ and $\epsilon_5(\alpha, \Psi_{\h})$.
\end{dfn}

\begin{lem} \label{lem:14}
Let $(\sigma, \tau)$ be a kite in $(\mbf{I}^2, v_0)$. Then
\begin{equation} \label{eqn:184}
\opn{MI} \bigl( \alpha, \beta \vert \opn{flip}(\sigma, \tau) \bigr) =
\opn{MI} (\alpha, \beta \vert \sigma, \tau)^{-1} .
\end{equation}
\end{lem}

\begin{proof}
Step 1. Assume $(\sigma, \tau)$ is $(\alpha, \beta)$-tiny and 
$\beta|_{\tau(\mbf{I}^2)}$ is smooth. 
We may assume that $\epsilon := \opn{side}(\tau)$ is positive
(since the case $\epsilon = 0$ is trivial).

Since the flip reverses the orientation, it follows that
\[ \opn{RP}_0  \bigl( \alpha, \beta \vert \opn{flip}(\sigma, \tau) \bigr) =
\opn{RP}_0 ( \alpha, \beta \vert \sigma, \tau)^{-1} ; \]
cf.\ Definition \ref{dfn:11}.
Next, by applying Proposition \ref{prop:8} to both $(\sigma, \tau)$ and \lb
$\opn{flip}(\sigma, \tau)$, we have 
\begin{equation} \label{eqn:182}
\Norm{ \log_H \bigl( 
\opn{MI} \bigl( \alpha, \beta \vert \opn{flip}(\sigma, \tau) \bigr) \bigr)
- \log_H \bigl( 
\opn{MI} (\alpha, \beta \vert \sigma, \tau)^{-1} \bigr) }
\leq 2 \cdot c_2(\alpha, \beta) \cdot \epsilon^4 \, . 
\end{equation}

\medskip \noindent
Step 2. Again we assume that $(\sigma, \tau)$ is $(\alpha, \beta)$-tiny
and $\epsilon := \opn{side}(\tau)$ is positive; but we do not assume smoothness.

Take $k \geq 0$. For each index $i \in \{ 1, \ldots, 4^k \}$ we let
\[ (\sigma_i, \tau_i) := (\sigma, \tau) \circ (\sigma^k_i, \tau^k_i)
= \opn{tes}^k_i (\sigma, \tau) . \]
Note that 
$\opn{side}(\tau_i) = (\smfrac{1}{2})^k \cdot \epsilon$. 
The subsets $\opn{good}(\tau, k)$ and $\opn{bad}(\tau, k)$ of 
$\{ 1, \ldots, 4^k \}$ were defined in Definition \ref{dfn:45}.
According to Lemma \ref{lem:51}, there are constants 
\[ a_j := a_j(\alpha, \beta, \tau(\mbf{I}^2)) \]
such that 
\[ \abs{ \opn{bad}(\tau, k) } \leq a_0 + a_1 \cdot 2^k \, . \]

If $i$ is a good index, then by step 1 we know that 
\[ \begin{aligned}
& \Norm{ \log_H \bigl( 
\opn{MI} ( \alpha, \beta \vert \opn{flip}(\sigma_i, \tau_i) ) \bigr)
- \log_H \bigl( 
\opn{MI} ( \alpha, \beta \vert \sigma_i, \tau_i)^{-1} \bigr) } \\
& \qquad \leq 2 \cdot  c_2(\alpha, \beta) \cdot (\smfrac{1}{2})^{4 k} \cdot 
\epsilon^4 \, .
\end{aligned} \]
And if $i$ is a bad index, then by Proposition \ref{prop:8}(1) we have
\begin{equation} \label{eqn:192}
\begin{aligned}
& \Norm{ \log_H \bigl( 
\opn{MI} ( \alpha, \beta \vert \opn{flip}(\sigma_i, \tau_i) ) \bigr)
- \log_H \bigl( 
\opn{MI} ( \alpha, \beta \vert \sigma_i, \tau_i)^{-1} \bigr) } \\
& \qquad \leq 2 \cdot  c_1(\alpha, \beta) \cdot (\smfrac{1}{2})^{2 k} \cdot 
\epsilon^2 \, .
\end{aligned}
\end{equation}
The conditions in property (ii) of Theorem \ref{thm:6} are satisfied. Hence,
using Lemma \ref{lem:41} and Proposition \ref{prop:3}, we obtain these
estimates:
\[ \begin{aligned}
& \Norm{ \log_H \bigl( 
\opn{MI} ( \alpha, \beta \vert \opn{flip}(\sigma, \tau) ) \bigr) -
\log_H \bigl( 
\opn{MI} ( \alpha, \beta \vert \sigma, \tau)^{-1} \bigr) } \\
& \quad = \Norm{ \log_H \Bigl( \,
\prod_{i = 4^k}^1 \, 
\opn{MI} ( \alpha, \beta \vert \opn{flip}(\sigma_i, \tau_i) ) \Bigr) \\
& \qquad \qquad - \log_H \Bigl( \,
\prod_{i = 4^k}^1 \, 
\opn{MI} ( \alpha, \beta \vert \sigma_i, \tau_i)^{-1} \Bigr) } \\
& \quad \leq 
\abs{ \opn{good}(\tau, k) } \cdot 2 \cdot  c_2(\alpha, \beta) \cdot
(\smfrac{1}{2})^{4 k} \cdot  \epsilon^4 \\
& \qquad \qquad +
\abs{ \opn{bad}(\tau, k) } \cdot 2 \cdot  c_1(\alpha, \beta) \cdot 
(\smfrac{1}{2})^{2 k} \cdot  \epsilon^2 \\
& \quad \leq 
4^k \cdot 2 \cdot  c_2(\alpha, \beta) \cdot
(\smfrac{1}{2})^{4 k} \cdot  \epsilon^4 +
(a_0 + a_1 \cdot 2^k) \cdot 2 \cdot  c_1(\alpha, \beta) \cdot 
(\smfrac{1}{2})^{2 k} \cdot  \epsilon^2 \, . 
\end{aligned} \]
As $k \to \infty$ the last term goes to $0$. Hence (\ref{eqn:184}) holds in
this case.

\medskip \noindent
Step 3. Now $(\sigma, \tau)$ is an arbitrary kite in $(\mbf{I}^2, v_0)$. 
We may assume that $\tau(\mbf{I}^2)$ is $2$-dimensional, for otherwise things
are trivial. As done in the beginning of the proof of Lemma \ref{lem:41}, we can
modify the setup so that $(\sigma, \tau)$ is a square kite and 
$\opn{len}(\sigma) \leq 1$. Take $k$ large enough so that all the kites
\[ (\sigma_i, \tau_i) := (\sigma, \tau) \circ (\sigma^k_i, \tau^k_i) , \]
$i \in \{ 1, \ldots, 4^k \}$, are $(\alpha, \beta)$-tiny. 
By Step 2 we know that 
\[ \opn{MI} \bigl( \alpha, \beta \vert \opn{flip}(\sigma_i, \tau_i) \bigr) =
\opn{MI} (\alpha, \beta \vert \sigma_i, \tau_i)^{-1} \]
holds for all $i$. Hence, using Lemma \ref{lem:41} and Proposition
\ref{prop:3}, we conclude that
\[ \begin{aligned}
& \opn{MI} ( \alpha, \beta \vert \opn{flip}(\sigma, \tau) ) = 
\prod_{i = 4^k}^1 \, 
\opn{MI} ( \alpha, \beta \vert \opn{flip}(\sigma_i, \tau_i) ) \\
& \qquad = 
\prod_{i = 4^k}^1 \, 
\opn{MI} ( \alpha, \beta \vert \sigma_i, \tau_i)^{-1} 
 = \Bigl( \prod_{i = 1}^{4^k} \, 
\opn{MI} ( \alpha, \beta \vert \sigma_i, \tau_i) \Bigr)^{-1} \\
& \qquad  = \opn{MI} ( \alpha, \beta \vert \sigma, \tau)^{-1} . 
\end{aligned} \]
\end{proof}

\subsection{The Turn}
The (counterclockwise) turn of $\mbf{I}^2$ is the linear automorphism
$\opn{turn} : \mbf{I}^2 \to \mbf{I}^2$ defined on vertices by
\[ \opn{turn}(v_0, v_1, v_2) := (v_1, (1, 1), v_0) . \]
The turn of the basic kite $(\sigma^0_1, \tau^0_1)$ is
\[ \opn{turn}(\sigma^0_1, \tau^0_1) 
:= \bigl( (v_0, v_1), \opn{turn} \Bigr) . \]
Given any kite $(\sigma, \tau)$ in $(\mbf{I}^2, v_0)$, its turn is the kite
\[ \opn{turn}(\sigma, \tau) := 
(\sigma, \tau) \circ \opn{turn}(\sigma^0_1, \tau^0_1) . \]
See figure \ref{fig:36}.

\begin{figure}
\includegraphics[scale=0.36]{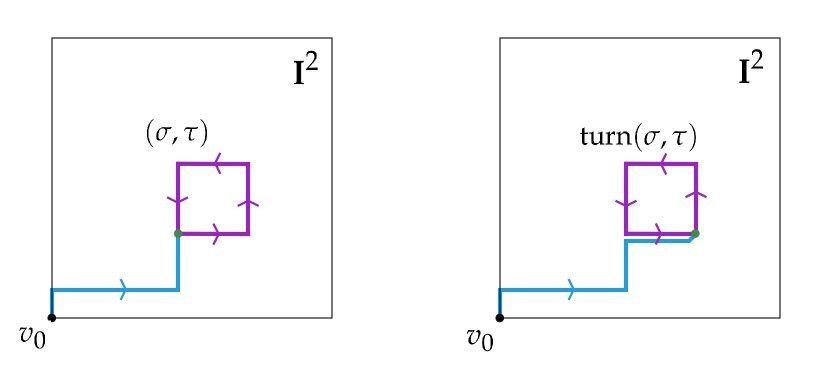}
\caption{A kite $(\sigma, \tau)$ in $(\mbf{I}^2, v_0)$, and the kite
$\opn{turn} (\sigma, \tau)$.} 
\label{fig:36}
\end{figure}

Observe that for kites $(\sigma_1, \tau_1)$ and $(\sigma_2, \tau_2)$ in 
$(\mbf{I}^2, v_0)$, one has
\begin{equation} \label{eqn:56}
\opn{turn} \bigl( (\sigma_1, \tau_1) \circ (\sigma_2, \tau_2) \bigr)
= (\sigma_1, \tau_1) \circ \opn{turn} (\sigma_2, \tau_2) . 
\end{equation}
If $(\sigma, \tau)$ is patterned on $\opn{sd}^k \mbf{I}^2$, then the boundaries
satisfy
\begin{equation} \label{eqn:185}
[ \partial \opn{turn}(\sigma, \tau) ] =  [ (\sigma, \tau) ]
\end{equation}
in $\bsym{\pi}_1(\opn{sk}_1 \opn{sd}^k \mbf{I}^2)$.

\begin{lem} \label{lem:42}
Let $(\sigma, \tau)$ be a kite in $(\mbf{I}^2, v_0)$, and let $k \geq 0$. Then 
\begin{equation*} \label{eqn:186}
\opn{MI} \bigl( \alpha, \beta \vert \opn{turn}(\sigma, \tau) \bigr) = 
\prod_{i = 1}^{4^k} \, 
\opn{MI} \bigl( \alpha, \beta \vert \opn{turn}( (\sigma, \tau) \circ 
(\sigma^k_i, \tau^k_i) ) \bigr) .
\end{equation*}
\end{lem}

\begin{proof}
The proof is very similar to that of Lemma \ref{lem:41}. As we showed there, it 
suffices to consider the case
$(\sigma, \tau) = (\sigma^0_1, \tau^0_1)$.
So we have to prove
\begin{equation} \label{eqn:187} 
\opn{MI} \bigl( \alpha, \beta \vert \opn{turn}(\sigma^0_1, \tau^0_1) \bigr) = 
\prod_{i = 1}^{4^k} \, 
\opn{MI} \bigl( \alpha, \beta \vert \opn{turn}(\sigma^k_i, \tau^k_i) \bigr) .
\end{equation}

For any $i \in \{ 1, \ldots, 4^k \}$ let
\[ (\sigma^{\natural}_i, \tau^{\natural}_i) :=
\opn{turn}(\sigma^0_1, \tau^0_1) \circ  (\sigma^k_i, \tau^k_i) . \]
The sequence
\[ \bsym{\rho}^{\natural} := \bigl( (\sigma^{\natural}_1, \tau^{\natural}_1),
\ldots, (\sigma^{\natural}_{4^k}, \tau^{\natural}_{4^k}) \bigr) \]
is a generating sequence of kites for $\opn{sd}^k \mbf{I}^2$.
Define elements $a^{\natural}_i$ and $h^{\natural}_i$ as 
in Definition \ref{dfn:33} and Lemma \ref{lem:12}, with respect to
the this new generating sequence $\bsym{\rho}^{\natural}$. 
By Proposition \ref{prop:3} we know that
\begin{equation} \label{eqn:188} 
\opn{MI} \bigl( \alpha, \beta \vert \opn{turn}(\sigma^0_1, \tau^0_1) \bigr) = 
\prod_{i = 1}^{4^k} \, h^{\natural}_i  \, .
\end{equation}
And by induction on $k$ it is not hard to see that 
\begin{equation} \label{eqn:189} 
\prod_{i = 1}^{4^k} \, a^{\natural}_i  = [ \partial \mbf{I}^2 ] . 
\end{equation}

Next consider the sequence of kites
\[ \bsym{\rho} := \bigl( (\sigma_1, \tau^{}_1),
\ldots, (\sigma^{}_{4^k}, \tau^{}_{4^k}) \bigr) , \]
where
\[ (\sigma_i, \tau_i) := \opn{turn}(\sigma^k_i, \tau^k_i) . \]
These kites are all aligned with $\bsym{\rho}^{\natural}$.
Take $a_i$ and $h_i$ defined as in Lemma \ref{lem:13}. We then have
\[ \prod_{i = 1}^{4^k} \, a_i = [ \partial \mbf{I}^2 ] . \]
Applying Lemma \ref{lem:13} we conclude that 
\[ \prod_{i = 1}^{4^k} \, h_i = \prod_{i = 1}^{4^k} \, h^{\natural}_i . \]
Combining this equation with (\ref{eqn:188})
we can deduce that equation (\ref{eqn:187}) is true.
\end{proof}

\begin{lem} \label{lem:16}
Let $(\sigma, \tau)$ be a kite in $(\mbf{I}^2, v_0)$. 
Then
\begin{equation} \label{eqn:190}
\opn{MI} \bigl( \alpha, \beta \vert \opn{turn}(\sigma, \tau) \bigr) =
\opn{MI} (\alpha, \beta \vert \sigma, \tau) .
\end{equation}
\end{lem}

\begin{proof}
The proof is organized like that of Lemma \ref{lem:14}, so we allow ourselves
to be less detailed.

\medskip \noindent
Step 1. Here we assume that $(\sigma, \tau)$ is an $(\alpha, \beta)$-tiny kite,
the number $\epsilon := \opn{side}(\tau)$ is positive, and 
$\beta|_{\tau(\mbf{I}^2)}$ is smooth. Going back to Definition \ref{dfn:11},
and using its notation mixed with the present notation, we have
\[ \opn{RP}_0(\alpha, \beta \vert \sigma, \tau) = 
\exp_H \bigl( \epsilon^2 \cdot \Psi_{\h}(g)(\til{\beta}(z)) \bigr) . \]
Recall that 
\[ g = \opn{MI} \bigl( \alpha \vert \sigma * (\tau \circ \sigma_{\mrm{pr}})
\bigr) . \]

The turn does not change the area of the square 
$Z = \tau(\mbf{I}^2)$, nor its midpoint $z$; it only changes the string that
leads from $v_0$ to $z$. Indeed, the string that is related to 
$\opn{turn}(\sigma, \tau)$ is $\sigma * (\tau \circ \rho')$, 
where $\rho'$ is the string
\[ \rho' := (v_0, v_1) * (v_1, y) * (y, w) , \]
in $\mbf{I}^2$ with  
$y := (\smfrac{1}{2}, 1)$ and $w := (\smfrac{1}{2}, \smfrac{1}{2})$.
The formula for the Riemann product is
\[ \opn{RP}_0(\alpha, \beta \vert \opn{turn}(\sigma, \tau)) = 
\exp_H \bigl( \epsilon^2 \cdot \Psi_{\h}(g')(\til{\beta}(z)) \bigr) , \]
with 
\[ g' := \opn{MI} \bigl( \alpha \vert \sigma * (\tau \circ \rho') \bigr) . \]

Let
$\rho'' := \sigma_{\mrm{pr}}^{-1} * \rho'$, 
which is a closed string based at $w$. The group element
\[ g'' := \opn{MI} ( \alpha \vert \tau \circ \rho'' ) \]
satisfies $g' = g \cdot g''$. Since  
$\opn{len}(\tau \circ \rho'') = 3 \epsilon$,
we know by Proposition \ref{prop:21} that 
\[ \Norm{ \Psi_{\h}(g'') - \bsym{1} } \leq 
c_5(\alpha, \Psi_{\h}) \cdot 3 \epsilon \, . \]
But by Proposition \ref{prop:13} we have
\[ \norm{ \Psi_{\h}(g) } \leq  \exp \bigl( c_4(\alpha, \Psi_{\h}) \cdot 5
\bigr) \, . \]
Therefore
\[ \begin{aligned}
& \Norm{ \log_H \bigl(  \opn{RP}_0(\alpha, \beta \vert \opn{turn}(\sigma, \tau))
\bigr) -
\log_H \bigl(  \opn{RP}_0(\alpha, \beta \vert \sigma, \tau) \bigr) } \\
& \qquad = \epsilon^2 \cdot \bigl( \Psi_{\h}(g) \circ 
( \Psi_{\h}(g'') - \bsym{1} ) \bigr) (\til{\beta}(z)) \\
& \qquad \leq
\epsilon^3 \cdot c \, , 
\end{aligned} \]
where we write
\[ c := 
\exp \bigl( c_4(\alpha, \Psi_{\h}) \cdot 5 \bigr) \cdot 
3 \cdot c_5(\alpha, \Psi_{\h}) \cdot \norm{\beta}_{\mrm{Sob}} \, . \]
Combining this estimate with Proposition \ref{prop:8}(2), we obtain
\begin{equation} \label{eqn:191}
\begin{aligned}
& \Norm{ \log_H \bigl(  \opn{MI}(\alpha, \beta \vert \opn{turn}(\sigma, \tau))
\bigr) -
\log_H \bigl(  \opn{MI}(\alpha, \beta \vert \sigma, \tau) \bigr) } \\
& \qquad \leq  2 \cdot c_2(\alpha, \beta) \cdot \epsilon^4 +
c \cdot \epsilon^3 \, .
\end{aligned} 
\end{equation}

\medskip \noindent
Step 2. In this step we assume that $(\sigma, \tau)$ is an
$(\alpha, \beta)$-tiny kite and $\epsilon := \opn{side}(\tau)$ is
positive; but no smoothness is assumed. 

Take $k \geq 0$, and define kites $(\sigma_i, \tau_i)$ and  sets
$\opn{good}(\tau, k)$
and $\opn{bad}(\tau, k)$ like in step 2 of the proof of Lemma \ref{lem:14}.
If $i \in \opn{good}(\tau, k)$ then by equation (\ref{eqn:191}) in step 1 we
know
that 
\[ \begin{aligned}
& \Norm{ \log_H \bigl(  \opn{MI}(\alpha, \beta \vert 
\opn{turn}(\sigma_i, \tau_i)) \bigr) -
\log_H \bigl(  \opn{MI}(\alpha, \beta \vert \sigma_i, \tau_i) \bigr) } \\
& \qquad \leq  2 \cdot c_2(\alpha, \beta) \cdot (\smfrac{1}{2})^{4 k} \cdot
\epsilon^4 + 
c \cdot (\smfrac{1}{2})^{3 k} \cdot \epsilon^3 \, .
\end{aligned} \]
For $i \in \opn{bad}(\tau, k)$ we use the estimate (\ref{eqn:192}). 
As in the proof of Lemma \ref{lem:14}, but using Lemma \ref{lem:42} instead of
Lemma \ref{lem:41}, we arrive at the estimate
\[ \begin{aligned}
& \Norm{ \log_H \bigl( 
\opn{MI} ( \alpha, \beta \vert \opn{turn}(\sigma, \tau) ) \bigr) -
\log_H \bigl( 
\opn{MI} ( \alpha, \beta \vert \sigma, \tau) \bigr) } \\
& \qquad = \Norm{ \log_H \Bigl( \,
\prod_{i = 1}^{4^k} \, 
\opn{MI} ( \alpha, \beta \vert \opn{turn}(\sigma_i, \tau_i) ) \Bigr) -
\log_H \Bigl( \,
\prod_{i = 1}^{4^k} \, 
\opn{MI} ( \alpha, \beta \vert \sigma_i, \tau_i) \Bigr) } \\
& \qquad \leq 
\abs{ \opn{good}(\tau, k) } \cdot 
\bigl( 2 \cdot c_2(\alpha, \beta) \cdot (\smfrac{1}{2})^{4 k} \cdot
\epsilon^4 + c \cdot (\smfrac{1}{2})^{3 k} \cdot \epsilon^3 \bigr) \\
& \qquad \quad +
\abs{ \opn{bad}(\tau, k) } \cdot \bigl( 2 \cdot  c_1(\alpha, \beta) \cdot
(\smfrac{1}{2})^{2 k} \cdot \epsilon^2 \bigr) \\
& \qquad \leq 
4^k \cdot \bigl( 2 \cdot c_2(\alpha, \beta) \cdot (\smfrac{1}{2})^{4 k} \cdot
\epsilon^4 + c \cdot (\smfrac{1}{2})^{3 k} \cdot \epsilon^3 \bigr) \\
& \qquad \quad + 
(a_0 + a_1 \cdot 2^k) \cdot 
\bigl( 2 \cdot  c_1(\alpha, \beta) \cdot
(\smfrac{1}{2})^{2 k} \cdot \epsilon^2 \bigr) \, . 
\end{aligned} \]
As $k \to \infty$ the last term goes to $0$. Hence (\ref{eqn:190}) holds in
this case.

\medskip \noindent
Step 3. For an arbitrary kite $(\sigma, \tau)$ in $(\mbf{I}^2, v_0)$ we prove
that (\ref{eqn:190}) holds using step 2, as was done in step 3 of the proof of  
Lemma \ref{lem:14}, but using Lemma \ref{lem:42} instead of Lemma \ref{lem:41}.
\end{proof}

\subsection{Putting it all Together}

\begin{thm} \label{thm:7}
Let 
\[ \mbf{C} / \mbf{I}^2  = (G, H, \Psi, \Phi_0, \Phi_{X}) \]
be a Lie quasi crossed module over $(X, x_0) := (\mbf{I}^2, v_0)$,  let
$(\alpha, \beta)$ be a connection-curvature pair  in $\mbf{C} / \mbf{I}^2$,
and let 
\[ \bsym{\rho} = \bigl( (\sigma_1, \tau_1), \, \ldots, \, (\sigma_m, \tau_m) 
\bigr) \]
be a sequence of kites in $(\mbf{I}^2, v_0)$ patterned on 
$\opn{sd}^k \mbf{I}^2$, for some $k \in \N$.
For $i \in \{ 1, \ldots, m \}$ define
\[ a_i := [\partial (\sigma_i, \tau_i)] \in 
\bsym{\pi}_1(\opn{sk}_1 \opn{sd}^k \mbf{I}^2) \]
and
\[ h_i := \opn{MI} ( \alpha, \beta \vert \sigma_i, \tau_i) \in H . \]

Suppose 
\[  w(\bsym{s}) \in \opn{Wrd}^{\pm 1}(\bsym{s}) = \opn{Wrd}^{\pm 1}(s_1, \ldots,
s_m) \]
is a word such that 
\[ w(a_1, \ldots, a_m) =  1 \]
in $\bsym{\pi}_1(\opn{sk}_1 \opn{sd}^k \mbf{I}^2)$.
Then 
\[ w(h_1, \ldots, h_m) = 1 \]
in $H$.
\end{thm}

\begin{proof}
We can find $l_i \in \{ 0, 1 \}$ and 
$j_i \in \{ 0, 1, 2, 3 \}$, such that the kites
\[ (\sigma'_i, \tau'_i) := \opn{flip}^{l_i} \bigl( 
\opn{turn}^{j_i} (\sigma_i, \tau_i) \bigr) \]
are aligned with the generating sequence 
$\opn{tes}^k \mbf{I}^2$.  Here the exponents
$l_i$ and $j_i$ refer to iteration of the corresponding operation. Let
$e_i := (-1)^{l_i}$, 
\[ a'_i := [\partial (\sigma'_i, \tau'_i)] \in 
\bsym{\pi}_1(\opn{sk}_1 \opn{sd}^k \mbf{I}^2) \]
and 
\[ h'_i := \opn{MI} (\alpha, \beta \vert \sigma'_i, \tau'_i) \in H . \]
Then 
$a_i^{e_i} = a'_i$, 
and by Lemmas \ref{lem:14} and \ref{lem:16} we also have
$h_i^{e_i} = h'_i$.

Consider the word
\[ u(\bsym{s}) := w \bigl( s_1^{e_1}, \ldots, s_m^{e_m} \bigr)
\in \opn{Wrd}^{\pm 1}(\bsym{s}) . \]
Then
\[ w(a_1, \ldots, a_m) = u(a'_1, \ldots, a'_m) \]
and
\[ w(h_1, \ldots, h_m) =u(h'_1, \ldots, h'_m) . \]
Now we can finish the proof with the use of Lemma \ref{lem:13}.
\end{proof}

\begin{cor} \label{cor:2}
Let 
\[ \mbf{C} / \mbf{I}^2  = (G, H, \Psi, \Phi_0, \Phi_{X}) \]
be a Lie quasi crossed module over $(X, x_0) := (\mbf{I}^2, v_0)$,  let
$(\alpha, \beta)$ be a connection-curvature pair  in $\mbf{C} / \mbf{I}^2$,
and let 
\[ \bsym{\rho} = \bigl( (\sigma_1, \tau_1), \, \ldots, \, 
(\sigma_{4^k}, \tau_{4^k}) \bigr) \]
be a tessellation of $\mbf{I}^2$ patterned on 
$\opn{sd}^k \mbf{I}^2$ \tup{(}cf.\ Definition \tup{\ref{dfn:4}}\tup{)}.
Then 
\[ \prod_{i = 1}^{4^k} \, 
\opn{MI} ( \alpha, \beta \vert \sigma_i, \tau_i)  = 
\opn{MI} (\alpha, \beta \vert \mbf{I}^2) . \]
\end{cor}

\begin{proof}
This is an immediate consequence of the theorem, taking
$m := 4^k + 1$,  
\[ w(\bsym{s}) := \Bigl( \prod\nolimits_{i = 1}^{m - 1} s_i \Bigr) 
\cdot s_m^{-1}  \]
and $(\sigma_m, \tau_m)$ the basic kite. 
\end{proof}

\cleardoublepage
\section{Stokes Theorem in Dimension Three}
\label{sec:stokes-3}

The goal of this section is to prove Theorem \ref{thm:10}, which is the first
version of the main result of the paper. (The second version, dealing with the
triangular case, is Theorem \ref{thm:22}.)

\subsection{Balloons and their Boundaries} \label{subsec:ballbound}

\begin{dfn}
Let $(X, x_0)$ be a pointed polyhedron. A {\em linear quadrangular balloon} 
\index{Quadrangular balloon}
in $(X, x_0)$ is a pair $(\sigma, \tau)$, where
$\sigma$ is a string in $X$ (see Definition \ref{dfn:3}), and
$\tau : \mbf{I}^3 \to X$ is a linear map. The conditions are that
$\sigma(v_0) = x_0$ and $\sigma(v_1) = \tau(v_0)$.
\end{dfn}

In other words, a linear quadrangular balloon is the $3$-dimensional version of
a linear quadrangular kite. See Figure \ref{fig:51} for an illustration. 

All balloons in this section are linear quadrangular ones; so we shall
simply call them balloons. (This will change in Section \ref{sec:simpl}.)
If the image of $\tau$ is a cube in $X$, then we call 
$(\sigma, \tau)$ a {\em cubical balloon}. The length of the side of $\tau$ is
denoted by $\opn{side}(\tau)$. If $\opn{side}(\tau) > 0$ then 
$(\sigma, \tau)$ is said to be {\em nondegenerate}.

If $(\sigma, \tau)$ is a balloon in $(X, x_0)$, and
$(\sigma', \tau')$ is a balloon (resp.\ a kite) in $(\mbf{I}^{3}, v_0)$, then 
the composition $(\sigma, \tau) \circ (\sigma', \tau')$
(defined like (\ref{eqn:18})) is a balloon  (resp.\ a kite) in $(X, x_0)$.

The $k$-th binary subdivision of $\mbf{I}^{3}$ is its cellular decomposition 
into $8^k$ little cubes, each of side $(\smfrac{1}{2})^k$. We denote this
decomposition by $\opn{sd}^k \mbf{I}^3$.

A map $\sigma : \mbf{I}^p \to \mbf{I}^3$
is said to be {\em patterned on $\opn{sd}^k \mbf{I}^3$} if it is linear, and
its image is a $p$-cell in $\opn{sd}^k \mbf{I}^3$.
A string $\sigma = (\sigma_1, \ldots, \sigma_n)$ in $(\mbf{I}^{3}, v_0)$ is
said to be patterned on $\opn{sd}^k \mbf{I}^3$ if for every $i$ the map
$\sigma_i : \mbf{I}^1 \to \mbf{I}^3$ is patterned on $\opn{sd}^k \mbf{I}^3$.
A kite or balloon $(\sigma, \tau)$ in $(\mbf{I}^{3}, v_0)$ is said to be
patterned on $\opn{sd}^k \mbf{I}^3$ if the string $\sigma$
and the map $\tau : \mbf{I}^p \to \mbf{I}^3$ ($p = 2, 3$) are patterned on 
$\opn{sd}^k \mbf{I}^3$.

\begin{figure}
\includegraphics[scale=0.27]{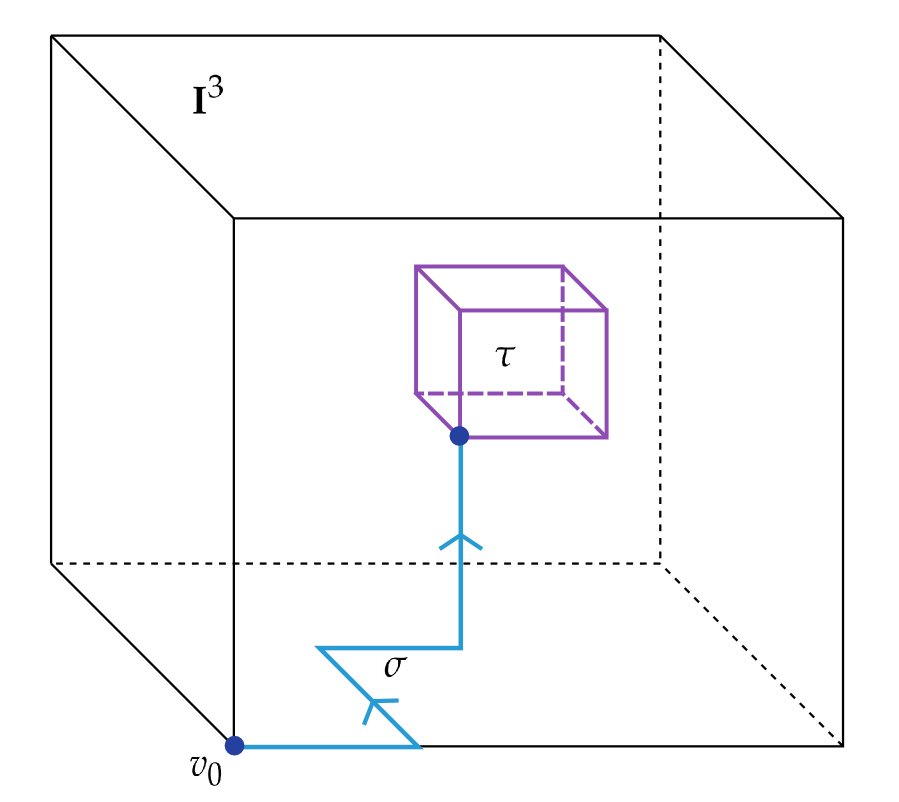}
\caption{A cubical balloon $(\sigma, \tau)$ in the pointed polyhedron
$(\mbf{I}^3, v_0)$.} 
\label{fig:51}
\end{figure}

The $k$-th binary tessellation of $\mbf{I}^{3}$ is the sequence 
\begin{equation} \label{eqn:141}
\opn{tes}^k \mbf{I}^3 = 
\bigl( \opn{tes}^k_1 \mbf{I}^3, \ldots, \opn{tes}^k_{8^k} \mbf{I}^3 \bigr)
= \bigl( (\sigma^{k}_1, \tau^{k}_1) , \ldots, 
(\sigma^{k}_{8^k}, \tau^{k}_{8^k}) \bigr)
\end{equation}
of balloons patterned on $\opn{sd}^k \mbf{I}^3$, defined as follows.
For $k = 0$ we have the basic balloon
$(\sigma^{0}_1, \tau^{0}_1)$,
where $\sigma^{0}_1$ is the empty string, and 
$\tau^{0}_1 : \mbf{I}^3 \to \mbf{I}^3$
is the identity map. 

For $k = 1$ we choose, once and for all, a sequence 
\[ \opn{tes}^1 \mbf{I}^3 = 
\bigl( (\sigma^{1}_1, \tau^{1}_1) , \ldots, 
(\sigma^{1}_{8}, \tau^{1}_{8}) \bigr) \]
balloons patterned on $\opn{sd}^1 \mbf{I}^3$, satisfying these conditions:
\begin{enumerate}
\rmitem{a}  Each of the
$3$-cells of $\opn{sd}^1 \mbf{I}^3$ occurs exactly once as 
$\tau^{1}_{i}(\mbf{I}^3)$ for some $i$.
\rmitem{b} The maps $\tau^{1}_{i}$ are positively oriented.
\rmitem{c} The length of each string $\sigma^{1}_i$ is at most $3$. 
\end{enumerate}
This can be done of course.

For $k \geq 1$ we use the recursive definition
\[ \opn{tes}^{k+1} \mbf{I}^3 :=  
(\opn{tes}^1 \mbf{I}^3) \circ (\opn{tes}^{k} \mbf{I}^3) . \]
Here we use the convention (\ref{eqn:19}) for composition of sequences.

Given a balloon $(\sigma, \tau)$ in a pointed polyhedron $(X, x_0)$, 
an numbers $k \in \N$, $i \in \{ 1, \ldots, 8^k \}$, let
\begin{equation} \label{eqn:250}
\opn{tes}^{k}_i (\sigma, \tau) := 
(\sigma, \tau) \circ \opn{tes}^{k}_i \mbf{I}^3 = 
(\sigma, \tau) \circ (\sigma^{k}_i, \tau^{k}_i) . 
\end{equation}
The $k$-th binary tessellation of $(\sigma, \tau)$ is the sequence (of length
$8^k$) of balloons 
\begin{equation} \label{eqn:251}
\opn{tes}^{k} (\sigma, \tau) := (\sigma, \tau) \circ (\opn{tes}^{k} \mbf{I}^3)
= \bigl( \opn{tes}^{k}_1 (\sigma, \tau), \ldots, 
\opn{tes}^{k}_{8^k} (\sigma, \tau) \bigr) . 
\end{equation}

\newpage
\begin{dfn}   \label{dfn:36} \mbox{} 
\begin{enumerate}
\item For $i \in \{ 1, \ldots, 6 \}$ let
$\partial_i \mbf{I}^{3} := (\sigma^{\flat}_i, \tau^{\flat}_i)$ 
be the kites depicted in Figure \ref{fig:46}.
The {\em boundary of $\mbf{I}^{3}$} is the sequence of kites
\[ \partial \mbf{I}^3 = 
\bigl( \partial_1 \mbf{I}^{3}, \ldots, \partial_6 \mbf{I}^{3} \bigr) = 
\bigl( (\sigma^{\flat}_1, \tau^{\flat}_1) , \ldots, 
(\sigma^{\flat}_{6}, \tau^{\flat}_{6}) \bigr) . \]

\item Given a balloon $(\sigma, \tau)$ in a pointed polyhedron $(X, x_0)$, 
let
\[ \partial_i (\sigma, \tau) := (\sigma, \tau) \circ \partial_i \mbf{I}^{3}
= (\sigma, \tau) \circ (\sigma^{\flat}_i, \tau^{\flat}_i) . \]
The {\em boundary} of $(\sigma, \tau)$ is the sequence of kites
\[ \partial (\sigma, \tau) := 
\bigl( \partial_1 (\sigma, \tau), \ldots, 
\partial_6 (\sigma, \tau) \bigr)  \]
in  $(X, x_0)$.
\end{enumerate}
\end{dfn}

\begin{figure}
\includegraphics[scale=0.50]{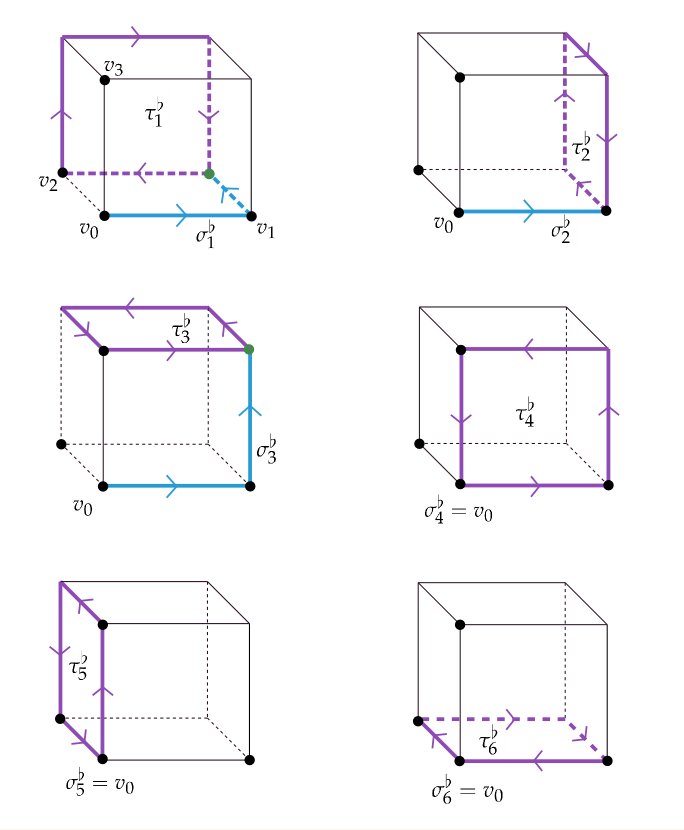}
\caption{The kites $(\sigma^{\flat}_i, \tau^{\flat}_i)$,
$i = 1, \ldots, 6$, that make up the boundary of $\mbf{I}^3$.} 
\label{fig:46}
\end{figure}

If $(\sigma, \tau)$ is a balloon in $(\mbf{I}^3, v_0)$
patterned on $\opn{sd}^k \mbf{I}^3$, then 
$\partial (\sigma, \tau)$ is a sequence (of length $6$) of kites  patterned on 
$\opn{sd}^k \mbf{I}^3$, and 
$\partial \partial (\sigma, \tau)$ is a sequence of closed strings patterned on 
$\opn{sd}^k \mbf{I}^3$.

\begin{dfn} \label{dfn:35}
Let $(X, x_0)$ be a pointed polyhedron, let
$\mbf{C} / X$ be a Lie quasi crossed module with additive feedback
over $(X, x_0)$, and
let $(\alpha, \beta)$ be a piecewise smooth connection-curvature pair in 
$\mbf{C} / X$. Given a balloon $(\sigma, \tau)$ in $(X, x_0)$, we define
\[ \opn{MI}(\alpha, \beta \vert \partial(\sigma, \tau)) :=
\prod_{1 = 1}^6 \, \opn{MI}(\alpha, \beta \vert \partial_i(\sigma, \tau))
\in H . \]
\end{dfn}

\subsection{Some Algebraic Topology}
Let us denote by $\bsym{\pi}_1(\opn{sk}_1 \opn{sd}^k \mbf{I}^3)$
the fundamental group of the topological space 
$\opn{sk}_1 \opn{sd}^k \mbf{I}^3$, based at $v_0$. 
If $\sigma$ is a closed string patterned on $\opn{sd}^k \mbf{I}^3$ and based at
$v_0$, then we denote its class in 
$\bsym{\pi}_1(\opn{sk}_1 \opn{sd}^k \mbf{I}^3)$
by $[\sigma]$.

We now look at homology groups. Fix $k \geq 0$. On each $p$-cell $\tau$ in 
$\opn{sd}^k \mbf{I}^3$, $p \in \{0, 1, 2\}$, let us choose an orientation
(there are two options for $p > 0$).  
We denote by $\mrm{C}_p(\opn{sk}_2 \opn{sd}^k \mbf{I}^3)$
the free abelian group based on these oriented cells, called the group of
$p$-chains. 
The direct sum on all $p$ is a complex (with the usual boundary operator), and
the $p$-th homology is the singular homology
$\mrm{H}_p(\opn{sk}_2 \opn{sd}^k \mbf{I}^3)$ of the topological space
$\opn{sk}_2 \opn{sd}^k \mbf{I}^3$. 
The homology class of a cycle $\sum_i n_i \tau_i$ is denoted by
$[\sum_i n_i \tau_i]$.
Observe that a string $\sigma = (\sigma_1, \ldots, \sigma_n)$ patterned on
$\opn{sd}^k \mbf{I}^3$ represents an element
$\sum_i \sigma_i \in \mrm{C}_1(\opn{sk}_2 \opn{sd}^k \mbf{I}^3)$.
And a linear map $\tau : \mbf{I}^2 \to \mbf{I}^3$ patterned on
$\opn{sd}^k \mbf{I}^3$ represents an element
$\tau \in \mrm{C}_2(\opn{sk}_2 \opn{sd}^k \mbf{I}^3)$.

Recall the free monoid with involution $\opn{Wrd}^{\pm 1}(\bsym{s})$ from
Subsection
\ref{subec:freemon}.

\begin{lem} \label{lem:20}
Let 
$\bigl( (\sigma_1, \tau_1), \ldots, (\sigma_n, \tau_n) \bigr)$
be a sequence of kites patterned on $\opn{sk}_2 \opn{sd}^k \mbf{I}^3$, for some
$k \geq 0$. Write
\[ a_i := [\partial (\sigma_i, \tau_i)] \in 
\bsym{\pi}_1(\opn{sk}_1 \opn{sd}^k \mbf{I}^3) . \]
Let $w(\bsym{s}) \in \opn{Wrd}^{\pm 1}(s_1, \ldots, s_n)$ be a word such that
$w(a_1, \ldots, a_n) = 1$. 
Then the  $2$-chain 
\[ w(\tau_1, \ldots, \tau_n) \in \mrm{C}_2(\opn{sk}_2 \opn{sd}^k \mbf{I}^3) \]
is a cycle.
\end{lem}

\begin{proof}
Say $m$ is the number of $1$-cells in $\opn{sd}^k \mbf{I}^3$. 
We choose orientations, and arrange these oriented $1$-cells in a sequence
$(\rho_1, \ldots, \rho_m)$.
For every $i$ let $u_i(\bsym{t}) \in$ 
$\opn{Wrd}^{\pm 1}(t_1, \ldots, t_m)$
be the word such that
\[ \partial (\sigma_i, \tau_i) = u_i(\rho_1, \ldots, \rho_m) \]
as strings patterned on $\opn{sd}^k \mbf{I}^3$.
Define
\[ u(\bsym{t}) := w( u_1(\bsym{t}), \ldots, u_n(\bsym{t})) \in \opn{Wrd}^{\pm 1}
(\bsym{t}) . \]
Then $u(\rho_1, \ldots, \rho_m)$ is a closed string, and its homotopy class is
trivial in the group $\bsym{\pi}_1(\opn{sk}_1 \opn{sd}^k \mbf{I}^3)$.
According to Lemma \ref{lem:15} we have
$u(\bsym{t})  \sim_{\mrm{can}} 1$ in 
$\opn{Wrd}^{\pm 1}(\bsym{t})$.

On the other hand we have 
$\partial(\tau_i) = \partial (\sigma_i, \tau_i)$
as additive $1$-chains, i.e.\ as elements of
$\mrm{C}_1(\opn{sk}_2 \opn{sd}^k \mbf{I}^3)$.
Hence
\[ \partial(w(\tau_1, \ldots, \tau_n)) = 
w(\partial (\sigma_1, \tau_1), \ldots, \partial (\sigma_n, \tau_n))
= u(\rho_1, \ldots, \rho_m) = 0 \]
as elements of $\mrm{C}_2(\opn{sk}_2 \opn{sd}^k \mbf{I}^3)$.
\end{proof}

\begin{lem} \label{lem:21}
Recall the boundary of $\mbf{I}^3$ from Definition \tup{\ref{dfn:36}}.
For any $i$ let
\[ a^{\flat}_i := [\partial (\sigma^{\flat}_{i}, \tau^{\flat}_{i})] \in 
\bsym{\pi}_1(\opn{sk}_1 \opn{sd}^0 \mbf{I}^3) \]
and
\[ b^{\flat}_i := \tau^{\flat}_i \in 
\mrm{C}_2(\opn{sk}_1 \opn{sd}^0 \mbf{I}^3) . \]
Then:
\begin{enumerate}
\item The fundamental group 
$\bsym{\pi}_1(\opn{sk}_1 \opn{sd}^0 \mbf{I}^3)$ is generated by
the sequence of elements
$(a^{\flat}_1, \ldots, a^{\flat}_6)$, 
and there is one relation:
\[  \prod_{i = 1, \ldots, 6} \, a^{\flat}_i = 1 . \]
Thus $\bsym{\pi}_1(\opn{sk}_1 \opn{sd}^0 \mbf{I}^3)$ is a free group
of rank $5$.  
\item The homology group $\mrm{H}_2(\opn{sk}_2 \opn{sd}^0 \mbf{I}^3)$
is a free abelian group of rank $1$, with basis 
$[\sum_{i = 1, \ldots, 6} \, \tau_{i}]$.
\end{enumerate}
\end{lem}

\begin{proof}
This is obvious from looking at the pictures.
\end{proof}

Given a balloon $(\sigma, \tau)$, its boundary 
$\partial(\sigma, \tau)$ is a sequence of $6$ kites, and
$\partial \partial(\sigma, \tau)$ is a sequence of $6$ closed strings.

\begin{lem} \label{lem:22} \mbox{}
\begin{enumerate}
\item The fundamental group 
$\bsym{\pi}_1(\opn{sk}_1 \opn{sd}^1 \mbf{I}^3)$ is generated by
the \linebreak sequence of closed strings 
\[ \partial \partial (\opn{tes}^1 \mbf{I}^3)
= \bigl( \partial \partial (\sigma^{k}_1, \tau^{k}_1) , \ldots, 
\partial \partial (\sigma^{k}_{8^k}, \tau^{k}_{8^k}) \bigr) .
\]
This sequence has length $48$. 
There are $20$ relations, and they are of two kinds:
\begin{enumerate}
\item For any two distinct kites 
$(\sigma_{i}, \tau_{i}), (\sigma_{j}, \tau_{j})$
in $\partial \partial (\opn{tes}^1 \mbf{I}^3)$
such that $\tau_i(\mbf{I}^2) = \tau_j(\mbf{I}^2)$, namely for any of
the $12$ interior faces of $\opn{sd}^1 \mbf{I}^3$, there is a relation
\[ [\partial (\sigma_{i}, \tau_{i})] = [g_i] \cdot
[\partial (\sigma_{j}, \tau_{j})^{-1}] \cdot [g_i^{-1}] \]
for some word $g_i$ in the $48$ generators.
\item For any of the $8$ balloons 
$(\sigma^{k}_i, \tau^{k}_i)$
in $\opn{tes}^1 \mbf{I}^3$ there is a
relation as in Lemma \tup{\ref{lem:21}(1)}.
\end{enumerate}
Thus $\bsym{\pi}_1(\opn{sk}_1 \opn{sd}^1 \mbf{I}^3)$ is a free group
of rank $28$.  
\item The homology group $\mrm{H}_2(\opn{sk}_2 \opn{sd}^1 \mbf{I}^3)$
is a free abelian group of rank $8$, with basis the $8$ boundaries of
the $8$ balloons in $\opn{tes}^1 \mbf{I}^3$.
\end{enumerate}
\end{lem}

\begin{proof}
Use Lemma \ref{lem:21} together with the Van-Campen and Mayer- \lb Vietoris
theorems.
\end{proof}

\subsection{Inert Forms} \label{subsec:inert}
Let $(X, x_0)$ be a pointed polyhedron, and let
\[ \mbf{C} / X = (G, H, \Psi, \Phi_0, \Phi_X) \]
be a Lie quasi crossed module with additive feedback over $(X, x_0)$.
Recall that the additive feedback
$\Phi_X$ is an element of 
$\mcal{O}_{\mrm{pws}}(X) \otimes \opn{Hom}(\h, \g)$, so for any point
$x \in X$ we have a linear function
$\Phi_X(x) : \h \to \g$. 

Suppose $Z \subset X$ is some sub-polyhedron.
As in equation (\ref{eqn:140}), but restricting to $Z$, for every $p$ we have an
$\mcal{O}_{\mrm{pws}}(Z)$-linear homomorphism
\[ \Phi_X|_Z : \Omega^p_{\mrm{pws}}(Z) \otimes \h \to 
\Omega^p_{\mrm{pws}}(Z) \otimes \g . \]
Note that for $p = 0$, an element 
\[ f \in \Omega^0_{\mrm{pws}}(Z) \otimes \h   = 
\mcal{O}_{\mrm{pws}}(Z) \otimes \h \]
is a piecewise smooth function $f : Z \to \h$; and in this case
\[ \Phi_X|_Z(f)(z) = \Phi_X(z)(f(z)) \in \g \]
for any point $z \in Z$. 

\begin{dfn} Let $Z \subset X$ be a subpolyhedron.
\begin{enumerate}
\item A function $f \in \mcal{O}_{\mrm{pws}}(Z) \otimes \h$ is called an {\em
inert function} (relative to $\mbf{C} / X$) if $\Phi_X|_Z(f) = 0$.

\item A form $\gamma \in  \Omega^p_{\mrm{pws}}(Z) \otimes \h$ is called an 
{\em inert $p$-form} \index{Inert form}
(relative to $\mbf{C} / X$) if $\Phi_X|_Z(\gamma) = 0$.
\end{enumerate}
\end{dfn}

Let $\gamma \in \Omega^p_{\mrm{pws}}(X) \otimes \h$.
Choose some linear coordinate system 
$\bsym{s} = (s_1, \ldots, s_n)$ on $X$. Also choose a smoothing triangulation
$\{ X_j \}_{j \in J}$ for $\gamma$; so that 
$\gamma|_{X_j} \in \Omega^p(X_j) \otimes \h$ for every $j$. See Subsection 
\ref{subsec:pws} for details. For each index $j$ let 
$f_{j, \bsym{i}} \in \mcal{O}(X_j) \otimes \h$
be the coefficients of $\gamma|_{X_j}$ relative to $\bsym{s}$, in the sense of
Definition \ref{dfn:27}. Namely
\begin{equation} \label{eqn:222}
\gamma|_{X_j} = \sum_{\bsym{i}} \, f_{j, \bsym{i}} \cdot 
\d s_{i_1} \wedge \cdots \wedge \d s_{i_p} \in 
\Omega^p(X_j) \otimes \h ,
\end{equation}
where $\bsym{i} = (i_1, \ldots, i_p)$ runs over the set of strictly increasing
multi-indices in \lb $\{ 1, \ldots, n \}^p$.

\begin{lem} \label{lem:47}
In the situation above, the following are equivalent:
\begin{enumerate}
\rmitem{i} The form $\gamma \in \Omega^p_{\mrm{pws}}(X) \otimes \h$ is inert.
\rmitem{iii} The functions  
$f_{j, \bsym{i}} \in \mcal{O}(X_j) \otimes \h$ are all inert.
\end{enumerate}
\end{lem}

We omit the easy proof. 

\begin{dfn}
The closed subgroup
\[ H_0 := \opn{Ker}(\Phi_0) \subset H \]
is called the {\em inertia subgroup} (relative to $\mbf{C} / X$).
Its Lie algebra
\[ \h_0 := \opn{Lie}(H_0) = \opn{Ker}(\opn{Lie}(\Phi_0)) \]
is called the {\em inertia subalgebra}.
\end{dfn}

\begin{prop}
The subgroup $H_0$ is central in $H$, and the subalgebra $\h_0$ is central
in $\h$.
\end{prop}

\begin{proof}
Take $h \in H_0$, so $\Phi_0(h) = 1$. By the Pfeiffer condition (i.e.\
condition ($*$) of Definition \ref{dfn:22}) we have
\[ \opn{Ad}_H(h) = \Psi(\Phi_0(h)) = \Psi(1) = \opn{id}_H , \]
which says that 
\[ h \in \opn{Ker}(\opn{Ad}_H) = \opn{Z}(H) . \]

Since $H_0 \subset  \opn{Z}(H)$ it follows that 
\[ \h_0 \subset \opn{Lie}(\opn{Z}(H)) \subset \opn{Z}(\h) . \]
(Note that $H$ could be disconnected, in which case 
$\opn{Z}(\h)$ could be bigger than
$\opn{Lie}(\opn{Z}(H))$.)
\end{proof}

Recall the notion of tame connection (Definition \ref{dfn:46}).

\begin{lem} \label{lem:45}
Let $\alpha$ be a tame connection for $\mbf{C} / X$, let $\sigma$ be a closed
string in $X$ based at $x_0$, and let
$g := \opn{MI}(\alpha \vert \sigma) \in G$. Then for any
$\lambda \in \h_0$ one has
\[ \Psi_{\h}(g)(\lambda) = \lambda . \]
\end{lem}

\begin{proof} Let $m$ be the number of pieces in $\sigma$, and choose $k$ large
enough such that $m \leq 2^{k+2}$. We may append to $\sigma$ a few
copies of the constant map $x_0$ at its end, so that now $\sigma$ has exactly 
$2^{k+2}$ pieces. The group element $g$ is unchanged. 

Denote by $\sigma'$ the closed string in
$\mbf{I}^2$ which is ``$k$-th subdivision'' of the boundary 
$\partial \mbf{I}^2$, based at $v_0$; so $\sigma'$ has 
$2^{k+2}= 4 \cdot 2^k$ pieces. We can construct a piecewise linear map
$f : \mbf{I}^2 \to X$ such that $f(v_0) = x_0$, and 
$f \circ \sigma' = \rho$ as strings. According to 
Propositions \ref{prop:12} and \ref{prop:5}(3) we have
\[ g = \opn{MI}(f^*(\alpha) \vert \sigma')
= \opn{MI}(f^*(\alpha) \vert \partial \mbf{I}^2)  . \]

Let $\beta \in \Omega^2_{\mrm{pws}}(X) \otimes \h$ be such that 
$(\alpha, \beta)$ is a connection-curvature pair for $\mbf{C} / X$.
By Proposition \ref{prop:11} the pair $(f^*(\alpha), f^*(\beta))$ is a
connection-curvature pair in $f^*(\mbf{C} / X)$.
Hence by Theorem \ref{thm:1} we have
$g = \Phi_0(h)$, where
\[ h := \opn{MI}(f^*(\alpha), f^*(\beta) \vert \mbf{I}^2) \in H . \]
Since $\lambda \in \opn{Lie}(\opn{Z}(H))$ it follows that
\[ \Psi_{\h}(g)(\lambda) = \Psi_{\h}(\Phi_0(h))(\lambda) = 
\opn{Ad}_{\h}(h)(\lambda) = \lambda . \]
\end{proof}

\begin{rem}
The lemma above shows that the holonomy group of a tame connection
$\alpha$ at $x_0$ acts trivially on the inertia subalgebra $\h_0$. Hence the
names.
\end{rem}

\begin{lem} \label{lem:46}
Let $\mbf{C} / X$ be a Lie quasi crossed module with additive feedback over
$(X, x_0)$, let $\alpha$ be a tame connection for $\mbf{C} / X$, and
let $Z \subset X$ be a subpolyhedron.
\begin{enumerate}
\item Let $f \in \mcal{O}_{\mrm{pws}}(Z) \otimes \h$ be an inert
function relative to $\mbf{C} / X$.
Then there is a unique function
\[ \Psi_{\h, \alpha}(f) : Z \to \h_0 \]
such that the following condition holds:
\begin{itemize}
\item[($*$)] Let $\sigma$ be a string in $X$, with initial point $x_0$
and terminal point $z := \sigma(v_1) \in Z$, and let
$g := \opn{MI}(\alpha \vert \sigma) \in G$. 
Then 
\[ \Psi_{\h, \alpha}(f)(z) = \Psi_{\h}(g)(f(z)) . \]
\end{itemize}
\item The operation $\Psi_{\h, \alpha}$
is linear over the ring $\mcal{O}_{\mrm{pws}}(Z)$.
\item The function $\Psi_{\h, \alpha}(f)$ is continuous.
\end{enumerate}
\end{lem}

\begin{proof}
Take a point $z \in Z$. Choose any string $\sigma$ connecting $x_0$ to $z$,
and let 
\[ \lambda  := \Psi_{\h}(g)(f(z)) \in \h \]
as in condition ($*$). Since $\alpha$ is a compatible connection 
(see Definition \ref{dfn:8}), and since $f$ is inert, we know that 
\[ \opn{Lie}(\Phi_0)(\lambda) = \Phi_X(x_0)(\lambda) = 
\opn{Ad}_{\g}(g)(\Phi_X(z)(f(z))) = 0 . \]
This shows that $\lambda \in \h_0$. 

If we were to choose another string $\sigma'$ with the same initial and terminal
points, then for 
$g' := \opn{MI}(\alpha \vert \sigma')$ we would have
\[ (\Psi_{\h}(g') \circ \Psi_{\h}(g)^{-1})(\lambda) = \lambda , \]
this according to Lemma \ref{lem:45}.
Hence $\lambda =  \Psi_{\h}(g')(f(z))$. 
We see that $\lambda$ is independent of the string $\sigma$, and 
we can define 
$\Psi_{\h, \alpha}(f)(z) := \lambda$.
We obtain a function $\Psi_{\h, \alpha}(f) : Z \to \h_0$.
This proves part (1). 

Part (2) is true because the operator $\Psi_{\h}(g) : \h \to \h$ is linear.

It remains to prove part (3). We need to prove that the function \lb 
$\Psi_{\h, \alpha}(f) : Z \to \h_0$ is continuous. So let's fix a point
$z_0 \in Z$, and a string $\sigma$ starting at $x_0$ and ending at $z_0$. Let
$g := \opn{MI}(\alpha \vert \sigma)$.
For any point $z \in Z$ let $\sigma_z : \mbf{I}^1 \to Z$ be the unique linear
map with initial point $z_0$ and terminal point $z$, and let
$g_z := \opn{MI}(\alpha \vert \sigma_z)$. So
\[ \Psi_{\h, \alpha}(f)(z) = \Psi_{\h}(g \cdot g_z)(f(z)) 
=  (\Psi_{\h}(g) \circ \Psi_{\h}(g_z))(f(z)) . \]
But according to Proposition \ref{prop:21} we know that 
$\Psi_{\h}(g_z) \to \bsym{1} \in \opn{End}(\h)$ as $z \to z_0$;
and $f(z) \to f(z_0)$ by continuity of $f$. 
Therefore 
\[ \Psi_{\h, \alpha}(f)(z) \to \Psi_{\h, \alpha}(f)(z_0) \] 
as $z \to z_0$.
\end{proof}

To summarize, the lemma says that there is an 
$\mcal{O}_{\mrm{pws}}(Z)$-linear homomorphism
\begin{equation} \label{eqn:224}
\Psi_{\h, \alpha} : \{ \text{inert functions on } Z \} 
\to \mcal{O}_{\mrm{cont}}(Z) \otimes \h_0 ,
\end{equation}
where $\mcal{O}_{\mrm{cont}}(Z)$ is the ring of continuous functions 
$Z \to \R$. 

Recall the module $\Omega^p_{\mrm{pwc}}(X)$ of piecewise continuous
differential forms, from Subsection \ref{subsec:pwc}.

\begin{prop} \label{prop:18}
Let $\alpha$ be a tame connection, and let
$\gamma \in \Omega^p_{\mrm{pws}}(X) \otimes \h$ be an inert form.
Then there is a unique piecewise continuous form
\[ \Psi_{\h, \alpha}(\gamma) \in \Omega^p_{\mrm{pwc}}(X) \otimes \h_0 \]
with the following property:
\begin{enumerate}
\item[($*$)] Choose a linear coordinate system 
$\bsym{s} = (s_1, \ldots, s_n)$ on $X$, and a smoothing
triangulation $\{ X_j \}_{j \in J}$ for $\gamma$. 
Let $f_{j, \bsym{i}} \in \mcal{O}(X_j) \otimes \h$
be the coefficients of $\gamma|_{X_j}$, as in \tup{(\ref{eqn:222})}.
Then
\[ \Psi_{\h, \alpha}(\gamma)|_{X_j} = 
\sum_{\bsym{i}} \, \Psi_{\h, \alpha}(f_{j, \bsym{i}}) \cdot 
\d s_{i_1} \wedge \cdots \wedge \d s_{i_p} . \]
\end{enumerate}
\end{prop}

Note that by Lemma \ref{lem:47} the functions $f_{j, \bsym{i}}$ are inert, 
so the continuous functions
$\Psi_{\h, \alpha}(f_{j, \bsym{i}}) \in \mcal{O}_{\mrm{cont}}(X_j)$ are defined.

\begin{proof}
This is immediate from the uniqueness of the coefficients
$f_{j, \bsym{i}}$, and the properties of the homomorphism
$\Psi_{\h, \alpha}$ listed in Lemma \ref{lem:46}.
\end{proof}

\begin{rem}
Actually a lot more can be said here. Presumably, one can show that 
\[ \opn{Ker}(\Phi_X) \subset \mcal{O}_{\mrm{pws}}(X) \otimes \h  \]
is the set of piecewise smooth sections of a piecewise smooth
vector bundle $E$ over $X$, which is a sub-bundle of the
trivial vector bundle $X \times \h$. 
And the operation $\Psi_{\h, \alpha}$ corresponds to a piecewise smooth
isomorphism of vector bundles
\[ E \iso  X \times \h_0 . \]

If so, then it would follow that for any inert form 
$\gamma \in \Omega^p_{\mrm{pws}}(Z) \otimes \h$, the form 
$\Psi_{\h, \alpha}(\gamma)$ is actually piecewise smooth (not just piecewise
continuous). And we would have an $\mcal{O}_{\mrm{pws}}(X)$-linear bijection
\[ \Psi_{\h, \alpha} : \{ \text{inert $p$-forms on } Z \} 
\to \Omega^p_{\mrm{pws}}(Z) \otimes \h_0 . \]
However we did not verify these assertions.
\end{rem}

Suppose we are given a piecewise continuous form
$\delta \in \Omega^p_{\mrm{pwc}}(X) \otimes \h$
and a piecewise linear map $\tau : \mbf{I}^p \to X$.
Extending the formula (\ref{eqn:221}) linearly to $\h_0$-valued forms we obtain
$\int_{\tau} \delta \in \h_0$.

\begin{dfn}[Integration of Inert Forms] \label{dfn:43}
\index{Twisted multiplicative integral of inert form}
Let  $\mbf{C} / X$ be a Lie quasi crossed module with additive feedback
over $(X, x_0)$. Given a tame connection $\alpha$, an inert form
$\gamma \in \Omega^p_{\mrm{pws}}(X) \otimes \h$ and a piecewise linear map
$\tau : \mbf{I}^p \to X$, we define the {\em twisted multiplicative integral}
\[ \opn{MI}(\alpha, \gamma \vert \tau) \in H_0 \] 
as follows: 
\[ \opn{MI}(\alpha, \gamma \vert \tau) :=
\exp_H \Bigl( \int_{\tau} \Psi_{\h, \alpha}(\gamma) \Bigr) . \]
\end{dfn}

\begin{prop} \label{prop:20}
Let $f : Y \to X$ and $\tau : \mbf{I}^p \to Y$ be piecewise linear maps
between polyhedra, let 
$\alpha \in \Omega^1_{\mrm{pws}}(X) \otimes \g$ 
be a tame connection, and let
$\gamma \in \Omega^p_{\mrm{pws}}(X) \otimes \h$ be an inert form. Then 
$f^*(\alpha) \in \Omega^1_{\mrm{pws}}(Y) \otimes \g$ 
is a tame connection, 
$f^*(\gamma) \in \Omega^p_{\mrm{pws}}(Y) \otimes \h$ is an inert form, and
\[ \opn{MI}(\alpha, \gamma \vert f \circ \tau) = 
\opn{MI} \bigl( f^*(\alpha), f^*(\gamma) \vert \tau \bigl) . \]
\end{prop}

\begin{proof}
The connection $f^*(\alpha)$ is tame for $f^*(\mbf{C} / X)$ by Corollary
\ref{cor:8}. It is easy to see directly from the definitions that $f^*(\gamma)$
is inert, and moreover 
\[ \Psi_{\h, f^*(\alpha)}(f^*(\gamma)) = 
f^*( \Psi_{\h, \alpha}(\gamma)) \in \Omega^p_{\mrm{pwc}}(Y) \otimes \h . \]
And Integration of piecewise continuous forms commutes with pullbacks (see
Subsection \ref{subsec:pwc}). 
\end{proof}

\subsection{Combinatorics and Integration} \label{subsec:combi3d}
In this subsection we work with the pointed polyhedron 
$(X, x_0) := \lb (\mbf{I}^3, v_0)$. We fix a Lie quasi crossed module
with additive feedback
\[ \mbf{C} / \mbf{I}^3 = (G, H, \Psi, \Phi_0, \Phi_X) , \] 
and a piecewise smooth connection-curvature pair 
$(\alpha, \beta)$ in $\mbf{C} / \mbf{I}^3$.
Let us also fix a euclidean norm $\norm{-}$ on $\h$, 
an open set $V_0(H)$ in $H$, and constants $\epsilon_0(H)$ and 
$c_0(H)$ as in Section \ref{sec:expon}.

\begin{lem} \label{lem:23}
Let $(\sigma, \tau)$ be a balloon in $(\mbf{I}^3, v_0)$. Then 
\[ \opn{MI} \bigl( \alpha, \beta \vert \partial (\sigma, \tau) \bigr)
\in H_0 . \]
\end{lem}

\begin{proof}
Let us write
$h := \opn{MI} \bigl( \alpha, \beta \vert \partial (\sigma, \tau) \bigr) \in H$
and 
$g := \opn{MI} \bigl( \alpha \vert \partial \partial (\sigma, \tau) \bigr) \in
G$.
According to Theorem \ref{thm:1} we have 
$\Phi_0(h) = g$. 

On the other hand, consider the word
$w(\bsym{s}) \in \opn{Wrd}^{\pm 1}(s_1, \ldots, s_{12})$ such that 
\[ \partial \partial \mbf{I}^3 = w(\rho_1, \ldots, \rho_{12}) , \]
where $\rho_1, \ldots, \rho_{12}$ are the oriented $1$-cells of 
$\opn{sd}^0 \mbf{I}^3$. Then 
$[w(\rho_1, \ldots, \rho_{12})] = 1$ in \linebreak
$\bsym{\pi}_1(\opn{sk}_1 \opn{sd}^0 \mbf{I}^3)$. 
By Lemma \ref{lem:15} we know that
$w(\bsym{s}) \sim_{\mrm{can}} 1$ in $\opn{Wrd}^{\pm 1}(\bsym{s})$. 
Now writing $g_i := \opn{MI}(\alpha \vert \rho_i)$
we have
$g = w(g_1, \ldots, g_{12}) = 1$ in the group $G$. 
We see that 
$h \in  \opn{Ker}(\Phi_0) = H_0$.
\end{proof}

\begin{lem} \label{lem:24}
Take $k = 0, 1$. Let 
$\bsym{\rho} = \bigl( (\sigma_i, \tau_i) \bigr)_{i = 1, \ldots, n}$
be a sequence of kites in $(\mbf{I}^{3}, v_0)$, all patterned on 
$\opn{sd}^k \mbf{I}^3$. Write
\[ a_i := [\partial (\sigma_i, \tau_i)] \in 
\bsym{\pi}_1(\opn{sk}_1 \opn{sd}^k \mbf{I}^3) \]
and 
\[ h_i := \opn{MI} (\alpha, \beta \vert \sigma_i, \tau_i) \in H . \]
Let $w(\bsym{s}) \in \opn{Wrd}^{\pm 1}(s_1, \ldots, s_n)$ be a word such that
\[ w(a_1, \ldots, a_n) = 1 \]
in $\bsym{\pi}_1(\opn{sk}_1 \opn{sd}^k \mbf{I}^3)$ and
\[ [ w(\tau_1, \ldots, \tau_n) ] = 0 \]
in $\mrm{H}_2(\opn{sk}_2 \opn{sd}^k \mbf{I}^3)$. Then
\[ w(h_1, \ldots, h_n) = 1 \]
in $H$.
\end{lem}

Observe that by Lemma \ref{lem:20} the chain 
$w(\tau_1, \ldots, \tau_n)$ is a cycle, so we can talk about its homology class.

\begin{proof}
First assume $k = 0$. Let $p_i \in \{ 1, \ldots, 6 \}$, 
$d_i \in \{ 0, 1 \}$ and $e_i \in \{ 0, 1, 2, 3 \}$ be such that
\[ \tau_i = \opn{flip}^{d_i} ( \opn{turn}^{e_i} (\tau^{\flat}_{p_i} )) . \]
as maps $\mbf{I}^2 \to \mbf{I}^3$.
Let
$a^{\flat}_i := [\partial (\sigma^{\flat}_{i}, \tau^{\flat}_{i})]$
and
$h^{\flat}_i := \opn{MI} (\alpha, \beta \vert \sigma^{\flat}_i,
\tau^{\flat}_i)$,
in the notation of Definition \ref{dfn:36}. There are words 
$v_i(\bsym{t}) \in \opn{Wrd}^{\pm 1}(t_1, \ldots, t_6)$
such that 
\[ [\sigma_i] * [\sigma^{\flat}_{p_i}]^{-1} = 
v_i(a^{\flat}_1, \ldots, a^{\flat}_6) \]
in the group 
$\bsym{\pi}_1(\opn{sk}_1 \opn{sd}^0 \mbf{I}^3)$.
According to Corollaries \ref{cor:2} and \ref{cor:6}, repeated, we have
\[ h_i = \opn{Ad}_H \bigl( v_i( h^{\flat}_1, \ldots, h^{\flat}_6 ) \bigr)
\bigl( (h^{\flat}_{p_i})^{(-1)^{d_i}} \bigr) . \]
Define
\[ u_i(\bsym{t}) := \opn{Ad}_{\opn{Wrd}^{\pm 1}(\bsym{t})}( v_i(\bsym{t}))
(y_{p_i}^{(-1)^{d_i}}) \in \opn{Wrd}^{\pm 1}(\bsym{t})  \]
using the conjugation operation from (\ref{eqn:249}). Then
$h_i = u_i(h^{\flat}_1, \ldots, h^{\flat}_6)$ and
$a_i = u_i(a^{\flat}_1, \ldots, a^{\flat}_6)$.

Let 
\[ u(\bsym{t}) := w \bigl( u_1(\bsym{t}), \ldots, u_n(\bsym{t}) \bigr) 
\in \opn{Wrd}^{\pm 1}(\bsym{t}) . \]
We know that
$u(a^{\flat}_1, \ldots, a^{\flat}_6) = 1$ in 
$\bsym{\pi}_1(\opn{sk}_1 \opn{sd}^0 \mbf{I}^3)$.
Hence by Lemmas \ref{lem:21}(1) and \ref{lem:15} we have
\[ u(\bsym{t}) \sim_{\mrm{can}} 
\opn{Ad}_{\opn{Wrd}^{\pm 1}(\bsym{t})} (v(\bsym{t})) \bigl( (y_1 \cdots y_6)^e 
\bigr)
\]
in $\opn{Wrd}^{\pm 1}(\bsym{t})$, for some word $v(\bsym{t})$ and integer $e$. 
Passing to the abelian group \linebreak 
$\mrm{H}_2(\opn{sk}_2 \opn{sd}^0 \mbf{I}^3)$, with additive notation, 
we get
\[ e \cdot [\tau^{\flat}_{1} + \cdots + \tau^{\flat}_{6}] = 
[u(\tau^{\flat}_{1} , \ldots , \tau^{\flat}_{6})] =
[w(\tau_1, \ldots, \tau_n)] = 0 . \]
Using Lemma \ref{lem:21}(2) we conclude that $e = 0$, and hence
$u(\bsym{t}) \sim_{\mrm{can}} 1$ in $\opn{Wrd}^{\pm 1}(\bsym{t})$.
Finally, evaluating in the group $H$ we get
\[ w(h_1, \ldots, h_n) = u(h^{\flat}_1, \ldots, h^{\flat}_6) = 1 . \]

For the case $k = 1$ the proof is the same, only using Lemma \ref{lem:22}
instead of Lemma \ref{lem:21}, and working with the  with words in the  monoid
$\opn{Wrd}^{\pm 1}(t_1, \ldots, t_{48})$ instead of 
$\opn{Wrd}^{\pm 1}(t_1, \ldots, t_{6})$.
\end{proof}

\subsection{Estimates} \label{subsec:est}
We continue with the setup of Subsection \ref{subsec:combi3d}.

Recall the constants $c_{2'}(\alpha, \beta)$ and $\epsilon_{2'}(\alpha, \beta)$
from Proposition \ref{prop:22}. \lb Among other things, these numbers satisfy
$c_{2'}(\alpha, \beta) \geq 1$ and $0 < \epsilon_{2'}(\alpha, \beta) \lb < 1$.

\begin{lem} \label{lem:26}
There are constants $c_6(\alpha, \beta)$ and
$\epsilon_6(\alpha, \beta)$ with these properties:
\begin{enumerate}
\rmitem{i} $c_6(\alpha, \beta) \geq c_{2'}(\alpha, \beta)$ and
\[ 0 < \epsilon_6(\alpha, \beta) \leq \min \bigl(
\smfrac{1}{6} \cdot c_6(\alpha, \beta)^{-1} \cdot \epsilon_0(H) , \
\epsilon_{2'}(\alpha, \beta) \bigr) \, . \]
\rmitem{ii} Suppose $(\sigma, \tau)$ is a cubical balloon
$(\mbf{I}^3, v_0)$
such that 
$\opn{side}(\tau) < \epsilon_6(\alpha, \beta)$ 
and $\opn{len}(\sigma) \leq 6$. Then
\[ \opn{MI} \bigl( \alpha, \beta \vert 
\partial (\sigma, \tau) \bigr) \in V_0(H) \]
and
\[ \Norm{ \log_H \bigl( \opn{MI} \bigl( \alpha, \beta \vert 
\partial (\sigma, \tau) \bigr) \bigr) }
\leq c_6(\alpha, \beta) \cdot \opn{side}(\tau)^3  \, . \]
\end{enumerate}
\end{lem}

\begin{proof}
The proof is basically the same as that of Lemma \ref{lem:9}, but using
Proposition \ref{prop:22} instead of Proposition \ref{prop:16}.
\end{proof}

Suppose $(\sigma, \tau)$ is a nondegenerate cubical balloon in
$(\mbf{I}^3, v_0)$. Let us write $\epsilon := \opn{side}(\tau)$; 
$Z := \tau(\mbf{I}^3)$, which is an oriented cube in $\mbf{I}^3$; and 
$z_0 := \tau(\smfrac{1}{2}, \smfrac{1}{2}, \smfrac{1}{2})$, the midpoint of $Z$.
Let $\bsym{s} = (s_1, s_2, s_3)$ be the orthonormal linear coordinate system on
$Z$ such that 
$\tau^*(s_i) = \epsilon \cdot t_i$.

Assume that the forms $\alpha|_Z$ and $\beta|_Z$ are smooth. Let 
$\til{\alpha}_1, \til{\alpha}_2, \til{\alpha}_3 \in 
\mcal{O}(Z) \otimes \g$
be the coefficients of $\alpha|_Z$ relative to $\bsym{s}$, namely 
\[ \alpha|_Z = \sum_{1 \leq i \leq 3} \, \til{\alpha}_i \cdot \d s_i . \]
And let
$\til{\beta}_{1, 2}, \til{\beta}_{1, 3}, \til{\beta}_{2, 3}  \in 
\mcal{O}(Z) \otimes \h$
be the coefficients of $\beta|_Z$ relative to $\bsym{s}$, namely 
\[ \beta|_Z = \sum_{1 \leq i < j \leq 3} \, 
\til{\beta}_{i, j} \cdot \d s_i \wedge \d s_j . \]

Recall the Lie algebra map
\[ \psi_{\h} = \opn{Lie}(\Psi_{\h}) : \g \to \opn{End}(\h) . \]
For any $i$ we get a function
\[ \psi_{\h}(\til{\alpha}_i) \in \mcal{O}(Z) \otimes \opn{End}(\h) ; \]
so for any $i, j, k$ there is a function
\[ \psi_{\h}(\til{\alpha}_i)(\til{\beta}_{j, k}) \in 
\mcal{O}(Z) \otimes \h . \]
Define the function
\begin{equation} \label{eqn:197}
 \begin{aligned}
& \til{\gamma} := 
\smfrac{\partial}{\partial s_1}(\til{\beta}_{2, 3}) 
- \smfrac{\partial}{\partial s_2}(\til{\beta}_{1, 3})
+ \smfrac{\partial}{\partial s_3}(\til{\beta}_{1, 2}) \\
& \qquad \quad + \psi_{\h}(\til{\alpha}_1)(\til{\beta}_{2, 3}) 
- \psi_{\h}(\til{\alpha}_2)(\til{\beta}_{1, 3}) 
+ \psi_{\h}(\til{\alpha}_3)(\til{\beta}_{1, 2})
\in \mcal{O}(Z) \otimes \h .
\end{aligned}
\end{equation}

Let $\sigma_{\mrm{pr}}$ be the following string in $Z$:
\begin{equation} \label{eqn:204}
\sigma_{\mrm{pr}} := \tau \circ \bigl( v_0, (\smfrac{1}{2}, 0, 0) \bigr) * 
\bigl( (\smfrac{1}{2}, 0, 0), (\smfrac{1}{2}, \smfrac{1}{2}, 0) \bigr) *
\bigl( (\smfrac{1}{2}, \smfrac{1}{2}, 0), 
(\smfrac{1}{2}, \smfrac{1}{2}, \smfrac{1}{2}) \bigr) . 
\end{equation}
So $\sigma_{\mrm{pr}}$ has initial point $\tau(v_0)$ and
terminal point $z_0$; see Figure \ref{fig:50}. Define
\begin{equation} \label{eqn:215}
g_0 := \opn{MI}(\alpha \vert \sigma_{\mrm{pr}}) \ \text{ and } \
g := \opn{MI}(\alpha \vert \sigma) 
\end{equation}
in $G$. 

\begin{figure}
\includegraphics[scale=0.39]{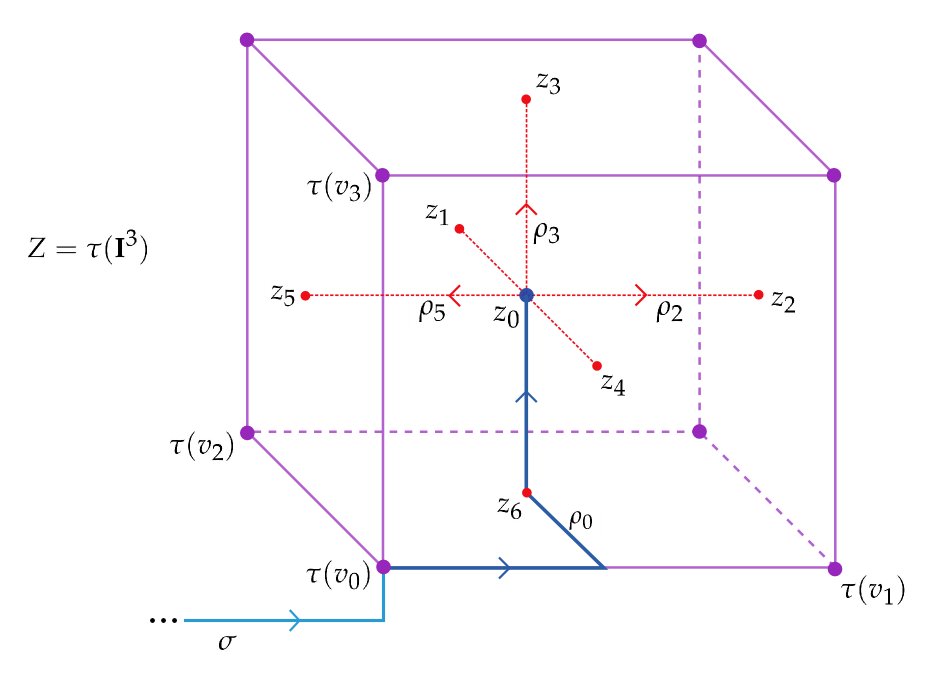}
\caption{Illustration for Subsection \ref{subsec:est}.} 
\label{fig:50}
\end{figure}

For $i \in \{ 1, \ldots, 6 \}$ we consider the points $z_i \in Z$ defined as
follows:
\[ z_1 := \tau(\smfrac{1}{2}, 1, \smfrac{1}{2}) , \
z_2 := \tau(1, \smfrac{1}{2}, \smfrac{1}{2}) , \ 
z_3 := \tau(\smfrac{1}{2}, \smfrac{1}{2}, 1) , \]
\[ z_4 := \tau(\smfrac{1}{2}, 0, \smfrac{1}{2}) , \
z_5 := \tau(0, \smfrac{1}{2}, \smfrac{1}{2}) , \ 
z_6 := \tau(\smfrac{1}{2}, \smfrac{1}{2}, 0) . \]
See Figure \ref{fig:50}. We define the strings
\begin{equation} \label{eqn:216}
\rho'_i := (\tau \circ \sigma^{\flat}_i) * 
(\tau \circ \tau^{\flat}_i \circ \sigma_{\mrm{pr}}) ,
\end{equation}
where $\sigma_{\mrm{pr}}$ is the probe in $\mbf{I}^2$; see formula 
(\ref{eqn:237}). 
So $\rho'_i$ has initial point $\tau(v_0)$ and terminal point $z_i$. 
Let
\[ g'_i := \opn{MI}(\alpha \vert \rho'_i) \in G . \]

For $i \in \{ 1, \ldots, 6 \}$ we define elements $\lambda_i \in \h$ as
follows:
\begin{equation} \label{eqn:217}
\begin{aligned}
& \lambda_1 := - \epsilon^2 \cdot \Psi_{\h}(g \cdot g'_1)
\bigl( \til{\beta}_{1, 3}(z_1) \bigr) , \\
& \lambda_2 := \epsilon^2 \cdot \Psi_{\h}(g \cdot g'_2)
\bigl( \til{\beta}_{2, 3}(z_2) \bigr) , \\
& \lambda_3 := \epsilon^2 \cdot \Psi_{\h}(g \cdot g'_3)
\bigl( \til{\beta}_{1, 2}(z_3) \bigr) , \\
& \lambda_4 := - \epsilon^2 \cdot \Psi_{\h}(g \cdot g'_4)
\bigl( \til{\beta}_{1, 3}(z_4) \bigr) , \\
&  \lambda_5 := \epsilon^2 \cdot \Psi_{\h}(g \cdot g'_5)
\bigl( \til{\beta}_{2, 3}(z_5) \bigr) , \\
& \lambda_6 :=  \epsilon^2 \cdot \Psi_{\h}(g \cdot g'_6)
\bigl( \til{\beta}_{1, 2}(z_6) \bigr) .
\end{aligned}
\end{equation}

It is easy to see from Definition \ref{dfn:11} that
\begin{equation} \label{eqn:198}
\exp_H(\lambda_i) = 
\opn{RP}_0(\alpha, \beta \vert \partial_i(\sigma, \tau)) .
\end{equation}

\begin{lem} \label{lem:27}
There are constants $c_7(\alpha, \beta)$ and
$\epsilon_7(\alpha, \beta)$ with these properties:
\begin{enumerate}
\rmitem{i} $c_7(\alpha, \beta) \geq c_6(\alpha, \beta)$ and
\[ 0 < \epsilon_7(\alpha, \beta) \leq 
\min \bigl( \epsilon_6(\alpha, \beta), \, 
\smfrac{1}{10} \epsilon_1(\alpha) \bigr) \, . \]
\rmitem{ii} Let $(\sigma, \tau)$ be a nondegenerate cubical balloon in
$(\mbf{I}^3, v_0)$ with $\epsilon := \opn{side}(\tau)$. Assume that 
$\epsilon < \epsilon_7(\alpha, \beta)$,
$\opn{len}(\sigma) \leq 6$,
and the forms $\alpha|_Z$ and $\beta|_Z$ are smooth.
Then, with $\lambda_i$ as in \tup{(\ref{eqn:217})}, the following estimate
holds:
\[ \Norm{ \log_H \bigl( \opn{MI}(\alpha, \beta \vert 
\partial(\sigma, \tau)) \bigr)
- \bosum_{i = 1}^6 \lambda_i }
\leq c_7(\alpha, \beta)  \cdot \epsilon^4 \, . \]
\end{enumerate}
\end{lem}

\begin{proof}
Let us set
\[ c := 12 \cdot \max \bigl( c_2(\alpha, \beta), \, 
\exp ( c_4(\alpha, \Psi_{\h}) \cdot 9 ) \cdot 
\norm{\beta}_{\mrm{Sob}} \bigr) \, . \]
According to Proposition \ref{prop:8}(2), if
$\epsilon < \epsilon_2(\alpha, \beta)$ then 
\begin{equation} \label{eqn:199}
\Norm{ \log_H \bigl( \opn{MI}(\alpha, \beta \vert \partial_i(\sigma, \tau)) 
\bigr) - \lambda_i } 
\leq c_2(\alpha, \beta) \cdot \epsilon^4 
\leq \smfrac{1}{12} c \cdot \epsilon^4 \, . 
\end{equation}
And by Proposition \ref{prop:13} we know that
\begin{equation*} \label{eqn:200}
\norm{\lambda_i} \leq \exp \bigl( c_4(\alpha, \Psi_{\h}) \cdot 9 \bigr) \cdot 
\norm{\beta}_{\mrm{Sob}} \cdot \epsilon^2 
\leq \smfrac{1}{12} c \cdot \epsilon^2 \, .  
\end{equation*}
Hence
\begin{equation} \label{eqn:201}
\Norm{ \log_H \bigl( 
\opn{MI}(\alpha, \beta \vert \partial_i(\sigma, \tau)) \bigr) }
\leq \smfrac{1}{6}  c \cdot  \epsilon^2 \, .  
\end{equation}
We will take 
\[ \epsilon_7(\alpha, \beta) := \min \bigl( 
\epsilon_6(\alpha, \beta), \, 
c^{- 1/2} \cdot \epsilon_0(H)^{1/2} \bigr) \, . \]

Now assume that
$\epsilon <  \epsilon_7(\alpha, \beta)$. Then by property (ii) of Theorem
\ref{thm:6}, used in conjunction with the bound (\ref{eqn:201}), we get
\[ \begin{aligned}
& \Norm{ \log_H \bigl( \boprod_{i = 1}^6  
\opn{MI}(\alpha, \beta \vert \partial_i(\sigma, \tau)) \bigr)
- \bosum_{i = 1}^6  
\opn{MI}(\alpha, \beta \vert \partial_i(\sigma, \tau)) } \\
& \qquad \leq 
c_0(H) \cdot (c \cdot \epsilon^2)^2 = 
c_0(H) \cdot c^2 \cdot \epsilon^4 \, . 
\end{aligned} \]
Combining this with (\ref{eqn:199}) we obtain 
\begin{equation} \label{eqn:202}
\begin{aligned}
& \Norm{ \log_H \bigl( \boprod_{i = 1}^6  
\opn{MI}(\alpha, \beta \vert \partial_i(\sigma, \tau)) \bigr)
- \bosum_{i = 1}^6 \lambda_i } \\
& \qquad \leq 
\bigl( c_0(H) \cdot c^2 + \smfrac{1}{2} c \bigr) \cdot \epsilon^4 \, . 
\end{aligned} 
\end{equation}
Thus the constant
\[ c_7(\alpha, \beta) := \max \bigl( 
c_0(H) \cdot c^2 + \smfrac{1}{2} c, \, c_6(\alpha, \beta) \bigr)  \]
works.
\end{proof}

\begin{dfn}
\begin{enumerate}
\item Let $(\sigma, \tau)$ be a square kite in $(\mbf{I}^3, v_0)$. We
will say that $(\sigma, \tau)$ is
{\em $(\alpha, \beta)$-tiny} (in this section) if 
$\opn{side}(\tau) < \epsilon_7(\alpha, \beta)$ and
$\opn{len}(\sigma) \leq 9$. 
\item Let $(\sigma, \tau)$ be a cubical balloon in $(\mbf{I}^3, v_0)$. We
will say that $(\sigma, \tau)$ is
{\em $(\alpha, \beta)$-tiny} if 
$\opn{side}(\tau) < \epsilon_7(\alpha, \beta)$ and
$\opn{len}(\sigma) \leq 6$. 
\end{enumerate}
\end{dfn}

Note that if $(\sigma, \tau)$ is an $(\alpha, \beta)$-tiny cubical balloon,
then the kites $\partial_i(\sigma, \tau)$ are all $(\alpha, \beta)$-tiny.

Let $(\sigma, \tau)$  be a cubical balloon in $(\mbf{I}^2, v_0)$.
For $i \in \{ 1, \ldots, 6 \}$ let $\rho_i$ be the linear map 
$\mbf{I}^1 \to Z$ such that $\rho_i(v_0) = z$ and 
$\rho_i(v_1) = z_i$. Recall the string $\sigma_{\mrm{pr}}$ from
(\ref{eqn:204}). 
For every $i \in \{ 1, \ldots, 6 \}$ the string
$\sigma_{\mrm{pr}} * \rho_i$ has initial point $\tau(v_0)$ and terminal point
$z_i$.
Let
\[ g_i := \opn{MI}(\alpha \vert \rho_i) \in G . \]

The form $\alpha|_Z$ is smooth, and hence we have
\[ \alpha' := \psi_{\h}(\alpha|_Z) \in 
\Omega^1(Z) \otimes \opn{End}(\h) .  \] 
As in Definition \ref{dfn:26} there is an associated constant form at $z_0$:
\[ \alpha'(z_0) \in \Omega^1_{\mrm{const}}(Z) \otimes \opn{End}(\h) . \]
In this way for every $i$ we get 
\[ \int_{\rho_i} \alpha'(z_0) \in \opn{End}(\h) . \]

\begin{lem} \label{lem:44}
There is a constant $c_8(\alpha, \beta)$ with this property:
\begin{enumerate}
\rmitem{$*$} Let $(\sigma, \tau)$ be a nondegenerate $(\alpha, \beta)$-tiny 
cubical balloon in $(\mbf{I}^3, v_0)$,
with $\epsilon := \opn{side}(\tau)$. 
Then, in the notation above, the estimate
\[ \Norm{ \Psi_{\h}(g \cdot g'_i) - 
\Psi_{\h}(g \cdot g_0) \circ 
\bigl( \bsym{1} + \int_{\rho_i} \alpha'(z_0) \bigr) }
\leq c_8(\alpha, \beta) \cdot \epsilon^2  \]
holds for every $i \in \{ 1, \ldots, 6 \}$.
\end{enumerate}
\end{lem}

\begin{proof}
Consider such a balloon. 
Recall the strings $\rho'_i$ from formula (\ref{eqn:216}). 
Since 
$\epsilon < \epsilon_7(\alpha, \beta) \leq 
\smfrac{1}{10} \epsilon_1(\alpha)$, we see that for every $i$ 
the closed string $\rho'_i * (\sigma_{\mrm{pr}} * \rho_i)^{-1}$ has 
\[ \opn{len}(\rho'_i * (\sigma_{\mrm{pr}} * \rho_i)^{-1}) <
10 \cdot \epsilon \leq \epsilon_1(\alpha) . \]
And this closed string bounds an area less than 
$10 \cdot \epsilon^2$. Hence by Corollary \ref{cor:7} we have
\[ \Norm{ \log_G \bigl( g'_i \cdot (g_0 \cdot g_i)^{-1} \bigr) } \leq
c_0(G) \cdot c_1(\alpha)^2 \cdot (10 \cdot \epsilon)^2 +
(10 \cdot \epsilon^2) \cdot \norm{\alpha}_{\mrm{Sob}} \, . \]
Therefore there is a constant $c'$, depending only on $\alpha$, such that 
\begin{equation} \label{eqn:205}
\Norm{ \Psi_{\h}(g'_i) - \Psi_{\h}(g_0 \cdot g_i) } \leq c' \cdot 
\epsilon^2 \, .
\end{equation}
By Proposition \ref{prop:13} we know that
\begin{equation*} \label{eqn:206}
\norm{ \Psi_{\h}(g) } \leq \exp \bigl( c_4(\alpha, \Psi_{\h}) \cdot 9 \bigr)
\cdot \norm{\alpha}_{\mrm{Sob}}  
\end{equation*}
and
\[ \norm{ \Psi_{\h}(g_0) } \leq \exp \bigl( 
c_4(\alpha, \Psi_{\h}) \cdot 2 \epsilon \bigr)
\cdot \norm{\alpha}_{\mrm{Sob}}  \, . \]
Therefore
\begin{equation*} \label{eqn:207}
\begin{aligned}
& \Norm{ \Psi_{\h}(g \cdot g'_i) - \Psi_{\h}(g \cdot g_0 \cdot g_i) } = 
\Norm{ \Psi_{\h}(g) \circ \bigl( \Psi_{\h}(g'_i) - 
\Psi_{\h}(g_0 \cdot g_i) \bigr) } \\
& \qquad \leq 
\Norm{ \Psi_{\h}(g) } \cdot 
\Norm{ \Psi_{\h}(g'_i) - \Psi_{\h}(g_0 \cdot g_i) } \\
& \qquad \leq 
\exp \bigl( c_4(\alpha, \Psi_{\h}) \cdot 9 \bigr)
\cdot \norm{\alpha}_{\mrm{Sob}} \cdot 
 c' \cdot \epsilon^2 \, .
\end{aligned}
\end{equation*}
On the other hand, by Proposition \ref{prop:21} we have the estimate
\begin{equation*} \label{eqn:208}
\Norm{ \Psi_{\h}(g_i) - \bigl( \bsym{1} + 
\int_{\rho_i} \alpha'(z_0) \bigr) }
\leq c_5(\alpha, \Psi_{\h}) \cdot \smfrac{1}{4} \epsilon^2 \, .
\end{equation*}
We conclude that the constant 
\[ \begin{aligned}
& c_8(\alpha, \beta) := 
\exp \bigl( c_4(\alpha, \Psi_{\h}) \cdot 9 \bigr)
\cdot \norm{\alpha}_{\mrm{Sob}} \cdot  c' \\
& \qquad  \qquad +
c_5(\alpha, \Psi_{\h}) \cdot \smfrac{1}{4} \cdot 
\exp \bigl( c_4(\alpha, \Psi_{\h}) \cdot 12 \bigr)
\cdot \norm{\alpha}_{\mrm{Sob}}^2 \, . 
\end{aligned} \]
works.
\end{proof}

\begin{lem} \label{lem:43}
There is a constant $c_9(\alpha, \beta)$ with this property:
\begin{enumerate}
\rmitem{$*$} Let $(\sigma, \tau)$ be a nondegenerate $(\alpha, \beta)$-tiny 
cubical balloon in $(\mbf{I}^3, v_0)$, with 
$\epsilon := \opn{side}(\tau)$. 
Then, in the notation above, the following estimate holds:
\[ \Norm{ \bosum_{i = 1}^6 \lambda_i
- \epsilon^3 \cdot \Psi_{\h}(g \cdot g_0)(\til{\gamma}(z_0)) } 
\leq \epsilon^4 \cdot c_9(\alpha, \beta) \, . \]
\end{enumerate}
\end{lem}

\begin{proof}
The strategy is to try to estimate the elements $\lambda_i$.

We begin with $i = 1$. Define
\[ \begin{aligned}
& \mu_1 := \epsilon^2 \cdot \Psi_{\h}(g \cdot g_0) \Bigl(
- \til{\beta}_{1, 3}(z_0) 
- \smfrac{1}{2} \cdot \epsilon \cdot
\psi_{\h}(\til{\alpha}_2(z_0))(\til{\beta}_{1, 3}(z_0)) \\
& \qquad \qquad
- \smfrac{1}{2} \cdot \epsilon \cdot 
(\smfrac{\partial}{\partial s_2} \til{\beta}_{1, 3})(z_0) \Bigr) .
\end{aligned} \]
The Taylor expansion to first order of the function $\til{\beta}_{1, 3}$ around
$z_0$ gives us
\begin{equation} \label{eqn:210}
\Norm{ \til{\beta}_{1, 3}(z_1) - \bigl( \til{\beta}_{1, 3}(z_0)
- \smfrac{1}{2} \epsilon \cdot 
(\smfrac{\partial}{\partial s_2} \til{\beta}_{1, 3})(z_0) \bigr) }
\leq \smfrac{1}{4} \epsilon^2 \cdot \norm{\beta}_{\mrm{Sob}} \, . 
\end{equation}
And Lemma \ref{lem:44}, for $i = 1$, gives
\begin{equation} \label{eqn:211}
\Norm{ \Psi_{\h}(g \cdot g'_i) - 
\Psi_{\h}(g \cdot g_0) \circ 
\bigl( \bsym{1} - \smfrac{1}{2} \epsilon \cdot \til{\alpha}_2(z_0) \bigr) }
\leq c_8(\alpha, \beta) \cdot \epsilon^2 \, . 
\end{equation}
Let
\[ \begin{aligned}
& d := \exp \bigl( c_4(\alpha, \Psi_{\h}) \cdot 9 \bigr)
\cdot \smfrac{1}{4} \cdot \norm{\beta}_{\mrm{Sob}} +
c_8(\alpha, \beta) \cdot 2 \cdot \norm{\beta}_{\mrm{Sob}} \\
& \qquad  + \exp \bigl( c_4(\alpha, \Psi_{\h}) \cdot 9 \bigr) \cdot 
\smfrac{1}{2} \cdot \norm{\alpha'}_{\mrm{Sob}} \cdot
\norm{\beta}_{\mrm{Sob}} \, . 
\end{aligned} \]
By combining the estimates (\ref{eqn:210}) and (\ref{eqn:211})  we obtain
\begin{equation} \label{eqn:212}
\norm{ \lambda_1 - \mu_1 } \leq d \cdot \epsilon^4 \, .
\end{equation}

We do the same thing for $\lambda_4$: let 
\[ \begin{aligned}
& \mu_4 := \epsilon^2 \cdot \Psi_{\h}(g \cdot g_0) \Bigl(
\til{\beta}_{1, 3}(z_0) 
- \smfrac{1}{2} \cdot \epsilon \cdot
\psi_{\h}(\til{\alpha}_2(z_0))(\til{\beta}_{1, 3}(z_0)) \\
& \qquad \qquad  
- \smfrac{1}{2} \cdot \epsilon \cdot 
(\smfrac{\partial}{\partial s_2} \til{\beta}_{1, 3})(z_0) \Bigr) .
\end{aligned} \]
The same sort of calculation that got us (\ref{eqn:212}), now gives the
inequality
\begin{equation} \label{eqn:213}
\norm{ \lambda_4 - \mu_4 } \leq d \cdot \epsilon^4
\end{equation}
And clearly 
\[ \mu_1 + \mu_4 = 
\epsilon^3 \cdot \Psi_{\h}(g \cdot g_0) \bigl( 
- \psi_{\h}(\til{\alpha}_2(z_0))(\til{\beta}_{1, 3}(z_0)) 
- (\smfrac{\partial}{\partial s_2} \til{\beta}_{1, 3})(z_0) \bigr) . \]
Plugging inequalities  (\ref{eqn:212}) and (\ref{eqn:213}) into this 
we get
\begin{equation} \label{eqn:214}
\begin{aligned}
& \Norm{ (\lambda_1 + \lambda_4) - 
\bigl( \epsilon^3 \cdot \Psi_{\h}(g \cdot g_0) \bigl( 
- \psi_{\h}(\til{\alpha}_2(z_0))(\til{\beta}_{1, 3}(z_0)) 
- (\smfrac{\partial}{\partial s_2} \til{\beta}_{1, 3})(z_0) \bigr) \bigr) } \\
& \qquad \leq \epsilon^4 \cdot 2 d \, . 
\end{aligned}
\end{equation}

By a similar calculation for the pairs of indices $(2, 5)$ and $(3, 6)$ 
we get these analogues of formula (\ref{eqn:214}): 
\[ \begin{aligned}
& \Norm{ (\lambda_2 + \lambda_5) - 
\bigl( \epsilon^3 \cdot \Psi_{\h}(g \cdot g_0) \bigl( 
\psi_{\h}(\til{\alpha}_1(z_0))(\til{\beta}_{2, 3}(z_0)) 
+ (\smfrac{\partial}{\partial s_1} \til{\beta}_{2, 3})(z_0) \bigr) \bigr) } \\
& \qquad \leq \epsilon^4 \cdot 2 d  
\end{aligned} \]
and 
\[ \begin{aligned}
& \Norm{ (\lambda_3 + \lambda_6) - 
\bigl( \epsilon^3 \cdot \Psi_{\h}(g \cdot g_0) \bigl( 
\psi_{\h}(\til{\alpha}_3(z_0))(\til{\beta}_{1, 2}(z_0)) 
+ (\smfrac{\partial}{\partial s_3} \til{\beta}_{1, 2})(z_0) \bigr) \bigr) } \\
& \qquad \leq \epsilon^4 \cdot 2 d \, . 
\end{aligned} \]
Thus, plugging in the value of $\til{\gamma}(z_0)$ from (\ref{eqn:197}) we get
\[ \Norm{ \bosum_{i = 1}^6 \lambda_i - 
\epsilon^3 \cdot \Psi_{\h}(g \cdot g_0) (\til{\gamma}(z_0)) }
\leq \epsilon^4 \cdot 6 d \, . \]
To finish we take
$c_8(\alpha, \beta) :=  6 d$.
\end{proof}

The last two lemmas combined give us:

\begin{lem} \label{lem:49}
Let $(\sigma, \tau)$ be a nondegenerate $(\alpha, \beta)$-tiny 
cubical balloon in $(\mbf{I}^3, v_0)$, with 
$\epsilon := \opn{side}(\tau)$,
$Z := \tau(\mbf{I}^3)$ and 
$z := (\smfrac{1}{2}, \smfrac{1}{2}, \smfrac{1}{2})$. 
Assume that $\alpha|_Z$ and $\beta|_Z$ are smooth. Let
$\til{\gamma} \in \mcal{O}(Z) \otimes \h$ be the function from 
equation \tup{(\ref{eqn:197})}, and let $g, g_0 \in G$ be the group elements
from equation \tup{(\ref{eqn:215})}. Then
\[ \Norm{ \log_H \bigl( \opn{MI}(\alpha, \beta \vert 
\partial(\sigma, \tau)) \bigr) 
- \epsilon^3 \cdot \Psi_{\h}(g \cdot g_0)(\til{\gamma}(z_0)) } 
\leq \epsilon^4 \cdot c_{10}(\alpha, \beta) \, , \]
where
\[ c_{10}(\alpha, \beta) := c_8(\alpha, \beta) + c_9(\alpha, \beta) . \]

If, moreover, $\til{\gamma}$ happens to be an inert function, then 
\[ \Psi_{\h}(g \cdot g_0) (\til{\gamma}(z_0)) = 
\Psi_{\h, \alpha}(\til{\gamma})(z_0) , \]
and therefore
\[ \Norm{ \log_H \bigl( \opn{MI}(\alpha, \beta \vert 
\partial(\sigma, \tau)) \bigr) 
- \epsilon^3 \cdot \Psi_{\h, \alpha}(\til{\gamma})(z_0) } 
\leq \epsilon^4 \cdot c_{10}(\alpha, \beta) \, . \]
\end{lem}

\subsection{Stokes Theorem}

Here again we are in the general situation: 
$(X, x_0)$ is a pointed polyhedron, 
\[ \mbf{C} / X = (G, H, \Psi, \Phi_0, \Phi_X) \]
is a Lie quasi crossed module with additive feedback, and
$(\alpha, \beta)$ is a piecewise smooth connection-curvature pair in 
$\mbf{C} / X$.

The Lie group map
$\Psi_{\h} : G \to \opn{GL}(\h)$ induces a Lie algebra map
\[ \psi_{\h} := \opn{Lie}(\Psi_{\h}) : \g \to \mfrak{gl}(\h) = 
\opn{End}(\h) . \]
By tensoring with $\Omega_{\mrm{pws}}(X)$ this induces a map of DG Lie algebras
\[ \psi_{\h} : \Omega_{\mrm{pws}}(X) \otimes \g \to 
\Omega_{\mrm{pws}}(X) \otimes \opn{End}(\h) . \]
In this way from the pair $\alpha \in \Omega^1_{\mrm{pws}}(X) \otimes \g$
and 
$\beta \in \Omega^2_{\mrm{pws}}(X) \otimes \h$ we get
\[ \psi_{\h}(\alpha)(\beta) \in \Omega^3_{\mrm{pws}}(X) \otimes \h . \]

\begin{dfn} \label{dfn:32}
\index{$3$-curvature}
Let $(X, x_0)$ be a pointed polyhedron, let
\[ \mbf{C} / X = (G, H, \Psi, \Phi_0, \Phi_X) \]
be a Lie quasi crossed module with additive feedback, and
let $(\alpha, \beta)$ be a piecewise smooth connection-curvature pair in 
$\mbf{C} / X$. The {\em $3$-curvature of $(\alpha, \beta)$} is the form
\[ \gamma := \d(\beta) + \psi_{\h}(\alpha)(\beta) \in 
\Omega_{\mrm{pws}}^3(X) \otimes \h . \]
\end{dfn}

Recall the notion of orientation of a polyhedron (Subsection
\ref{subsec:orient}). If $Z$ is an oriented cube, and 
$(s_1, s_2, s_3)$ is a positively oriented orthonormal linear coordinate system
on $Z$, then 
\[ \opn{or}(Z) = \d s_1 \wedge \d s_2 \wedge \d s_3 . \] 

\begin{lem} \label{lem:48}
Assume that $(X, x_0) = (\mbf{I}^3, v_0)$, so that we are in the situation of
Subsection \tup{\ref{subsec:est}}. 
Let $Z \subset \mbf{I}^3$ be an oriented nondegenerate cube, such that 
$\alpha|_Z$ and $\beta|_Z$ are smooth. Let
$\til{\gamma} \in \mcal{O}(Z) \otimes \h$ be the function from formula
\tup{(\ref{eqn:197})}. Then $\til{\gamma}$ is the coefficient of $\gamma|_Z$;
Namely
\[ \gamma|_Z = \til{\gamma} \cdot \opn{or}(Z) . \]
\end{lem}

\begin{proof}
This amounts to expanding the formula in Definition \ref{dfn:32} into
coordinates.
\end{proof}

\begin{lem} \label{lem:39}
Let $(\sigma, \tau)$ be a balloon in $(X, x_0)$. Then
\[ \opn{MI} \bigl(\alpha, \beta \vert \partial (\sigma, \tau) \bigr) 
\in H_0 . \]
\end{lem}

\begin{proof}
By definition we have
\[ \opn{MI} \bigl(\alpha, \beta \vert \partial (\sigma, \tau) \bigr) = 
\prod_{i = 1, \ldots, 6} \,
\opn{MI} \bigl(\alpha, \beta \vert (\sigma, \tau) \circ 
(\sigma^{\flat}_i, \tau^{\flat}_i) \bigr)  \]
in the notation of Definition \ref{dfn:36}.
For every index $i$ we have, by Theorem \ref{thm:1}, an equality
\[ \Phi_0 \bigl( \opn{MI} \bigl(\alpha, \beta \vert (\sigma, \tau) \circ 
(\sigma^{\flat}_i, \tau^{\flat}_i) \bigr) =
\opn{MI} \bigl(\alpha \vert \partial \bigl(
(\sigma, \tau) \circ (\sigma^{\flat}_i, \tau^{\flat}_i) \bigr) \bigr)  \]
in the group $G$. 
Now from Figure \ref{fig:46} we see that the closed string
\[ \partial \partial \mbf{I}^3 =
\prod_{i = 1, \ldots, 6} \, \partial (\sigma^{\flat}_i, \tau^{\flat}_i) \]
is cancellation equivalent to the empty string in the monoid of strings
in $\mbf{I}^3$. Therefore the closed string
\[ \partial \partial (\sigma, \tau) = 
\prod_{i = 1, \ldots, 6} \, \partial \bigl(
(\sigma, \tau) \circ (\sigma^{\flat}_i, \tau^{\flat}_i) \bigr) \]
is cancellation equivalent to the empty string in the monoid of strings in 
$X$. By Proposition \ref{prop:1} we can conclude that 
\[ \prod_{i = 1, \ldots, 6} \,
\opn{MI} \bigl(\alpha, \beta \vert \partial \bigl(
(\sigma, \tau) \circ (\sigma^{\flat}_i, \tau^{\flat}_i) \bigr) \bigr) = 1 \]
in $G$. Hence
\[ \Phi_0 \bigl( 
\opn{MI} \bigl(\alpha, \beta \vert \partial (\sigma, \tau) \bigr)
\bigr) = 1 . \]
\end{proof}

\begin{lem} \label{lem:40}
Let $(\sigma, \tau)$ be a balloon in $(X, x_0)$. Then there is a piecewise
linear map $f : (\mbf{I}^3, v_0) \to (X, x_0)$, and a cubical balloon 
$(\sigma', \tau')$ in $(\mbf{I}^3, v_0)$, such that 
$\opn{len}(\sigma') \leq 2$, $f|_{\tau'(\mbf{I}^3)}$ is linear, and
\[ (\sigma, \tau) = f \circ (\sigma', \tau') \]
as balloons in $(X, x_0)$.
\end{lem}

\begin{proof}
Just like the proof of Proposition \ref{prop:14}.
\end{proof}

\begin{lem} \label{lem:25}
Let $(\sigma, \tau)$ be a balloon in $(X, x_0)$. 
Take any $k \geq 0$. Then 
\[ \opn{MI} \bigl( \alpha, \beta \vert  
\partial(\opn{tes}^k (\sigma, \tau)) \bigr) =
\opn{MI} \bigl( \alpha, \beta \vert  \partial (\sigma, \tau ) \bigl) . \]
\end{lem}

\begin{proof}
Due to the recursive nature of the tessellations it is enough to consider the
case $k=1$. As in the first paragraph in the proof of Lemma \ref{lem:41}, and
using Lemma \ref{lem:40}, we can assume that 
$(X, x_0) = (\mbf{I}^2, v_0)$ and 
$(\sigma, \tau) = (\sigma^0_1, \tau^0_1)$, the basic balloon.

The assertion is now an easy consequence of Lemma \ref{lem:24} and Corollary 
\ref{cor:2}.
\end{proof}

\begin{thm}[Nonabelian $3$-dimensional Stokes Theorem]
\label{thm:10}
\index{Nonabelian Stokes Theorem for cubes}
Let $(X, x_0)$ be a pointed polyhedron, let
$\mbf{C} / X$ be a Lie quasi crossed module with additive feedback, 
let $(\alpha, \beta)$ be a piecewise smooth connection-curvature pair in 
$\mbf{C} / X$, and let $\gamma$ be the $3$-curvature of $(\alpha, \beta)$.
Then:
\begin{enumerate}
\item The form $\gamma$ is inert. 
\item For any balloon $(\sigma, \tau)$ in $(X, x_0)$ one has
\[  \opn{MI} \bigl( \alpha, \beta \vert \partial (\sigma, \tau)
\bigr) = 
\opn{MI} (\alpha, \gamma \vert \tau)  \]
in $H$.
\end{enumerate}
\end{thm}

Note that assertion (1) of the theorem is a {\em generalized Bianchi identity}. 

\begin{proof}
(1) Since $\gamma$ is a $3$-form, it is enough to prove that the form 
$\gamma|_Z \in \Omega_{\mrm{pws}}^3(Z) \otimes \h$
is inert for every cube $Z$ in $X$. 
Given such a cube $Z$, choose a linear map $f : \mbf{I}^3 \to X$
such that $Z = f(\mbf{I}^3)$. Then $\gamma|_Z$ is inert if and only if
$f^*(\gamma) \in \Omega^3_{\mrm{pws}}(\mbf{I}^3) \otimes \h$
is inert, with respect to the induced Lie quasi crossed module with additive
feedback $f^*(\mbf{C} / X)$. Hence we might as well assume that 
$(X, x_0) = (\mbf{I}^3, v_0)$. 

Now the singular locus of $\alpha$ and the singular locus of $\beta$ are
contained in a finite union of polygons in $\mbf{I}^3$. So by continuity it is
enough to show that $\gamma|_Z$ is inert for cubes
$Z \subset \mbf{I}^3$ such that $\alpha|_Z$ and $\beta|_Z$ are smooth. 
In this case $\gamma$ is also smooth. 

Given such a cube $Z$, choose an orientation on it, and let
$\til{\gamma} \in \mcal{O}(Z) \otimes \h$ be the coefficient of 
$\gamma|_Z$ (in the sense of Definition \ref{dfn:24}). 
We have to prove that the function $\til{\gamma}$ is inert; namely that 
$\til{\gamma}(z) \in \opn{Ker}(\Phi_X(z))$ for every $z \in Z$. 
Again by continuity, it is enough to look at $z \in \opn{Int} Z$.

We are allowed to move the cube $Z$ around the point $z$ and to shrink it. 
Hence it is enough
to take an  $(\alpha, \beta)$-tiny balloon $(\sigma, \tau)$, and 
to show that $\til{\gamma}(z) \in \opn{Ker}(\Phi_X(z))$ for the midpoint
$z := \tau(\smfrac{1}{2}, \smfrac{1}{2}, \smfrac{1}{2})$ of
$Z := \tau(\mbf{I}^3)$. 

Let $g$ and $g_0$ be the group elements from equation
(\ref{eqn:215}). By Lemma \ref{lem:39} we have 
\[ \opn{MI} \bigl( \alpha, \beta \vert \partial (\sigma, \tau) \in H_0 , \]
and by Lemma \ref{lem:26} we have
\[ \opn{MI} \bigl( \alpha, \beta \vert 
\partial (\sigma, \tau) \bigr) \in V_0(H) . \]
Taking logarithms we see that 
\[ \log_H \bigl( \opn{MI} \bigl( \alpha, \beta \vert \partial (\sigma, \tau)
\bigr) \in \h_0 . \]
By Lemma \ref{lem:49} it follows that the distance of the
element \lb $\Psi_{\h}(g \cdot g_0)(\til{\gamma}(z))$ from the linear subspace
$\h_0$ is at most $c_{10} \cdot \epsilon^4$.

Now by Proposition \ref{prop:13}(1) there is a uniform bound on the norm of the
operator $\Psi_{\h}(g \cdot g_0)^{-1}$; say $c'$. So the distance of
$\til{\gamma}(z)$
from the subspace
$\opn{Ker}(\Phi_X(z)) \subset \h$
is at most $c' \cdot c \cdot \epsilon$.

Since we can make $\epsilon$ arbitrarily small, we conclude that
$\til{\gamma}(z) \in \lb \opn{Ker}(\Phi_X(z))$. 

\medskip \noindent
(2) By the functoriality of 
$\opn{MI} \bigl( \alpha, \beta \vert \partial (\sigma, \tau)\bigr)$
and 
$\opn{MI} (\alpha, \gamma \vert \tau)$
(see Propositions \ref{prop:10} and \ref{prop:20}), and using the construction
of Lemma \ref{lem:40}, we can assume that 
$(X, x_0) = (\mbf{I}^3, v_0)$ and 
$\opn{len}(\sigma) \leq 2$. 

Take any $k \geq 0$. Consider the sequence
\[ \opn{tes}^k(\sigma, \tau) = 
\bigl( (\sigma, \tau) \circ (\sigma^{k}_i, \tau^{k}_i) 
\bigr)_{i = 1, \ldots, 8^k} \]
which was defined in Subsection \ref{subsec:ballbound}.
The twisted multiplicative integration of inert forms is multiplicative:
\[ \opn{MI} (\alpha, \gamma \vert \tau) = 
\prod_{i = 1, \ldots, 8^k} \, 
\opn{MI} \bigl( \alpha, \gamma \vert \tau \circ \tau^k_i) . \]
On the other hand, by Lemma \ref{lem:25} we know that
\[ \opn{MI} \bigl( \alpha, \beta \vert \partial (\sigma, \tau) \bigr) = 
\prod_{i = 1, \ldots, 8^k} \, 
\opn{MI} \bigl( \alpha, \beta \vert \partial
((\sigma, \tau) \circ (\sigma^{k}_i, \tau^{k}_i)) \bigr) . \]
So it suffices to prove that 
\[ \opn{MI} \bigl( \alpha, \beta \vert \partial
((\sigma, \tau) \circ (\sigma^{k}_i, \tau^{k}_i)) \bigr) =
\opn{MI} \bigl( \alpha, \gamma \vert \tau \circ \tau^k_i) \]
for every $i \in \{ 1, \ldots, 8^k \}$. 

If $k$ is large enough then all the balloons in
$\opn{tes}^k (\sigma, \tau)$
are $(\alpha, \beta)$-tiny. We conclude that it suffices to prove 
the equality
\[  \opn{MI} \bigl( \alpha, \beta \vert \partial (\sigma, \tau)
\bigr) = 
\opn{MI} (\alpha, \gamma \vert \tau)  \]
for an $(\alpha, \beta)$-tiny balloon $(\sigma, \tau)$ in $(\mbf{I}^3, v_0)$.

Suppose we are given an $(\alpha, \beta)$-tiny balloon $(\sigma, \tau)$, and a
natural number $k$. Let $\epsilon := \opn{side}(\tau)$. 
For any index $i \in \{ 1, \ldots, 8^k \}$ let
$Z_i := \lb (\tau \circ \tau^k_i)(\mbf{I}^3)$, which is an oriented cube in
$\mbf{I}^3$ of side $(\smfrac{1}{2})^k \cdot \epsilon$, and let
$z_i :=  (\tau \circ \tau^k_i)(\smfrac{1}{2}, \smfrac{1}{2}, \smfrac{1}{2})$.
Define
\[ \opn{good}(\tau, k) := \{ i \mid \alpha|_{Z_i} \textup{ and } 
\beta|_{Z_i} \textup { are smooth } \} , \]
and define $\opn{bad}(\tau, k)$ to be the complement of $\opn{good}(\tau, k)$ in
$\{ 1, \ldots, 8^k \}$.
Since the singular loci of $\alpha$ and $\beta$ are contained in a finite union
of polygons, it follows that 
\[ \abs{ \opn{bad}(\tau, k) } \leq 
a_2(\alpha, \beta) \cdot 4^k + a_0(\alpha, \beta) \]
for some constants $a_0(\alpha, \beta), a_2(\alpha, \beta)$, that are
independent of $k$; cf.\ Lemma \ref{lem:51}.

For $i \in \opn{good}(\tau, k)$ let
$\til{\gamma}_i \in \mcal{O}(Z_i) \otimes \h$ be the coefficient of 
$\gamma|_{Z_i}$, in the sense of Definition \tup{\ref{dfn:24}}, 
and let
\[ \mu_i := (\smfrac{1}{2})^{3k} \cdot \epsilon^3 \cdot 
\Psi_{\h, \alpha}(\til{\gamma}_i)(z_i) \in \h_0  . \]
According to Lemmas \ref{lem:49} and \ref{lem:48} we know that 
\[ \Norm{ \log_H \bigl( \opn{MI} \bigl(\alpha, \beta \vert 
\partial ( (\sigma, \tau) \circ (\sigma^k_i, \tau^k_i) ) \bigr) \bigr)
- \mu_i } \leq 
(\smfrac{1}{2})^{4k} \cdot \epsilon^4 \cdot c_{10}(\alpha, \beta) \, .\]

On the other hand, for $i \in \opn{bad}(\tau, k)$ let
$\mu_i := 0 \in \h$. 
According to Lemma \ref{lem:26} we have the inequality
\[ \Norm{ \log_H \bigl( \opn{MI} \bigl(\alpha, \beta \vert 
\partial ( (\sigma, \tau) \circ (\sigma^k_i, \tau^k_i) ) \bigr) \bigr)
- \mu_i } \leq 
(\smfrac{1}{2})^{3k} \cdot \epsilon^3 \cdot c_6(\alpha, \beta) \, .\]

Let 
\[ \opn{RS}_k(\alpha, \gamma \vert \tau) := 
\sum_{i = 1}^{8^k} \, \mu_i \in \h_0 . \]
Applying property (iv) of Theorem \ref{thm:6} with the inequalities above we
get
\[ \begin{aligned}
& \Norm{ \log_H \bigl( \opn{MI} \bigl(\alpha, \beta \vert \sigma, \tau \bigr)
\bigr) - \opn{RS}_k(\alpha, \gamma \vert \tau) } \\
& \qquad \leq 
c_0(H) \cdot \bigl( 
8^k \cdot (\smfrac{1}{2})^{4k} \cdot \epsilon^4 \cdot c_{10}(\alpha, \beta) 
\\
& \qquad \qquad +
(a_2(\alpha, \beta) \cdot 4^k + a_0(\alpha, \beta)) \cdot 
(\smfrac{1}{2})^{3k} \cdot \epsilon^3 \cdot c_6(\alpha, \beta) \bigr) .
\end{aligned} \]
We see that the difference tends to $0$ as $k \to \infty$. But
\[ \lim_{k \to  \infty} \opn{RS}_k(\alpha, \gamma \vert \tau) = 
\log_H \bigl( \opn{MI}(\alpha, \gamma \vert \tau) \bigr) . \]
\end{proof}

\begin{rem}
Part (1) of Theorem \ref{thm:10} is a variant of the Bianchi identity.
\end{rem}

\cleardoublepage
\section{Multiplicative Integration on Triangular Kites}
\label{sec:simpl}

In this section $(X, x_0)$ is a pointed manifold. 
Recall that by Convention \ref{conv:500} this means that $X$ is a smooth
manifold with sharp
corners (cf.\ Definition \ref{dfn:505}); so it could be a polyhedron.
We introduce triangular kites in $(X, x_0)$ and the corresponding multiplicative
integration. 

\subsection{Triangular Kites and balloons}
Recall that the polyhedra $\mbf{I}^1$ and $\bsym{\Delta}^1$ are identified via
the linear isomorphism $\bsym{\Delta}^1 \iso \mbf{I}^1$ that on vertices is
$(v_0, v_1) \mapsto (v_0, v_1)$. 

A {\em piecewise smooth path} in $X$ is a 
piecewise smooth map $\sigma: \mbf{I}^1 \to X$. 
When convenient we shall view such a path as a piecewise smooth map 
$\bsym{\Delta}^1 \to X$ using the identification above.

Suppose $\sigma_1, \sigma_2 : \mbf{I}^1 \to X$ are piecewise smooth paths
satisfying $\sigma_1(v_1) = \sigma_2(v_0)$. 
Their {\em product} is the piecewise smooth path
\[ \sigma_1 * \sigma_2 : \mbf{I}^1 \to X \]
defined as follows:
\[ (\sigma_1 * \sigma_2)(a) := 
\begin{cases}
\sigma_1(2 a) & \text{ if } 0 \leq a \leq \smfrac{1}{2} \\
\sigma_2(2 a - 1) & \text{ if } \smfrac{1}{2} \leq a \leq 1 .
\end{cases} \]
Note that this is the standard product used in homotopy theory, and it is
distinct from the concatenation operation on strings. In particular this product
is not associative nor unital.

The {\em inverse} of a piecewise smooth path
$\sigma : \mbf{I}^1 \to X$ is the piecewise smooth path
$\sigma^{-1} : \mbf{I}^1 \to X$ defined by
\[ \sigma^{-1}(a) := \sigma(1 - a) . \]

Suppose $Z$ is a polyhedron, $f : Z \to X$ is a piecewise smooth map, and 
$\sigma = (\sigma_1, \ldots, \sigma_m)$ is a
string in $Z$, with $m \geq 1$. This data gives rise to a piecewise smooth path
$f \circ \sigma$ in $X$ defined as follows:
\begin{equation} \label{eqn:227}
f \circ \sigma := \bigl( (f \circ \sigma_1) * (f \circ \sigma_2) \bigr)
* \cdots * (f \circ \sigma_m) .
\end{equation}
In this formula we view each $f \circ \sigma_i$ as a piecewise smooth path in
$X$, and the multiplication $*$ is that of paths. 
For the empty string $\sigma$ (i.e.\ $m = 0$) the path 
$f \circ \sigma$ is not always defined; but if $X$ has a base point $x_0$, then
we usually define $f \circ \sigma$ to be the constant path $x_0$.

\begin{dfn}
A {\em piecewise smooth triangular kite}
\index{Piecewise smooth triangular kite}
in $(X, x_0)$ is a pair 
$(\sigma, \tau)$, consisting of piecewise smooth maps
$\sigma : \mbf{I}^1 \to X$ and
$\tau : \bsym{\Delta}^2 \to X$, 
satisfying $\sigma(v_0) = x_0$ and
$\sigma(v_1) = \tau(v_0)$.
\end{dfn}

See Figure \ref{fig:80} for an illustration.

The {\em boundary} of $\bsym{\Delta}^2$ is the string
\begin{equation} \label{eqn:160}
\partial \bsym{\Delta}^2 := (v_0, v_1) * (v_1, v_2) * (v_2, v_0) 
\end{equation}
(consisting of $3$ pieces) in $\bsym{\Delta}^2$.

\begin{dfn} \label{dfn:31}
Let $(\sigma, \tau)$ be a piecewise smooth triangular kite in \lb $(X, x_0)$. 
Its {\em boundary} is the piecewise smooth path
\[ \partial (\sigma, \tau) := \bigl( \sigma * 
(\tau \circ \partial \bsym{\Delta}^2) \bigr) * \sigma^{-1} . \]
\end{dfn}

See Figure \ref{fig:81}.

\begin{dfn}
A {\em piecewise smooth triangular balloon} 
\index{Piecewise smooth triangular balloon}
in $(X, x_0)$ is a pair 
$(\sigma, \tau)$, consisting of piecewise smooth maps
$\sigma : \mbf{I}^1 \to X$ and
$\tau : \bsym{\Delta}^3 \to X$, 
satisfying $\sigma(v_0) = x_0$ and
$\sigma(v_1) = \tau(v_0)$.
\end{dfn}

See Figure \ref{fig:82} for an illustration.

Let $(Z, z_0)$ be a pointed polyhedron. 
By {\em linear triangular kite} in $(Z, z_0)$ we mean the obvious variant of
linear quadrangular kite. Namely this is a pair 
$(\sigma, \tau)$, consisting of a string $\sigma$ in $Z$ and a linear map
$\tau : \bsym{\Delta}^3 \to X$. These must satisfy 
$\sigma(v_0) = z_0$ and
$\sigma(v_1) = \tau(v_0)$.
Likewise we define {\em linear triangular balloons}.

\begin{dfn}
The {\em boundary} of $\bsym{\Delta}^3$ is the sequence of linear triangular
kites
\[ \partial \bsym{\Delta}^3 =
\bigl( \partial_1 \bsym{\Delta}^3, 
\partial_2 \bsym{\Delta}^3, 
\partial_3 \bsym{\Delta}^3, 
\partial_4 \bsym{\Delta}^3 \bigr) \]
in $(\bsym{\Delta}^3, v_0)$
defined as follows.
\begin{itemize}
\item Let $\sigma^{\flat}_2, \sigma^{\flat}_3, \sigma^{\flat}_4$ be the
empty strings in $\bsym{\Delta}^3$.  And let 
$\sigma^{\flat}_1 : \mbf{I}^1 \to \bsym{\Delta}^3$ 
be the linear map defined on vertices by
\[ \sigma^{\flat}_1(v_0, v_1) := (v_0, v_1) . \]

\item Let 
$\tau^{\flat}_i : \bsym{\Delta}^2 \to \bsym{\Delta}^3$
be the linear maps given on vertices by:
\[ \tau^{\flat}_1(v_0, v_1, v_2) := (v_1, v_2, v_3) , \]
\[ \tau^{\flat}_2(v_0, v_1, v_2) := (v_0, v_1, v_3) , \]
\[ \tau^{\flat}_3(v_0, v_1, v_2) := (v_0, v_3, v_2) , \]
\[ \tau^{\flat}_4(v_0, v_1, v_2) := (v_0, v_2, v_1) . \]

\item The kite are 
\[ \partial_i \bsym{\Delta}^3 := (\sigma^{\flat}_i, \tau^{\flat}_i) . \]
\end{itemize}
\end{dfn}

See Figure \ref{fig:83}.

Warning: the kites $(\sigma^{\flat}_i, \tau^{\flat}_i)$ above should be confused
with the quadrangular kites from Definition \ref{dfn:36}, despite the shared
notation. 

\begin{figure} 
\includegraphics[scale=0.3]{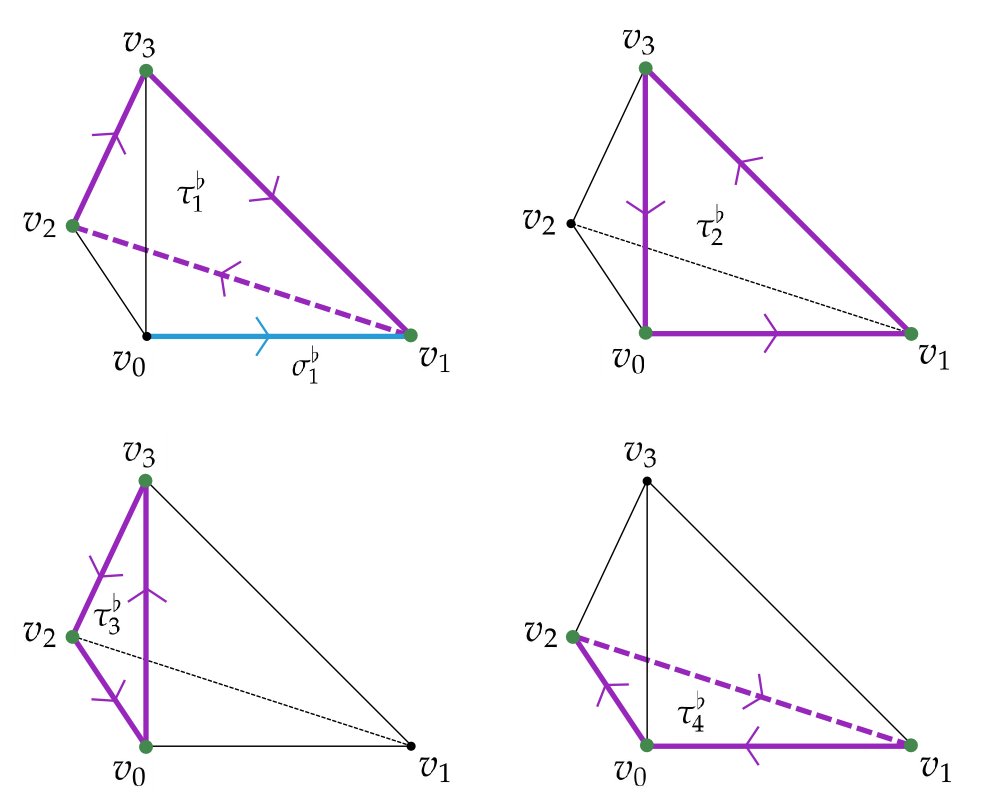}
\caption{The boundary of $\mbf{\Delta}^3$.} 
\label{fig:83}
\end{figure}

\begin{dfn} \label{dfn:30}
Let $(\sigma, \tau)$ be a piecewise smooth triangular balloon in $(X, x_0)$. Its
{\em
boundary} is the sequence of piecewise smooth triangular kites
\[ \partial (\sigma, \tau) :=
\bigl( \partial_1 (\sigma, \tau), \, \partial_2 (\sigma, \tau), \,
\partial_3 (\sigma, \tau), \, \partial_4 (\sigma, \tau) \bigr) \]
in $(X, x_0)$, where
\[ \partial_i (\sigma, \tau) :=  
(\sigma, \tau)  \circ (\sigma^{\flat}_i, \tau^{\flat}_i) . \]
\end{dfn}

See Figure \ref{fig:86} for an illustration.

\subsection{MI on Triangular Kites}
As before, $(X, x_0)$ is a pointed manifold.

Let
\[ \mbf{C} := (G, H, \Psi, \Phi_0) \]
be a Lie quasi crossed module (Definition \ref{dfn:22}). As usual we write
$\g := \opn{Lie}(G)$ and $\h := \opn{Lie}(H)$. 
An additive feedback for 
$\mbf{C}$ over $(X, x_0)$ is an element
\[ \Phi_X \in \mcal{O}(X) \otimes \opn{Hom}(\h, \g) \]
satisfying condition ($**$) of Definition \ref{dfn:21}.
Just like in Definition \ref{dfn:6}, we call the data 
\[ \mbf{C} / X := (G, H, \Psi, \Phi_0, \Phi_X) \]
a {\em Lie quasi crossed module with additive feedback} over $(X, x_0)$.

\begin{dfn}
Let $\alpha \in \Omega^1(X) \otimes \g$, and let $\sigma$ be a piecewise smooth
path in $X$.
Consider the piecewise smooth differential form
\[ \alpha' := \sigma^*(\alpha) \in 
\Omega^1_{\mrm{pws}}(\mbf{I}^1) \otimes \g . \]
We define
\[ \opn{MI}(\alpha \vert \sigma) := 
\opn{MI}(\alpha' \vert \mbf{I}^1) \in G  \]
(cf.\ Definition \ref{dfn:28}).
\end{dfn}

\begin{prop} \label{prop:23}
Let $\alpha \in \Omega^1(X) \otimes \g$.
\begin{enumerate}
\item Suppose $\sigma_1$ and $\sigma_2$ are piecewise smooth paths in $X$ such
that 
$\sigma_1 * \sigma_2$ is defined. Then
\[ \opn{MI}(\alpha \vert \sigma_1 * \sigma_2) = 
\opn{MI}(\alpha \vert \sigma_1) \cdot \opn{MI}(\alpha \vert \sigma_2) . \]

\item Let $\sigma$ be a piecewise smooth path in $X$. Then
\[ \opn{MI}(\alpha \vert \sigma^{-1}) = \opn{MI}(\alpha \vert \sigma)^{-1} . \]

\item Suppose $\sigma'$ is a string in a polyhedron $Z$, $f : Z \to X$ is a
piecewise smooth map, and 
$\sigma := f \circ \sigma'$ is the path in $X$ gotten by the operation
\tup{(\ref{eqn:227})}. Then
\[ \opn{MI}(\alpha \vert \sigma) = 
\opn{MI}(f^*(\alpha) \vert \sigma') . \]
\end{enumerate}
\end{prop}

\begin{proof}
(1) We could use the $2$-dimensional Stokes Theorem; but there is an elementary
proof. Let 
$\sigma^1_1, \sigma^1_2 : \mbf{I}^1 \to \mbf{I}^1$ be the linear maps belonging
to $\opn{tes}^1 \mbf{I}^1$, as in Definition \ref{dfn:41}. 
Let $\sigma : \mbf{I}^1 \to X$ be the piecewise smooth map such that
$\sigma \circ \sigma^1_i = \sigma_i$
for $i = 1, 2$. And let $\alpha' := \sigma^*(\alpha)$. 
Now by definition of $*$ we have
$\sigma = \sigma_1 * \sigma_2$ as paths in $X$. Hence
\[ \opn{MI}(\alpha \vert \sigma_1 * \sigma_2) = 
\opn{MI}(\alpha \vert \sigma) = \opn{MI}(\alpha' \vert \mbf{I}^1) . \]
Likewise
\[ \opn{MI}(\alpha \vert \sigma_i) = 
\opn{MI}(\alpha \vert \sigma \circ \sigma^1_i) = 
\opn{MI}(\alpha' \vert \sigma^1_i) . \]
On the other hand, by Proposition \ref{prop:5}(3) we know that
\[ \opn{MI}(\alpha' \vert \mbf{I}^1) = 
\opn{MI}(\alpha' \vert \sigma^1_1) \cdot \opn{MI}(\alpha' \vert \sigma^1_2)  .
\]

\medskip \noindent
(2) This is immediate from Proposition \ref{prop:1}(2).

\medskip \noindent
(3) This follows part (1) and induction on the number of pieces in the string
$\sigma'$.
\end{proof}

In order to define MI on triangular kites we shall need the following geometric
construction. Consider the linear triangular kite 
$(\sigma', \tau')$ in $(\mbf{I}^2, v_0)$, where the string $\sigma'$ has one
linear piece $\sigma' : \mbf{I}^1 \to \mbf{I}^2$ defined on vertices by
\[ \sigma'(v_0, v_1) := (v_0, (\smfrac{1}{2}, \smfrac{1}{2}) ) . \]
The linear map 
$\tau' : \bsym{\Delta}^2 \to \mbf{I}^2$ is defined on vertices by
\[ \tau'(v_0, v_1, v_2) := ( (\smfrac{1}{2}, \smfrac{1}{2}),
(1, \smfrac{1}{2}), (\smfrac{1}{2}, 1) ) . \]
We also need the linear quadrangular kite 
$(\sigma', \tau'')$ in $(\mbf{I}^2, v_0)$, where the linear map 
$\tau'' : \mbf{I}^2 \to \mbf{I}^2$ is defined on vertices by
\[ \tau''(v_0, v_1, v_2) := ( (\smfrac{1}{2}, \smfrac{1}{2}),
(1, \smfrac{1}{2}), (\smfrac{1}{2}, 1) ) . \]
We write
$Z := \sigma'(\mbf{I}^1)$ and $Y := \tau'(\bsym{\Delta}^2)$. 

Consider the  canonical  linear embedding 
$\bsym{\Delta}^2 \to  \mbf{I}^2$ 
which on vertices is
\[ (v_0, v_1, v_2) \mapsto (v_0, v_1, v_2) . \]
Let $h : \mbf{I}^2 \to \bsym{\Delta}^2$ be the piecewise linear retraction
which is linear on $\bsym{\Delta}^2$ and on the the triangle complementary to
it, and satisfies $h(1, 1) = v_1$. See Figure \ref{fig:55}.

\begin{lem}
There exists a piecewise linear retraction
$g : \mbf{I}^2 \to Y \cup Z$, such that
\[ g \circ \tau'' = \tau' \circ h \]
as piecewise linear maps $\mbf{I}^2 \to \mbf{I}^2$.
\end{lem}

\begin{proof}
Easy exercise. Cf.\ proof of Proposition \ref{prop:14}. And see Figure
\ref{fig:84}.
\end{proof}

\begin{figure} 
\includegraphics[scale=0.23]{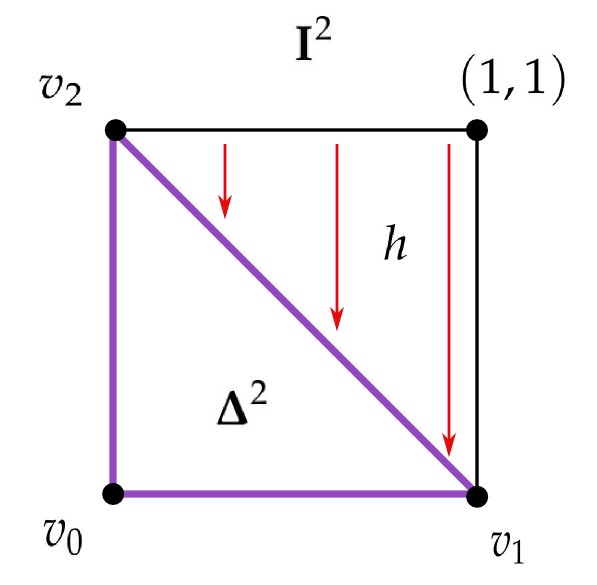}
\caption{The piecewise linear retraction $h : \mbf{I}^2 \to \bsym{\Delta}^2$.} 
\label{fig:55}
\end{figure}

\begin{figure} 
\includegraphics[scale=0.29]{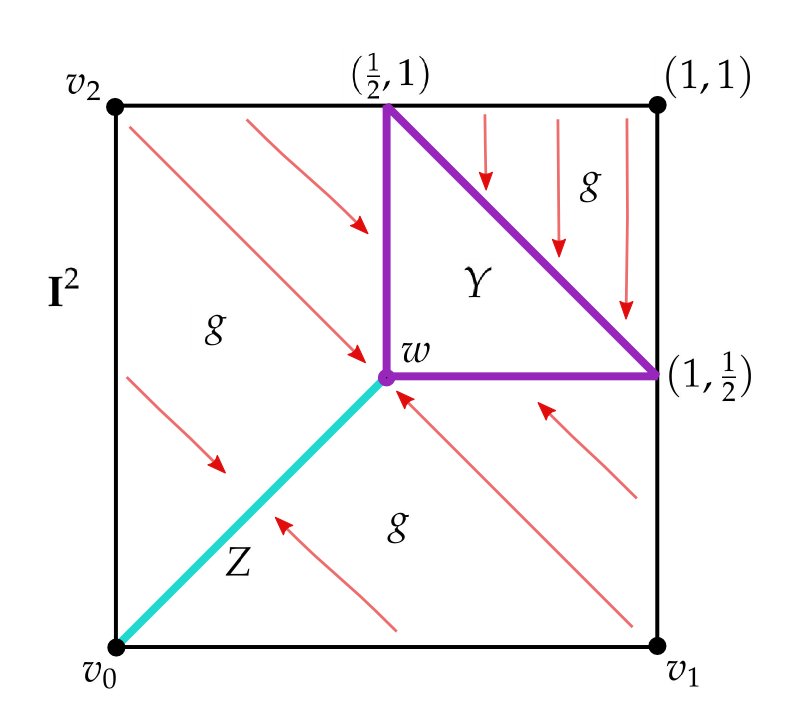}
\caption{The piecewise linear retraction $g : \mbf{I}^2 \to Y \cup Z$.
Here $w := (\smfrac{1}{2}, \smfrac{1}{2})$.} 
\label{fig:84}
\end{figure}

\begin{figure} 
\includegraphics[scale=0.35]{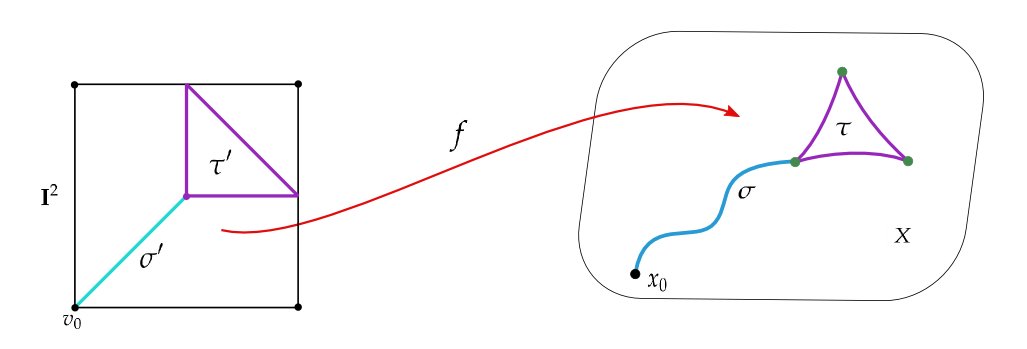}
\caption{The piecewise smooth map $f : \mbf{I}^2 \to X$
such that 
$f \circ (\sigma', \tau') = (\sigma, \tau)$.} 
\label{fig:85}
\end{figure}

Let us fix such a retraction $g$, which we also view as a piecewise linear map
$g : \mbf{I}^2 \to \mbf{I}^2$. 

\begin{dfn} \label{dfn:42}
\index{Multiplicative integral on triangular kites}
Let $\alpha \in \Omega^1(X) \otimes \g$ and 
$\beta \in \Omega^1(X) \otimes \h$. Given a piecewise smooth triangular kite 
$(\sigma, \tau)$ in $(X, x_0)$ we define its {\em multiplicative integral}
\[  \opn{MI}(\alpha, \beta \vert \sigma, \tau) \in H \]
as follows. Let $f : \mbf{I}^2 \to X$ be the unique piecewise smooth map such
that:
\begin{itemize}
\item $f \circ \sigma'  = \sigma$ as maps
$\mbf{I}^1 \to X$.
\item $f \circ \tau'  = \tau$ as maps
$\bsym{\Delta}^2 \to X$.
\item $f = f \circ g$ as maps $\mbf{I}^2 \to X$, where 
$g : \mbf{I}^2 \to \mbf{I}^2$ is the chosen retraction.
\end{itemize}
(See Figure \ref{fig:85}.) 
We get differential forms
\[ \alpha' := f^*(\alpha) \in 
\Omega^1_{\mrm{pws}}(\mbf{I}^2) \otimes \g \]
and
\[ \beta' := f^*(\beta) \in 
\Omega^2_{\mrm{pws}}(\mbf{I}^2) \otimes \h . \]
We define
\[ \opn{MI}(\alpha, \beta \vert \sigma, \tau) :=
\opn{MI}(\alpha', \beta' \vert \sigma', \tau'') , \]
where $(\sigma', \tau'')$ is the linear quadrangular kite in
$(\mbf{I}^2, v_0)$ defined above, and \lb
$\opn{MI}(\alpha', \beta' \vert \sigma', \tau'')$
is the multiplicative integral from Definition \ref{dfn:16}.
\end{dfn}

This definition might seem strange; but we shall soon see that it has all the
expected good properties.

\begin{prop}[Functoriality in $X$]
Let $e : (Y, y_0) \to (X, x_0)$ be a map of pointed manifolds, let
$\alpha \in \Omega^1(X) \otimes \g$, let
$\beta \in \Omega^2(X) \otimes \h$, 
and let $(\sigma, \tau)$ be a piecewise smooth triangular kite in $(Y, y_0)$.
Then
\[ \opn{MI}(\alpha, \beta \vert e \circ \sigma, e \circ \tau) = 
\opn{MI} \bigl( e^*(\alpha), e^*(\beta) \vert  \sigma, \tau \bigl) , \]
where the latter is calculated with respect to the Lie quasi crossed module
with additive feedback $e^*(\mbf{C} / X)$.
\end{prop}

\begin{proof}
This is immediate from the definition, since 
$(e \circ f)^* = f^* \circ e^*$ for a piecewise smooth map 
$f : \mbf{I}^2 \to Y$.
\end{proof}

\begin{prop}[Comparison to Quadrangular Kites] \label{prop:24}
Let $(Z, z_0)$ be a pointed polyhedron, 
$e : (Z, z_0) \to (X, x_0)$ a piecewise smooth map preserving base points, 
$(\sigma, \tau)$ a piecewise smooth triangular kite in $(X, x_0)$, 
$(\sigma', \tau')$ a  linear quadrangular kite in $(Z, z_0)$, 
$\alpha \in \Omega^1(X) \otimes \g$ and $\beta \in \Omega^2(X) \otimes \h$.
Assume that $\sigma = e \circ \sigma'$ as paths in $X$, and 
$\tau \circ h = e \circ \tau'$ as maps $\mbf{I}^2 \to X$, where
$h : \mbf{I}^2 \to \bsym{\Delta}^2$ is the retraction in
Figure \tup{\ref{fig:55}}.
Then
\[ \opn{MI}(\alpha, \beta \vert \sigma, \tau) =
\opn{MI}(e^*(\alpha), e^*(\beta) \vert \sigma', \tau') , \]
where the latter is calculated with respect to the Lie quasi crossed module
with additive feedback $e^*(\mbf{C} / X)$.
\end{prop}

\begin{proof}
The piecewise smooth map $f : \mbf{I}^2 \to X$ used in Definition \ref{dfn:42}
can be factored as $f = e \circ f'$ for a piecewise linear map 
$f' : (\mbf{I}^2, v_0) \to (Z, z_0)$ of pointed polyhedra.
The assertion now follows from Proposition \ref{prop:10}.
\end{proof}

The notion of compatible connection from Definition \ref{dfn:8} makes sense
here too, only we have to consider 
$\alpha \in \Omega^1(X) \otimes \g$
(and replace ``string'' with ``piecewise smooth path''). 

\begin{prop}[Moving the Base Point] \label{prop:25}
Let $\alpha \in \Omega^1(X) \otimes \g$
be a connection compatible with $\mbf{C} / X$, and let
$\rho$ be a piecewise smooth path in $X$, with initial point $x_0$ and terminal
point $x_1$. Define
$g := \opn{MI}(\alpha \vert \rho)$, 
and let
\[ \mbf{C}^g / X = (G, H^g, \Psi, \Phi_0^g, \Phi_X) \]
be the Lie quasi crossed module with additive feedback over $(X, x_1)$
constructed in Subsection \tup{\ref{subsec:moving}}.
Given a form 
$\beta \in \Omega^2(X) \otimes \h$ and a piecewise smooth triangular kite
$(\sigma, \tau)$ in  $(X, x_1)$, consider the element
\[ \opn{MI}^g(\alpha, \beta \vert \sigma, \tau) \in H^g  \]
calculated with respect to the Lie quasi crossed module
with additive feedback $\mbf{C}^g / X$. Then 
\[ \Psi(g) \bigl( \opn{MI}^g(\alpha, \beta \vert \sigma, \tau) \bigr)  =
\opn{MI}(\alpha, \beta \vert \rho * \sigma, \tau)  \]
in $H$. 
\end{prop}

\begin{proof}
It is possible to find a polyhedron $Z$, points  $z_0, z_1 \in Z$, a string
$\rho'$ in $Z$ with initial point $z_0$ and terminal point $z_1$, 
a linear quadrangular kite $(\sigma', \tau')$ in 
$(Z, z_1)$ and a piecewise smooth map 
$e : Z \to X$, such that $\rho = e \circ \rho'$ as paths and 
$(\sigma, \tau) = e \circ (\sigma', \tau')$ as kites. 
Define
$\alpha' := e^*(\alpha)$ and $\beta' := e^*(\beta)$.

Let $\mbf{C}' / Z := e^*(\mbf{C} / X)$ be the induced 
Lie quasi crossed module with additive feedback over the pointed polyhedron
$(Z, z_0)$. By Proposition \ref{prop:23}(3) we know that 
$\opn{MI}(\alpha' \vert \rho') = g$. Examining the definition of
$\mbf{C}^g / X$ in Subsection \ref{subsec:moving} we see that
\[ e^*(\mbf{C}^g / X) = (\mbf{C}')^g / Z \]
as Lie quasi crossed modules with additive feedback over the pointed polyhedron
$(Z, z_1)$. Therefore 
$\opn{MI}^g(\alpha', \beta' \vert \sigma', \tau')$
is unambiguous. 

By Proposition \ref{prop:24} we know that
\[ \opn{MI}^g(\alpha', \beta' \vert \sigma', \tau') = 
\opn{MI}^g(\alpha, \beta \vert \sigma, \tau ) \]
and
\[ \opn{MI}(\alpha', \beta' \vert \rho' * \sigma', \tau') = 
\opn{MI}(\alpha, \beta \vert  \rho * \sigma, \tau ) . \]
Finally, by Theorem \ref{thm:14} we have
\[ \opn{MI}^g(\alpha', \beta' \vert \sigma', \tau') = 
\opn{MI}(\alpha', \beta' \vert \rho' * \sigma', \tau') . \]
\end{proof}

\subsection{Stokes Theorems}
We continue with the setup from before: $(X, x_0)$ is a pointed manifold, and
\[ \mbf{C} / X = (G, H, \Psi, \Phi_0, \Phi_X) \]
be a Lie quasi crossed module with additive feedback. 

Just like in Definition \ref{dfn:19}, a {\em connection-curvature pair} for
$\mbf{C} / X$ is a pair $(\alpha, \beta)$, with 
$\alpha \in \Omega^1(X) \otimes \g$ and 
$\beta \in \Omega^1(X) \otimes \h$,
such that $\alpha$ is a compatible connection, and  the differential equation
\[ \Phi_X(\beta) = \d(\alpha) + \smfrac{1}{2} [\alpha, \alpha] \]
holds in $\Omega^2(X) \otimes \g$.

Recall that for a piecewise smooth triangular kite $(\sigma, \tau)$, its
boundary
$\partial (\sigma, \tau)$ was defined in Definition \ref{dfn:31}.

\begin{thm}[Stokes Theorem for Triangles] \label{thm:520}
\index{Nonabelian Stokes Theorem for triangles}
Let $(X, x_0)$ be a pointed manifold, let
\[ \mbf{C} / X := (G, H, \Psi, \Phi_0, \Phi_X) \]
be a Lie quasi crossed module with additive feedback over $(X, x_0)$,
let $(\alpha, \beta)$ be a \lb connection-curvature pair for
$\mbf{C} / X$, and let
$(\sigma, \tau)$ be a piecewise smooth triangular kite in $(X, x_0)$. Then
\[ \Phi_0 \bigl( \opn{MI}(\alpha, \beta \vert \sigma, \tau) \bigr) = 
\opn{MI}(\alpha, \beta \vert \partial (\sigma, \tau)) . \]
\end{thm}

\begin{proof}
As in the proof of Proposition \ref{prop:25}, we can find a pointed
polyhedron $(Z, z_0)$, a linear quadrangular kite $(\sigma', \tau')$ in 
$(Z, z_0)$ and a piecewise smooth map $f : Z \to X$, such that  
$(\sigma, \tau) = f \circ (\sigma', \tau')$ as kites. 
Define
$\alpha' := f^*(\alpha)$ and $\beta' := f^*(\beta)$.
According to Proposition \ref{prop:23}(3) we have
\[ \opn{MI}(\alpha, \beta \vert \partial (\sigma, \tau)) = 
\opn{MI}(\alpha' \vert \partial (\sigma', \tau')) . \]
And by Proposition \ref{prop:24} we have
\[ \opn{MI}(\alpha, \beta \vert \sigma, \tau)  = 
\opn{MI}(\alpha', \beta' \vert \sigma', \tau')  . \]
Finally by Theorem \ref{thm:1} we know that
\[ \opn{MI}(\alpha' \vert \partial (\sigma', \tau')) = 
\Phi_0 \bigl( \opn{MI}(\alpha', \beta' \vert \sigma', \tau') \bigr) . \]
\end{proof}

For the $3$-dimensional Stokes Theorem we shall need an auxiliary  geometric
construction, similar to the one in the
previous subsection. Consider the linear triangular balloon  
$(\sigma', \tau')$ in $(\mbf{I}^3, v_0)$, where the string $\sigma'$ has one
linear piece $\sigma' : \mbf{I}^1 \to \mbf{I}^3$ defined on vertices by
\[ \sigma'(v_0, v_1) := \bigl( v_0, (\smfrac{1}{2}, \smfrac{1}{2}, 0) \bigr) .
\]
The linear map 
$\tau' : \bsym{\Delta}^3 \to \mbf{I}^3$ is defined on vertices by
\[ \tau'(v_0, v_1, v_2, v_3) := \bigl( (\smfrac{1}{2}, \smfrac{1}{2}, 0),
(1, \smfrac{1}{2}, 0), (\smfrac{1}{2}, 1, 0), 
(\smfrac{1}{2}, \smfrac{1}{2}, \smfrac{1}{2}) \bigr) . \]
We also need the linear quadrangular balloon 
$(\sigma', \tau'')$ in $(\mbf{I}^3, v_0)$, where the linear map 
$\tau'' : \mbf{I}^3 \to \mbf{I}^3$ is defined on vertices by the same formula
as $\tau'$. We let 
$Z := \sigma'(\mbf{I}^1)$ and $Y := \tau'(\bsym{\Delta}^3)$.
See Figure \ref{fig:87}.

\begin{figure} 
\includegraphics[scale=0.44]{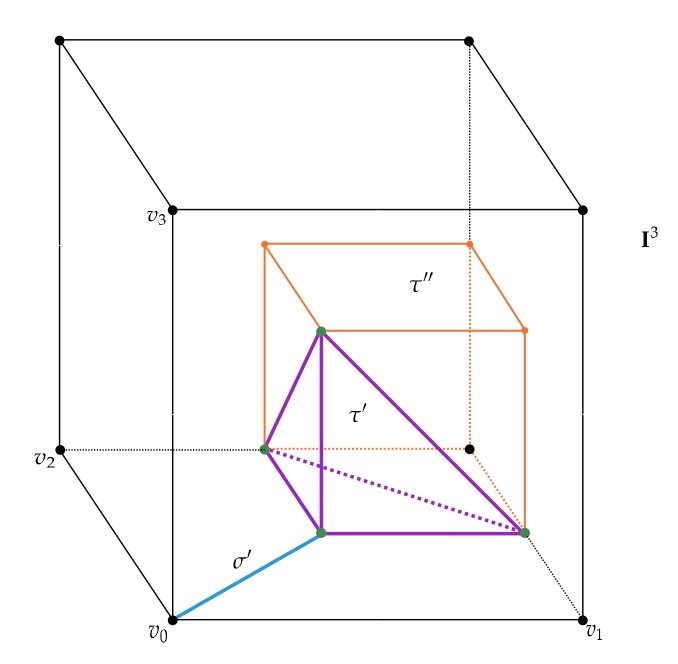}
\caption{The triangular linear balloon $(\sigma', \tau')$ and the 
quadrangular linear balloon $(\sigma', \tau'')$ in the 
pointed polyhedron $(\mbf{I}^3, v_0)$.} 
\label{fig:87}
\end{figure}
 
Consider the canonical linear embedding 
$\bsym{\Delta}^3 \to  \mbf{I}^3$ 
which on vertices is
\[ (v_0, \ldots, v_3) \mapsto (v_0, \ldots, v_3) . \]
Let $h : \mbf{I}^3 \to \bsym{\Delta}^3$ be the piecewise linear retraction
$h := h_2 \circ h_1$, where $h_1$ and $h_2$ are the retractions shown in Figure
\ref{fig:57}.

\begin{figure} 
\includegraphics[scale=0.32]{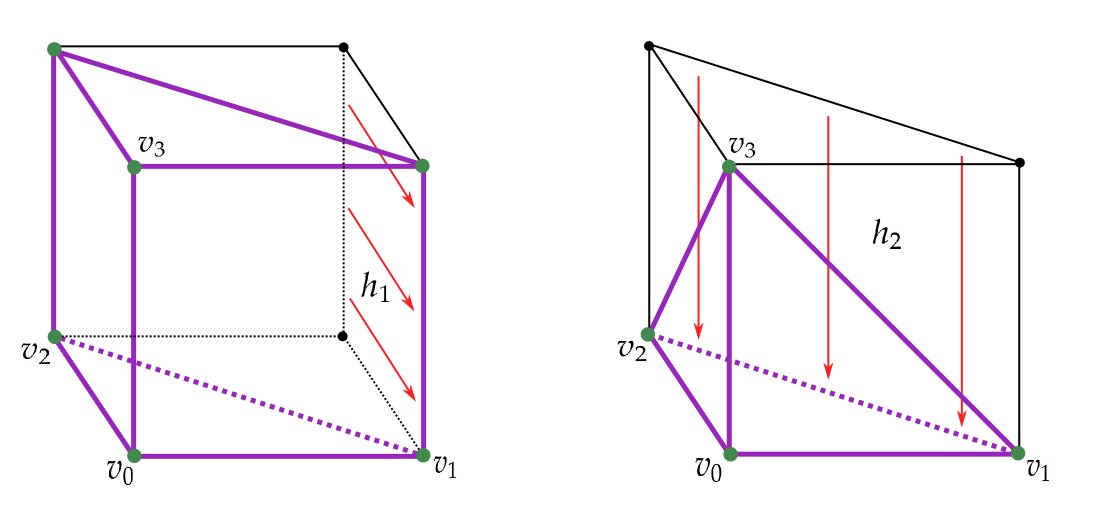}
\caption{The piecewise linear map $h_1$ retracts the cube $\mbf{I}^3$ to a
prism. The piecewise linear map $h_2$ retracts the prism to the tetrahedron
$\bsym{\Delta}^3$.} 
\label{fig:57}
\end{figure}

\begin{lem}
There is a piecewise linear retraction
$g : \mbf{I}^3 \to Z \cup Y$, such that
\[ g \circ \tau'' = \tau' \circ h \]
as maps $\mbf{I}^3 \to \mbf{I}^3$.  
\end{lem}

\begin{proof}
Nice exercise in piecewise linear geometry.
\end{proof}

Let us fix such a retraction$g$, which we also view as a piecewise linear map
$g : \mbf{I}^3 \to \mbf{I}^3$.

The notions of {\em tame connection} and 
{\em inert differential form}, from Subsections \ref{subsec:ccpair}
\ref{subsec:inert} respectively, make sense here, and all pertinent results
hold. Given a tame connection 
$\alpha \in \Omega^1(X) \otimes \g$, 
an inert form $\gamma \in \Omega^3(X) \otimes \h$
and  a piecewise smooth balloon $(\sigma, \tau)$ in $(X, x_0)$, we define the
{\em twisted
multiplicative integral}
\[ \opn{MI}(\alpha, \gamma \vert \sigma, \tau) \in H_0  \] 
as follows. Let $f : \mbf{I}^3 \to X$ be the unique piecewise smooth map
satisfying:
\begin{itemize}
\item $f \circ \sigma'  = \sigma$ as maps
$\mbf{I}^1 \to X$.
\item $f \circ \tau'  = \tau$ as maps
$\bsym{\Delta}^3 \to X$.
\item $f = f \circ g$ as maps $\mbf{I}^3 \to X$, where 
$g: \mbf{I}^3 \to \mbf{I}^3$ is the chosen retraction.
\end{itemize}
Now 
\[ \mbf{C}' / \mbf{I}^3 := f^*(\mbf{C} / X) = (G, H, \Psi, \Phi_0, f^*(\Phi_X))
\]
is a Lie quasi crossed module with additive feedback
over $(\mbf{I}^3, v_0)$,
\[ \alpha' := f^*(\alpha) \in \Omega^1_{\mrm{pws}}(\mbf{I}^3) \otimes \g \]
is a tame connection for $\mbf{C}' / X$, and 
\[ \gamma' := f^*(\gamma) \in \Omega^3_{\mrm{pws}}(\mbf{I}^3) \otimes \h \]
is an inert differential form. We define
\begin{equation} \label{eqn:161}
\opn{MI}(\alpha, \gamma \vert \sigma, \tau) :=
\opn{MI} \bigl( \alpha', \gamma' \vert \tau'') ,
\end{equation}
where $\tau'' : \mbf{I}^3 \to \mbf{I}^3$ is the linear map described above,
and 
$\opn{MI} \bigl( \alpha', \gamma' \vert \tau'')$
is the twisted multiplicative integral from Definition \ref{dfn:43}.

\begin{rem}
The reason that the path $\sigma$ is explicit in the expression \lb
$\opn{MI}(\alpha, \gamma \vert \sigma, \tau)$ 
is that this operation could depend on the homotopy class of $\sigma$, in case
the manifold $X$ is not simply connected. Since polyhedra are always simply
connected, we did not have to worry about strings in Definition \ref{dfn:43}.

It is easy to see that $\opn{MI}(\alpha, \gamma \vert \sigma, \tau)$ 
makes sense for an inert $p$-form $\gamma$ and a ``$p$-dimensional balloon''
$(\sigma, \tau)$, for any $p \geq 0$. Moreover, for $p = 2$ this coincides with
the nonabelian MI of Definition \ref{dfn:42}.
\end{rem}

Recall the boundary $\partial (\sigma, \tau)$ of a balloon 
$(\sigma, \tau)$ in $(X, x_0)$, from Definition \ref{dfn:30}.
As is our convention for the multiplicative integral of a sequence, we write
\[ \opn{MI}(\alpha, \beta \vert \partial (\sigma, \tau)) :=
\prod_{i = 1, \ldots, 4} \ 
\opn{MI}(\alpha, \beta \vert \partial_i(\sigma, \tau) ) . \]

As in Definition \ref{dfn:32}, the $3$-curvature of $(\alpha, \beta)$ is
the form
\[ \gamma := \d(\beta) + \psi_{\h}(\alpha)(\beta) \in 
\Omega^3(X) \otimes \h . \]

Here is the second version of the main result of the paper (the first version
was Theorem \ref{thm:10}.)

\begin{thm}[Stokes Theorem for Tetrahedra] \label{thm:22}
\index{Nonabelian Stokes Theorem for tetrahedra}
Let $(X, x_0)$ be a pointed manifold, let
$\mbf{C} / X$ be a Lie quasi crossed module with additive feedback, 
let $(\alpha, \beta)$ be a connection-curvature pair for 
$\mbf{C} / X$, and let $\gamma$ be the $3$-curvature of $(\alpha, \beta)$.
Then:
\begin{enumerate}
\item The differential form $\gamma$ is inert. 
\item For any piecewise smooth balloon $(\sigma, \tau)$ in $(X, x_0)$ one has
\[  \opn{MI} \bigl( \alpha, \beta \vert \partial (\sigma, \tau)
\bigr) = 
\opn{MI} (\alpha, \gamma \vert \sigma, \tau)  \]
in $H$.
\end{enumerate}
\end{thm}

\begin{proof}
(1) It suffices to prove that for any piecewise smooth map 
$f : \mbf{I}^3 \to X$ the form
\[ \gamma' := f^*(\gamma) \in \Omega^3_{\mrm{pws}}(\mbf{I}^3) \otimes \h \]
is inert. Define
\[ \alpha' := f^*(\alpha) \in \Omega^1_{\mrm{pws}}(\mbf{I}^3) \otimes \g \]
and
\[ \beta' := f^*(\beta) \in \Omega^2_{\mrm{pws}}(\mbf{I}^3) \otimes \h . \]
Then $(\alpha', \beta')$ is a connection curvature pair in 
$\mbf{C}' / \mbf{I}^3$ (this is a variant of Proposition \ref{prop:11}(1)), and
$\gamma'$ is the $3$-curvature of $(\alpha', \beta')$.
But according to Theorem \ref{thm:10}(1) the form $\gamma'$ is inert.

\medskip \noindent
(2) Let $f : (\mbf{I}^3, v_0) \to (X, x_0)$ be the pointed piecewise smooth map
constructed just before (\ref{eqn:161}). 
By definition we have
\[ \opn{MI}(\alpha, \gamma \vert \sigma, \tau) :=
\opn{MI} \bigl( \alpha', \gamma' \vert \tau'') . \]
By Proposition \ref{prop:24} we know that
\[ \opn{MI} \bigl( \alpha, \beta \vert \partial (\sigma, \tau) \bigr) = 
\opn{MI} \bigl( \alpha', \beta' \vert \partial (\sigma', \tau') \bigr) . \]
And by Theorem \ref{thm:10}(2) we know that
\[ \opn{MI} \bigl( \alpha', \beta' \vert \partial (\sigma', \tau'') \bigr) =
\opn{MI} \bigl( \alpha', \gamma' \vert \tau'') . \]
It remains to prove that
\begin{equation} \label{eqn:228}
\opn{MI} \bigl( \alpha', \beta' \vert \partial (\sigma', \tau') \bigr) = 
\opn{MI} \bigl( \alpha', \beta' \vert \partial (\sigma', \tau'') \bigr) .
\end{equation}

By definition we have
\[ \opn{MI} \bigl( \alpha', \beta' \vert \partial (\sigma', \tau') \bigr) = 
\prod_{i = 1, \ldots, 4} \ 
\opn{MI}(\alpha, \beta \vert \partial_i(\sigma', \tau') )  \]
and
\[ \opn{MI} \bigl( \alpha', \beta' \vert \partial (\sigma', \tau'') \bigr) = 
\prod_{i = 1, \ldots, 6} \ 
\opn{MI}(\alpha, \beta \vert \partial_i(\sigma', \tau'') ) . \]
Using Proposition \ref{prop:24} and looking at Figures
\ref{fig:87}, \ref{fig:83} and \ref{fig:46} we see that
\[ \opn{MI} \bigl( \alpha', \beta' \vert \partial_1(\sigma', \tau'') \bigr) = 1 
, \]
\[ \opn{MI} \bigl( \alpha', \beta' \vert \partial_2(\sigma', \tau'') \bigr) = 1
, \]
\[ \opn{MI} \bigl( \alpha', \beta' \vert \partial_3(\sigma', \tau'')  \bigr) = 
\opn{MI} \bigl( \alpha', \beta' \vert \partial_1(\sigma', \tau')  \bigr)  , \]
\[ \opn{MI} \bigl( \alpha', \beta' \vert \partial_4(\sigma', \tau'')  \bigr) = 
\opn{MI} \bigl( \alpha', \beta' \vert \partial_2(\sigma', \tau')  \bigr)  , \]
\[ \opn{MI} \bigl( \alpha', \beta' \vert \partial_5(\sigma', \tau'')  \bigr) = 
\opn{MI} \bigl( \alpha', \beta' \vert \partial_3(\sigma', \tau')  \bigr)   \]
and
\[ \opn{MI} \bigl( \alpha', \beta' \vert \partial_6(\sigma', \tau'')  \bigr) = 
\opn{MI} \bigl( \alpha', \beta' \vert \partial_4(\sigma', \tau')  \bigr) .  \]
Thus equation (\ref{eqn:228}) is true. 
\end{proof}

\subsection{Rationality of the Multiplicative Integral} \label{subsec:ration}
Let $\K$ be a subfield of $\R$. As common in algebraic geometry, let us denote
by $\mbf{A}^n(\K)$ the set of $\K$-rational points of $\mbf{A}^n(\R)$; namely
those points $x \in \mbf{A}^n(\R)$ whose coordinates satisfy $t_i(x) \in \K$,
$i = 1, \ldots, n$. 

By (embedded) {\em polyhedron defined over $\K$} we mean a polyhedron 
$X \subset \mbf{A}^n(\R)$, such that all the
vertices of $X$ belong to $\mbf{A}^n(\K)$. 
For such $X$, and for every field $L$ such that $\K \subset L \subset \R$, we
can talk about the set $X(L)$ of $L$-rational points of $X$.
The real polyhedron is of course $X(\R)$.

Suppose $Y$ is another polyhedron defined over $\K$. It makes sense to talk
about linear maps $f : X \to Y$ defined over $\K$. The condition is of course
that for any vertex $v \in X$ the point $f(v)$ is in $Y(\K)$. 

The simplices $\bsym{\Delta}^p$ are defined over $\mbb{Q}$. Therefore we have
the notion of linear triangulation of $X$ defined over $\K$. This lets 
us consider piecewise linear maps $f : Y \to X$ defined over $\K$.
In particular we have {\em kites in $X$ defined over
$\K$}. 

The polyhedron $X$ has a $\K$-subalgebra 
$\mcal{O}_{\mrm{alg}}(X(\K)) \subset \mcal{O}(X)$
of {\em $\K$-valued algebraic functions on $X(\K)$}, which is the restriction
to  $X(\K)$ of the polynomial algebra 
$\K[t_1, \ldots, t_n] \subset \mcal{O}(\mbf{A}^n(\R))$. 
Similarly one can define the DG algebra
$\Omega_{\mrm{alg}}(X(\K))$ of {\em $\K$-valued algebraic differential forms},
which is contained in $\Omega(X)$.
Using linear triangulations defined over $\K$ we can also consider the 
$\K$-algebra $\mcal{O}_{\mrm{pwa}}(X(\K))$ of {\em $\K$-valued piecewise
algebraic functions}, and the DG
algebra $\Omega_{\mrm{pwa}}(X(\K))$ of {\em $\K$-valued piecewise algebraic
differential forms}, which is a DG subalgebra of $\Omega_{\mrm{pws}}(X)$.

Suppose $X$ is a polyhedron and $Y \subset \mbf{A}^m$ is an
affine algebraic variety, both defined over $\K$. 
A map $f : X \to Y$ is said to be algebraic and defined over $\K$ if
for every $i$ the function $t_i \circ f : X(\K) \to \R$ belongs to
$\mcal{O}_{\mrm{alg}}(X(\K))$. Equivalently, $f$ extends to a map of algebraic
varieties $\til{f} : \mbf{A}^n \to \mbf{A}^m$ that's defined over $\K$.
Using linear triangulations defined over $\K$ we can also consider 
piecewise algebraic maps $f : X \to Y$ defined over $\K$. 

Now assume that we are given a Lie quasi crossed module with additive feedback
\[ \mbf{C}  = (G, H, \Psi, \Phi_0, \Phi_X) \]
in which $(X, x_0)$ is a pointed polyhedron defined over $\K$, 
the groups $G$ and $H$ are {\em affine unipotent} linear algebraic
groups defined over $\K$, the maps
$\Psi : G \times H \to H$ and $\Phi_0 : H \to G$ are maps of algebraic
varieties defined over $\K$, and 
\[ \Phi_X \in \mcal{O}_{\mrm{pwa}}(X(\K)) \otimes_{\K} 
\opn{Hom}_{\K}(\h(\K), \g(\K)). \]
We say that $\mbf{C} / X$ is an {\em algebraic unipotent quasi crossed
module with additive feedback defined over $\K$}. 
Note that the Lie groups in this setup are $G(\R)$ and
$H(\R)$. Also note that exponential maps $\exp_G : \g \to G$ and
$\exp_H : \h \to H$ are isomorphisms of algebraic varieties defined over $\K$
-- this is because the Lie algebras $\h$ and $\g$ are nilpotent.

The polyhedron $\mbf{I}^1$ is defined over $\mbb{Q}$. 
Suppose we are given a piecewise algebraic differential form
\[ \alpha \in \Omega^p_{\mrm{pwa}}(\mbf{I}^1(\K)) \otimes_{\K} \g(\K) . \]
It is not hard to show (using the ODE (\ref{eqn:225}), and the fact that 
$\exp_G$ is algebraic) that the map
$g : \mbf{I}^1 \to G$  of (\ref{eqn:252}) is piecewise algebraic and defined
over $\K$.  In particular we get
\[ g(1) = \opn{MI}(\alpha \vert \mbf{I}^1) \in G(\K) . \]
From this it follows that for any polyhedron $X$ defined over $\K$, any
\[ \alpha \in \Omega^1_{\mrm{alg}}(X(\K)) \otimes_{\K} \g(\K)  \]
and any piecewise algebraic string $\sigma$ in $X$ defined over $\K$, one has
\[ \opn{MI}(\alpha \vert \sigma) \in G(\K) . \]

We are led to make the following conjecture. A connection-curvature pairs
$(\alpha, \beta)$ is said to be algebraic if 
$\alpha \in \Omega^1_{\mrm{alg}}(X(\K)) \otimes_{\K} \g(\K)$ and
$\alpha \in \Omega^2_{\mrm{alg}}(X(\K)) \otimes_{\K} \h(\K)$.

\begin{conj} \label{conj:520}
Let $\mbf{C} / X$ be an algebraic unipotent quasi crossed
module with additive feedback defined over $\K$, let 
$(\alpha, \beta)$ be an algebraic connection-curvature pair in $\mbf{C} / X$, 
and let $(\sigma, \tau)$ be a piecewise algebraic kite in $(X, x_0)$.
Then 
\[ \opn{MI}(\alpha, \beta \vert \sigma, \tau) \in H(\K) \ . \]
\end{conj}

It seems that the very recent paper \cite{BGNT} may have a proof of this
conjecture in the special case when $\mbf{C} / X$ is a {\em crossed module} and 
{\em the $3$-curvature of $(\alpha, \beta)$ vanishes}. See Subsection
\ref{subsec:PDEs}.

\subsection{A Conjecture on Descent Data} \label{subsec:conj.desc}
Suppose $\f$ is a  cosimplicial nilpotent quantum type DG Lie algebra, just
like in Subsection \ref{subsec:CosDGLie}. Take an MC element 
$\til{\omega} = \{ \til{\omega}^{p, q} \}$ in the
Thom-Sullivan normalization $\til{\mrm{N}}(\f)$.
Then 
\[ \til{\omega}^{0, 0} \in \mcal{O}(\bsym{\Delta}^0) \otimes \f^{0, 1}
= \f^{0, 1} , \] 
\[ \til{\omega}^{1, 1} \in \Omega^1(\bsym{\Delta}^1) \otimes \f^{1, 0} \]
and 
\[ \til{\omega}^{2, 2} \in \Omega^2(\bsym{\Delta}^2) \otimes \f^{2, -1} . \]
Let us write 
\[ \omega := \til{\omega}^{0, 0} \in \f^{0, 1}  \]
(this notation is different from the one on Subsection \ref{subsec:CosDGLie}),
\[ g := \opn{MI}( \til{\omega}^{1, 1} \vert \bsym{\Delta}^1) \in 
\opn{exp}(\f^{1, 0}) \]
and 
\[ a :=  \opn{MI}( \til{\omega}^{2, 1} , \til{\omega}^{2, 2} 
\vert \bsym{\Delta}^2 ) \in 
\opn{exp}(\f^{2, -1})_{\til{\omega}^{2, 0}(v_0)}  . \]
This last multiplicative integral takes place in the Lie quasi crossed module
with additive feedback 
\[  \mbf{C}_{\til{\omega}^{2, 0}(v_0)}   / \bsym{\Delta}^2  =  
\bigl( G, H_{\til{\omega}^{2, 0}(v_0)}, 
\Psi, \Phi_{\til{\omega}^{2, 0}(v_0)}, \Phi_{\bsym{\Delta}^2} \bigr)  \]
as in (\ref{eqn:522}). In this way we assign to each 
$\til{\omega} \in \opn{MC}(\til{\mrm{N}}(\f))$ a triple 
\begin{equation} \label{eqn:523}
\opn{MI}(\til{\omega} \vert \bsym{\Delta}) := (\omega, g, a) .
\end{equation}

On the other hand, for every $p$ the nilpotent quantum type DG Lie algebra 
$\f^{p, \bdot}$ has its {\em Deligne $2$-groupoid} 
$\opn{Del}(\f^{p, \bdot})$, also known as the Deligne crossed groupoid; see 
\cite[Section 6]{Ye5}. The collection 
$\{ \opn{Del}(\f^{p, \bdot}) \}_{p \in \N}$ is then a cosimplicial crossed
groupoid. As explained in \cite[Section 5]{Ye3}, there is the set of
descent data 
$\opn{Desc}(\opn{Del}(\f))$, and its quotient set 
$\ol{\opn{Desc}}(\opn{Del}(\f))$.

Consider an MC element $\til{\omega} \in \opn{MC}(\til{\mrm{N}}(\f))$, 
and the triple 
$\opn{MI}(\til{\omega} \vert \bsym{\Delta}) = (\omega, g, a)$ from
(\ref{eqn:523}). The $2$-dimensional Stokes Theorem 
(Theorem \ref{thm:520}) implies that $(\omega, g, a)$ satisfies the ``failure
of $1$-cocycle'' condition of \cite[Definition 5.2]{Ye3}.
The $3$-dimensional Stokes Theorem (Theorem \ref{thm:22}) implies that 
$(\omega, g, a)$ satisfies the ``twisted $2$-cocycle'' condition of loc.\ cit.
This means that 
$(\omega, g, a) \in \opn{Desc}(\opn{Del}(\f))$.

Recall that the quotient set of $\opn{MC}(\til{\mrm{N}}(\f))$ by the action of
the gauge group $\opn{exp}(\til{\mrm{N}}(\f)^0)$ is denoted by 
$\ol{\opn{MC}}(\til{\mrm{N}}(\f))$.

\begin{conj} \label{conj:521}
Let $\f$ be a cosimplicial nilpotent quantum type DG Lie algebra. Then the
function 
\[ \opn{MI}(- \vert \bsym{\Delta}) : \opn{MC}(\til{\mrm{N}}(\f)) \to
\opn{Desc}(\opn{Del}(\f)) \]
induces a bijection 
\[ \ol{\opn{MC}}(\til{\mrm{N}}(\f)) \to \ol{\opn{Desc}}(\opn{Del}(\f)) . \]
\end{conj}

\cleardoublepage

\cleardoublepage                              
\printindex

\end{document}